\documentclass[11pt]{article}

\usepackage{amsfonts, latexsym, verbatim}
\usepackage{graphics}
\usepackage{slashed}
\usepackage{url}
\usepackage{pifont}
\usepackage{upgreek}
\usepackage{times}
\usepackage{calligra}

\usepackage{hyperref}
\hypersetup{
    colorlinks,
    citecolor=red,
    filecolor=green,
    linkcolor=blue,
    urlcolor=black
}

\usepackage[a4paper,
            bindingoffset=0in,
            left=0.9in,
            right=0.9in,
            top=0.8in,
            bottom=0.8in,
            footskip=0.3in]{geometry}
\usepackage[open,openlevel=1]{bookmark}
\usepackage{pgfplots}
\usepackage{amsmath}
\usepackage{tikz}
\usetikzlibrary{cd}
\usetikzlibrary{arrows}
\usetikzlibrary{decorations.markings}
\usetikzlibrary{arrows.meta}
\usetikzlibrary{plotmarks}
\usetikzlibrary{decorations.pathmorphing}
\usetikzlibrary{decorations.pathreplacing}
\usetikzlibrary{positioning}
\usetikzlibrary{fadings}
\usetikzlibrary{intersections}
\usetikzlibrary{knots}
\usepackage[parfill, skip=.25\baselineskip plus 1pt, indent=20pt]{parskip}
\usepackage[font=small,labelfont=bf, margin=0.3cm]{caption}
\usepackage{changepage}
\usepackage{sidecap}
\usepackage{multirow}
\usepackage{wrapfig}
\usepackage{amssymb}
\usepackage{fancyvrb}
\usepackage{color}
\usepackage[inline,shortlabels]{enumitem}
\usepackage{adjustbox}
\usepackage{amsthm}
\usepackage{tablefootnote}
\usepackage{mathrsfs}
\usepackage{tensor}
\usepackage{xparse}
\usepackage[cal=boondoxo]{mathalfa}

\allowdisplaybreaks

\newcommand{\E}{\mathcal{E}}
\newcommand{\Z} {\mathbb{Z}}
\newcommand{\N} {\mathbb{N}}
\newcommand{\Q} {\mathbb{Q}}
\newcommand{\C} {\mathbb{C}}
\newcommand{\R} {\mathbb{R}}
\newcommand{\F} {\mathbb{F}}
\newcommand{\K} {\mathbb{K}}

\renewcommand{\S}{\mathcal{S}}

\newcommand{\norm}[1]{\left\lVert #1\right\rVert}
\newcommand{\abs}[1]{\left\lvert #1\right\rvert}
\newcommand{\brac}[1]{\left( #1\right)}
\newcommand{\sbrac}[1]{\left[ #1\right]}
\newcommand{\inner}[1]{\left\< #1\right\>}
\newcommand{\set}[1]{\left\{#1\right\}}
\newcommand{\bs}[1]{\boldsymbol{\mathbf{#1}}}

\newcommand{\eva}[1]{\left.#1\right|}

\newcommand{\cmmnt}[1]{\ignorespaces}

\newcommand\Item[1][]{%
  \vspace{2mm}
  \ifx\relax#1\relax  \item \else \item[#1] \fi
  \abovedisplayskip=0pt\abovedisplayshortskip=0pt~\vspace*{-\baselineskip}\vspace*{-1.5mm}}
  
  \newcommand\iitem[1][]{%
  \vspace{2mm}
  \ifx\relax#1\relax  \item \else \item[#1] \fi
  \abovedisplayskip=0pt\abovedisplayshortskip=0pt~\vspace*{-\baselineskip}}

\setlist[enumerate,1]{label=\roman*., leftmargin=8mm}
\setlist[itemize,1]{label=\textbullet, leftmargin=6mm}

\DeclareMathOperator{\tr}{tr}

\DeclareMathOperator{\curl}{curl}
\DeclareMathOperator{\Dom}{Dom}

\newcommand{\mb}{\mathbf}

\newcommand{\ph}{\,\cdot\,}

\newcommand{\<}{\langle}
\renewcommand{\>}{\rangle}

\newcommand{\loc} {\textup{loc}}

\let\div\undefined
\DeclareMathOperator{\div}{div}
\let\curl\undefined
\DeclareMathOperator{\curl}{curl}
\newcommand{\lpc}{\Delta}
\newcommand{\grad}{\nabla}

\renewcommand{\d}{\textup{d}}

\newcommand{\beq}{\begin{equation}}
\newcommand{\eeq}{\end{equation}}
\newcommand{\beqs}{\begin{equation*}}
\newcommand{\eeqs}{\end{equation*}}
\newcommand{\beqa}{\begin{equation}\begin{aligned}}
\newcommand{\eeqa}{\end{aligned}\end{equation}}
\newcommand{\beqas}{\begin{equation*}\begin{aligned}}
\newcommand{\eeqas}{\end{aligned}\end{equation*}}

\def\be{\begin{equation}}
\def\ee{\end{equation}}
\def\bea{\begin{eqnarray}}
\def\eea{\end{eqnarray}}
\def\beas{\begin{eqnarray*}}
\def\eeas{\end{eqnarray*}}

\def\pa{\partial }

\def\lv{\left\vert}
\def\rv{\right\vert}

\def\bcr{\begin{color}{red}}
\def\bcb{\begin{color}{blue}}
\def\ec{\end{color}}

\def\tr{\tilde\rho}

\def\bcr{\begin{color}{red}}
\def\ec{\end{color}}
\def\l{\lambda}

\newtheorem{theorem}{Theorem}[section]
\newtheorem{definition}[theorem]{Definition}
\newtheorem{proposition}[theorem]{Proposition}

\newtheorem{corollary}[theorem]{Corollary}
\newtheorem{lemma}[theorem]{Lemma}
\newtheorem{remark}[theorem]{Remark}

\title{Nonradial stability of expanding Goldreich-Weber stars}

\author{Mahir Had\v zi\'c\thanks{Department of Mathematics, University College London, London WC1E 6XA, UK. Email: \href{mailto:m.hadzic@ucl.ac.uk}{m.hadzic@ucl.ac.uk}.}, \ Juhi Jang\thanks{Department of Mathematics, University of Southern California, Los Angeles, CA 90089, USA, and Korea Institute for Advanced Study, Seoul, Korea.  Email: \href{mailto:juhijang@usc.edu}{juhijang@usc.edu}.}, \ King Ming Lam \thanks{Department of Mathematics, University College London, London WC1E 6XA, UK. Email: \href{mailto:king.lam.19@ucl.ac.uk}{king.lam.19@ucl.ac.uk}.
Delft Institute of Applied Mathematics, Delft University of Technology, 2628 CD Delft, Netherlands. Email: \href{mailto:K.M.Lam@tudelft.nl}{K.M.Lam@tudelft.nl}.}}

\date{}

\begin{document}

\maketitle

\begin{abstract}
Goldreich-Weber solutions constitute a finite-parameter of expanding and collapsing solutions to the mass-critical Euler-Poisson system.
Two subclasses of this family correspond to compactly supported density profiles suitably modulated by the dynamic radius of the star that expands at the self-similar rate 
$\l(t)_{t\to\infty}\sim t^{\frac23}$ and linear rate $\l(t)_{t\to\infty}\sim t$ respectively. We prove two results: 
any linearly expanding Goldreich-Weber star is nonlinearly stable, while  
any given self-similarly expanding Goldreich-Weber star is codimension-$4$ nonlinearly stable against irrotational perturbations.

The codimension-$4$ condition in the latter result is optimal and reflects the presence of $4$ unstable directions in the linearised dynamics in self-similar coordinates, which are
induced by the conservation of the energy and the momentum. This result can be viewed as a codimension-$1$ nonlinear stability of the moduli space of self-similarly expanding Goldreich-Weber stars against irrotational perturbations.
\end{abstract}

\tableofcontents

\section{Introduction}

\subsection{The Euler-Poisson system}

We consider
 a fundamental model of a self-gravitating compressible fluid, given by the Euler-Poisson system.
The unknowns are the fluid density $\rho\ge0$,  the velocity vector ${\bf u}$, the fluid pressure $p\ge0$, and the gravitational potential $\phi$.
They solve the system
\begin{alignat}{3}
\partial_t\rho+\grad\cdot(\rho\mb u)&=0, \qquad &&\text{ in } \ \Omega(t), \label{E:CONT}\\
\rho \left(\pa_t + \mb u\cdot\nabla\right)\mb u+\grad p+\rho\grad\phi&=\mb 0, \qquad &&\text{ in } \ \Omega(t), \label{E:MOM} \\
\lpc\phi&=4\pi\rho, \ \qquad &&\text{ in } \ \mathbb R^3.\label{E:GRAV}
\end{alignat}
Here the pressure $p$ satisfies the mass-critical polytropic equation of state
\be\label{E:EOS}
p=\rho^{\frac43},
\ee 
and the star is isolated, which translates into the asymptotic boundary condition for the gravitational potential:
\begin{align}\label{E:AF}
\lim_{|\mb x|\to\infty}\phi(t,\mb x) &=0.
\end{align}
We refer to the system~\eqref{E:CONT}--\eqref{E:AF} as the (EP)$_{\frac43}$-system.
Moreover $\Omega(t): = \{{\bf x}\,\big| \,\rho(t,{\bf x})>0\}$ is the interior of the support of star density. Note that in~\eqref{E:GRAV} the density is trivially extended 
by $0$ to the complement of $\Omega(t)$. 
Since we allow the boundary to move, we must complement equations~\eqref{E:CONT}--\eqref{E:AF} with suitable boundary conditions at the vacuum free boundary $\pa\Omega(t)$. We  assume the classical kinematic boundary condition
\begin{align}\label{E:NORMALVEL}
\mathcal V_{\pa\Omega(t)} = {\bf u} \cdot {\bf n} \qquad \quad \text{ on } \ \pa\Omega(t), 
\end{align}
which states that the normal velocity of the boundary $\mathcal V_{\pa\Omega(t)}$ equals the normal component of the velocity vector field; here ${\bf n}$ is the outward pointing unit normal to $\pa\Omega(t)$.
It is well-known that the presence of the vacuum boundary 
complicates the local-in-time well-posedness problem, as the acoustic cones degenerate as the speed of sound $c_s$ defined through
\[
c_s^2:= \frac{d p}{d \rho}  = \frac43 \rho^{\frac13},
\]
becomes $0$ at the vacuum boundary. The resolution comes by imposing a condition on initial data, that specifies the rate of decay of the initial density to $0$ as we approach the vacuum boundary. This condition is known as the physical vacuum condition and reads
\begin{align}\label{E:PHYSICALVACUUM}
\nabla (c_s^2)\cdot {\bf n}\Big|_{\pa\Omega_0} <0.
\end{align}
The Euler-Poisson system \eqref{E:CONT}-\eqref{E:AF} possesses the following important conserved quantities -- mass, momentum and energy, given respectively by:
\begin{align}
M[\rho]&:=\int_{\R^3}\rho\;\d\mb x, \label{E:mass in Euler}\\
\mb W[\rho,\mb u]&:=\int_{\R^3}\rho\mb u\;\d\mb x, \label{E:momentum in Euler}\\
E[\rho,{\bf u}]&:=\int_{\R^3}\brac{{1\over 2}\rho|\mb u|^2+3\rho^{4\over 3}+{1\over 2}\rho\phi}\d\mb x. \label{E:energy in Euler}
\end{align}

The mass-criticality associated with the polytropic index $\frac43$ in~\eqref{E:EOS} is simply a statement that the natural self-similar rescaling of the problem also preserves the total mass $M[\rho]$. Namely, for any $\l>0$ and ${\bf x}_0\in\mathbb R^3$
one can check that if $(\rho,\mb u)$ is a classical solution of the (EP)$_{\frac43}$-system,  
then $(\tilde{\rho},\tilde{\mb u})$ defined by
\begin{align}
\rho(t,\mb x)&=\lambda^{-3}\tilde{\rho}\brac{\frac{t}{\lambda^{\frac32}},{\mb x-\mb x_0\over\lambda}} \label{E:RHOSCALED}\\
\mb u(t,\mb x)&=\lambda^{\frac12}\tilde{\mb u}\brac{\frac{t}{\lambda^{\frac32}},{\mb x-\mb x_0\over\lambda}}
\end{align}
is also a solution to the (EP)$_{\frac43}$-system as functions of the rescaled variables $(s,\mb y)$:
\[
s={t\over\lambda^{\frac32}},\qquad\qquad \mb y={\mb x-\mb x_0\over\lambda}.
\]
Relation~\eqref{E:RHOSCALED} readily implies that the total mass is conserved under this change of variables.


\newcommand{\Apar}{A\!\partial}
\newcommand{\cApar}{\A\!\partial}

We shall mostly work in the Lagrangian coordinates in this article, as they are particularly well suited to the analysis of fluids featuring a vacuum boundary. 
Let $\bs\eta(t,\mb x)$ be the the fluid flow map, defined through
\[
\partial_t\bs\eta=\mb u\circ\bs\eta\qquad\text{with}\qquad\bs\eta(0,\mb x)=\bs\eta_0(\mb x),
\]
where $\mb u\circ\bs\eta(t,\mb x)=\mb u(t,\bs\eta(t,\mb x))$. The spatial domain is then fixed for all time as $\Omega_0:=\bs\eta_0^{-1}(\Omega(0))$. To reformulate the (EP)$_{\frac43}$-system in the new variables, we introduce 
\begin{align*}
\mb v&=\mb u\circ\bs\eta\hfill\tag{Lagrangian velocity}\\
f&=\rho\circ\bs\eta\hfill\tag{Lagrangian density}\\
\psi&=\phi\circ\bs\eta\hfill\tag{Lagrangian potential}\\
A&=(\grad\bs\eta)^{-1}\hfill\tag{inverse of the deformation tensor}\\
J&=\det(\grad\bs\eta)\hfill\tag{Jacobian determinant}\\
a&=JA\hfill\tag{cofactor matrix of the deformation tensor}
\end{align*}
%
%
Under this change of coordinates, the continuity equation becomes $fJ=f_0J_0$ and the momentum equation \eqref{E:MOM} in the domain $\Omega_0$ 
reads
\be\label{E:MOMLAGR}
\partial_t\mb v+\frac1{w^3}\partial_k(A_\bullet^k(w^{4} J^{-\frac13})+A\grad\psi=\mb 0,
\ee
where $A\grad:=A^k\partial_k$ and we have introduced the enthalpy $w$
\begin{align}\label{E:ENTHALPYDEF}
w: = (f_0J_0)^{\frac13}.
\end{align}
Moreover, $\psi$ solves the Poisson equation
\be
(A\grad)\cdot(A\grad)\psi=4\pi f_0J_0 J^{-1}.
\ee
For details of the Lagrangian description of the Euler-Poisson system, we refer to~\cite{HaJa2017}. 


\subsection{Goldreich-Weber stars}

\renewcommand{\b}{\mathcal{b}}

The mass-criticality of the problem allows for the existence of a special class of expanding solutions, known as the Goldreich-Weber stars~\cite{GoWe1980}. The reason such solutions exist is, roughly speaking, because the scaling properties of the Euler-Poisson system in the mass critical case allows us to scale solutions while maintain the overall mass. This suggests that natural solutions that evolve in time under this scaling exist (note that solutions must conserve overall mass in time).
For reader's convenience we provide a brief summary of this special class of solutions of~\eqref{E:CONT}--\eqref{E:AF} which has been analysed in~\cite{GoWe1980,Makino92,FuLin,DengXiangYang}. A comprehensive overview can be found in \cite{HaJa2018-1}.

We let
\be\label{E:KDEF}
\mathcal{K}:=4\pi\lpc^{-1}
\ee
i.e. $\mathcal K f (\mb x)= - \int_{\mathbb R^3} \frac{f(\mb y)}{|\mb x-\mb y|}\d \mb y$ for any $f\in L^2(\mathbb R^3)$.
Observe that
\begin{align*}
\psi(\mb x)&=(\mathcal{K}\rho)(\bs\eta(\mb x))=-\int{\rho(\mb y)\over|\bs\eta(\mb x)-\mb y|}\d\mb y
=-\int{f(\mb z)J(\mb z)\over|\bs\eta(\mb x)-\bs\eta(\mb z)|}\d\mb z
=-\int{f_0(\mb z)J_0(\mb z)\over|\bs\eta(\mb x)-\bs\eta(\mb z)|}\d\mb z\\
&=-{1\over\lambda}\int{f_0(\mb z)\over|\mb x-\mb z|}\d\mb z
={1\over\lambda}\mathcal{K}(f_0)(\mb x).
\end{align*}
We look for spherically symmetric solutions to the Euler-Poisson system of the form $\bs\eta(t,\mb x)=\lambda(t)\mb x$, and assume without loss of generality that $\lambda(0)=1$. Under this affine ansatz, the momentum equation~\eqref{E:MOMLAGR} reduces to
\[
\ddot{\lambda}\lambda^2\mb x+{1\over f_0}\grad(f_0^{\frac43})+\grad\mathcal Kf_0=\mb 0.
\]
Assuming spherically symmetry we get
\[
\ddot{\lambda}\lambda^2+{1\over rf_0}\partial_r (f_0^{\frac43})+{1\over r}\partial_r\mathcal Kf_0=0.
\]
Since we can separate variables above, we look for a $\delta\in\mathbb R$ and a $\delta$-dependent solution $(\l,f_0)=(\l_\delta,f_0^\delta)$ so that
\begin{align}
\ddot{\lambda}(t)\lambda(t)^2&=\delta, \label{E:GW1}\\
{4\over r}\partial_r \bar w_\delta+{1\over r}\partial_r\mathcal K(\bar w_\delta^3)&=-\delta,\label{E:GW2}
\end{align}
where $\bar w_\delta$ is the enthalpy associated with $f_0^\delta$ satisfying 
\begin{align}
(\bar w_\delta)^3 := f_0^\delta.
\end{align}
We also equip~\eqref{E:GW1} with initial data
\begin{align}\label{E:GW3}
\l(0)=1, \qquad \ \dot\l(0)=\l_1\in\mathbb R.
\end{align}
It can be shown that there exists a negative constant $\tilde\delta<0$ such that the solution $(\l_\delta(t),\bar w_\delta)$ to~\eqref{E:GW1}--\eqref{E:GW3} exists for all $\delta\ge\tilde\delta$, see~\cite{FuLin,HaJa2018-1}, whereby $\l_\delta(\cdot)$ either blows up in finite positive time, or exists globally for all $t\ge0$. Moreover, for any such $\delta\ge\tilde\delta$, the enthalpy profile $\bar w_\delta$ is compactly supported, has finite total mass, and by adapting the value $\bar w_\delta(0)$ it can be normalised to be supported on the interval $r\in[0,R]$ for a fixed $R>0$. At the vacuum boundary, by analogy to 
the classical Lane-Emden stars~\cite{HaJa2018-1}, the Goldreich-Weber star 
satisfies the so-called physical vacuum condition, which in this context reads
\begin{align}\label{E:PHYSICALVACUUM GW}
\bar w_\delta'(r)\Big|_{r=R}<0.
\end{align}

\subsubsection{Self-similarly expanding Goldreich-Weber stars}

The self-similarly expanding Goldreich-Weber stars are the subclass of solutions to~\eqref{E:GW1}--\eqref{E:GW3} of total energy $0$, for which $\l_\delta(\cdot)$ exists for all $t\ge0$.
Since the total conserved energy of the above affine motion is easily seen to be
\begin{align}
E_\delta(t) = \left(\l_1^2+2\delta\right) \int 2\pi f_0^\delta z^4\,\d z,
\end{align}
solutions with vanishing energy necessitate $\delta<0$. For any such $\tilde\delta\le \delta<0$, equation~\eqref{E:GW1} with~\eqref{E:GW3}
is explicitly solvable with
\be\label{E:LAMBDADELTA}
\lambda_\delta(t)=\left(1 + \frac32\lambda_1 t \right)^{2/3}, \qquad \ \l_1^2= -2\delta.
\ee
In particular, for any $\l_1>0$ we obtain an expanding solution with the explicit rate of expansion $\l_\delta(t)\sim_{t\to\infty} t^{\frac23}$. This is the self-similarly expanding Goldreich-Weber solution.



\begin{definition}[Self-similarly expanding Goldreich-Weber solutions]\label{self-similar GW def}
To any $\delta\in[\tilde\delta,0)$ we associate the Goldreich-Weber (GW) star which constitutes a solution of the mass-critical free-boundary Euler-Poisson system~\eqref{E:CONT}--\eqref{E:AF}:  
\begin{align}
\bar\rho(t,\mb x) = \l_\delta(t)^{-3} \bar w_\delta^3\left(\frac{|\mb x|}{\l_\delta(t)}\right), \qquad \ \bar{\mb u}(t,\mb x) = \frac{\dot\l_\delta(t)}{\l_\delta(t)}\mb x, \qquad \ \bar{\Omega}(t) = B_{\l_\delta(t)}(\mb 0),
\end{align}
with $\l_\delta(t)$ given by~\eqref{E:LAMBDADELTA} with $\l_1>0$ and $\bar w_\delta$ the normalised solution to~\eqref{E:GW2} as above.
\end{definition}

These solutions are spherical symmetric about the origin, have zero momentum $\mb W[\bar\rho,\bar{\mb u}]=0$ and zero energy $E[\bar\rho,\bar{\mb u}]=0$. Without loss of generality, this 
can be assumed by setting our frame of reference.

\begin{remark}\label{Momentum remark}
The Galilean invariance of the Euler-Poisson system~\eqref{E:CONT}--\eqref{E:AF} implies the conservation of momentum. If we change our frame of reference, we can obtain an enlarged family of the GW-solutions with arbitrary momentum $\bar{\mb W}\in\R^3$. More precisely, for any motion $\mb p(t)=\mb p_0+t \mb p_1$ we can
obtain a new solution via
\begin{align*}
\bar\rho_{\mb p}(t,\mb x)&=\bar\rho(t,\mb x-\mb p(t)),\\
\bar{\mb u}_{\mb p}(t,\mb x)&=\bar{\mb u}(t,\mb x-\mb p(t))+\mb p_1,
\end{align*}
or equivalently $\bs\eta_{\mb p}(t,\mb x)=\bs\eta(t,\mb x)+\mb p(t)$ in Lagrangian coordinates. It is easy to verify that $(\bar\rho_{\mb p},\bar{\mb u}_{\mb p})$ so obtained solves the Euler-Poisson system with the total  momentum $\mb W[\bar\rho_{\mb p},\bar{\mb u}_{\mb p}]=M[\bar\rho,\bar{\mb u}]\mb p_1$ and energy $E[\bar\rho_{\mb p},\bar{\mb u}_{\mb p}]={1\over 2}M[\bar\rho,\bar{\mb u}]|\mb p_1|^2$. The freedom to choose $\mb p_1\in\mathbb R^3$ thus parametrises the three degrees of freedom associated with the total linear momentum, and this will play a role in our analysis.
\end{remark}

\subsubsection{Linearly expanding Goldreich-Weber stars}

In the case
\begin{align}
\delta>0\qquad\text{ or }\qquad\delta=0\text{ with }\lambda_1>0\qquad\text{ or }\qquad\delta\in(\tilde\delta,0)\text{ with }\lambda_1>\sqrt{2|\delta|},\label{Linearly GW condition}
\end{align}
the solution $\l_\delta(\cdot)$ exists for all $t\ge0$ and expands indefinitely at a linear rate, i.e. there exists
a constant $c > 0$ such that
\begin{align*}
\lim_{t\to\infty}\dot\lambda(t)=c.
\end{align*}
These solutions have strictly positive energy
\begin{align}
E_{\delta,\lambda_1}(t) = \left(\l_1^2+2\delta\right) \int 2\pi f_0^\delta z^4\,\d z>0.
\end{align}
We refer to such solutions of~\eqref{E:GW1}--\eqref{E:GW3} as the {\em linearly expanding Goldreich-Weber} solutions.

\begin{definition}[Linearly expanding Goldreich-Weber solutions]\label{linear GW def}
To any $\delta,\lambda_1$ satisfying \eqref{Linearly GW condition} we associate the Goldreich-Weber (GW) star which constitutes a solution of the mass-critical free-boundary Euler-Poisson system~\eqref{E:CONT}--\eqref{E:AF}:  
\begin{align}
\bar\rho(t,\mb x) = \l_{\delta,\lambda_1}(t)^{-3} \bar w_\delta^3\left(\frac{|\mb x|}{\l_{\delta,\lambda_1}(t)}\right), \quad \ \bar{\mb u}(t,\mb x) = \frac{\dot\l_{\delta,\lambda_1}(t)}{\l_{\delta,\lambda_1}(t)}\mb x, \quad \ \bar{\Omega}(t) = B_{\l_{\delta,\lambda_1}(t)}(\mb 0),
\end{align}
with $\l_\delta(t)$ the solution to \eqref{E:GW1}, \eqref{E:GW3} and $\bar w_\delta$ the normalised solution to~\eqref{E:GW2} as above.
\end{definition}


Unless stated otherwise, we shall drop the subscript $\delta$ in the definition of the GW-solution, as this will create no confusion
in the analysis.

\subsection{Main results and review}

The two results to be presented in this paper are a generalisation of nonlinear stability of GW-stars against radial perturbations shown in \cite{HaJa2018-1} by the first two authors. We prove nonlinear stability against non-radial perturbations. In Section \ref{self-similar GW} and \ref{linear GW} respectively, we will prove the following two theorems.

\begin{theorem}[Informal statement]\label{theorem 1}
The class of self-similarly expanding Goldreich-Weber stars is co-dimension 1 non-linearly stable under irrotational perturbations.
\end{theorem}

\begin{theorem}[Informal statement]\label{theorem 2}
The linearly expanding Goldreich-Weber stars are non-linearly stable (against general perturbations).
\end{theorem}

The precise statements will be provided 
in Section \ref{self-similar GW} and \ref{linear GW} respectively. More precisely, Theorem \ref{theorem 1} corresponds to Theorem \ref{T:MAIN} and Corollary \ref{C:MAIN}, while Theorem \ref{theorem 2} corresponds to Theorem \ref{T:MAIN - linear}.

First discovered class of nontrivial global solutions to the Euler-Poisson system are the classical Lane-Emden (LE) stars~\cite{Chandrasekhar}. Their linear stability is well-known to depend on the size of the polytropic exponent in the general pressure law $p = \rho^\gamma$, $\frac65\le \gamma<2$. Very few rigorous nonlinear results are available on the dynamics in the vicinity of LE-steady states, we refer to~\cite{Rein2003} for some rigorous statements about the stability in the subcritical range $\frac43<\gamma<2$ and to~\cite{Jang2008,Jang2014} for rigorous nonlinear instability analysis in the supercritical range $\frac65\le\gamma<\frac 43$. In the context of nonradial stability, recent works~\cite{JaMa2020,LiZe2021,Lin-Wang-Zhu} treat this question from the Lagrangian and the Eulerian perspective respectively. In the critical case $\gamma=\frac43$, the very existence of the GW-stars shows that the LE-steady states are embedded in a larger family of collapsing/expanding solutions, and are therefore unstable. Our main result can be viewed as a definitive nonradial instability statement about the mass-critical LE-solutions, improving upon the radial nonlinear stability shown by the first two authors~\cite{HaJa2018-1}. We emphasise that in the presence of viscosity, the parabolic effect takes over and various asymptotic stability results are available~\cite{LXZ1,LXZ2}. We also mention recent works \cite{ChHeWaYu,Cheng-Cheng-Lin} on global existence result with radial symmetry in the class of weak solutions and conditional behaviour of strong solutions. 

The driving stabilisation mechanism that allows for the global existence in Theorem~\ref{T:MAIN} is the expansion of the support of the background GW-star. Intuitively expansion translates into dispersion, since the total mass is preserved.
When there is no vacuum boundary present, the dispersion induced by the expansion was used by Grassin~\cite{Grassin98}, Serre~\cite{Se1997}, and Rozanova~\cite{Ro} to give examples of  global-in-time solutions to the compressible Euler flows. We also mention here that there has been a recent surge of activity on the problem of existence of collapsing self-gravitating flows, which are characterised by the finite-time implosion of the fluid density. We refer the reader to~\cite{GHJ2021a,GHJ2021b,GHJ2021c,GHJS2021, AlHaSc2023, Sandine2023} and for a discussion of the various features of the collapsing and expanding stellar dynamics, we refer to the overview paper~\cite{Ha2023}.

The GW-stars belong to a class of so-called affine motions. In the context of compressible flows the notion of an affine motion goes back to the works of Ovsiannikov~\cite{Ov1956} and Dyson~\cite{Dyson1968}.
In the presence of vacuum, Sideris~\cite{Sideris2017} showed the existence of a finite-parameter family of compactly supported expanding affine flows, whose nonlinear stability was shown by Had\v{z}i\'c and Jang~\cite{HaJa2016-2} and Shkoller and Sideris~\cite{ShSi} for the pure Euler flows. For expanding profiles with small initial densities, but not necessarily close to the Sideris solutions, see~\cite{PHJ2021}. Further results in this direction, in the nonisentropic setting and in the presence of heat convection can be found in~\cite{Rickard1,Rickard2,RHJ2021}. A similar method works for the Euler-Poisson system and global-in-time flows were shown to exist in both the gravitational and electrostatic case~\cite{HaJa2017}, where the Euler part of the flow entirely dominates the gravitational/electrostatic response of the model. Another application of an expansion-induced stabilisation is the work of Parmeshwar~\cite{Par2022} where an $N$-body configuration of expanding stars is shown to exist globally in-time. If damping is present in Euler flows it can drive sublinear expansion of Barenblatt-like solutions, 
see \cite{LuZe2016,Ze2017,Ze2021}.

Our result concerning the self-similarly expanding GW stars in Theorem \ref{theorem 1} has one notable difference to the above results. The stability of linearly expanding GW stars in Theorem \ref{theorem 2} is in fact easier than that for self-similarly expanding GW stars in Theorem \ref{theorem 1}. The reason is that the linearly expanding GW stars expand faster than the self-similarly expanding ones and thus the effects of dispersion-via-expansion are stronger in the proof of Theorem~\ref{theorem 2}. One important consequence is that the gravitational forces in the linearly expanding case are of subleading order and from the analysis it is apparent that the results do not depend on the attractive/repulsive nature of the force field. By contrast, our result on the self-similarly expanding GW stars in Theorem \ref{theorem 1} are profoundly sensitive to the attractive nature of the gravitational force and require more sophisticated estimates. This is particularly felt in the linearised stability analysis in Section~\ref{S:LINCO}.

The plan of the paper is as follows. In Section~\ref{Notation} we introduce important notational conventions and some key objects that 
will play a role throughout the paper. We introduce the gravity and the pressure operators $\mb G$ and $\mb P$ expressed in Lagrangian coordinates
and include several preparatory lemmas. In Section~\ref{self-similar GW} we provide a precise formulation of Theorem~\ref{theorem 1} and provide its proof. 
One of the main difficulties is to obtain coercivity of the associated linearised operator, see Section~\ref{S:LINCO}. In Section~\ref{linear GW} we provide a rigorous formulation
and proof of Theorem~\ref{theorem 2}. Finally, in Appendix~\ref{GW appendix} we provide an overview of several technical tools used throughout the paper, including various properties 
of the spherical harmonics, as well as some weighted Poincar\'e inequalities.

{\bf Acknowledgments.}
M. Had\v zi\'c's research is supported by the EPSRC Early Career Fellowship EP/S02218X/1.
J. Jang's research is supported in part by the NSF grants DMS-2009458 and DMS-2306910. 
K.-M. Lam was supported by the EPSRC studentship grant EP/R513143/1 when undertaking this research; now supported by the NWO grant OCENW.M20.194.

\section{Notation and preliminary lemmas}\label{Notation}

As Section \ref{self-similar GW} and \ref{linear GW} are devoted to the self-similarly expanding GW stars and linearly expanding GW stars respectively, we will write $\bar w$ to denote the the self-similarly expanding GW stars and linearly expanding GW stars enthalpy profile respectively (see Definition \ref{self-similar GW def} and \ref{linear GW def}) in these sections.

\newcommand{\A}{\mathcal{A}}
\newcommand{\J}{\mathcal{J}}
\renewcommand{\a}{\mathcal{a}}
\renewcommand{\K}{\mathcal{K}}
\newcommand{\q}{\mathcal{q}}
\newcommand{\qgrad}{\blacktriangledown}

\newcommand{\Rd}{\mathcal{R}}

\renewcommand{\F}{\mathcal{F}}
\renewcommand{\Z}{\mathcal{Z}}
\renewcommand{\Q}{\mathcal{Q}}
\newcommand{\pt}{/\!\!\!\partial}
\newcommand{\pr}{X_r}


Since the gaseous Euler-Poisson system is degenerate near the vacuum boundary, we will need to make use of weighted Sobolev spaces. Let $L^2(B_R,w)$ denote the $L^2$ space on $B_R$ weighted by a non-negative weight $w$. Of crucial importance in this paper are the weighted inner products
\begin{align}
\<g,h\>_k&:=\int_{B_R}gh\bar w^k\d\mb x, \\
\<\mb g,\mb h\>_k&:=\int_{B_R}\mb g\cdot\mb h\bar w^k\d\mb x, 
\end{align}
defined for any scalar fields $g, h\in L^2(B_R,\bar w^k)$  and vector fields $\mb g, \mb h\in L^2(B_R,\bar w^k)^3$. The weighted inner product for tensor fields are defined in the same way. The associated norm is then given  by 
\begin{align}\label{E:WEIGHTEDNORMDEF}
\|f\|_{k}^2 = \int_{B_R}|f(\mb x)|^2 \bar w(\mb x)^k\d\mb x.
\end{align}

To capture the structure of the roughly spherical stars, we will need to use the following specially defined radial and tangential derivatives in our analysis. We define
\begin{align}
\pr&:=x^i\partial_i=r\partial_r\label{pr}\\
\pt_{i}&:=\epsilon_{ijk}x^j\partial_k\label{pt_i}\\
\pt_{ij}&:=x^i\partial_j-x^j\partial_i\label{pt_ij}
\end{align}
where $\epsilon_{ijk}$ is the alternating symbol (see Definition \ref{Standard notations}). Note that $\pt_{ij}=\epsilon_{ijk}\pt_{k}$. We denote
\begin{align*}
\div\bs\theta&:=\grad\cdot\bs\theta\\
[\curl\bs\theta]^k_l&:=\partial_l\theta^k-\partial_k\theta^l\\
\div_\A\bs\theta&:=(\A\grad)\cdot\bs\theta\\
[\curl_\A\bs\theta]^k_l&:=\cApar_l\theta^k-\cApar_k\theta^l
\end{align*}
where $\A\grad:=\A^k\partial_k$ and $\cApar_i:=\A_i^k\partial_k$.

%
Let $(\bar\rho,\bar{\mb u})$ be a given self-similarly or linearly expanding GW-flow from Definition~\ref{self-similar GW def} and \ref{linear GW def} with the corresponding radius $R\l(t)$ 
and the associated enthalpy $\bar w:[0,R]\to\mathbb R_+$.
In order to study the stability of the flow, we will follow the strategy introduced in~\cite{HaJa2018-1,HaJa2016-2} and renormalise the equation by introducing a new unknown
\be\label{E:KSIDEF}
\bs\xi(t,\mb x)={\bs\eta(t,\mb x)\over\lambda(t)}.
\ee
We suitably renormalise the inverse of the Jacobian gradient and the Jacobian determinant, so that
\begin{alignat*}{3}
\A&:=(\grad\bs\xi)^{-1}&&=\lambda A\\
\J&:=\det(\grad\bs\xi)&&=\lambda^{-3}J\\
\a&:=\J\A&&=\lambda^{-2}a \\
\Phi&:= -\int{f_0(\mb z)\J_0(\mb z)\over|\bs\xi(\mb x)-\bs\xi(\mb z)|}\d\mb z && = \l\psi 
\end{alignat*}

We will work mainly with the perturbation variable defined by
\be\label{E:THETADEF}
\bs\theta(\mb x):=\bs\xi(\mb x)-\mb x,
\ee 
which measures the deviation of the nonlinear flow to the background GW profile. 

As will see later, in these new variable, the pressure and gravity term in the Euler-Poisson system take the following form
\begin{align}
\mb P&:=\bar w^{-3}\partial_k(\bar w^4(\A^k_\bullet\J^{-1/3}-I^k_\bullet)), \label{E:P}\\
\mb G&:=\A\grad\Phi-\mathcal{K}\grad\bar w^3 \label{E:GDEF}
\end{align}

In the rest of the paper
we will use some fairly standard notations which we collect here for reader's convenience.

\begin{definition}[Standard notations]\label{Standard notations}\mbox{}
\begin{enumerate}
\item Greek letter superscript on derivatives are multi-index notation for derivatives. For example, $\pt^\alpha=\pt_1^{\alpha_1}\pt_2^{\alpha_2}\pt_3^{\alpha_3}$ where $\alpha=(\alpha_1,\alpha_2,\alpha_3)$. And $|\alpha|=\alpha_1+\alpha_2+\alpha_3$.
\item Roman letter indices such as $i,j,k,l,m$ on derivatives and vector or tensor fields are assumed to range over $\{1,2,3\}$. However, this does not apply to $s$ which we reserved to denote the rescaled time variable. Also, it does not apply when they are indices of non-vector or non-tensor objects, for example $\Psi_{lm}$ and $\Lambda_{lm}$ in Section \ref{self-similar GW}.
\item The Einstein summation convention will be used, i.e. repeated indices on derivatives and vector or tensor fields are summed over. For example, $\partial_i\theta^i=\sum_{i=1}^3\partial_i\theta^i$. However, this does not apply to non-vector or non-tensor objects, for example $\Psi_{lm}$ and $\Lambda_{lm}$ in Section \ref{self-similar GW}.
\item $I$ denotes the identity matrix, $\delta_{ij}$ or $\delta^i_j$ the Kronecker delta, $\epsilon_{ijk}$ the alternating symbol (Levi-Civita symbol).
\item $C$ will denote generic ``analyst's constant'', whose exact value can change from line to line and term to term. When appearing in equalities, it can potentially denote any real constant, but when appearing in inequalities, it is generally assumed to be non-negative.   We will use subscript to emphasise its dependence on certain variables, for example $C_\delta$ is a constant that depends on $\delta$.
\item $\mb e_i$ ($i=1,2,3$) denotes the standard basis of $\R^3$, while $\mb e_r$ denotes the radial unit vector $\mb x/|\mb x|$.
\end{enumerate}
\end{definition}

Now we will define some important special new notations that the reader probably will not have seen before.

\begin{definition}[Special notations]\label{Special notations}\mbox{}
\begin{enumerate}
\item We will denote $\partial_\bullet$ as a generic derivative, so it can be any of $\partial_s$, $\partial_i$, $\pt$ or $\pr$.
\item We will use $\bullet$ to denote an unspecified index, or to emphasise the vectorial/tensorial nature of non-scalar quantities. For example if $A$ is a matrix, we can write $A^k_\bullet$.
\item When the exact value/ordering of the indices is not important, we shall often write
$\<\star\>$ for a generic term that looks like $\star$ to avoid invoking indices. For example, $\<C\A\grad\bs\theta\>$ could represent a term like $C\A^i_j\partial_k\theta^l$ for some $i,j,k,l$ and constant $C\in\R$.
\item We will write $\Rd[\star]$ to denote terms that can be bounded by $\star$, e.g. $|\Rd[S_nE_n]|\lesssim S_nE_n$.
\item We will write $\mb 1_\star$ to denote the usual indicator function, and write $\mb 1[\star]$ to denote the Iverson bracket. For example, $\mb 1_A(\mb x)=\mb 1[\mb x\in A]$.
\end{enumerate}
\end{definition}


\subsection{Essential lemmas and definitions}


We next state and prove some essential preparatory lemmas and definitions that will be used throughout the paper.
The following Hodge-type estimate allows us to estimate the norm of the gradient of a quantity in terms of the div and curl of the quantity term and a lower order term. 

\begin{lemma}[Hodge-type bound]\label{Hodge bound}
Let $k\geq 0$. For any $\bs\theta$ we have
\begin{align*}
\|\grad\bs\theta\|_{k+2}^2\lesssim\|\grad\cdot\bs\theta\|_{k+2}^2+\|\grad\times\bs\theta\|_{k+2}^2+\norm{\bs\theta}_{k}^2
\end{align*}
\end{lemma}


\begin{proof}
We have
\begin{align*}
\int |\grad\bs\theta|^2\bar w^{k+2}\d\mb x
&=\int(\partial_j\theta^i)(\partial_j\theta^i)\bar w^{k+2}\d\mb x
=\int(\partial_j\theta^i)(\partial_i\theta^j+[\curl\bs\theta]^i_j)\bar w^{k+2}\d\mb x\\
&=\int(\partial_j\theta^i)(\partial_i\theta^j)\bar w^{k+2}\d\mb x+\int(\partial_j\theta^i)[\curl\bs\theta]^i_j\bar w^{k+2}\d\mb x\\
&=-\int(\partial_i\partial_j\theta^i)(\theta^j)\bar w^{k+2}\d\mb x+\int(\partial_j\theta^i)[\curl\bs\theta]^i_j\bar w^{k+2}\d\mb x\\
&\quad-(k+2)\int(\partial_j\theta^i)(\theta^j)\bar w^{k+1}\partial_i\bar w\d\mb x\\
&=\int(\partial_i\theta^i)(\partial_j\theta^j)\bar w^{k+2}\d\mb x+\int(\partial_j\theta^i)[\curl\bs\theta]^i_j\bar w^{k+2}\d\mb x\\
&\quad+(k+2)\brac{\int(\partial_i\theta^i)(\theta^j)\bar w^{k+1}\partial_j\bar w\d\mb x-\int(\partial_j\theta^i)(\theta^j)\bar w^{k+1}\partial_i\bar w\d\mb x}\\
&\lesssim\delta'\int |\grad\bs\theta|^2\bar w^{k+2}\d\mb x+{1\over\delta'}\brac{\int(|\grad\cdot\bs\theta|^2+|\grad\times\bs\theta|^2)\bar w^{k+2}\d\mb x+\int |\bs\theta|^2\bar w^{k}\d\mb x}.
\end{align*}
Picking $\delta'$ small enough, we are done. 
\end{proof}


\subsubsection{Basic bounds on the gravity term $\mb G$}\label{Pertaining the gravity term}

We recall here the definition~\eqref{E:GDEF} of $\mb G$. The following lemma is a structural identity that
allows us to estimate the gravitational term $\mb G$ more conveniently.

\begin{lemma}
Recall the definition~\eqref{E:GDEF} of the gravitational term $\mb G$. We then have the identity:
\begin{align}
\mb G&=\K_{\bs\xi}\grad\cdot(\A_\bullet\bar w^3)-\mathcal{K}\grad\bar w^3   \label{E:G}\\
&=\K_{\bs\xi}((\A-I)\grad\bar w^3-\bar w^3\A^i_m\A^l_\bullet\partial_i\partial_l\theta^m)+(\K_{\bs\xi}-\mathcal{K})\grad\bar w^3,
\end{align}
where
\begin{align}
(\K_{\bs\xi}g)(\mb x)&:=-\int{g(\mb z)\over|\bs\xi(\mb x)-\bs\xi(\mb z)|}\d\mb z\label{E:NONLOCALKEY}
\end{align}
\end{lemma}


\begin{proof}
Note that formally
\begin{align*}
(\grad\K\rho)(\mb x)&=-\int\grad_{\mb x}{\rho(\mb z)\over|\mb x-\mb z|}\d\mb z
=\int\grad_{\mb z}{\rho(\mb z)\over|\mb x-\mb z|}\d\mb z
=-\int{\grad\rho(\mb z)\over|\mb x-\mb z|}\d\mb z
=(\K\grad\rho)(\mb x)
\end{align*}
and so
\begin{align*}
A\grad\psi(\mb x)&=(\grad\phi)(\bs\eta(\mb x))=(\grad\K\rho)(\bs\eta(\mb x))=(\K\grad\rho)(\bs\eta(\mb x))
=-\int{\grad\rho(\mb y)\over|\bs\eta(\mb x)-\mb y|}\d\mb x\\
&=-\int{A\grad f(\mb z)J(\mb z)\over|\bs\eta(\mb x)-\bs\eta(\mb z)|}\d\mb z=-\int{a\grad(fJJ^{-1})(\mb z)\over|\bs\eta(\mb x)-\bs\eta(\mb z)|}\d\mb z
=-\int{a\grad(\bar w^3J^{-1})(\mb z)\over|\bs\eta(\mb x)-\bs\eta(\mb z)|}\d\mb z\\
&=-\int{\grad\cdot(a\bar w^3J^{-1})(\mb z)\over|\bs\eta(\mb x)-\bs\eta(\mb z)|}\d\mb z
=-\int{\grad\cdot(A\bar w^3)(\mb z)\over|\bs\eta(\mb x)-\bs\eta(\mb z)|}\d\mb z={1\over\lambda^2}(\K_{\bs\xi}\grad\cdot(\A\bar w^3))(\mb x),
\end{align*}
where we denote $\grad\cdot M=\partial_iM^i$ for a matrix $M$ and recall~\eqref{E:NONLOCALKEY}. We then have
\[\A\grad\Phi=\lambda^2A\grad\psi(\mb x)=\K_{\bs\xi}\grad\cdot(\A\bar w^3).\]
Now we have
\begin{align*}
\mb G&=\K_{\bs\xi}\grad\cdot(\A_\bullet\bar w^3)-\mathcal{K}\grad\bar w^3
=\K_{\bs\xi}(\grad\cdot(\A\bar w^3)-\grad\bar w^3)+(\K_{\bs\xi}-\mathcal{K})\grad\bar w^3\\
&=\K_{\bs\xi}((\A-I)\grad\bar w^3-\bar w^3\A^i_m\A^l_\bullet\partial_i\partial_l\theta^m)+(\K_{\bs\xi}-\mathcal{K})\grad\bar w^3. 
\end{align*}
\end{proof}

Since the gravity term is a non-local term, we need to estimate a convolution-like operator. However, rather than the convolution kernel $|\mb x-\mb z|^{-1}$ we actually need to estimate $|\bs\xi(\mb x)-\bs\xi(\mb z)|^{-1}$. The next lemma tells us how to reduce the latter to the former, which will allow us to estimate using the Young convolution inequality.

\begin{lemma}\label{distance estimate}
Let $\bs\xi$ be as in \eqref{E:THETADEF}. For any $\mb x,\mb z\in B_R$ we have
\begin{align*}
|\mb x-\mb z|&\leq\|\A\|_{L^\infty(B_R)}|\bs\xi(\mb x)-\bs\xi(\mb z)|\\
|\partial_s^{a}\pt_{\mb x}^{\beta}\bs\xi(\mb x)-\partial_s^{a}\pt_{\mb z}^{\beta}\bs\xi(\mb z)|&\leq\|\grad\partial_s^{a}\pt^{\beta}\bs\xi\|_{L^\infty(B_R)}|\mb x-\mb z|
\end{align*}
\end{lemma}
\begin{proof}
Using the mean value inequality we have
\begin{align*}
|\mb x-\mb z|&=|\bs\xi^{-1}\bs\xi(\mb x)-\bs\xi^{-1}\bs\xi(\mb z)|\\
&\leq\|\grad\bs\xi^{-1}\|_{L^\infty(B_R)}|\bs\xi(\mb x)-\bs\xi(\mb z)|=\|\A\|_{L^\infty(B_R)}|\bs\xi(\mb x)-\bs\xi(\mb z)|
\end{align*}
and
\begin{align*}
|\partial_s^{a_i}\pt_{\mb x}^{\beta_i}\bs\xi(\mb x)-\partial_s^{a_i}\pt_{\mb z}^{\beta_i}\bs\xi(\mb z)| & \leq\|\grad\partial_s^{a_i}\pt^{\beta_i}\bs\xi\|_{L^\infty(B_R)}|\mb x-\mb z|. 
\end{align*}
\end{proof}

Since we cannot commute extra weights into the non-local gravity term, the radial derivatives, which affect the powers in the weight, need to be estimated differently to avoid possible loss of regularity via unfavourable weights.
Using methods from \cite{HaJa2017}, the following two lemmas provide the way to do this. More precisely, the radial derivative can be estimated with curl, divergence and tangential derivatives. And this is useful because the curl and divergence of the gravity term consist only of local or non-linear terms, which we can estimate.

\begin{lemma}\label{Breakdown for X_r}
For any vector field $\tilde{\mb G}\in H^1_\loc$
\begin{align*}
|\pr\tilde{\mb G}|^2\lesssim|r\grad\cdot\tilde{\mb G}|^2+|r\grad\times\tilde{\mb G}|^2+\sum_{k=1}^3|\pt_k\tilde{\mb G}|^2,
\end{align*}
where we recall the notation \eqref{pr}-\eqref{pt_ij}.
\end{lemma}
\begin{proof}
Note that
\begin{align*}
|\mb x\cdot\tilde{\mb G}|^2&=(x^i\tilde{G}^i)(x^j\tilde{G}^j)=(x^j\tilde{G}^i)(x^i\tilde{G}^j)
=(x^j\tilde{G}^i)(x^j\tilde{G}^i)+(x^j\tilde{G}^i)(x^i\tilde{G}^j-x^j\tilde{G}^i)\\
&=|\mb x|^2|\tilde{\mb G}|^2-{1\over 2}(x^i\tilde{G}^j-x^j\tilde{G}^i)(x^i\tilde{G}^j-x^j\tilde{G}^i)
=r^2|\tilde{\mb G}|^2-|\mb x\times\tilde{\mb G}|^2
\end{align*}
We have by definition
\[
\partial_i={x^j\over r^2}\pt_{ji}+{x^i\over r^2}\pr, \ \ i=1,2,3.
\]
The divergence and the curl of $\tilde{\mb G}$ can be written as 
\begin{align*}
r^2\grad\cdot\tilde{\mb G}=x^j\pt_{ji}\tilde{\mb G}^i+\mb x\cdot\pr\tilde{\mb G} \qquad  \ \text{ and } \qquad \ 
r^2\grad\times\tilde{\mb G}=x^j\pt_{j\bullet}\times\tilde{\mb G}+\mb x\times\pr\tilde{\mb G}
\end{align*}
We then obtain 
\begin{align*}
r^2|\pr\tilde{\mb G}|^2=|r^2\grad\cdot\tilde{\mb G}-x^j\pt_{ji}G^i|^2+|r^2\grad\times\tilde{\mb G}-x^j\pt_{j\bullet}\times\tilde{\mb G}|^2
\end{align*}
from which we deduce the result. 
\end{proof}

\begin{lemma}[The div-curl structure of the gravitational term]\label{div and curl of G}
Let $\mb G$ be as in \eqref{E:G}. We have
\begin{align*}
\grad\cdot\mb G&=(I-\A)\grad\cdot\mb G+(I-\A)\grad\cdot\grad\K\bar w^3+4\pi\bar w^3(\J^{-1}-1)\\
\grad\times\mb G&=(I-\A)\grad\times\mb G+(I-\A)\grad\times\grad\K\bar w^3.
\end{align*}
\end{lemma}
\begin{proof}
By definition $\mb G=\A\grad\Phi-\grad\K\bar w^3$, so
\begin{align*}
\A\grad\cdot\mb G&=(\A\grad)\cdot(\A\grad)\Phi-\A\grad\cdot\grad\K\bar w^3
=(\A\grad)\cdot(\A\grad)\Phi+(I-\A)\grad\cdot\grad\K\bar w^3-4\pi\bar w^3
\end{align*}
And we have
\begin{align}
(\A\grad)\cdot(\A\grad)\Phi(\mb x)&=\lambda^3(A\grad)\cdot(A\grad)\psi(\mb x)=\lambda^3(\grad\cdot\grad\K\rho)(\bs\eta(\mb x))\nonumber\\
&=\lambda^34\pi\rho(\bs\eta(\mb x))=\lambda^34\pi f(\mb x)
=\lambda^34\pi\bar w^3J^{-1}=4\pi\bar w^3\J^{-1}.\label{E:weighted divergence of gravity}
\end{align}
So we get the first formula. Proof for the second formula is similar but we use $(\A\grad)\times(\A\grad)=0$. 
\end{proof}

The next lemma lets us deal with time and tangential derivatives on non-local terms and its kernel, as we will need when dealing with the gravity term. 

\begin{lemma}\label{L:ENERGYLEMMA1}
\begin{enumerate}
\item[{\em (i)}]
For any $K:B_R\times B_R\to\R$ sufficiently smooth and $g\in H_0^1(B_R)$ we have
\begin{align*}
\pt_{i,\mb x}\int_{B_R}K(\mb x,\mb z)g(\mb z)\d\mb z=\int_{B_R}\brac{g(\mb z)(\pt_{i,\mb x}+\pt_{i,\mb z})K(\mb x,\mb z)+
K(\mb x,\mb z)\pt_{i,\mb z}g(\mb z)}\d\mb z,
\end{align*}
where we recall the notation \eqref{pr}-\eqref{pt_ij}.
\item[{\em (ii)}]
For any $\bs\theta:B_R\to\R^3$ sufficiently smooth and $\mb x,\mb y\in B_R$ we have
\begin{align}
|\partial_s^a\pt^\beta\bs\theta(\mb x)-\partial_s^a\pt^\beta\bs\theta(\mb z)|&\leq\|\grad\partial_s^a\pt^\beta\bs\theta\|_{L^\infty(B_R)}|\mb x-\mb z| \label{E:PART21}\\
|\pt^\beta\mb x-\pt^\beta\mb z|&\leq|\mb x-\mb z| \label{E:PART22}
\end{align}
\end{enumerate}
\end{lemma}


\begin{proof}
For part (i), integrate by parts to get
\begin{align*}
\int_{B_R}K(\mb x,\mb z)\pt_{i,\mb z}g(\mb z)\d\mb z=-\int_{B_R}g(\mb z)\pt_{i,\mb z}K(\mb x,\mb z)\d\mb z.
\end{align*}
For part (ii), use the mean value inequality to get~\eqref{E:PART21}. Bound~\eqref{E:PART22} follows from $\pt_{ij}x^k=x^i\delta_j^k-x^j\delta_i^k$.
\end{proof}


\subsubsection{Basic bounds on the pressure term $\mb P$}\label{Pertaining the pressure term}

In order to apply the high order energy method, will need to estimate derivatives of the pressure term $\partial_s^a\pr^b\pt^{\beta}\mb P$ and $\partial^\gamma\mb P$. 
Recall that $\mb P:=\bar w^{-3}\partial_k(\bar w^4(\A^k\J^{-1/3}-I^k))$ by~\eqref{E:P} and therefore we will need to compute the commutators between the operator $\partial_s^a\pr^b\pt^{\beta}$ and $\partial^\gamma$ and the weighted derivative $\bar w^{-3}\partial_k(\bar w^4\cdot)$. Lemma~\ref{pressure-outer-commutator-tangential} deals with the case when no radial derivatives are present,
while Lemma~\ref{pressure-outer-commutator} includes the radial derivatives. 
Lemmas~\ref{pressure-outer-commutator-tangential} and~\ref{pressure-outer-commutator} are necessary to control all the  non-``top-order'' contributions coming from
$\partial_s^a\pr^b\pt^{\beta}\mb P$ and $\partial^\gamma\mb P$ by our energy norms.


\begin{lemma}\label{pressure-outer-commutator-tangential}
For any tensor field $T^k_i$ sufficiently smooth, we have
\begin{align}
\pt^\beta\brac{\bar w^{-3}\partial_k(\bar w^4T^k_i)}
&=\bar w^{-3}\partial_k(\bar w^4\pt^\beta T^k_i)+\sum_{|\beta'|\leq|\beta|-1}\<C\bar w^{-3}\grad(\bar w^4\pt^{\beta'}T)\>
\end{align}
for $i=1,2,3$, where we recall notations defined in Definition \ref{Special notations}.
\end{lemma}


\begin{proof}
We will prove this by induction. Assume this is true for $\beta$, then we have
\begin{align*}
\pt_j\pt^\beta\brac{\bar w^{-3}\partial_k(\bar w^4T^k_i)}
&=\bar w^{-3}\partial_k(\bar w^4\pt_j\pt^\beta T^k_i)-\bar w^{-3}\epsilon_{jkl}\partial_l(\bar w^4\pt^\beta T^k_i)
+\pt_j\sum_{|\beta'|\leq|\beta|-1}\<C\bar w^{-3}\grad(\bar w^4\pt^{\beta'}T)\>\\
&=\bar w^{-3}\partial_k(\bar w^4\pt_j\pt^\beta T^k_i)+\sum_{|\beta'|\leq|\beta|}\<C\bar w^{-3}\grad(\bar w^4\pt^{\beta'}T)\>.
\end{align*}
where we used the commutation relation for $[\pt_j,\partial_k]$ from Lemma \ref{Commutation relations}.
\end{proof}


The use of radial derivatives naturally changes the weighting structure, which is one of the key observations that makes the high-order energy argument possible and goes back 
to~\cite{JaMa2015}.


\begin{lemma}\label{pressure-outer-commutator}
For any tensor field $T^k_i$ sufficiently smooth, we have
\begin{align}
\pr\brac{\bar w^{-c}\partial_k(\bar w^{1+c}T^k_i)}&=\bar w^{-(1+c)}\partial_k(\bar w^{2+c}\pr T^k_i)\nonumber\\
&\quad+(1+c)(T^k_i\pr\partial_k\bar w)+(\partial_k\bar w)\pt_{kj}T^j_i-\bar w\partial_kT^k_i\nonumber\\
\pt_j\brac{\bar w^{-c}\partial_k(\bar w^{1+c}T^k_i)}&=\bar w^{-c}\partial_k(\bar w^{1+c}\pt_jT^k_i)\nonumber\\
&\quad-\epsilon_{jkl}((1+c)(\partial_l\bar w)T^k_i+\bar w\partial_lT^k_i))\nonumber\\
\partial_j\brac{\bar w^{-c}\partial_k(\bar w^{1+c}T^k_i)}&=\bar w^{-(2+c)}\partial_k(\bar w^{3+c}\partial_jT^k_i)\nonumber\\
&\quad+(1+c)(T^k_i\partial_j\partial_k\bar w)+(\partial_j\bar w)\partial_kT^k_i-2(\partial_j T^k_i)\partial_k\bar w\nonumber\\
\pr^d\pt^\beta\brac{\bar w^{-3}\partial_k(\bar w^4T^k_i)}&=
\bar w^{-(3+d)}\partial_k(\bar w^{4+d}\pr^d\pt^\beta T^k_i)\nonumber\\
&\quad+\brac{\sum_{\substack{d'\leq d\\|\beta'|\leq|\beta|-1}}+\sum_{\substack{d'\leq d-1\\|\beta'|\leq|\beta|+1}}}
\<C\omega\pr^{d'}\pt^{\beta'}T\>\nonumber\\
&\quad+\brac{\sum_{\substack{d'\leq d-1\\|\beta'|\leq|\beta|-1}}+\sum_{\substack{d'\leq d-2\\|\beta'|\leq|\beta|}}}\<C\pr^{d'}\pt^{\beta'}\grad T\>\nonumber\\
&\quad+\brac{\sum_{\substack{d'\leq d\\|\beta'|\leq|\beta|-1}}+\sum_{\substack{d'\leq d-1\\|\beta'|\leq|\beta|}}}\<C\bar w\pr^{d'}\pt^{\beta'}\grad T\>\\
\partial^\gamma\brac{\bar w^{-3}\partial_k(\bar w^4T^k_i)}&=
\bar w^{-(3+2|\gamma|)}\partial_k(\bar w^{4+2|\gamma|}\partial^\gamma T^k_i)\nonumber\\
&\quad+\sum_{|\gamma'|\leq|\gamma|-1}
\brac{\<C\omega\partial^{\gamma'}T\>+\<C\omega\partial^{\gamma'}\grad T\>}\nonumber
\end{align}
for any $c\geq 0$ and $i=1,2,3$, where $\omega$ denotes some derivatives of $\bar w$. Here we used notations defined in Definition \ref{Special notations}.
\end{lemma}


\begin{proof}
First note that
\[x^i\pr T^i=x^ix^j\partial_jT^i=r^2\partial_iT^i+x^j(x^i\partial_j-x^j\partial_i)T^i=r^2\partial_iT^i-x^i\pt_{ij}T^j.\]
Using this we have
\begin{align*}
\pr\brac{\bar w^{-c}\partial_k(\bar w^{1+c}T^k_i)}&=
\pr\brac{(1+c)T^k_i\partial_k\bar w+\bar w\partial_kT^k_i}\\
&=(1+c)\brac{(\pr T^k_i)\partial_k\bar w+T^k_i\pr\partial_k\bar w}+(\mb x\cdot\grad\bar w)\partial_kT^k_i
+\bar w\pr\partial_kT^k_i\\
&=(1+c)\brac{(\pr T^k_i)\partial_k\bar w+T^k_i\pr\partial_k\bar w}
+(\mb x\cdot\grad\bar w)r^{-2}(x^k\pr T^k_i+x^k\pt_{kj}T^j_i)\\
&\quad+\bar w\partial_k\pr T^k_i-\bar w\partial_kT^k_i\\
&=(1+c)\brac{(\pr T^k_i)\partial_k\bar w+T^k_i\pr\partial_k\bar w}
+(\partial_k\bar w)(\pr T^k_i+\pt_{kj}T^j_i)\\
&\quad+\bar w\partial_k\pr T^k_i-\bar w\partial_kT^k_i\\
&=\bar w^{-(1+c)}\partial_k(\bar w^{2+c}\pr T^k_i)+(1+c)(T^k_i\pr\partial_k\bar w)
+(\partial_k\bar w)\pt_{kj}T^j_i-\bar w\partial_kT^k_i\\
\pt_j\brac{\bar w^{-c}\partial_k(\bar w^{1+c}T^k_i)}&=\bar w^{-c}\partial_k(\bar w^{1+c}\pt_jT^k_i)-\epsilon_{jkl}\bar w^{-c}\partial_l(\bar w^{1+c}T^k_i)\\
\partial_j\brac{\bar w^{-c}\partial_k(\bar w^{1+c}T^k_i)}
&=
\partial_j\brac{(1+c)T^k_i\partial_k\bar w+\bar w\partial_kT^k_i}\\
&=(1+c)\brac{(\partial_j T^k_i)\partial_k\bar w+T^k_i\partial_j\partial_k\bar w}+(\partial_j\bar w)\partial_kT^k_i+\bar w\partial_j\partial_kT^k_i\\
&=\bar w^{-(2+c)}\partial_k(\bar w^{3+c}\partial_j T^k_i)+(1+c)(T^k_i\partial_j\partial_k\bar w)+(\partial_j\bar w)\partial_kT^k_i-2(\partial_j T^k_i)\partial_k\bar w
\end{align*}
where we used commutation relations from Lemma \ref{Commutation relations}. The final two formulas can be proven by induction.
\end{proof}


The next lemma deals with the terms we get when we apply $\partial_s^a\pr^b\pt^{\beta}$ or $\partial^\gamma$ to $\A\J^{-1/3}-I$.

\begin{lemma}\label{pressure structure lemma}
Let
\begin{align}
T:=\A\J^{-1/3}-I\label{E:T}.
\end{align}
Recall notations defined in Definition \ref{Special notations}. For $a\leq 0$ and $|\gamma|>0$, we have
\begin{align}
\partial_\bullet T&=T_T[\partial_\bullet\grad\theta], \\
\partial_s^a\pr^d\pt^\beta T
&=T_T[\partial_s^a\pr^b\pt^\beta\grad\theta]+T_{R:a,\beta,d}\\
\partial^\gamma T
&=T_T[\partial^\gamma\grad\theta]+T_{R:\gamma}.
\end{align}
where
\begin{align}
T_T[M]^k&:=-\J^{-1/3}\brac{\A^k_m\A^l+{1\over 3}\A^k\A^l_m}M^m_l,\qquad\qquad k=1,2,3\label{E:T_T}\\
T_{R:a,\beta,d}&:=\J^{-1/3}\sum^{a+d+|\beta|}_{c=2}\sum_{\substack{\sum_{i=1}^c(a_i,d_i,\beta_i)=(a,d,\beta)\\|a_i|+|d_i|+|\beta_i|>0}}\<C\>\<\A\>^{1+c}\prod_{i=1}^c\<\partial_s^{a_i}\pr^{d_i}\pt^{\beta_i}\grad\bs\theta\>\\
T_{R:\gamma}&:=\J^{-1/3}\sum^{|\gamma|}_{c=2}\sum_{\substack{\sum_{i=1}^c\gamma_i=\gamma\\|\gamma_i|>0}}\<C\>\<\A\>^{1+c}\prod_{i=1}^c\<\partial^{\gamma_i}\grad\bs\theta\>.
\end{align}
We write $T_{R:\beta,d}:=T_{R:0,\beta,d}$.
\end{lemma}
\begin{proof}
Applying Lemma \ref{derivative-formula} we get that
\begin{align*}
\partial_\bullet(\A^k\J^{-1/3}-I^k)
&=-\J^{-1/3}\A^k_m\A^l\partial\partial_l\theta^m-{1\over 3}\J^{-1/3}\A^k\A^l_m\partial_\bullet\partial_l\theta^m\\
&=-\J^{-1/3}\brac{\A^k_m\A^l+{1\over 3}\A^k\A^l_m}\partial_\bullet\partial_l\theta^m. 
\end{align*}
Hence $\partial_\bullet T^k=T_T[\partial_\bullet\grad\theta]^k$. By repeated application of this we get the next two formulas. 
\end{proof}

We have from \eqref{E:EP in self-similar}
\begin{align*}
\partial_s^a\pr^b\pt^\beta\mb P=\partial_s^a\pr^b\pt^\beta\brac{\bar w^{-3}\partial_k(\bar w^4T^k)}
\end{align*}
One can see from the last few lemmas (we will prove this properly in Section \ref{Estimating the non-linear part of the pressure term} and \ref{Estimating the non-linear part of the pressure term - linear}) that the leading order term is in fact $\mb P_d\partial_s^a\pr^b\pt^\beta\bs\theta$, where
\begin{align}
\mb P_d\bs\theta&:=\bar w^{-3-d}\partial_k(\bar w^{4+d}T_T[\grad\bs\theta]^k)\\
&\ =-\bar w^{-(3+d)}\partial_k\brac{\bar w^{4+d}\brac{\A^k_m\A^l+{1\over 3}\A^k\A^l_m}\partial_l\theta^m}
\end{align}
Let $\mb P_{d,L}$ be the linear part of $\mb P_d$, i.e.
\begin{align}
\mb P_{d,L}\bs\theta&:=-\bar w^{-(3+d)}\partial_k\brac{\bar w^{4+d}\brac{I^k_mI^l+{1\over 3}I^kI^l_m}\partial_l\theta^m}\nonumber\\
&=-{1\over 3\bar w^{3+d}}\grad(\bar w^{4+d}\grad\cdot\bs\theta)-{1\over\bar w^{3+d}}\partial_k(\bar w^{4+d}\grad\theta^k) \label{E:PDLDEF}
\end{align}
In doing energy estimates, the term $\<\partial_s^a\pr^b\pt^\beta\mb P,\partial_s^a\pr^b\pt^\beta\bs\theta\>$ and $\<\partial_s^a\pr^b\pt^\beta\mb P,\partial_s^{a+1}\pr^b\pt^\beta\bs\theta\>$ will arise. Using the lemmas in this subsection, we will show in Section \ref{Estimating the non-linear part of the pressure term} and \ref{Estimating the non-linear part of the pressure term - linear} that $\mb P$ here can be reduced to $\mb P_{d,L}$ modulo remainder terms that can be estimated. The following identity is needed for that purpose.

\begin{lemma}\label{P-inner-product}
For any vector field $\bs\theta_1,\bs\theta_2$ sufficiently smooth we have
\begin{align}
&\<\mb P_d\bs\theta_1,\bs\theta_2\>_{3+d}\nonumber\\
&=\int\bigg((\cApar_m\bs\theta_1)\cdot(\cApar_m\bs\theta_2)+{1\over 3}(\div_{\A}\bs\theta_1)(\div_{\A}\bs\theta_2)
-{1\over 2}[\curl_{\A}\bs\theta_1]^m_j[\curl_{\A}\bs\theta_2]^m_j\bigg)\J^{-1/3}\bar w^{4+d}\d\mb x\\
&\<\mb P_{d,L}\bs\theta_1,\bs\theta_2\>_{3+d}\nonumber\\
&=\int\brac{(\partial_m\bs\theta_1)\cdot(\partial_m\bs\theta_2)+{1\over 3}(\div\bs\theta_1)(\div\bs\theta_2)-{1\over 2}[\curl\bs\theta_1]^m_j[\curl\bs\theta_2]^m_j}\bar w^{4+d}\d\mb x
\end{align}
\end{lemma}
\begin{proof}
We have
\begin{align*}
\<\mb P_d\bs\theta_1,\bs\theta_2\>_{3+d}&=
\int\brac{(\cApar_l\theta_1^m)\cdot(\cApar_m\theta_2^l)+{1\over 3}(\div_{\A}\bs\theta_1)(\div_{\A}\bs\theta_2)}\J^{-1/3}\bar w^{4+d}\d\mb x.
\end{align*}
We are done for $\mb P_d$ noting that
\begin{align*}
[\curl_{\A}\bs\theta_1]^m_j[\curl_{\A}\bs\theta_2]^m_j&=(\cApar_j\theta_1^m-\cApar_m\theta_1^j)(\cApar_j\theta_2^m-\cApar_m\theta_2^j)\\
&=2(\cApar_j\theta_1^m)(\cApar_j\theta_2^m)-2(\cApar_j\theta_1^m)(\cApar_m\theta_2^j).
\end{align*}
Similarly for $\mb P_{d,L}$. 
\end{proof}

\section{Nonradial stability of self-similarly expanding Goldreich-Weber stars}\label{self-similar GW}

\subsection{Formulation and statement of the result}

\subsubsection{Equation in self-similar coordinates}



In this section we will take the enthalpy $\bar w$ to be the profile associated with the self-similarly expanding GW star from Definition~\ref{self-similar GW def}. To study the stability of self-similarly expanding GW stars, we want to write our variables as a perturbation from the model GW star. To that end we will use the rescaled variable $\bs\xi$ (equation \eqref{E:KSIDEF}) introduced in Section \ref{Notation} adapted to the expanding background profile and also write the problem in self-similar time variables. We introduce the self-similar time coordinate $s$ adapted to the expanding profile via
\[
\frac{ds}{dt}=\l(t)^{-\frac32}.
\]
We then have the following change of coordinate formula $\partial_t=\lambda^{-3/2}\partial_s$. The condition $\ddot{\lambda}\lambda^2=\delta$ \eqref{E:GW1} becomes
\begin{align}
\delta=\lambda^{1/2}\partial_s(\lambda^{-3/2}\partial_s\lambda)={\partial_s^2\lambda\over\lambda}-{3\over 2}{(\partial_s\lambda)^2\over\lambda^2}
=\partial_s\brac{\partial_s\lambda\over\lambda}-{1\over 2}{(\partial_s\lambda)^2\over\lambda^2}=-{1\over 2}\b^2\label{E:delta and b relation}
\end{align}
where
\begin{align}\label{E:BDEF}
\b  := -\frac{\partial_s\l}{\l} = -\sqrt{2|\delta|}<0.
\end{align}
Then the Euler-Poisson equations~\eqref{E:MOMLAGR} becomes
\begin{align*}
\mb 0
&=\partial_t\mb v+(f_0J_0)^{-1}\partial_k(A^k(f_0J_0)^{4/3}J^{-1/3})+A\grad\psi\\
&=\lambda^{-3/2}\partial_s(\lambda^{-3/2}\partial_s(\lambda\bs\xi))+\lambda^{-2}(f_0J_0)^{-1}\partial_k(\A^k(f_0J_0)^{4/3}\J^{-1/3})+\lambda^{-2}\A\grad\Phi
\end{align*}
Times the equation by $\lambda^2$ we get
\begin{align*}
\mb 0&=\lambda^{1/2}\partial_s(\lambda^{-3/2}\partial_s(\lambda\bs\xi))+(f_0J_0)^{-1}\partial_k(\A^k(f_0J_0)^{4/3}\J^{-1/3})+\A\grad\Phi\\
&=\brac{\partial_s^2\bs\xi+{1\over 2}{\partial_s\lambda\over\lambda}\partial_s\bs\xi+\brac{{\partial_s^2\lambda\over\lambda}-{3\over 2}{(\partial_s\lambda)^2\over\lambda^2}}\bs\xi}+(f_0J_0)^{-1}\partial_k(\A^k(f_0J_0)^{4/3}\J^{-1/3})
+\A\grad\Phi\\
&=\brac{\partial_s^2\bs\xi-{1\over 2}\b\partial_s\bs\xi+\delta\bs\xi}+(f_0J_0)^{-1}\partial_k(\A^k(f_0J_0)^{4/3}\J^{-1/3})+\A\grad\Phi
\end{align*}
So the Euler-Poisson equations in terms of $\bs\xi$~\eqref{E:KSIDEF} is:
\begin{align}\label{E:KSIEQUATION}
\partial_s^2\bs\xi-{1\over 2}\b\partial_s\bs\xi+\delta\bs\xi+\frac1{f_0J_0}\partial_k(\A^k(f_0J_0)^{4/3}\J^{-1/3})+\A\grad\Phi = \mb 0.
\end{align}

The self-similarly expanding GW-star is a particular $s$-independent solution of~\eqref{E:KSIEQUATION} of the form $\bs\xi(\mb x)\equiv \mb x$ and $f_0=\bar w^3$.
Before formulating the stability problem, we must first make the use of the labelling gauge freedom and fix the choice of the initial enthalpy $(f_0J_0)^{1/3}$ for the general perturbation to be exactly
identical to the background enthalpy $\bar w$, i.e. we set
\begin{align}\label{E:GAUGECHOICE}
(f_0J_0)^{1/3} = \bar w \ \qquad \text{ on }\ B_R(\mb 0).
\end{align}
Equation~\eqref{E:GAUGECHOICE} can be re-written in the form $\rho_0\circ\bs\eta_0 \det[\grad\bs\eta_0]=\bar w^3$ on the initial domain $B_R(\mb 0)$. By a result of 
Dacorogna-Moser~\cite{DM} and similarly to~\cite{HaJa2018-1,HaJa2016-2} there exists a choice of an initial bijective map $\bs\eta_0:B_R(\mb 0)\to\Omega(\mb 0)$ so that~\eqref{E:GAUGECHOICE} holds true. The gauge fixing condition~\eqref{E:GAUGECHOICE}
is necessary as it constrains the freedom to arbitrary relabel the particles at the initial time.
%


\begin{lemma}[Euler-Poisson in self-similar coordinate]
With respect to the self-similarly expanding profile $(\lambda,\bar w)$ from Definition~\ref{self-similar GW def}, the perturbation $\bs\theta$ defined in~\eqref{E:THETADEF} formally solves 
\begin{align}\label{E:EP in self-similar}
\partial_s^2\bs\theta-{1\over 2}\b\partial_s\bs\theta+\delta\bs\theta+\mb P+\mb G=\mb 0,
\end{align}
where the nonlinear pressure operator $\mb P$ and the nonlinear gravity operator $\mb G$ are defined in \eqref{E:P} and \eqref{E:G}.
\end{lemma}


\begin{proof}
Recall that the GW-enthalpy satisfies
\begin{align}
\mb 0=\delta\mb x+4\grad\bar w+\grad\K\bar w^3
\label{E:equation for bar-w}
\end{align}
Using the gauge condition~\eqref{E:GAUGECHOICE},
the momentum equation~\eqref{E:KSIEQUATION} becomes
\begin{align*}
\bar w^3\brac{\partial_s^2\bs\theta-{1\over 2}\b\partial_s\bs\theta+\delta\bs\theta}+\partial_k(\bar w^4(\A^k\J^{-1/3}-I^k))+\bar w^3(\A\grad\Phi-\grad\mathcal{K}\bar w^3)=\mb 0.
\end{align*}
Hence, we can write the momentum equation as
\begin{align*}
\mb 0 
& =\partial_s^2\bs\theta-{1\over 2}\b\partial_s\bs\theta+\delta\bs\theta+\underbrace{\bar w^{-3}\partial_k(\bar w^4(\A^k\J^{-1/3}-I^k))}_{\mb P}+\underbrace{\A\grad\Phi-\grad\mathcal{K}\bar w^3}_{\mb G}.
\end{align*}
\end{proof}


\subsubsection{Total energy and momentum}


Next we will give expressions for the total momentum and energy in terms of $\bs\theta$. We will write the expressions in a way that separates the linear and non-linear terms of $\bs\theta$ clearly. To that end, 
we first derive the following identity.


\begin{lemma}\label{L:ENERGYAUX}
For any $\bs\theta$ sufficiently smooth we have the identity
\begin{align*}
\int_{B_R}\brac{\bar w^4\grad\cdot\bs\theta+{1\over 2}\b^2\bar w^3\bs\theta\cdot\mb x-{1\over 2}\bar w^3(\K_{\bs\xi}^{(1)}\bar w^3)}\d\mb x=0,
\end{align*}
where 
\begin{align}
(\K_{\bs\xi}^{(1)}g)(\mb x)&:=\int{(\mb x-\mb z)\cdot(\bs\theta(\mb x)-\bs\theta(\mb z))\over|\mb x-\mb z|^3}g(\mb z)\d\mb z,\label{E:KKSIONE}
\end{align}
and we recall~\eqref{E:BDEF}.
\end{lemma}
\begin{proof}
We have
\begin{align*}
\int\bar w(\mb x)^3(\K_{\bs\xi}^{(1)}\bar w^3)(\mb x)\d\mb x&=\int\int\bar w(\mb x)^3\bar w(\mb z)^3{(\mb x-\mb z)\cdot(\bs\theta(\mb x)-\bs\theta(\mb z))\over|\mb x-\mb z|^3}\d\mb z\d\mb x\\
&=2\int\int\bar w(\mb x)^3\bar w(\mb z)^3{(\mb x-\mb z)\cdot\bs\theta(\mb x)\over|\mb x-\mb z|^3}\d\mb z\d\mb x\\
&=2\int\bar w(\mb x)^3\bs\theta(\mb x)\cdot\int\bar w(\mb z)^3{\mb x-\mb z\over|\mb x-\mb z|^3}\d\mb z\d\mb x
=2\int\bar w^3\bs\theta\cdot\grad\K\bar w^3\d\mb x
\end{align*}
Also, using (\ref{E:equation for bar-w}), we have
\begin{align*}
\int\brac{\bar w^4\grad\cdot\bs\theta-w^3\bs\theta\cdot\grad\K\bar w^3}\d\mb x
&=\int\brac{-\bs\theta\cdot\grad\bar w^4-w^3\bs\theta\cdot\grad\K\bar w^3}\d\mb x\\
&=\int\delta\bar w^3\bs\theta\cdot\mb x\d\mb x
=-{1\over 2}\b^2\int\bar w^3\bs\theta\cdot\mb x\d\mb x. 
\end{align*}
\end{proof}

With this identity we can now derive the expression for the momentum and the energy in terms of $\bs\theta$.


\begin{lemma}[Momentum and energy in self-similar coordinate]\label{Momentum and energy in self-similar coordinate}
Fix a $\delta\in[\tilde\delta,0)$.
In self-similar Lagrangian coordinates introduced above, the total momentum~\eqref{E:momentum in Euler} and energy~\eqref{E:energy in Euler} 
are respectively denoted by
\begin{align*}
\mb W_\delta[\bs\theta](s) &:=\mb W_\delta(s,\bs\theta(s),\partial_s\bs\theta(s)),\\
E_\delta[\bs\theta](s)& :=E_\delta(s,\bs\theta(s),\partial_s\bs\theta(s)),
\end{align*}
where
\begin{align*}
\mb W_\delta(s,\bs\theta(s),\partial_s\bs\theta(s))
&=\bar{\mb W}+{1\over\lambda(s)^{1/2}}\int(\partial_s\bs\theta(s)-\b\bs\theta(s))\bar w^3\d\mb x,\\
E_\delta(s,\bs\theta(s),\partial_s\bs\theta(s))
&=\bar E+{1\over\lambda(s)}\int\brac{{1\over 2}\bar w^3\brac{|\partial_s\bs\theta(s)-\b\bs\theta(s)|^2-\b\mb x\cdot\brac{2\partial_s\bs\theta(s)-{5\over 2}\b\bs\theta(s)}}}\d\mb x\\
&\quad+{1\over\lambda(s)}\int\brac{3\bar w^4\brac{\J(s)^{-{1\over 3}}-1+{1\over 3}\grad\cdot\bs\theta(s)}+{1\over 2}\bar w^3(\K_{\bs\xi}-\K-\K_{\bs\xi}^{(1)})\bar w^3(s)}\d\mb x, 
\end{align*}
and 
$\bar{\mb W}:=\mb W_\delta[\mb 0]=\mb 0$ and $\bar E:=E_\delta[\mb 0]=0$ are respectively the momentum and energy of the GW star given by~Definition~\ref{self-similar GW def}.
\end{lemma}


\begin{proof}
We clearly have
\begin{align*}
\mb W[\rho,\mb u]&
=\int fJ\partial_t\bs\eta\;\d\mb x
=\int f_0J_0\partial_t\bs\eta\;\d\mb x
=\int\bar w^3\lambda^{-3/2}\partial_s(\lambda(\mb x+\bs\theta))\d\mb x\\
&=\bar{\mb W}+{1\over\lambda^{1/2}}\int(\partial_s\bs\theta-\b\bs\theta)\bar w^3\d\mb x,\\
E[\rho,\mb u]&
=\int\brac{{1\over 2}f|\partial_t\bs\eta|^2+3f^{4\over 3}+{1\over 2}f\psi}J\d\mb x
=\int\brac{{1\over 2}f_0J_0|\partial_t\bs\eta|^2+3J^{-{1\over 3}}(f_0J_0)^{4\over 3}+{1\over 2}f_0J_0\psi}\d\mb x\\
&=\int\brac{{1\over 2}\bar w^3|\lambda^{-3/2}\partial_s(\lambda(\mb x+\bs\theta))|^2+{3\over\lambda}\J^{-{1\over 3}}\bar w^4+{1\over 2\lambda}\bar w^3\Phi}\d\mb x\\
&=\bar E+{1\over\lambda}\int\bigg({1\over 2}\bar w^3(|\partial_s\bs\theta-\b\bs\theta|^2-2\b\mb x\cdot(\partial_s\bs\theta-\b\bs\theta))
+3(\J^{-{1\over 3}}-1)\bar w^4+{1\over 2}\bar w^3(\K_{\bs\xi}-\K)\bar w^3\bigg)\d\mb x\\
&=\bar E+{1\over\lambda}\int\brac{{1\over 2}\bar w^3\brac{|\partial_s\bs\theta-\b\bs\theta|^2-\b\mb x\cdot\brac{2\partial_s\bs\theta-{5\over 2}\b\bs\theta}}}\d\mb x\\
&\quad+{1\over\lambda}\int\brac{3\bar w^4\brac{\J^{-{1\over 3}}-1+{1\over 3}\grad\cdot\bs\theta}+{1\over 2}\bar w^3(\K_{\bs\xi}-\K-\K_{\bs\xi}^{(1)})\bar w^3}\d\mb x.
\end{align*}
When re-writing $E[\rho,\mb u]$ above, we have used Lemma~\ref{L:ENERGYAUX} and~(\ref{E:equation for bar-w}).
\end{proof}

\begin{remark}\label{Momentum remark 1}
If instead we consider the GW solutions $\bs\eta_{\mb p}$ translated at constant velocity $\mb p_1$ as in Remark \ref{Momentum remark}, we will get $\bar{\mb W}=M\mb p_1$ and $\bar E={1\over 2}M|\mb p_1|^2$ instead of $\bar{\mb W}=\mb 0$ and $\bar E=0$.
\end{remark}

\subsubsection{High-order energies and the main theorem}\label{sec:2.1.3}



We now introduce high-order weighted Sobolev norm that measures the size of the deviation $\bs\theta$
{\em without} time derivatives. Recall the notation in Section \ref{Notation}. Assuming that $(s,{\bf y})\mapsto \bs\theta(s,\mb y)$ is a sufficiently smooth field, for any $n\in\mathbb N_0$ and $s\ge0$ we let
\begin{align}
Z_n(s)&:=\sum_{|\beta|+b\leq n}\|\pr^b\pt^{\beta}\bs\theta\|_{3+b}+\sum_{c\leq n}\|\grad^c\bs\theta\|_{3+2c}\\
\Z_n(s)&:=\sup_{\tau\in[0,s]}Z_n(\tau). \label{E:ZNTOTAL}
\end{align}
Next we define energy norms {\em with} time-derivatives - they will be a basis of our high-order energy method explained in Section~\ref{S:ENERGYESTIMATES}.
\begin{align*}
S_n(s)&:=\sum_{\substack{a+|\beta|+b\leq n\\a>0}}\brac{\|\partial_s^{a+1}\pr^b\pt^{\beta}\bs\theta\|_{3+b}^2+\|\partial_s^a\pr^b\pt^{\beta}\bs\theta\|_{3+b}^2+\|\partial_s^a\grad\pr^b\pt^{\beta}\bs\theta\|_{4+b}^2}\\
S_{n,c}(s)&:=\sum_{\substack{a+|\beta|+b\leq n\\a>0\\|\beta|+b\leq c}}\brac{\|\partial_s^{a+1}\pr^b\pt^{\beta}\bs\theta\|_{3+b}^2+\|\partial_s^a\pr^b\pt^{\beta}\bs\theta\|_{3+b}^2+\|\partial_s^a\grad\pr^b\pt^{\beta}\bs\theta\|_{4+b}^2}\\
S_{n,c,d}(s)&:=\sum_{\substack{a+|\beta|+b\leq n\\a>0\\|\beta|+b\leq c\\b\leq d}}\brac{\|\partial_s^{a+1}\pr^b\pt^{\beta}\bs\theta\|_{3+b}^2+\|\partial_s^a\pr^b\pt^{\beta}\bs\theta\|_{3+b}^2+\|\partial_s^a\grad\pr^b\pt^{\beta}\bs\theta\|_{4+b}^2}\\
Q_n(s)&:=\sum_{\substack{a+c\leq n+1\\a>0}}\|\partial_s^a\grad^c\bs\theta\|_{3+2c}^2\\
Q_{n,d}(s)&:=\sum_{\substack{a+c\leq n+1\\a>0\\c\leq d+1}}\|\partial_s^a\grad^c\bs\theta\|_{3+2c}^2
\end{align*}
Note that $S_{n,n,n}=S_n$. We will also use the convention that $S_{n,-1}=0$ etc. 

\begin{remark}
The indexing above is needed to describe qualitative differences between taking the time, the angular, and the radial derivatives. We shall need this distinction
to later close our estimates via a delicate induction argument in high-order spaces, which not only depends on the number of derivatives, but also on the order in which the derivatives are taken.  
\end{remark}


We define the total instant energy via
\begin{align}\label{E:ENDEF}
E_n:=S_n+Q_n.
\end{align}
We shall run the energy identity using $E_n$; energies $S_n$ and $Q_n$ will be used  for high-order estimates near the vacuum boundary and near the origin respectively.  
In particular, the control afforded by $Q_n$ is stronger near the origin, while $S_n$ is stronger near the boundary. Finally we define
\begin{align}
\S_\bullet(s)&:=\sup_{\tau\in[0,s]}S_\bullet(\tau)+\int_0^sS_\bullet(\tau)\d\tau,\label{E:SBULLET}\\
\Q_\bullet(s)&:=\sup_{\tau\in[0,s]}Q_\bullet(\tau)+\int_0^sQ_\bullet(\tau)\d\tau, \label{E:QBULLET}\\
\E_n(s)&:=\sup_{\tau\in[0,s]}E_n(\tau)+\int_0^sE_n(\tau)\d\tau, \label{E:TOTALNORM}
\end{align}
where $\bullet$ stands for indices of the form $n$; $n,c$; $n,c,d$ in~\eqref{E:SBULLET}, and of the form $n$; $n,c$ in~\eqref{E:QBULLET}.
The norms~\eqref{E:SBULLET}--\eqref{E:TOTALNORM} will play the role of the ``left hand side'' in the high-order energy identities.

\begin{remark}
We emphasise that the higher order energies $E_n$ we defined (always with a subscript $n\in\N_0$) are different from the total conserved energy $E$ (and $E_\delta$) defined in \eqref{E:energy in Euler}. Where no confusion arises, we will refer to both as ``energy''.
\end{remark}

In this section, we make the following a priori assumption:
\begin{flalign}
\text{\noindent\textbf{A priori assumption:} $E_\bullet,Z_\bullet\leq\epsilon$ where $\epsilon>0$ is some small constant.} && \label{A priori assumption}
\end{flalign}

%


We now state our main theorem.


\begin{theorem}[Nonlinear stability of GW stars]\label{T:MAIN}
Let $n\geq 21$.
 There exists $\tilde\delta\le \delta^*<0$ such that for any $\delta\in(\delta^*,0)$ the associated GW expanding star from Definition~\ref{self-similar GW def} is codimension-$4$ 
 nonlinearly stable in the class of irrotational perturbations. More precisely, there exists an $\epsilon_0>0$ such that for any initial data $(\bs\theta(0),\partial_s\bs\theta(0))$ satisfying
\begin{align}
E_n(0)+Z_n(0)^2&\leq\epsilon_0\\
\mb W_\delta(0,\bs\theta(0),\partial_s\bs\theta(0))&=\mb W_\delta[\mb 0]=:\bar{\mb W}=\mb 0\label{initial momentum condition}\\
E_\delta(0,\bs\theta(0),\partial_s\bs\theta(0))&=E_\delta[\mb 0]=:\bar E=0\label{initial energy condition}\\
\curl_A\partial_t\bs\eta(0)&=\mb 0,\label{initial irrotational condition}
\end{align}
the associated solution $s\mapsto( \bs\theta(s,\ph),  \partial_s\bs\theta(s,\ph) )$ to \eqref{E:EP in self-similar} exists for all $s\geq 0$ and is unique in the class of all data with finite norm $E_n+Z_n^2$. Moreover, there exists a constant $C>0$ such that
\begin{align*}
E_n(s)+Z_n(s)^2\leq C\epsilon_0\qquad\text{for all}\qquad s\geq 0,
\end{align*}
and $E_n(s)$ decays exponentially fast in $s$.
\end{theorem}


Note that condition \eqref{initial irrotational condition} is the Lagrangian statement that the fluid velocity is irrotational.
The momentum and energy constraints \eqref{initial momentum condition}--\eqref{initial energy condition} define the
codimension-$4$ ``manifold" of initial data.

Heuristically speaking, since momentum and energy are conserved quantities for the Euler-Poisson system, it is necessary that our perturbation does not alter the momentum and energy if our background solution were to be the right asymptotic-in-time ``limit''. Indeed, the momentum and energy constraints \eqref{initial momentum condition}--\eqref{initial energy condition} are necessary in our stability analysis due to the presence of growing modes of the linearised operator in self-similar coordinates induced by the conservation of the energy and momentum. However if our perturbation does alter the momentum $\mb W$ and energy $E$ away from $\mb 0$ and $0$ such that $E={1\over 2}|\mb W|^2/M$, our proof can be easily adapted to show that it still leads to global existence with the solution staying close to a GW star for all time, but one translated at constant velocity $\mb p_1$ with $\mb W=M\mb p_1$ and $E={1\over 2}M|\mb p_1|^2$ as described in Remarks \ref{Momentum remark} and \ref{Momentum remark 1}. In this sense the ``manifold'' of GW-solutions is codimension-$1$ nonlinearly stable in the class of irrotational perturbations, even though each individual GW-star is only codimension-$4$ stable.
%
In particular, given any initial data $(\rho_0,\mb u_0)$ such that $E[\rho_0,\mb u_0]={1\over 2}|\mb W[\rho_0,\mb u_0]|^2/M[\rho_0,\mb u_0]$, we can change our frame of reference and subtract a constant velocity of $\mb p_1=\mb W[\rho_0,\mb u_0]|/M[\rho_0,\mb u_0]$ from $\mb u_0$ to obtain
\begin{align*}
\mb W[\rho_0,\mb u_0-\mb p_1]&=\mb W[\rho_0,\mb u_0]-M[\rho_0,\mb u_0]\mb p_1=\mb 0\\
E[\rho_0,\mb u_0-\mb p_1]&=E[\rho_0,\mb u_0]-\int_{\R^3}\rho_0\mb u_0\cdot\mb p_1\d\mb x+{1\over 2}\int_{\R^3}\rho_0|\mb p_1|^2\d\mb x\\
&=E[\rho_0,\mb u_0]-\mb p_1\cdot\mb W[\rho_0,\mb u_0]+{1\over 2}|\mb p_1|^2M[\rho_0,\mb u_0]=0
\end{align*}
So in this new frame of reference, the constraints \eqref{initial momentum condition} and \eqref{initial energy condition} are satisfied. 

To formalise this, note that for any $\mb p_1\in\R^3$, $\bs\theta=\mb p_1t\lambda(t)^{-1}$ is a global-in-time solution to \eqref{E:EP in self-similar} which corresponds to a Lagrangian description of a GW-star translated by a constant velocity.
Then, as a corollary of Theorem~\ref{T:MAIN} we have the following result.

\begin{corollary}\label{C:MAIN}
Let $n\geq 21$. There exists $\tilde\delta\le \delta^*<0$ such that for any $\delta\in(\delta^*,0)$, the ``manifold"  of GW-stars $(\bar\rho_{\mb p},\bar{\mb u}_{\mb p})$, $\mb p=\mb p_1t$, $\mb p_1\in\mathbb R^3$ (from Remark~\ref{Momentum remark}) is codimension-$1$ nonlinearly stable in the class of irrotational perturbations. More precisely, for given any initial data $(\tilde{\bs\theta}(0),\partial_s\tilde{\bs\theta}(0))$ define
\begin{align*}
\bs\theta_0&=\tilde{\bs\theta}(0)\\
(\partial_s\bs\theta)_0&=\partial_s\tilde{\bs\theta}(0)-\mb p_1,
\end{align*}
where 
\[
\mb p_1={W_\delta(0,\tilde{\bs\theta}(0),\partial_s\tilde{\bs\theta}(0))\over M[\bar w^3]}.
\]
Then, there exists an $\epsilon_0>0$ such that for any initial data $(\tilde{\bs\theta}(0),\partial_s\tilde{\bs\theta}(0))$ such that
\begin{align*}
(E_n+Z_n^2)[\bs\theta_0,(\partial_s\bs\theta)_0]&\leq\epsilon_0\\
E_\delta(0,\tilde{\bs\theta}(0),\partial_s\tilde{\bs\theta}(0))&={1\over 2}|\mb W_\delta(0,\tilde{\bs\theta}(0),\partial_s\tilde{\bs\theta}(0))|^2/M[\bar w^3]\\
\curl_A\partial_t\bs\eta(0)&=\mb 0
\end{align*}
the associated solution $s\mapsto(\tilde{\bs\theta}(s,\ph),  \partial_s\tilde{\bs\theta}(s,\ph))$ to \eqref{E:EP in self-similar} exists for all $s\geq 0$ and is unique in the class of all data with finite norm $(E_n+Z_n^2)[\tilde{\bs\theta}]$. Moreover, there exists a constant $C>0$ such that
\begin{align*}
(E_n+Z_n^2)[\bs\theta](s)\leq C\epsilon_0\qquad\text{for all}\qquad s\geq 0,
\end{align*}
where $\bs\theta=\tilde{\bs\theta}-\mb p_1t\lambda(t)^{-1}$, and $E_n[\bs\theta](s)$ decays exponentially fast in $s$.
\end{corollary}


\begin{remark}\label{Our goal is not to optimise}
Our goal is not to optimise the number $n$ of derivatives in our spaces. As usual, the size of $n$ is conditioned by the Hardy-Sobolev type embeddings, which allow us to bound
the $L^\infty$-norms of contributions with less than $\lfloor\frac n2\rfloor$ derivatives by $\bar w^k$-weighted Sobolev norms.
\end{remark}


\begin{remark}
The subclass of expanding GW-stars with non-zero total energy (in the frame of reference of $\mb 0$ momentum) consists of stars that expand at a linear rate in time, i.e. {\em not} at the self-similar rate considered above. This
problem leads to a stronger damping effect which allows the ``Euler part" of the flow to  dominate the dynamics. The stability of such GW-stars is the content of Section \ref{linear GW}.
\end{remark}




{\bf Local-in-time well-posedness.}
The presence of vacuum is known to pose challenges 
in the well-posedness theory for compressible fluid flows. To develop a 
satisfactory local existence and uniqueness theory, one needs to impose an additional assumption on the initial data - the so-called physical vacuum condition~\eqref{E:PHYSICALVACUUM}. In the works of Jang and Masmoudi~\cite{JaMa2015} and independently Coutand and Shkoller~\cite{CoSh2012} the local well-posedness for the compressible Euler equations was shown in the Lagrangian coordinates (for a more recent treatment in Eulerian coordinates see~\cite{IfTa2020}).
From the point of view of regularity theory, gravity represents a lower order term, so the techniques from~\cite{JaMa2015,CoSh2012} can be adapted to obtain a local-in-time well-posedness result for the free boundary EP-system~\cite{Jang2014,GuLe2016,LXZ,HaJa2017}. In particular, 
a simple adaptation of the methods in~\cite{JaMa2015,HaJa2017} yields the following local well-posedness result in the weighted high-order energy space $E_n+Z_n^2$.


\begin{theorem}[Local well-posedness]\label{T:LOCAL}
Let $n\geq 21$. Then for any given initial data $(\bs\theta(0), \partial_s \bs\theta(0))$ such that $E_n(0)+Z_n(0)^2<\infty$, there exist some $T>0$ and a unique solution $(\bs\theta, \partial_s\bs\theta):[0,T]\times B_R\to\R^3\times \R^3$ to \eqref{E:EP in self-similar} such that $E_n(s)+Z_n(s)^2\leq2(E_n(0)+Z_n(0)^2)$ for all $s\in[0,T]$.
\end{theorem}


Theorem~\ref{T:LOCAL} is a starting point for the continuity argument that will culminate in the proof of Theorem~\ref{T:MAIN}.

\subsubsection{Proof strategy}


The basic idea behind the global existence in Theorem~\ref{T:MAIN} is the presence of the damping term $-\frac12\b\pa_s\bs\theta$ in~\eqref{E:EP in self-similar}, which clearly
suggests a stabilising mechanism in the problem. Such a term appears as a direct consequence of the expanding character of the underlying GW-motion (and it would be of the opposite sign if we were linearising about a collapsing GW-star). This stabilisation effect was first exhibited in~\cite{HaJa2018-1} where the purely radial version of Theorem~\ref{T:MAIN} was established. 

Since the problem features the vacuum free boundary satisfying the physical vacuum condition~\eqref{E:PHYSICALVACUUM}, we use weighted high-order energy spaces introduced by Jang and Masmoudi~\cite{JaMa2015}. The key idea to overcome a possible loss of derivatives is to introduce increasing powers of $\bar w$ into the function spaces, as we increase the number of radial derivatives, but not the tangential ones. In particular,
 the proof of the main result is based on a high-order energy method which necessitates commuting the 
equation~\eqref{E:EP in self-similar} with operators of the form $\pa_s^a\pr^n\pt^\beta$. To understand the energy contribution from the combined
 pressure and gravity term $\mb P+\mb G$ (see~\eqref{E:P}--\eqref{E:G}), we must linearise~\eqref{E:EP in self-similar}. As shown in Lemma~\ref{L:LINEAREP}, this linearisation reads
\begin{align}\label{E:LINDYNINTRO}
\partial_s^2\bs\theta-{1\over 2}\b\partial_s\bs\theta+\mb L\bs\theta = \mb 0,
\end{align}
where the linearised operator $\mb L$ takes the form
\begin{align}\label{E:LINDEFINTRO}
\mb L\bs\theta:=-{4\over 3}\grad(\bar w^{-2}\grad\cdot(\bar w^3\bs\theta))-\grad\K\grad\cdot(\bar w^3\bs\theta),
\end{align}
where we recall~\eqref{E:BDEF} and~\eqref{E:KDEF}. The fundamental challenge with respect to the radial result~\cite{HaJa2018-1}
is to show the coercivity of the operator $\mb L$ in suitably weighted spaces, dictated by our local-in-time well-posedness theory.

The difficulty in proving a useful coercivity bound for the operator $\mb L$ lies in the antagonism between the nonlocal nature of the gravitational interaction described
by $\mb G$ in~\eqref{E:G}, and the Lagrangian perspective, which is naturally imposed on us by the problem. 
The operator $\mb L$ has a nontrivial unstable space, spanned by the eigenvectors $\mb x$ and the standard basis $\mb e_i$, $i=1,2,3$. The 4-dimensional nature of the unstable space is a reflection of the energy and momentum conservation laws, which in self-similar variables induce formally unstable modes. 

Nonradial linearised analysis around the Lane-Emden stars ($\delta=\l_1=0$) is given in~\cite{JaMa2020} where the non-negativity of the associated quadratic form is shown using the expansion in spherical harmonics. In this work we work in a similar spirit, but our linear analysis around the GW-stars improves upon~\cite{JaMa2020} considerably, as we show strict quantitative coercivity bound
\begin{align}\label{E:LINCOERCIVEINTRO}
\<\mb L\bs\theta,\bs\theta\>_3\gtrsim\int_{B_R}\bar w^{-2}|\lpc\Psi|^2\d\mb x+\int_{\R^3}|\grad\Psi|^2\d\mb x, \ \ 
\end{align}
under the crucial orthogonality conditions
\be\label{E:ORTHCOND}
\<\bs\theta,\mb x\>_3=0=\<\bs\theta,\mb e_i\>_3,  \qquad \qquad  i=1,2,3,
\ee 
where $\Delta\Psi = \text{div}(\bar w^3\bs\theta)$. 
This is the central estimate of Section~\ref{S:LINCO} (see Theorem~\ref{linear operator coercivity})
and it relies on a careful decomposition in spherical harmonics. It is non-trivial as it requires a careful use of the above orthogonality conditions to obtain quantitative lower bounds for the $0$-th and the $1$-st order spherical harmonics. In the former case, the problem essentially reduces to the radial coercivity bound from~\cite{HaJa2018-1}, while the analysis of 
the projection of $\mb L$ onto $1$-st order spherical harmonics requires a careful use of Sturm-Liouville theory, see Lemma~\ref{L:1MMODE}, a related argument was used in~\cite{JaMa2020}.

One of the main challenges is that the quantity $\int_{B_R}\bar w^{-2}|\lpc\Psi|^2\d\mb x+\int_{\R^3}|\grad\Psi|^2\d\mb x$ 
on the right-hand side of~\eqref{E:LINCOERCIVEINTRO} a priori does not appear useful for the energy estimates
as we need to control the norms $\|\bs\theta\|_3^2 + \|\nabla\bs\theta\|_4^2$, which are localised to the set $B_R$ by definition, see~\eqref{E:WEIGHTEDNORMDEF}. An intermediate step towards a resolution of this issue is to relate the general estimate~\eqref{E:LINCOERCIVEINTRO} (which holds for any sufficiently smooth map $\bs\theta$), to the nonlinear dynamics. In Section~\ref{Momentum and energy},  
by linearising the nonlinear energy-momentum constraints
\be\label{E:EMCONSTRAINTS}
E[\rho,\mb u]= \bar E, \qquad \qquad \mb W[\rho,\mb u] = \bar{\mb W},
\ee
we obtain effective ODEs (modulo lower order nonlinear terms) that allow to dynamically control the inner products $\<\bs\theta,\mb x\>_3$ and $\<\bs\theta,\mb e_i\>_3$. 
With this in hand we prove in Proposition~\ref{cor1} a high-order differentiated version of the bound~\eqref{E:LINCOERCIVEINTRO} for the solutions of~\eqref{E:LINDYNINTRO} satisfying the constraints~\eqref{E:EMCONSTRAINTS}:
\begin{align}
\int_{B_R}\bar w^{-2}|\grad\cdot(\bar w^3\partial_s^a\pt^{\beta}\bs\theta)|^2&\d\mb x + \|\partial_s^{a+1}\pt^{\beta}\bs\theta\|_3^2
\lesssim\<\mb L\partial_s^a\pt^{\beta}\bs\theta,\partial_s^a\pt^{\beta}\bs\theta\>_3+\|\partial_s^{a+1}\pt^{\beta}\bs\theta\|_3^2
+\text{l.o.t}. \label{E:STEP2INTRO}
\end{align}

The final and crucial step toward useful lower bounds is to exploit the irrotationality assumption $\nabla\times\mb u=0$ to obtain a dynamic control over $\|\bs\theta\|_3^2 + \|\nabla\bs\theta\|_4^2$. Looking at~\eqref{E:STEP2INTRO}, this necessitates a careful examination of the $\bar w$-weighted divergence appearing on the left-hand side. It is clear that any vectorfield such that $\grad\cdot(\bar w^3\bs\theta)=0$ formally belongs to the kernel of $\mb L$ and therefore, to obtain strict coercivity, we must mod out this infinite-dimensional kernel. The orthogonal complement with respect to the $\langle\cdot,\cdot\rangle_3$-inner product consists precisely of the gradients, so the 
first key observation is the content of Lemma~\ref{L:CURL1}, which roughly states that $\pa_s^a\bs\theta$ is a gradient modulo ``good" terms for $a\ge1$, assuming $\nabla\times\mb u=0$. Here, in simplest possible terms, the issue is that the irrotationality in Lagrangian variables creates error terms that a priori seem problematic, but luckily all such terms can be absorbed into a pure gradient. The second key ingredient is Lemma~\ref{g-identity}, which is an exact identity relating the norm of the weighted divergence of $\bs\theta$ to the weighted norms of the derivative of $\bs\theta$. This can be viewed as a form of ``elliptic regularity". Finally, we use these ingredients in the central statement of Section~\ref{S:CI} - Proposition~\ref{L-estimate} - to show that natural energy norms obtained via integration-by-parts from~\eqref{E:LINDYNINTRO} control the weighted norms of the pure time derivatives
of $\bs\theta$. In Proposition~\ref{L-estimate-tan} we treat also the angular derivatives in our operators, and the same statement as in the previous proposition holds, modulo the presence of a linear (and therefore not small) contribution, which fortunately involves one angular derivative less. This decoupling structure enables us to use a careful inductive procedure to eventually close the nonlinear estimates.

The nonlinear arguments are presented in Sections~\ref{S:ENERGYESTIMATES} and~\ref{S:EE2}. The global nonlinear stability will follow from the bound
\be\label{E:NONLINEARBOUND}
\E_n\lesssim\E_n(0)+(\E_n+\Z_n^2)^{1/2}\E_n 
\ee
in the regime where $E_n+Z_n^2$ is sufficiently small. To prove such a bound, we commute~\eqref{E:EP in self-similar} with high-order derivatives 
and while the discussion above
refers to the extraction of coercive bounds for the linear part of the operator, we are still left with the nonlinear estimates. Propositions~\ref{P-reduction-2}, \ref{P-reduction-1}, and~\ref{G non-linear bound} show that the deviation of the pressure term $\mb P$ and the gravity term $\mb G$ from its linearisation, can be controlled by the good trilinear error
$(\E_n+\Z_n^2)^{1/2}\E_n$ modulo some terms that scale like the linear norms, but always decouple at the top order of differentiation, so that they involve, for example, ``one spatial derivative less and one time derivative more". This decoupling is crucial for the closure of the estimates, and the key effective reduction to the linear problem is formulated in Theorem~\ref{estimate-theorem}. This feature of the problem suggests that we can show~\eqref{E:NONLINEARBOUND} inductively by taking derivatives in the right order. Key energy bounds for the nonlinear contributions from the pressure and the gravity are presented in Sections~\ref{Near boundary energy estimate section} and~\ref{Near origin energy estimate section} respectively. 
The final continuity argument and the exponential decay based on~\eqref{E:NONLINEARBOUND} is presented in Section~\ref{Bootstrapping scheme and final theorem}.



\subsection{Linearisation and coercivity}\label{S:LINCO}


\subsubsection{The linear and non-linear part of Euler-Poisson system}

The proof of Theorem~\ref{T:MAIN} crucially relies on good coercive properties of the linearisation around the background GW-star. In the next lemma we formally derive the linearised Euler-Poisson system.


\begin{lemma}[Linearised Euler-Poisson]\label{L:LINEAREP}
The formal linearisation of~(\ref{E:EP in self-similar}) reads
\begin{align}\label{E:LINDYN}
\partial_s^2\bs\theta-{1\over 2}\b\partial_s\bs\theta+\mb L\bs\theta = \mb 0
\end{align}
where
\begin{align}\label{E:LINDEF}
\mb L\bs\theta:=-{4\over 3}\grad(\bar w^{-2}\grad\cdot(\bar w^3\bs\theta))-\grad\K\grad\cdot(\bar w^3\bs\theta)
\end{align}
and we recall~\eqref{E:BDEF}.
Moreover, the formal linearisation of the gravitational contribution $\mb G$~(\ref{E:G}) is given by the operator
\begin{align}
\mb G_L\bs\theta&:=\bs\theta\cdot\grad\grad\K\bar w^3-\grad\K\grad\cdot(\bar w^3\bs\theta)=\K_{\bs\xi}^{(1)}\grad\bar w^3-\K\partial_i(\bar w^3\grad\theta^i),\label{E:linearised G}
\end{align}
where we recall~\eqref{E:NONLOCALKEY} and~\eqref{E:KKSIONE}.
\end{lemma}


\begin{proof}
Since $\grad\bs\xi=I+\grad\bs\theta$, to first order (in $\bs\theta$) we have $\A=I-\grad\bs\theta$ and $\J=1+\grad\cdot\bs\theta$. So to first order we have
\begin{align*}
\A\J^{-1/3}&=(I-\grad\bs\theta)(1+\grad\cdot\bs\theta)^{-1/3}
=(I-\grad\bs\theta)\brac{1-{1\over 3}\grad\cdot\bs\theta}
=\brac{1-{1\over 3}\grad\cdot\bs\theta}I-\grad\bs\theta
\end{align*}
and
\begin{align*}
{1\over\bar w^3}\partial_k(\bar w^4(\A^k\J^{-1/3}-I^k))&=-{1\over 3\bar w^3}\grad(\bar w^4\grad\cdot\bs\theta)-{1\over\bar w^3}\partial_k(\bar w^4\grad\theta^k)\\
&=-{4\over 3}\grad(\bar w\grad\cdot\bs\theta)-4\grad(\bs\theta\cdot\grad\bar w)+4\bs\theta\cdot\grad\grad\bar w\\
&=-{4\over 3}\grad(\bar w^{-2}\grad\cdot(\bar w^3\bs\theta))+4\bs\theta\cdot\grad\grad\bar w\\
&=-{4\over 3}\grad(\bar w^{-2}\grad\cdot(\bar w^3\bs\theta))-\bs\theta\cdot\grad(\delta\mb x+\grad\K\bar w^3)\\
&=-{4\over 3}\grad(\bar w^{-2}\grad\cdot(\bar w^3\bs\theta))-\delta\bs\theta-\bs\theta\cdot\grad\grad\K\bar w^3
\end{align*}
Since
\begin{align*}
|\bs\xi(\mb x)-\bs\xi(\mb z)|^2&=
|\mb x-\mb z+\bs\theta(\mb x)-\bs\theta(\mb z)|^2
=|\mb x-\mb z|^2+2(\mb x-\mb z)\cdot(\bs\theta(\mb x)-\bs\theta(\mb z))+|\bs\theta(\mb x)-\bs\theta(\mb z)|^2,
\end{align*}
to first order we have
\[
{1\over|\bs\xi(\mb x)-\bs\xi(\mb z)|}={1\over|\mb x-\mb z|}\brac{1-{(\mb x-\mb z)\cdot(\bs\theta(\mb x)-\bs\theta(\mb z))\over|\mb x-\mb z|^2}}.
\]
So to first order we have
\begin{align*}
\Phi(\mb x)&=-\int{\bar w(\mb z)^3\over|\bs\xi(\mb x)-\bs\xi(\mb z)|}\d\mb z
=-\int{\bar w(\mb z)^3\over|\mb x-\mb z|}\d\mb z+\int{(\mb x-\mb z)\cdot(\bs\theta(\mb x)-\bs\theta(\mb z))\over|\mb x-\mb z|^3}\bar w(\mb z)^3\d\mb z\\
&=(\K\bar w^3)(\mb x)+\int\brac{-\bs\theta(\mb x)\cdot\grad_{\mb x}{1\over|\mb x-\mb z|}-\bs\theta(\mb z)\cdot\grad_{\mb z}{1\over|\mb x-\mb z|}}\bar w(\mb z)^3\d\mb z\\
&=(\K\bar w^3)(\mb x)+\bs\theta\cdot\grad(\K\bar w^3)(\mb x)-(\K\grad\cdot(\bar w^3\bs\theta))(\mb x)
\end{align*}
and
\begin{align*}
\A\grad\Phi&=(I^i-\grad\theta^i)\partial_i(\K\bar w^3+\bs\theta\cdot\grad\K\bar w^3-\K\grad\cdot(\bar w^3\bs\theta))\\
&=\grad\K\bar w^3-(\grad\theta^i)\partial_i\K\bar w^3+\grad(\bs\theta\cdot\grad\K\bar w^3)-\grad\K\grad\cdot(\bar w^3\bs\theta)\\
&=\grad\K\bar w^3+\bs\theta\cdot\grad\grad\K\bar w^3-\grad\K\grad\cdot(\bar w^3\bs\theta)
=\grad\K\bar w^3+\mb G_L\bs\theta,
\end{align*}
where we have used~\eqref{E:linearised G} in the last line.
Therefore the linearisation of the momentum equation~\eqref{E:EP in self-similar}  takes the form~\eqref{E:LINDYN}.
Note that
\begin{align*}
\mb G_L\bs\theta&=\bs\theta\cdot\grad\K\grad\bar w^3-\K\partial_i(\grad\bar w^3\theta^i)-\K\partial_i(\bar w^3\grad\theta^i)\\
&=\int{(\mb x-\mb z)\cdot(\bs\theta(\mb x)-\bs\theta(\mb z))\over|\mb x-\mb z|^3}\grad\bar w(\mb z)^3\d\mb z-\K\partial_i(\bar w^3\grad\theta^i)
=\K_{\bs\xi}^{(1)}\grad\bar w^3-\K\partial_i(\bar w^3\grad\theta^i),
\end{align*}
which completes the proof of the lemma.  
\end{proof}




Finally, it will be important to keep track of the precise structure of the nonlinear correction $\mb G-\mb G_{L}\bs\theta$, which is given in the next lemma.


\begin{lemma}[Non-linear part of gravity term]\label{L:GRAVITYONE}
We have
\begin{align}
\mb G-\mb G_{L}\bs\theta&=\K_{\bs\xi}(\A_l^i(\partial_k\theta^l)(\grad\theta^k)\partial_i\bar w^3-\bar w^3(\A^i_m\A^l_\bullet-I^i_mI^l_\bullet)\partial_i\partial_l\theta^m)\nonumber\\
&\quad-(\K_{\bs\xi}-\K)\partial_i(\bar w^3\grad\theta^i)+(\K_{\bs\xi}-\mathcal{K}-\K_{\bs\xi}^{(1)})\grad\bar w^3\label{E:nonlinear G}
\end{align}
\end{lemma}


\begin{proof}
Since $\A=(\grad\bs\xi)^{-1}$, we have
\begin{align*}
I^i_j=\A^i_k\partial_j\xi^k=\A^i_k(I^k_j+\partial_j\theta^k).
\end{align*}
Therefore $\A^i_j-I^i_j=-\A^i_k\partial_j\theta^k$. We have
\begin{align*}
\mb G-\mb G_{L}\bs\theta&=\K_{\bs\xi}((\A-I)\grad\bar w^3-\bar w^3\A^i_m\A^l_\bullet\partial_i\partial_l\theta^m)+(\K_{\bs\xi}-\mathcal{K}-\K_{\bs\xi}^{(1)})\grad\bar w^3
+\K\partial_i(\bar w^3\grad\theta^i)\\
&=\K_{\bs\xi}((\grad\theta^i-\A_k^i\grad\theta^k)\partial_i\bar w^3-\bar w^3(\A^i_m\A^l_\bullet-I^i_mI^l_\bullet)\partial_i\partial_l\theta^m)\\
&\quad+(\K-\K_{\bs\xi})\partial_i(\bar w^3\grad\theta^i)+(\K_{\bs\xi}-\mathcal{K}-\K_{\bs\xi}^{(1)})\grad\bar w^3\\
&=\K_{\bs\xi}(\A_l^i(\partial_k\theta^l)(\grad\theta^k)\partial_i\bar w^3-\bar w^3(\A^i_m\A^l_\bullet-I^i_mI^l_\bullet)\partial_i\partial_l\theta^m)\\
&\quad-(\K_{\bs\xi}-\K)\partial_i(\bar w^3\grad\theta^i)+(\K_{\bs\xi}-\mathcal{K}-\K_{\bs\xi}^{(1)})\grad\bar w^3. 
\end{align*}
\end{proof}

We next derive helpful identities for the operators $\K_{\bs\xi}-\K$ and $\K_{\bs\xi}-\mathcal{K}-\K_{\bs\xi}^{(1)}$
appearing on the right-hand side of~\eqref{E:nonlinear G}. We first note that
\begin{align}
(\K_{\bs\xi}-\mathcal{K})g(\mb x)&=-\int_{\R^3}K_1(\mb x,\mb z)g(\mb z)\d\mb z, \\
(\K_{\bs\xi}-\mathcal{K}-\K_{\bs\xi}^{(1)})g(\mb x)&=-\int_{\R^3}K_2(\mb x,\mb z)g(\mb z)\d\mb z,
\end{align}
where
\begin{align}
K_1(\mb x,\mb z):&={1\over|\bs\xi(\mb x)-\bs\xi(\mb z)|}-{1\over |\mb x-\mb z|},\label{E:K1DEF}\\
K_2(\mb x,\mb z):&={1\over|\bs\xi(\mb x)-\bs\xi(\mb z)|}-{1\over |\mb x-\mb z|}+{(\mb x-\mb z)\cdot(\bs\theta(\mb x)-\bs\theta(\mb z))\over|\mb x-\mb z|^3}. \label{E:K2DEF}
\end{align}
In the following lemma, we write $K_1$ and $K_2$ explicitly in terms of $\bs\theta$, which will play a role in our energy estimates. In particular, we see that $\bs\theta$ appears at least linearly in $K_1$, and at least quadratically in $K_2$.


\begin{lemma}\label{varpi lemma}
We have
\begin{align}
K_2(\mb x,\mb z)&=K_1(\mb x,\mb z)+{(\mb x-\mb z)\cdot(\bs\theta(\mb x)-\bs\theta(\mb z))\over|\mb x-\mb z|^3}\nonumber\\
&=-{1\over 2}{|\bs\theta(\mb x)-\bs\theta(\mb z)|^2\over|\mb x-\mb z|^3}
+{3\over 4|\mb x-\mb z|}\brac{2{(\mb x-\mb z)\cdot(\bs\theta(\mb x)-\bs\theta(\mb z))\over|\mb x-\mb z|^2}+{|\bs\theta(\mb x)-\bs\theta(\mb z)|^2\over|\mb x-\mb z|^2}}^2\nonumber\\
&\qquad\qquad\qquad\qquad\varpi_{{1\over 2}}\brac{2{(\mb x-\mb z)\cdot(\bs\theta(\mb x)-\bs\theta(\mb z))\over|\mb x-\mb z|^2}+{|\bs\theta(\mb x)-\bs\theta(\mb z)|^2\over|\mb x-\mb z|^2}},
\end{align}
where 
\begin{align}
\varpi_{q}(y):=\int_0^1{1-z\over(1+yz)^{q+2}}\d z, \qquad \quad y>-1, \quad \  q\in\mathbb R. \label{varpi}
\end{align}
\end{lemma}


\begin{proof}
Let  $q\in\mathbb R\setminus \{-1,0\}$. Then for $y>-1$ and $y\neq 0$
\begin{align*}
\int_0^1{1-z\over(1+yz)^{q+2}}\d z&=-{1\over(q+1)y}\sbrac{{{1-z}\over(1+yz)^{q+1}}}_0^1
-{1\over(q+1)y}\int_0^1{1\over(1+yz)^{q+1}}\d z\\
&={1\over(q+1)y}\brac{1+{1\over qy}\sbrac{{1\over(1+yz)^q}}_0^1}
={1\over q(q+1)y^2}\brac{-1+qy+{1\over(1+y)^{q}}}
\end{align*}
and thus 
\begin{equation}\label{formula}
{1\over(1+y)^{q}}=1-qy+q(q+1)y^2\varpi_{q}(y), \qquad \quad y>-1
\end{equation}
where we note that \eqref{formula} trivially holds for $y=0$ and $q=-1,0$. 
Since
\begin{align*}
|\bs\xi(\mb x)-\bs\xi(\mb z)|^2&=|\mb x-\mb z+\bs\theta(\mb x)-\bs\theta(\mb z)|^2
=|\mb x-\mb z|^2+2(\mb x-\mb z)\cdot(\bs\theta(\mb x)-\bs\theta(\mb z))+|\bs\theta(\mb x)-\bs\theta(\mb z)|^2,
\end{align*}
we have
\begin{align*}
{|\bs\xi(\mb x)-\bs\xi(\mb z)|^2\over|\mb x-\mb z|^2}=1+2{(\mb x-\mb z)\cdot(\bs\theta(\mb x)-\bs\theta(\mb z))\over|\mb x-\mb z|^2}+{|\bs\theta(\mb x)-\bs\theta(\mb z)|^2\over|\mb x-\mb z|^2}.
\end{align*}
Hence by applying \eqref{formula} with $y=2{(\mb x-\mb z)\cdot(\bs\theta(\mb x)-\bs\theta(\mb z))\over|\mb x-\mb z|^2}+{|\bs\theta(\mb x)-\bs\theta(\mb z)|^2\over|\mb x-\mb z|^2}$ and $q=\frac12$, we see that 
\begin{align*}
{|\mb x-\mb z|\over|\bs\xi(\mb x)-\bs\xi(\mb z)|}&=1-\brac{{(\mb x-\mb z)\cdot(\bs\theta(\mb x)-\bs\theta(\mb z))\over|\mb x-\mb z|^2}+{1\over 2}{|\bs\theta(\mb x)-\bs\theta(\mb z)|^2\over|\mb x-\mb z|^2}}\\
&\quad+{3\over 4}\brac{2{(\mb x-\mb z)\cdot(\bs\theta(\mb x)-\bs\theta(\mb z))\over|\mb x-\mb z|^2}+{|\bs\theta(\mb x)-\bs\theta(\mb z)|^2\over|\mb x-\mb z|^2}}^2\\\
&\qquad\varpi_{{1\over 2}}\brac{2{(\mb x-\mb z)\cdot(\bs\theta(\mb x)-\bs\theta(\mb z))\over|\mb x-\mb z|^2}+{|\bs\theta(\mb x)-\bs\theta(\mb z)|^2\over|\mb x-\mb z|^2}}.
\end{align*}
Therefore, we obtain 
\begin{align*}
K_2(\mb x,\mb z)&=\underbrace{{1\over|\bs\xi(\mb x)-\bs\xi(\mb z)|}-{1\over |\mb x-\mb z|}}_{=K_1(\mb x,\mb z)}+{(\mb x-\mb z)\cdot(\bs\theta(\mb x)-\bs\theta(\mb z))\over|\mb x-\mb z|^3}\\
&=-{1\over 2}{|\bs\theta(\mb x)-\bs\theta(\mb z)|^2\over|\mb x-\mb z|^3}
+{3\over 4|\mb x-\mb z|}\brac{2{(\mb x-\mb z)\cdot(\bs\theta(\mb x)-\bs\theta(\mb z))\over|\mb x-\mb z|^2}+{|\bs\theta(\mb x)-\bs\theta(\mb z)|^2\over|\mb x-\mb z|^2}}^2\\
&\qquad\qquad\qquad\qquad\varpi_{{1\over 2}}\brac{2{(\mb x-\mb z)\cdot(\bs\theta(\mb x)-\bs\theta(\mb z))\over|\mb x-\mb z|^2}+{|\bs\theta(\mb x)-\bs\theta(\mb z)|^2\over|\mb x-\mb z|^2}}. 
\end{align*}
\end{proof}



\subsubsection{Coercivity of $\mb L$}


A fundamental prerequisite for the understanding of the nonlinear stability is a good linear stability theory. This entails a precise understanding of the coercivity properties of the operator $\mb L$ and this is the subject of this section. 

For sufficiently smooth $\bs\theta$, we have
\begin{align*}
\<\mb L\bs\theta_1,\bs\theta_2\>_3&=\int_{B_R}\brac{{4\over 3}\bar w^{-2}\grad\cdot(\bar w^3\bs\theta_2)\grad\cdot(\bar w^3\bs\theta_1)+\grad\cdot(\bar w^3\bs\theta_2)\K\grad\cdot(\bar w^3\bs\theta_1)}\d\mb x
\end{align*}
Note this is defined in a weak sense for $\bs\theta_i$ ($i=1,2$) such that $\grad\cdot(\bar w^3\bs\theta_i)\in L^2(B_R,\bar w^{-2})$. We see that $\mb L$ is symmetric under $\<\ph,\ph\>_3$ since
\[
\int\grad\cdot(\bar w^3\bs\theta_2)\K\grad\cdot(\bar w^3\bs\theta_1)\d\mb x=-\int\int{\grad\cdot(\bar w^3\bs\theta_2)(\mb x)\grad\cdot(\bar w^3\bs\theta_1)(\mb y)\over|\mb x-\mb y|}\d\mb x\d\mb y.
\]

Before stating the main theorem, we first characterise the growing modes for the linearised dynamics.


\begin{proposition}[Growing modes]\label{Eigenfunctions for L}
Let $\mb e_i$ ($i=1,2,3$) be the standard basis of $\R^3$. Then $\mb e_i$ and $\mb x$ are eigenfunctions for $\mb L$ with eigenvalue $\delta$ and $3\delta$ respectively.
\end{proposition}


\begin{proof}
Let $\mb f\in\mathbb R^3$ be a constant vector. Since $\mb 0=\delta\mb x+4\grad\bar w+\grad\K\bar w^3$, we have
\begin{align*}
\mb L\mb f&=-\grad\brac{{4\over 3}\bar w^{-2}\grad\cdot(\bar w^3\mb f)+\K\grad\cdot(\bar w^3\mb f)}
=-\grad\brac{{4\over 3}\bar w^{-2}\mb f\cdot\grad\bar w^3+\K\mb f\cdot\grad\bar w^3}\\
&=-\grad\brac{4\mb f\cdot\grad\bar w+\mb f\cdot\grad\K\bar w^3}
=\grad\brac{\delta\mb f\cdot\mb x}
=\delta\mb f
\end{align*}
And
\begin{align*}
\mb L\mb x&=-\grad\brac{{4\over 3}\bar w^{-2}\grad\cdot(\bar w^3\mb x)+\K\grad\cdot(\bar w^3\mb x)}
=-\grad\brac{4\mb x\cdot\grad\bar w+\K\mb x\cdot\grad\bar w^3+4\bar w+3\K\bar w^3}\\
&=-\grad\brac{4\mb x\cdot\grad\bar w+\mb x\cdot\grad\K\bar w^3+4\bar w+\K\bar w^3}
=\grad\brac{\delta\mb x\cdot\mb x-4\bar w-\K\bar w^3}
=2\delta\mb x+\delta\mb x=3\delta\mb x
\end{align*}
where we have used
\begin{align*}
\K(\mb x\cdot\grad\bar w^3)(\mb y)&=-\int{\mb x\cdot\grad\bar w(\mb x)^3\over|\mb y-\mb x|}\d\mb x
=\int\brac{{\bar w(\mb x)^3\grad\cdot\mb x\over|\mb y-\mb x|}+\bar w(\mb x)^3\mb x\cdot\grad_{\mb x}{1\over|\mb y-\mb x|}}\d\mb x\\
&=\int\brac{{3\bar w(\mb x)^3\over|\mb y-\mb x|}+\bar w(\mb x)^3\mb x\cdot{\mb y-\mb x\over|\mb y-\mb x|^3}}\d\mb x
=\int\brac{{2\bar w(\mb x)^3\over|\mb y-\mb x|}+\bar w(\mb x)^3\mb y\cdot{\mb y-\mb x\over|\mb y-\mb x|^3}}\d\mb x\\
&=-2\K\bar w^3+\mb y\cdot\grad\K\bar w^3. 
\end{align*}
\end{proof}



The main result of this section states that if the perturbation $\bs \theta$ is orthogonal to the four eigenvectors from Proposition~\ref{Eigenfunctions for L}, then the operator $\mb L$ is non-negative and we provide a quantitative lower bound.

\begin{theorem}[Non-negativity of $\mb L$]\label{linear operator coercivity}
Recall that $\bar w=\bar w_\delta$ and $\mb L$~\eqref{E:LINDEF} depends on $\delta$. There exists  $\epsilon>0$ such that for any $\delta\in(-\epsilon,0)$ the following holds. If $\bs\theta$ is such that $\|\bs\theta\|_3+\|\grad\bs\theta\|_4<\infty$ and
\be
\<\bs\theta,\mb x\>_3=0=\<\bs\theta,\mb e_i\>_3,\qquad \ i=1,2,3
\ee
then we have
\begin{align}
\<\mb L\bs\theta,\bs\theta\>_3\gtrsim\int_{B_R}\bar w^{-2}|\lpc\Psi|^2\d\mb x+\int_{\R^3}|\grad\Psi|^2\d\mb x
\end{align}
where the constants do not depend on $\delta$, and $\Psi$ is the gravitational potential induced by the flow disturbance 
$\bar w^3 \bs \theta$:
\begin{align*}
\Psi&:={1\over 4\pi}\K\grad\cdot(\bar w^3\bs\theta)\in H^1(\R^3)\cap C^1(\R^3)\\
\lpc\Psi&\;=\grad\cdot(\bar w^3\bs\theta)\in L^2(B_R,\bar w^{-2})
\end{align*}

\end{theorem}


The proof of Theorem~\ref{linear operator coercivity} is a simple consequence of Lemmas~\ref{L:00MODE}--\ref{L:HIGHERMODES}.
Our strategy is to use spherical harmonics to break down the problem into a sequence of scalar problems for each individual mode, by analogy to~\cite{JaMa2020}. The modes $l=0,1$ correspond to radial and translational motion, and therefore, although formally unstable, can be factored out from the dynamics through suitable orthogonality conditions.

\begin{lemma}[Spherical harmonics decomposition]\label{Spherical harmonics decomposition}
Suppose $\bs\theta$ is such that $\|\bs\theta\|_3+\|\grad\bs\theta\|_4<\infty$. Then 
\begin{align}
g&:=\grad\cdot(\bar w^3\bs\theta)\in L^2(B_R,\bar w^{-2})\\
\Psi(\mb x) &:= \frac1{4\pi} \K g(\mb x)\in H^1(\R^3)\cap C^1(\R^3),
\end{align}
and they 
can be expanded in spherical harmonics 
\begin{align}
g(\mb x)&=\sum_{l=0}^\infty\sum_{m=-l}^lg_{lm}(r)Y_{lm}(\mb x)\qquad\qquad\text{ on }\ B_R,\label{E:g expansion}\\
\Psi(\mb x) &= \sum_{l=0}^\infty\sum_{m=-l}^l \Psi_{lm}(r) Y_{lm}(\mb x)\qquad\qquad\text{ on }\ \R^3, \label{E:PSIDEF}
\end{align}
that converge in $L^2(B_R,\bar w^{-2})$ and $L^2(\R^3)$ respectively, where the spherical harmonics $Y_{lm}$ are introduced in Appendix~\ref{A:SPHERICALHARMONICS}. 
Moreover, $\Psi_{lm}$ are related to $g_{lm}$ by
\begin{align}
\Psi_{lm}(r) &= {-1\over 2l+1}\brac{\int^{r}_0{y^{l+2}\over r^{l+1}}g_{lm}(y)\d y+\int_{r}^R{r^l\over y^{l-1}}g_{lm}(y)\d y}\label{g and Psi relation}\\
g_{lm}&=\lpc^{\<l\>}\Psi_{lm}:=\brac{{1\over r^2}\brac{r^2\Psi_{lm}'}'-{l(l+1)\over r^2}\Psi_{lm}}.\label{E:LAPLACEL}
\end{align}
With this, the following identity holds:
\begin{align}
\<\mb L\bs\theta,\bs\theta\>_3
&=\sum_{l=0}^\infty\sum_{m=-l}^l \Lambda_{lm},
\end{align}
where
\begin{align}
\Lambda_{lm}: = \int_0^R\brac{{4\over 3}\bar w^{-2}g_{lm}^2+4\pi g_{lm}\Psi_{lm}}r^2\d r, \qquad \  l\ge0, \ m\in\{-l,\dots,l\}.\label{E:LAMBDALMDEF}
\end{align}
\end{lemma}

\begin{proof}
From $\|\bs\theta\|_3+\|\grad\bs\theta\|_4<\infty$, Corollary \ref{weight upgrade} of Hardy-Poincar\'e inequality means that we have $\|\bs\theta\|_2+\|\grad\bs\theta\|_4<\infty$. This immediately  gives that $g\in L^2(B_R,\bar w^{-2})$. Since $\Psi$ is a convolution of $g$ with the kernel $\|\ph\|^{-1}$, where $g$ is trivially extended by 0 on $\mathbb R^3\setminus B_R$, standard computation shows $\Psi\in C^1(\R^3)\cap H^1(\R^3)$. Since spherical harmonics form an $L^2$ basis (see \cite{Atkinson Han, Jackson, Courant Hilbert} and Appendix~\ref{A:SPHERICALHARMONICS}), we have the spherical harmonics expansion \eqref{E:g expansion}-\eqref{E:PSIDEF} for $g$ and $\Psi$ in $L^2$.

By Lemma \ref{potential decomposition} we have
\[{1\over|\mb x-\mb y|}=4\pi\sum_{l=0}^\infty\sum_{m=-l}^l{1\over 2l+1}{\min\{|\mb x|,|\mb y|\}^l\over\max\{|\mb x|,|\mb y|\}^{l+1}}Y_{lm}(\mb y)Y_{lm}(\mb x)\]
which converge uniformly on all compact set in $\{(\mb x,\mb y):|\mb x|\not=|\mb y|\}$. So we have
\begin{align*}
\mathcal{K}g(\mb x)
&=-4\pi\sum_{l=0}^\infty\sum_{m=-l}^l{1\over 2l+1}Y_{lm}(\mb x)\brac{\int_{B_{|\mb x|}(\mb 0)}{|\mb y|^l\over|\mb x|^{l+1}}gY_{lm}\d\mb y+\int_{B_{|\mb x|}(\mb 0)^c}{|\mb x|^l\over|\mb y|^{l+1}}gY_{lm}\d\mb y}\\
&=-4\pi\sum_{l=0}^\infty\sum_{m=-l}^l{1\over 2l+1}Y_{lm}(\mb x)\brac{\int^{|\mb x|}_0{y^{l+2}\over|\mb x|^{l+1}}g_{lm}\d y+\int_{|\mb x|}^R{|\mb x|^l\over y^{l-1}}g_{lm}\d y}.
\end{align*}
We therefore conclude that 
\[
\Psi_{lm}(r) = {-1\over 2l+1}\brac{\int^{r}_0{y^{l+2}\over r^{l+1}}g_{lm}(y)\d y+\int_{r}^R{r^l\over y^{l-1}}g_{lm}(y)\d y}
\]
since spherical harmonics expansion is unique (using standard Hilbert space theory and the fact that spherical harmonics forms a $L^2$ basis for $L^2$ functions on the sphere). Inverting this expression, we get (\ref{E:LAPLACEL}).

Now using the spherical harmonics expansion for $g$ and $\Psi$, we get
\begin{align*}
\<\mb L\bs\theta,\bs\theta\>_3
&=\int\brac{{4\over 3}\bar w^{-2}|\grad\cdot(\bar w^3\bs\theta)|^2+\grad\cdot(\bar w^3\bs\theta)\K\grad\cdot(\bar w^3\bs\theta)}\d\mb x\\
&=\int\brac{{4\over 3}\bar w^{-2}|g|^2+4\pi g\Psi}\d\mb x=\sum_{l=0}^\infty\sum_{m=-l}^l\Lambda_{lm},
\end{align*}
with $\Lambda_{lm}$ as in~\eqref{E:LAMBDALMDEF}. 
\end{proof}

From~\cite{HaJa2018-1} we have the following lemma.


\begin{lemma}\label{radial positivity}
There exists $\epsilon>0$ such that for any $\delta\in(-\epsilon,0)$ and the associated $\bar w=\bar w_\delta$, we have
\[\<\mathcal{L}\varphi,\varphi\>_{\bar w^3r^4}\gtrsim\|\varphi'\|_{\bar w^4r^4}^2+\|\varphi\|_{\bar w^3r^4}^2\qquad\text{whenever}\qquad\<\varphi,1\>_{\bar w^3r^4}=0\]
where the constants do not depend on $\delta$, and
\begin{align*}
\mathcal{L}\varphi&:=-{4\over 3\bar w^3 r^4}\partial_r\brac{\bar w^4 r^4\partial_r\varphi}+3\delta\varphi\\
\<f,g\>_{\bar w^kr^4}&:=\int_0^Rf(r)g(r)\bar w(r)^kr^4\d r.
\end{align*}
\end{lemma}


We shall use Lemma~\ref{radial positivity} to obtain coercivity for the quadratic form $\Lambda_{00}$ under the orthogonality assumption $\<\bs\theta,\mb x\>_3=0$.
\begin{lemma}[$l=0$ mode bound]\label{L:00MODE}
Suppose $\bs\theta$ is as in Lemma \ref{Spherical harmonics decomposition} and $\<\bs\theta,\mb x\>_3=0$. Then we have $\Psi_{00}'(r)=0$ for $r\geq R$ and
\begin{align}\label{E:ZEROZEROBOUND}
\Lambda_{00} \gtrsim  \int_0^R\brac{\bar w^{-2}g_{00}^2+(\Psi_{00}')^2}r^2\d r,
\end{align}
where we recall~\eqref{E:LAMBDALMDEF}.
\end{lemma}


\begin{proof}
From $\|\bs\theta\|_3+\|\grad\bs\theta\|_4<\infty$, Corollary \ref{weight upgrade} of Hardy-Poincar\'e inequality means that we have $\|\bs\theta\|_2+\|\grad\bs\theta\|_4<\infty$. It follows that $\bar w^3\bs\theta$ is well defined on $\partial B_R$ (trace theorem) and must vanish there. Since $\bar w^3\bs\theta=\grad\Psi+\mb C$ where $\mb C$ is divergence-free, we have
\[\int_{\partial B_R}\partial_r\Psi\;\d S=\int_{\partial B_R}\grad\Psi\cdot\d\mb S=\int_{\partial B_R}\bar w^3\bs\theta\cdot\d\mb S=0.\]
It follows that $\Psi_{00}'(R)=0$. Now taking the derivative of \eqref{g and Psi relation} and using $g_{00}(r)=0$ for $r>R$, we see that in fact we must have
\begin{align}
\Psi_{00}'(r)=0\qquad\text{for}\qquad r\geq R.\label{00 boundary condition second}
\end{align}

From the orthogonality condition $\<\bs\theta,\mb x\>_3=0$ we infer that 
\begin{align*}
0&=\<\bs\theta,\mb x\>_3={1\over 2}\int\bar w^3\bs\theta\cdot\grad|\mb x|^2\;\d\mb x={1\over 2}\int g|\mb x|^2\;\d\mb x.
\end{align*}
This means
\begin{align}
\int_0^Rg_{00}(r)r^4\d r&=0\label{00 g boundary condition}.
\end{align}
and therefore by \eqref{E:LAPLACEL} and \eqref{00 boundary condition second} in terms of $\Psi_{00}$,
\begin{align}
\int_0^R\Psi_{00}'(r)r^3\d r&=0.\label{00 boundary condition}
\end{align}

Since $\Psi\in H^1(\R^3)\cap C^1(\R^3)$, we have that $\partial_r\Psi\in L^2(\R^3)\cap C(\R^3\setminus\{\mb 0\})$. So $\partial_r\Psi$ has spherical harmonics expansion
$\partial_r\Psi=\sum_{l=0}^\infty\sum_{m=-l}^l\Psi_{r,lm}Y_{lm}$
in $L^2(\R^3)$ with
\begin{align}
\Psi_{r,lm}(r)={1\over 4\pi r^2}\int_{\partial B_r}(\partial_r\Psi )Y_{lm}\d S
={1\over 4\pi r^2}\partial_r\int_{B_r}\Psi Y_{lm}\d S
=\partial_r\Psi_{lm}(r)=\Psi_{lm}'(r).\label{Psi expansion radial derivative}
\end{align}
If we denote
\begin{align}\label{E:VARPHIDEF}
\varphi : = \Psi_{00}'/(r\bar w^3),
\end{align}
then by~(\ref{00 boundary condition}) we have
\begin{align*}
0=\int_0^R\Psi_{00}'(r)r^3\d r=\int_0^R\varphi(r)\bar w^3r^4\d r
\end{align*}
and thus $\<\varphi,1\>_{\bar w^3r^4}=0$. 
Using \eqref{E:LAPLACEL} and (\ref{00 boundary condition second}), we get 
\begin{align*}
\Lambda_{00}&=\int_0^R\brac{{4\over 3}\bar w^{-2}g_{00}^2+4\pi g_{00}\Psi_{00}}r^2\d r
=\int_0^R\brac{{4\over 3r^2}\bar w^{-2}\brac{\brac{r^2\Psi_{00}'}'}^2+4\pi\brac{r^2\Psi_{00}'}'\Psi_{00}}\d r\\
&=\int_0^R\brac{{4\over 3r^2}\bar w^{-2}\brac{\brac{r^2\Psi_{00}'}'}^2-4\pi r^2(\Psi_{00}')^2}\d r+4\pi R^2\Psi_{00}'(R)\Psi_{00}(R)\\
&=\int_0^R\brac{{4\over 3r^2}\bar w^{-2}\brac{\brac{r^3\bar w^3\varphi}'}^2-4\pi\varphi^2\bar w^6r^4}\d r
\end{align*}
Now since $0=3\delta+4\lpc\bar w+4\pi\bar w^3$ as in (\ref{E:equation for bar-w}),  we see that 
\begin{align*}
\Lambda_{00}&=\int_0^R\brac{{4\over 3r^2}\bar w^{-2}\brac{\brac{r^3\bar w^3\varphi}'}^2+(3\delta+4\lpc\bar w)\varphi^2\bar w^3r^4}\d r\\
&=\int_0^R\bigg({4\over 3r^2}\bar w^{-2}\brac{3r^2\bar w^3\varphi+3r^3\bar w^2\bar w'\varphi+r^3\bar w^3\varphi'}^2
+4(r^2\bar w')'\varphi^2\bar w^3r^2+3\delta\varphi^2\bar w^3r^4\bigg)\d r\\
&=\int_0^R\bigg({4\over 3}\brac{3r\bar w^2\varphi+3r^2\bar w\bar w'\varphi+r^2\bar w^2\varphi'}^2
-4r^2\bar w'(\varphi^2\bar w^3r^2)'+3\delta\varphi^2\bar w^3r^4\bigg)\d r\\
&=\int_0^R\bigg({4\over 3}\big(9r^2\bar w^4\varphi^2+9r^4\bar w^2(\bar w')^2\varphi^2+r^4\bar w^4(\varphi')^2
+18r^3\bar w^3\bar w'\varphi^2+6r^4\bar w^3\bar w'\varphi\varphi'+6r^3\bar w^4\varphi\varphi'\big)\\
&\qquad\qquad-4(2\varphi\varphi'\bar w^3\bar w'r^4+3\varphi^2\bar w^2(\bar w')^2r^4+2\varphi^2\bar w^3\bar w'r^3)+3\delta\varphi^2\bar w^3r^4\bigg)\d r\\
&=\int_0^R\bigg({4\over 3}\brac{9r^2\bar w^4\varphi^2+r^4\bar w^4(\varphi')^2+12r^3\bar w^3\bar w'\varphi^2+6r^3\bar w^4\varphi\varphi'}
+3\delta\varphi^2\bar w^3r^4\bigg)\d r\\
&=\int_0^R\bigg({4\over 3}r^4\bar w^4(\varphi')^2+3\delta\varphi^2\bar w^3r^4\bigg)\d r
=\<\mathcal{L}\varphi,\varphi\>_{\bar w^3r^4}\\
&\gtrsim\|\varphi'\|_{\bar w^4r^4}^2+\|\varphi\|_{\bar w^3r^4}^2
=\int_0^R((\varphi')^2\bar w^4+\varphi^2\bar w^3)r^4\d r
\end{align*}
Then for $\epsilon$ small enough, we get
\begin{align*}
\Lambda_{00}&\gtrsim\int_0^R\bigg(\epsilon\brac{{4\over 3}r^4\bar w^4(\varphi')^2+3\delta\varphi^2\bar w^3r^4}+\varphi^2\bar w^3r^4\bigg)\d r\\
&=\int_0^R\brac{\epsilon\brac{{4\over 3}\bar w^{-2}g_{00}^2-4\pi(\Psi_{00}')^2}+\bar w^{-3}(\Psi_{00}')^2}r^2\d r.
\end{align*}
Finally by choosing $\epsilon$ small enough we get~\eqref{E:ZEROZEROBOUND}.
\end{proof}


In order to prove positivity of the higher modes, we will need the following lemma which provides an estimate from below for $\Lambda_{lm}$ by an elliptic operator; a related bound was also used in~\cite{JaMa2020}.

\begin{lemma}\label{L:Lambda below bound}
Suppose $\bs\theta$ is as in Lemma \ref{Spherical harmonics decomposition}. Then for any $l\ge0, m\in\{-l,\dots,l\}$, we have
\begin{align*}
\Lambda_{lm}&\geq 4\pi\int_0^R \brac{-\lpc^{\<l\>}-3\pi\bar w^2}(\Psi_{lm})\Psi_{lm}r^2\d r.
\end{align*}
\end{lemma}
\begin{proof}
We have
\begin{align*}
\Lambda_{lm}&=\int_0^R\brac{{4\over 3}\bar w^{-2}g_{lm}^2+4\pi g_{lm}\Psi_{lm}}r^2\d r\\
&=\int_0^R\abs{{2\over\sqrt{3}\bar w}g_{lm}+2\pi\sqrt{3}\bar w\Psi_{lm}}^2r^2\d r-4\pi\int_0^R\brac{g_{lm}\Psi_{lm}+3\pi\bar w^2\Psi_{lm}^2}r^2\d r\\
&\geq 4\pi\int_0^R \brac{-\lpc^{\<l\>}-3\pi\bar w^2}(\Psi_{lm})\Psi_{lm}r^2\d r. 
\end{align*}
\end{proof}

With this bound from below by an elliptic operator, we can prove the positivity of $\Lambda_{lm}$ using elliptic ODE theory.

\begin{lemma}[$l=1$ modes bound]\label{L:1MMODE}
Suppose $\bs\theta$ is as in Lemma \ref{Spherical harmonics decomposition} and $\<\bs\theta,\mb e_i\>_3=0$ for $i=1,2,3$. 
Then we have $\Psi_{1m}(r)=0$ for $r\geq R$ and
\begin{align}
\Lambda_{1m}\gtrsim\int_0^R\brac{\bar w^{-2}g_{1m}^2r^2+\Psi_{1m}'^2r^2+\Psi_{1m}^2}\d r, \qquad \ m=-1,0,1, \label{E:1MBOUND}
\end{align}
where we recall~\eqref{E:LAMBDALMDEF}.
\end{lemma}


\begin{proof}
For these modes, we adapt the method of proof as found in \cite{JaMa2020} that makes use of the Sturm-Liouville theory.
We have by Lemma \ref{L:Lambda below bound}
\begin{align*}
\Lambda_{1m}&\geq 4\pi \<A_1\Psi_{1m},\Psi_{1m}\>_{r^2}
\end{align*}
where $\<y_1,y_2\>_{r^2}:=\int_0^Ry_1y_2r^2\d r$ and  
\begin{align}
A_1:=-\lpc^{\<1\>}-3\pi\bar w^2.
\end{align}
As this operator $A_1$ resembles the operator $A$ analyzed in \cite{JaMa2020} (cf. (7.15) of \cite{JaMa2020}), by arguing analogously, we deduce that the operator $A_1$ has the Friedrichs extension in the Hilbert space induced by the inner product $\<y_1,y_2\>_{r^2}$, denoted by the same $A_1$. Moreover it is of Sturm-Liouville type and the eigenvalues are simple  under the Dirichlet boundary condition on $r=R$, i.e. $y(R)=0$  (cf. Section VII of \cite{JaMa2020}).   

We next claim the least eigenvalue $\mu_1$ of $A_1$ is strictly positive.  Let $\phi_1$ be an associated eigenfunction such that $A_1\phi_1=\mu_1 \phi_1$. Since $\phi_1$ must have no zeros on $(0,R)$ by Sturm-Liouville theory,  we may assume that $\phi_1(r)>0$ for $r\in(0,R)$ so that $\phi_1'(R)\leq 0$ and $\phi_1(R)=0$. In fact we must have $\phi_1'(R)<0$, for if $\phi_1'(R)=0$, then $\phi_1$ must be the zero function, which is a contradiction.
To see the latter assertion, note that $A_1$ is a second order ODE operator with $C^1$ coefficients away from the origin. Picard-Lindel\"of existence theorem implies that for any $\epsilon>0$ the solution $u$ on $(\epsilon,R]$ satisfying $u'(R)=u(R)=0$ must be unique. Since $u=0$ is such a solution, we must have $\phi_1'=u=0$. On the other hand, recalling $\lpc(4\bar w)=-3\delta-4\pi\bar w^3$, we see that $A_1\bar w'=0$. 
Note that $\bar w'(R)\not=0$, so $\bar w'\not\in\Dom A_1$ where $\Dom A_1$ denotes the domain of $A_1$ under the Sturm-Liouville theory framework.  
By using $A_1\bar w'=0$, the properties of $\phi_1$ and integration by parts, we have
\begin{align*}
0&=\<A_1\bar w',\phi_1\>_{r^2}=\<\bar w',A_1\phi_1\>_{r^2}+R^2\bar w'(R)\phi_1'(R)
=\mu_1\<\bar w',\phi_1\>_{r^2}+R^2\bar w'(R)\phi_1'(R).
\end{align*}
Since $\bar w'(r)<0$ for $r\in(0,R]$, we see that $\<\bar w',\phi_1\>_{r^2}<0$. Also $R^2\bar w'(R)\phi_1'(R)> 0$. Therefore we must have $\mu_1> 0$. 

By the orthogonality condition
\begin{align*}
0&=\<\bs\theta,\mb e_i\>_3=\int_{B_R}\bar w^3\bs\theta\cdot\grad x^i\;\d\mb x
=\int_{B_R}gx^i\;\d\mb x,
\end{align*}
we conclude that 
$0=\int_0^Rg_{1m}r^3\d r$
and therefore
\begin{align}
\Psi_{1m}(r)&=0 \qquad\text{for}\qquad r\geq R\label{1m boundary condition}.
\end{align}
This (\ref{1m boundary condition}) means that $\Psi_{1m}\in\Dom A_1$, it follows that
\begin{equation}\label{inequality1m}
\Lambda_{1m}\geq 4\pi\<A_1\Psi_{1m},\Psi_{1m}\>_{r^2}\geq 4\pi\mu_1\<\Psi_{1m},\Psi_{1m}\>_{r^2}\geq 0.
\end{equation}

The second inequality of \eqref{inequality1m} implies 
\begin{align*}
\int_0^R\brac{\Psi_{1m}'^2+{2\over r^2}\Psi_{1m}^2-3\pi\bar w^2\Psi_{1m}^2}r^2\d r\geq\mu_1\int_0^R\Psi_{1m}^2r^2\d r
\end{align*}
which we can rewrite as
\begin{align*}
&(1+\epsilon)\int_0^R\brac{\Psi_{1m}'^2+{2\over r^2}\Psi_{1m}^2-3\pi\bar w^2\Psi_{1m}^2}r^2\d r\\
&\geq\epsilon\int_0^R\brac{\Psi_{1m}'^2+{2\over r^2}\Psi_{1m}^2-3\pi\bar w^2\Psi_{1m}^2}r^2\d r+\mu_1\int_0^R\Psi_{1m}^2r^2\d r\\
&\geq\epsilon\int_0^R\brac{\Psi_{1m}'^2r^2+2\Psi_{1m}^2}\d r+(\mu_1-3\epsilon\pi\bar w(0)^2)\int_0^R\Psi_{1m}^2r^2\d r
\end{align*}
Chose $\epsilon$ small enough so that the last term is non-negative. Hence we see that
\begin{align*}
\int_0^R\brac{\Psi_{1m}'^2+{2\over r^2}\Psi_{1m}^2-3\pi\bar w^2\Psi_{1m}^2}r^2\d r\gtrsim\int_0^R\brac{\Psi_{1m}'^2r^2+\Psi_{1m}^2}\d r
\end{align*}
Together with \eqref{inequality1m} we deduce that 
\begin{align*}
\Lambda_{1m}=\int_0^R\brac{{4\over 3}\bar w^{-2}g_{1m}^2+4\pi g_{1m}\Psi_{1m}}r^2\d r\gtrsim\int_0^R\brac{\Psi_{1m}'^2r^2+\Psi_{1m}^2}\d r
\end{align*}
We can rewrite this as, for some $C>0$,
\begin{align*}
(1+\epsilon)\Lambda_{1m}
&\geq\epsilon\int_0^R\brac{{4\over 3}\bar w^{-2}g_{1m}^2+4\pi g_{1m}\Psi_{1m}}r^2\d r+C\int_0^R\brac{\Psi_{1m}'^2r^2+\Psi_{1m}^2}\d r\\
&=\epsilon\int_0^R\brac{{4\over 3}\bar w^{-2}g_{1m}^2r^2+4\pi\brac{(r^2\Psi_{1m}')'-2\Psi_{1m}}\Psi_{1m}}\d r
+C\int_0^R\brac{\Psi_{1m}'^2r^2+\Psi_{1m}^2}\d r\\
&=\epsilon\int_0^R\brac{{4\over 3}\bar w^{-2}g_{1m}^2r^2-4\pi\brac{\Psi_{1m}'^2r^2+2\Psi_{1m}^2}}\d r
+C\int_0^R\brac{\Psi_{1m}'^2r^2+\Psi_{1m}^2}\d r
\end{align*}
Choosing $\epsilon$ small enough we obtain~\eqref{E:1MBOUND}. 
\end{proof}

%


\begin{lemma}[$l\geq 2$ modes bound]\label{L:HIGHERMODES}
Suppose $\bs\theta$ is as in Lemma \ref{Spherical harmonics decomposition}. Then for $l\geq 2$,
\begin{align}
\Lambda_{lm}\gtrsim\int_0^R\bar w^{-2}g_{lm}^2r^2\d r+\int_0^\infty\brac{\Psi_{lm}'^2r^2+l(l+1)\Psi_{lm}^2}\d r,\qquad\ m\in\{-l,\dots,l\}.\label{E:Higher modes bound}
\end{align}
\end{lemma}
\begin{proof}
For these higher modes, we use a continuity argument. We have by Lemma \ref{L:Lambda below bound}
\begin{align}
\Lambda_{lm}&\geq 4\pi\int_0^R\brac{-\lpc^{\<l\>}-3\pi\bar w^2}(\Psi_{lm})\Psi_{lm}r^2\d r\nonumber\\
&=4\pi\int_0^R\brac{\Psi_{lm}'^2+{l(l+1)\over r^2}\Psi_{lm}^2-3\pi\bar w^2\Psi_{lm}^2}r^2\d r-4\pi R^2\Psi_{lm}(R)\Psi_{lm}'(R)\nonumber\\
&=4\pi\int_0^R\brac{\Psi_{lm}'^2+{l(l+1)\over r^2}\Psi_{lm}^2-3\pi\bar w^2\Psi_{lm}^2}r^2\d r
+4\pi\brac{\int_R^\infty\brac{\Psi_{lm}'^2r^2+(r^2\Psi_{lm}')'\Psi_{lm}}\d r}\nonumber\\
&=4\pi\int_0^\infty\brac{\Psi_{lm}'^2+{l(l+1)\over r^2}\Psi_{lm}^2-3\pi\bar w^2\Psi_{lm}^2}r^2\d r\label{E:Lambda below for higher modes}\\
&=4\pi\int_0^\infty\brac{-\lpc^{\<1\>}-3\pi\bar w^2}(\Psi_{lm})\Psi_{lm}r^2\d r+4\pi\int_0^\infty(l(l+1)-2)\Psi_{lm}^2\d r
\end{align}
where we used
\[g_{lm}(r)=\lpc^{\<l\>}\Psi_{lm}(r)={1\over r^2}(r^2\Psi_{lm}')'-{l(l+1)\over r^2}\Psi_{lm}=0\qquad\text{for}\qquad r>R.\]

Recall that $\bar w=\bar w_\delta$ depends on $\delta$. In the proof of Lemma \ref{L:1MMODE} we have shown that
\begin{align*}
\int_0^{R}\brac{-\lpc^{\<1\>}-3\pi\bar w_\delta^2}(y)yr^2\d r\geq 0
\end{align*}
for all $y\in H^2([0,R],r^2)$ such that $y(R)=0$. In fact when $\bar w_\delta=\bar w_0$ (the Lane-Emden star), the same analysis can be extended to any $R'\geq R$ to give rise to 
\begin{align}\label{E:NONNEG}
\int_0^{R'}\brac{-\lpc^{\<1\>}-3\pi\bar w_0^2}(y)yr^2\d r\geq 0
\end{align}
for all $y\in H^2([0,R'],r^2)$ such that $y(R')=0$. To do so, we replace $\bar w_0'$ (used to argue the non-negativity of the least eigenvalue) with $\tilde{w}'$, where $\tilde{w}:=-{1\over 4}\K\bar w_0^3$ (recall $\bar w_0=0$ for $r>R$). Note that $\tilde{w}$ is $C^3(\R^3)$, or $C^3([0,\infty))$ as a function of the radial variable. By \eqref{E:GW2}, we see that $\tilde{w}'=\bar w_0'$ on $[0,R]$. Moreover, $\tilde{w}'<0$ on $(0,\infty)$. Since $\lpc\tilde{w}=-\pi\bar w_0^3$, taking $\partial_r$ we get $\lpc^{\<1\>}\tilde{w}=-3\pi\bar w_0^2\bar w_0'=-3\pi\bar w_0^2\tilde{w}'$. So we have $(-\lpc^{\<1\>}-3\pi\bar w_0^2)\tilde{w}'=0$ on $[0,\infty)$ which allows us to apply the same proof in Lemma \ref{L:1MMODE}. 

Let
\begin{align*}
y_{R'}(r)=\Psi_{lm}(r)-\Psi_{lm}(R)\brac{R\over R'}^{l+1}{r\over R'}
\end{align*}
From \eqref{g and Psi relation} we see that $y_{R'}(R')=0$. 
By using $\lpc^{\<1\>}r=0$ and applying \eqref{E:NONNEG} with $y=y_{R'}$, we obtain 
\begin{align*}
&\int_0^{R'}\brac{-\lpc^{\<1\>}-3\pi\bar w_0^2}(\Psi_{lm})\Psi_{lm}r^2\d r\\
&=\int_0^{R'}\brac{-\lpc^{\<1\>}-3\pi\bar w_0^2}(y_{R'}(r))\brac{y_{R'}(r)+\Psi_{lm}(R)\brac{R\over R'}^{l+1}{r\over R'}
}r^2\d r\\
&\quad-\int_0^{R}3\pi\bar w_0^2\Psi_{lm}(r)\Psi_{lm}(R)\brac{R\over R'}^{l+1}{r\over R'}r^2\d r\\
&\geq\int_0^{R'}\brac{-\lpc^{\<1\>}-3\pi\bar w_0^2}(y_{R'}(r))\Psi_{lm}(R)\brac{R\over R'}^{l+1}{r^3\over R'}\d r
-\int_0^{R}3\pi \bar w_0^2\Psi_{lm}(r)\Psi_{lm}(R)\brac{R\over R'}^{l+1}{r^3\over R'}\d r
\end{align*}
Denote the last two integral terms by $K$. By integrating by parts and using the boundary condition $y_{R'}(R')=0$, 
\begin{align*}
K&=-R'^2y_{R'}'(R')\Psi_{lm}(R)\brac{R\over R'}^{l+1} 
-\int_0^{R}3\pi \bar w_0^2y_{R'} (r)\Psi_{lm}(R)\brac{R\over R'}^{l+1}{r^3\over R'}\d r\\  
&\quad-\int_0^{R}3\pi \bar w_0^2\Psi_{lm}(r)\Psi_{lm}(R)\brac{R\over R'}^{l+1}{r^3\over R'}\d r\\
&=-R'^2\Psi_{lm}'(R')\Psi_{lm}(R)\brac{R\over R'}^{l+1}+R'(\Psi_{lm}(R))^2 \brac{R\over R'}^{2l+2}\\
&\quad+3\pi \int_0^{R} \bar w_0^2(\Psi_{lm}(R))^2 \brac{R\over R'}^{2l+2}{r^4\over R'^2}\d r
-6\pi\int_0^{R}\bar w_0^2\Psi_{lm}(r)\Psi_{lm}(R)\brac{R\over R'}^{l+1}{r^3\over R'}\d r\\
&\to 0\qquad\text{as}\qquad R'\to\infty
\end{align*}
when $l\geq 1$, where we used \eqref{g and Psi relation} to see for example that $\Psi_{lm}'(R')\to 0$ as $R'\to\infty$.

Therefore we have proven\footnote{
The proof of~\eqref{E:NONNEGE2} can be easily adapted to correct an inconsistency appearing in  \cite{JaMa2020} and establish the non-negativity of the quadratic form $\langle\mb L\bs\theta,\bs\theta\rangle_3$ around the Lane-Emden stars. 
}  
that for any $l\geq 1$,
\begin{multline}
\int_0^\infty\brac{\Psi_{lm}'^2+{l(l+1)\over r^2}\Psi_{lm}^2-3\pi\bar w_0^2\Psi_{lm}^2}r^2\d r
\geq\int_0^\infty(l(l+1)-2)\Psi_{lm}^2\d r\quad\text{for all}\quad\Psi_{lm}. \label{E:NONNEGE2}
\end{multline}
So we have
\begin{multline*}
\int_0^\infty\brac{\Psi_{lm}'^2+{l(l+1)\over r^2}\Psi_{lm}^2-3\pi\bar w_\delta^2\Psi_{lm}^2}r^2\d r\\
\geq\underbrace{\int_0^\infty(l(l+1)-2)\Psi_{lm}^2\d r-3\pi\norm{(\bar w_\delta^2-\bar w_0^2)r^2}_{L^\infty}\int_0^\infty\Psi_{lm}^2\d r}_{=M}
\end{multline*}
For sufficiently small $\delta$ we have
\begin{align*}
M&\geq(l(l+1)-3)\int_0^\infty\Psi_{lm}^2\d r 
\end{align*}
which leads to 
\begin{align*}
\Lambda_{lm}\geq 4\pi(l(l+1)-3)\int_0^\infty\Psi_{lm}^2\d r\geq 0.
\end{align*}

Observe that 
\begin{align*}
&(1+\epsilon)\int_0^\infty\brac{\Psi_{lm}'^2+{l(l+1)\over r^2}\Psi_{lm}^2-3\pi\bar w_\delta^2\Psi_{lm}^2}r^2\d r\\
&\geq\epsilon\int_0^\infty\brac{\Psi_{lm}'^2+{l(l+1)\over r^2}\Psi_{lm}^2-3\pi\bar w_\delta^2\Psi_{lm}^2}r^2\d r+(l(l+1)-3)\int_0^\infty\Psi_{lm}^2\d r.
\end{align*}
Choosing $\epsilon>0$ small enough we see that
\begin{align*}
\int_0^\infty\brac{\Psi_{lm}'^2+{l(l+1)\over r^2}\Psi_{lm}^2-3\pi\bar w_\delta^2\Psi_{lm}^2}r^2\d r
&\gtrsim\int_0^\infty\brac{\Psi_{lm}'^2r^2+(l(l+1)-4)\Psi_{lm}^2}\d r\\
&\gtrsim\int_0^\infty\brac{\Psi_{lm}'^2r^2+l(l+1)\Psi_{lm}^2}\d r.
\end{align*}
We have
\begin{align*}
\Lambda_{lm}&=\int_0^R\brac{{4\over 3}\bar w^{-2}g_{lm}^2+4\pi g_{lm}\Psi_{lm}}r^2\d r
\gtrsim\int_0^\infty\brac{\Psi_{lm}'^2r^2+l(l+1)\Psi_{lm}^2}\d r.
\end{align*}
We can rewrite this as, for some $C>0$,
\begin{align*}
(1+\epsilon)\Lambda_{lm}&\geq\epsilon\int_0^R\brac{{4\over 3}\bar w^{-2}g_{lm}^2+4\pi g_{lm}\Psi_{lm}}r^2\d r
+C\int_0^\infty\brac{\Psi_{lm}'^2r^2+l(l+1)\Psi_{lm}^2}\d r\\
&=\epsilon\int_0^R\brac{{4\over 3}\bar w^{-2}g_{lm}^2r^2+4\pi \brac{(r^2\Psi_{lm}')'-l(l+1)\Psi_{lm}}\Psi_{lm}}\d r\\
&\quad+C\int_0^\infty\brac{\Psi_{lm}'^2r^2+l(l+1)\Psi_{lm}^2}\d r\\
&={4\over 3}\epsilon\int_0^R\bar w^{-2}g_{lm}^2r^2\d r+4\pi\epsilon\int_0^\infty\brac{\Psi_{lm}'^2r^2-l(l+1)\Psi_{lm}^2}\d r\\
&\quad+C\int_0^\infty\brac{\Psi_{lm}'^2r^2+l(l+1)\Psi_{lm}^2}\d r
\end{align*}
where we used the fact that
\begin{multline*}
4\pi\int_0^R\brac{\Psi_{lm}'^2+{l(l+1)\over r^2}\Psi_{lm}^2}r^2\d r-4\pi R^2\Psi_{lm}(R)\Psi_{lm}'(R)
=4\pi\int_0^\infty\brac{\Psi_{lm}'^2+{l(l+1)\over r^2}\Psi_{lm}^2}r^2\d r
\end{multline*}
proved in \eqref{E:Lambda below for higher modes}. Choosing $\epsilon>0$ small enough we obtain the desired \eqref{E:Higher modes bound}. 
\end{proof}

{\em Proof of Theorem~\ref{linear operator coercivity}.}
Combining all the bounds we have for each $l,m$ from Lemmas~\ref{L:00MODE}--\ref{L:HIGHERMODES}, we have
\begin{align*}
\<\mb L\bs\theta,\bs\theta\>_3
&=\sum_{l=0}^\infty\sum_{m=-l}^l\Lambda_{lm}
\gtrsim\sum_{l=0}^\infty\sum_{m=-l}^l\int_0^R\bar w^{-2}g_{lm}^2r^2\d r+\int_0^\infty\brac{\Psi_{lm}'^2r^2+l(l+1)\Psi_{lm}^2}\d r.
\end{align*}
We know
\begin{align*}
\int_{B_R}\bar w^{-2}|\lpc\Psi|^2\d\mb x=4\pi\sum_{l=0}^\infty\sum_{m=-l}^l\int_0^R\bar w^{-2}g_{lm}^2r^2\d r.
\end{align*}
It remains to show that
\begin{align}
\int_{\R^3}|\grad\Psi|^2\d\mb x=4\pi\sum_{l=0}^\infty\sum_{m=-l}^l\int_0^\infty\brac{\Psi_{lm}'^2r^2+l(l+1)\Psi_{lm}^2}\d r\label{norm of grad Psi}
\end{align}
Since $\grad\Psi\in L^2(\R^3)^3$, it has a vector spherical harmonics expansion in $L^2(\R^3)^3$ \cite{Barrera Estevez Giraldo, Freeden Schreiner},
\begin{align}
\grad\Psi=\sum_{l=0}^\infty\sum_{m=-l}^l\brac{\Psi_{lm}^{[0]}\mb Y^{[0]}_{lm}+\Psi_{lm}^{[1]}\mb Y^{[1]}_{lm}+\Psi_{lm}^{[2]}\mb Y^{[2]}_{lm}}.\label{E:grad Psi expansion}
\end{align}
where
\begin{align*}
\mb Y^{[0]}_{lm}=Y_{lm}\hat{\mb r},\qquad\qquad
\mb Y^{[1]}_{lm}=r\grad Y_{lm},\qquad\qquad
\mb Y^{[2]}_{lm}=\mb r\times\grad Y_{lm}
\end{align*}
are the vector spherical harmonics \cite{Barrera Estevez Giraldo, Freeden Schreiner}. We have
\begin{align*}
\Psi^{[0]}_{lm}(r)={1\over r^2}\int_{\partial B_r}\grad\Psi\cdot\mb Y^{[0]}_{lm}\d S={1\over r^2}\int_{\partial B_r}(\partial_r\Psi)Y_{lm}\d S=\Psi_{lm}'(r)
\end{align*}
using \eqref{Psi expansion radial derivative}. And
\begin{align*}
\Psi^{[1]}_{lm}(r)&={1\over l(l+1)r^2}\int_{\partial B_r}\grad\Psi\cdot\mb Y^{[1]}_{lm}\d S
=-{1\over l(l+1)r^2}\int_{\partial B_r}(\Psi r\lpc Y_{lm}+\Psi\hat{\mb r}\cdot\grad Y_{lm})\d S\\
&={1\over r^3}\int_{\partial B_r}\Psi Y_{lm}\d S={1\over r}\Psi_{lm}(r)
\end{align*}
where we used the fact that $\lpc Y_{lm}=-l(l+1)r^{-2}Y_{lm}$. Also,
\begin{align*}
\Psi^{[2]}_{lm}(r)&={1\over l(l+1)r^2}\int_{\partial B_r}\grad\Psi\cdot\mb Y^{[2]}_{lm}\d S
=-{1\over l(l+1)r^2}\int_{\partial B_r}\Psi\grad\cdot(\mb r\times\grad Y_{lm})\d S=0.
\end{align*}
Evaluating $\int_{\R^3}|\grad\Psi|^2\d\mb x$ using \eqref{E:grad Psi expansion} we get \eqref{norm of grad Psi}. This completes the proof of \eqref{norm of grad Psi}.\hfill\qed





\subsection{Momentum and energy}\label{Momentum and energy}


The energy and momentum conservation account for a four-dimensional freedom in the parameter space of the self-similar Goldreich-Weber solutions, see Definition~\ref{self-similar GW def}. 
We shall require that the initial perturbation belongs to a codimension $4$ ``manifold" of initial data so that they have the same total momentum and total energy as the background GW star, i.e. \eqref{initial momentum condition} and \eqref{initial energy condition}. We will show that the linearisation of this requirement allows us to dynamically control 
the inner products
\[
\<\bs\theta,\mb x\>_3, \ \ \<\bs\theta,\mb e_i\>_3, \qquad \ i=1,2,3,
\]
modulo nonlinear terms,
which is necessary for the proof of linear coercivity in Theorem~\ref{linear operator coercivity}.
Hence, by fixing the total momentum and energy, we will be able to apply the non-negativity results we have for the linear operator $\mb L$ to control $\int_{B_R}\bar w^{-2}|\grad\cdot(\bar w^3\partial_s^a\pt^{\beta}\bs\theta)|^2\d\mb x$ with $\<\mb L\partial_s^a\pt^{\beta}\bs\theta,\partial_s^a\pt^{\beta}\bs\theta\>_3+\|\partial_s^{a+1}\pt^{\beta}\bs\theta\|_3^2$ modulo a correction involving non-linear terms. This is the main result of this section, stated and proved in Proposition~\ref{cor1}.




Firstly, the momentum condition \eqref{initial momentum condition} gives us the following.

\begin{lemma}\label{momentum lemma}
Let $\bs\theta$ be a solution of~\eqref{E:EP in self-similar} in the sense of Theorem \ref{T:LOCAL}, and such that $\mb W=\bar{\mb W}$ \eqref{initial momentum condition}. Then 
\begin{align*}
-{1\over 2}\<\partial_s^{a+1}\pt^{\beta}\bs\theta,\mb e_i\>_3^2=\delta\<\partial_s^a\pt^{\beta}\bs\theta,\mb e_i\>_3^2, \qquad \ a\ge0, \, |\beta|\ge0.
\end{align*}
\end{lemma}
\begin{proof}
From Lemma \ref{Momentum and energy in self-similar coordinate}, we see that when $\mb W_\delta[\bs\theta]=\bar{\mb W}$ we have
\begin{align*}
\<\partial_s\bs\theta,\mb e_i\>_3=\b\<\bs\theta,\mb e_i\>_3\qquad\text{for}\qquad i=1,2,3.
\end{align*}
and hence for any $a$ with $\bs\theta$ sufficiently smooth,
\begin{align}
\<\partial_s^{a+1}\bs\theta,\mb e_i\>_3=\b\<\partial_s^a\bs\theta,\mb e_i\>_3\qquad\text{for}\qquad i=1,2,3.\label{E:momentum condition}
\end{align}
Now note that, using integration by parts,
\begin{align}
\<\partial_s^{a+1}\pt_j\pt^{\beta'}\bs\theta,\mb e_i\>_3&=-\<\partial_s^{a+1}\pt^{\beta'}\bs\theta,\pt_j\mb e_i\>_3\\
&=0=\b\<\partial_s^a\pt_j\pt^{\beta'}\bs\theta,\mb e_i\>_3\qquad\text{for}\qquad i=1,2,3.
\end{align}
We are done noting $\delta=-{1\over 2}\b^2$ (\ref{E:delta and b relation}).
\end{proof}



We now turn our attention to the energy condition \eqref{initial energy condition}.



\begin{lemma}
Let $\bs\theta$ be a solution of~\eqref{E:EP in self-similar} in the sense of Theorem \ref{T:LOCAL}, and such that $E=\bar E$ \eqref{initial energy condition}. Then 
\begin{align}
{5\over 2}\b^2\<\partial_s^a\bs\theta,\mb x\>_3&=2\b\<\partial_s^{a+1}\bs\theta,\mb x\>_3
-\int\brac{\bar w^3\partial_s^a|\partial_s\bs\theta-\b\bs\theta|^2+6\bar w^4\partial_s^a\brac{\J^{-{1\over 3}}-1+{1\over 3}\grad\cdot\bs\theta}}\d\mb x\nonumber\\
&\quad-\int\bar w^3\partial_s^a(\K_{\bs\xi}-\K-\K_{\bs\xi}^{(1)})\bar w^3\d\mb x.\label{E:energy condition}
\end{align}
\end{lemma}


\begin{proof}
From Lemma \ref{Momentum and energy in self-similar coordinate}, we see that when $E_\delta[\bs\theta]=\bar E$ we have
\begin{align*}
{5\over 2}\b^2\<\bs\theta,\mb x\>_3&=2\b\<\partial_s\bs\theta,\mb x\>_3\\
&\quad-\int\bigg(\bar w^3|\partial_s\bs\theta-\b\bs\theta|^2+6\bar w^4\brac{\J^{-{1\over 3}}-1+{1\over 3}\grad\cdot\bs\theta}
+\bar w^3(\K_{\bs\xi}-\K-\K_{\bs\xi}^{(1)})\bar w^3\bigg)\d\mb x.
\end{align*}
And hence for any $a\ge0$ the identity~\eqref{E:energy condition} easily follows.
\end{proof}


The next lemma is needed to estimate the term with $\K_{\bs\xi}-\K-\K_{\bs\xi}^{(1)}$.



\begin{lemma}\label{K_2 lemma}
Let $n\geq 20$ and $a+|\beta|\leq n$ with $a>0$. We have
\begin{align*}
|\partial_s^a(\pt_{\mb x}+\pt_{\mb z})^\beta K_2(\mb x,\mb z)|&\lesssim{(E_n+Z_n^2)^{1/2}\over|\mb x-\mb z|^2}\sum_{\substack{0<a'\leq a\\\beta'\leq\beta}} |\partial_s^{a'}\pt^{\beta'}\bs\theta(\mb x)-\partial_s^{a'}\pt^{\beta'}\bs\theta(\mb z)|\\
&\quad+{E_n^{1/2}\over|\mb x-\mb z|^2}\sum_{\beta'\leq\beta}|\pt^{\beta'}\bs\theta(\mb x)-\pt^{\beta'}\bs\theta(\mb z)|,
\end{align*}
where we recall $K_2$~\eqref{E:K2DEF}.
\end{lemma}
\begin{proof}
From Lemma \ref{varpi lemma}, 
\begin{align*}
K_2(\mb x,\mb z)
&=-{1\over 2}{|\bs\theta(\mb x)-\bs\theta(\mb z)|^2\over|\mb x-\mb z|^3}
+{3\over 4|\mb x-\mb z|}\brac{2{(\mb x-\mb z)\cdot(\bs\theta(\mb x)-\bs\theta(\mb z))\over|\mb x-\mb z|^2}+{|\bs\theta(\mb x)-\bs\theta(\mb z)|^2\over|\mb x-\mb z|^2}}^2\\
&\qquad\qquad\qquad\qquad\qquad\varpi_{{1\over 2}}\underbrace{\brac{2{(\mb x-\mb z)\cdot(\bs\theta(\mb x)-\bs\theta(\mb z))\over|\mb x-\mb z|^2}+{|\bs\theta(\mb x)-\bs\theta(\mb z)|^2\over|\mb x-\mb z|^2}}}_{:=y(\mb x,\mb z)}
\end{align*}
Note that $|y(\mb x,\mb z)|\lesssim\|\grad\bs\theta\|_{L^\infty}$. Our a priori assumption \eqref{A priori assumption} together with the embedding theorems \ref{Near boundary embedding theorem} and \ref{Near origin embedding theorem} mean that $\|\grad\bs\theta\|_{L^\infty}$ is bounded by a small constant. So we can assume $|y(\mb x,\mb z)|\leq 1/2$. Then from the definition of $\varpi_q$ \eqref{varpi} we can see that
\begin{align*}
\varpi_{{1\over 2}}^{(k)}(y(\mb x,\mb z))\lesssim 1\qquad\text{for any}\qquad k\geq 0.
\end{align*}
Now using part (ii) of Lemma~\ref{L:ENERGYLEMMA1}, chain and product rule for derivatives and the embedding theorems \ref{Near boundary embedding theorem} and \ref{Near origin embedding theorem}, we can see that $\partial_s^a(\pt_{\mb x}+\pt_{\mb z})^\beta K_2(\mb x,\mb z)$ satisfies the stated bounds. 
\end{proof}

So the energy condition \eqref{initial energy condition} gives us the following.

\begin{lemma}\label{energy gravity lemma}
Let $n\geq 20$ and $a+|\beta|\leq n$ with $a>0, |\beta|\ge0$. Let $\bs\theta$ be a solution of~\eqref{E:EP in self-similar} in the sense of Theorem \ref{T:LOCAL}, and such that $E=\bar E$ \eqref{initial energy condition}\cmmnt{, where $\bar E$ is the total energy associated with the background GW-solution}. Then 
\begin{align*}
\abs{3\delta\<\partial_s^a\pt^{\beta}\bs\theta,\mb x\>_3^2+{24\over 25}\<\partial_s^{a+1}\pt^{\beta}\bs\theta,\mb x\>_3^2}\lesssim S_{n,|\beta|-1,0}+C_\delta(E_n+Z_n^2)^{1/2}E_n.
\end{align*}
\end{lemma}
\begin{proof}
First we deal with the case $|\beta|=0$. From (\ref{E:energy condition}) we get
\begin{align}
2\<\partial_s^{a+1}\bs\theta,\mb x\>_3&={5\over 2}\b\<\partial_s^a\bs\theta,\mb x\>_3
+\b^{-1}\int\brac{\bar w^3\partial_s^a|\partial_s\bs\theta-\b\bs\theta|^2+6\bar w^4\partial_s^a\brac{\J^{-{1\over 3}}-1+{1\over 3}\grad\cdot\bs\theta}}\d\mb x\notag\\
&\quad+\b^{-1}\int\bar w^3\partial_s^a(\K_{\bs\xi}-\K-\K_{\bs\xi}^{(1)})\bar w^3\d\mb x.\label{E:MOD1}
\end{align}
With the embedding theorems \ref{Near boundary embedding theorem} and \ref{Near origin embedding theorem}, it is easy to see that
\begin{align*}
\abs{\int\brac{\bar w^3\partial_s^a\lv \partial_s\bs\theta-\b\bs\theta\rv^2+6\bar w^4\partial_s^a\brac{\J^{-{1\over 3}}-1+{1\over 3}\grad\cdot\bs\theta}}\d\mb x}
\lesssim_\delta(E_n+Z_n^2)^{1/2}E_n^{1/2}.
\end{align*}
Now using Lemmas~\ref{L:ENERGYLEMMA1}--\ref{K_2 lemma} and Young's convolution inequality we have
\begin{align}
\abs{\int\bar w^3\partial_s^a(\K_{\bs\xi}-\K-\K_{\bs\xi}^{(1)})\bar w^3\d\mb x}&=\abs{\int\int\bar w^3(\mb x)\bar w^3(\mb z)\partial_s^aK_2(\mb x,\mb z)\d\mb x\d\mb z}
\lesssim(E_n+Z_n^2)^{1/2}E_n^{1/2}.\label{estimate second order gravity}
\end{align}
Therefore, upon taking the square of~\eqref{E:MOD1} and using the simple bound
$\lv \<\partial_s^a\bs\theta,\mb x\>_3\rv\lesssim E_n^{\frac12}$, we obtain
\begin{align*}
\abs{4\<\partial_s^{a+1}\bs\theta,\mb x\>_3^2-{25\over 4}\b^2\<\partial_s^a\bs\theta,\mb x\>_3^2}\lesssim_\delta(E_n+Z_n^2)^{1/2}E_n,
\end{align*}
which concludes the proof for when $|\beta|=0$ since $\delta=-{1\over 2}\b^2$ (recall~\eqref{E:delta and b relation}).

Now note that, using integration by parts,
\begin{align*}
|\<\partial_s^{a+1}\pt_j\pt^{\beta'}\bs\theta,\mb x\>_3|=|\<\partial_s^{a+1}\pt^{\beta'}\bs\theta,\pt_j\mb x\>_3|\lesssim S_{n,|\beta'|,0}^{1/2}.
\end{align*}
Similarly we have $|\<\partial_s^{a}\pt_j\pt^{\beta'}\bs\theta,\mb x\>_3|\lesssim S_{n,|\beta'|,0}^{1/2}$. Noting that $|\delta|\lesssim 1$ and we are done.
\end{proof}

Finally, the momentum condition \eqref{initial momentum condition} and the energy condition \eqref{initial energy condition} together gives us the following proposition.

\begin{proposition}\label{cor1}
Let $n\geq 20$ and $a+|\beta|\leq n$ with $a>0$. Let $\bs\theta$ be a solution of~\eqref{E:EP in self-similar} in the sense of Theorem \ref{T:LOCAL} such that $\mb W=\bar{\mb W}$ \eqref{initial momentum condition} and $E=\bar E$ \eqref{initial energy condition}. Then
\begin{align}\label{E:COERCIVEBOUNDKEY2}
|\b|^2\int_{B_R}\bar w^{-2}|\grad\cdot(\bar w^3\partial_s^a\pt^{\beta}\bs\theta)|^2\d\mb x
&\lesssim\<\mb L\partial_s^a\pt^{\beta}\bs\theta,\partial_s^a\pt^{\beta}\bs\theta\>_3+{49\over 50}\|\partial_s^{a+1}\pt^{\beta}\bs\theta\|_3^2\nonumber\\
&\quad+S_{n,|\beta|-1,0}+C_\delta(E_n+Z_n^2)^{1/2}E_n.
\end{align}
\end{proposition}


\begin{proof}
Let
\begin{align}
\tilde{\bs\theta}&=\partial_s^a\pt^{\beta}\bs\theta-{\<\partial_s^a\pt^{\beta}\bs\theta,\mb x\>_3\over\|\mb x\|_3^2}\mb x-\sum_{i=1}^3{\<\partial_s^a\pt^{\beta}\bs\theta,\mb e_i\>_3\over\|\mb e_i\|_3^2}\mb e_i\label{E:TILDETHETADEC}\\
{\bs\theta'}&=\partial_s^{a+1}\pt^{\beta}\bs\theta-{\<\partial_s^{a+1}\pt^{\beta}\bs\theta,\mb x\>_3\over\|\mb x\|_3^2}\mb x-\sum_{i=1}^3{\<\partial_s^{a+1}\pt^{\beta}\bs\theta,\mb e_i\>_3\over\|\mb e_i\|_3^2}\mb e_i
\end{align}
Then $\<\tilde{\bs\theta},\mb x\>_3=0=\<\tilde{\bs\theta},\mb e_i\>_3$ for $i=1,2,3$,  and
\begin{align}\label{E:MOD4}
\|\partial_s^{a+1}\pt^{\beta}\bs\theta\|_3^2=\|{\bs\theta'}\|_3^2+{\<\partial_s^{a+1}\pt^{\beta}\bs\theta,\mb x\>_3^2\over\|\mb x\|_3^2}+\sum_{i=1}^3{\<\partial_s^{a+1}\pt^{\beta}\bs\theta,\mb e_i\>_3^2\over\|\mb e_i\|_3^2}.
\end{align}
Since $\mb x$ and $\mb e_i$ are eigenfunctions of $\mb L$ with eigenvalues $3\delta$ and $\delta$ respectively, we have
\begin{align}
&\<\mb L\partial_s^a\pt^{\beta}\bs\theta,\partial_s^a\pt^{\beta}\bs\theta\>_3\notag\\
&=\Bigg\<\mb L\tilde{\bs\theta}+3\delta{\<\partial_s^a\pt^{\beta}\bs\theta,\mb x\>_3\over\|\mb x\|_3^2}\mb x+\sum_{i=1}^3\delta{\<\partial_s^a\pt^{\beta}\bs\theta,\mb e_i\>_3\over\|\mb e_i\|_3^2}\mb e_i,
\tilde{\bs\theta}+{\<\partial_s^a\pt^{\beta}\bs\theta,\mb x\>_3\over \|\mb x\|_3^2}\mb x+\sum_{i=1}^3{\<\partial_s^a\pt^{\beta}\bs\theta,\mb e_i\>_3\over\|\mb e_i\|_3^2}\mb e_i\Bigg\>_3\notag\\
&=\<\mb L\tilde{\bs\theta},\tilde{\bs\theta}\>_3+3\delta{\<\partial_s^a\pt^{\beta}\bs\theta,\mb x\>_3^2\over\|\mb x\|_3^2}+\sum_{i=1}^3\delta{\<\partial_s^a\pt^{\beta}\bs\theta,\mb e_i\>_3^2\over\|\mb e_i\|_3^2}. \label{E:MOD3}
\end{align}
We use Lemmas~\ref{momentum lemma} and~\ref{energy gravity lemma} to control the last two terms on the right-most side of~\eqref{E:MOD3} to get
\begin{align}
\<\mb L\tilde{\bs\theta},\tilde{\bs\theta}\>_3&\leq\<\mb L\partial_s^a\pt^{\beta}\bs\theta,\partial_s^a\pt^{\beta}\bs\theta\>_3+{24\over 25} {\<\partial_s^{a+1}\pt^{\beta}\bs\theta,\mb x\>_3^2\over\|\mb x\|_3^2} + \frac12 \sum_{i=1}^3{\<\partial_s^{a+1}\pt^{\beta}\bs\theta,\mb e_i\>_3^2\over\|\mb e_i\|_3^2}\notag\\
&\quad+CS_{n,|\beta|-1,0}+C_\delta(E_n+Z_n^2)^{1/2}E_n \\
& \le \<\mb L\partial_s^a\pt^{\beta}\bs\theta,\partial_s^a\pt^{\beta}\bs\theta\>_3+{24\over 25}\|\partial_s^{a+1}\pt^{\beta}\bs\theta\|_3^2
+CS_{n,|\beta|-1,0}+C_\delta(E_n+Z_n^2)^{1/2}E_n,  \label{E:MOD2}
\end{align}
where we have used~\eqref{E:MOD4} in the last line.
We now use the decomposition~\eqref{E:TILDETHETADEC} and then apply Theorem~\ref{linear operator coercivity} (with $\bs\theta=\tilde{\bs\theta}$) to obtain
\begin{align}
&\epsilon\int_{B_R}\bar w^{-2}|\grad\cdot(\bar w^3\partial_s^a\pt^{\beta}\bs\theta)|^2\d\mb x \notag \\
& \leq C\epsilon\int_{B_R}\bar w^{-2}|\grad\cdot(\bar w^3\tilde{\bs\theta})|^2\d\mb x+C\epsilon {\<\partial_s^a\pt^{\beta}\bs\theta,\mb x\>_3^2\over\|\mb x\|_3^2}+
C\epsilon \sum_{i=1}^3{\<\partial_s^a\pt^{\beta}\bs\theta,\mb e_i\>_3^2\over\|\mb e_i\|_3^2}\notag\\
&\leq   \<\mb L\tilde{\bs\theta},\tilde{\bs\theta}\>_3 + \frac{C\epsilon}{|\delta|} \|\partial_s^{a+1}\pt^{\beta}\bs\theta\|_3^2
+CS_{n,|\beta|-1,0}+C_\delta(E_n+Z_n^2)^{1/2}E_n\notag\\
& \leq \<\mb L\partial_s^a\pt^{\beta}\bs\theta,\partial_s^a\pt^{\beta}\bs\theta\>_3+{24\over 25}\|\partial_s^{a+1}\pt^{\beta}\bs\theta\|_3^2
+ \frac{C\epsilon}{|\delta|} \|\partial_s^{a+1}\pt^{\beta}\bs\theta\|_3^2
+CS_{n,|\beta|-1,0}+ C_\delta(E_n+Z_n^2)^{1/2}E_n\notag\\
& \leq \<\mb L\partial_s^a\pt^{\beta}\bs\theta,\partial_s^a\pt^{\beta}\bs\theta\>_3+{49\over 50}\|\partial_s^{a+1}\pt^{\beta}\bs\theta\|_3^2
+CS_{n,|\beta|-1,0}+C_\delta(E_n+Z_n^2)^{1/2}E_n,
\end{align}
where we have chosen $\epsilon$ small enough so that $C\epsilon\lesssim 1$ in the second line and then further shrink $\epsilon$ so that $\frac{C\epsilon}{|\delta|}<\frac1{50}$ in the fourth line. Note that since $\delta=-{1\over 2}\b^2$ (recall~\eqref{E:delta and b relation}), the dependence of $\epsilon$ on $\b$ is $\epsilon\sim|\b|^2$. We have used Lemmas~\ref{momentum lemma} and~\ref{energy gravity lemma} in the second bound, and~\eqref{E:MOD2} in the third bound.
\end{proof}


\subsection{Coercivity via irrotationality}\label{S:CI}


Note that Proposition~\ref{cor1} only 
controls the weighted divergence $g=\grad\cdot(\bar w^3\bs\theta)$ and not the norms of $\bs\theta$ in our energy spaces. It is therefore still not strong enough for our energy estimates in Sections~\ref{S:ENERGYESTIMATES}--\ref{S:EE2}.  To derive the coercivity we seek, we must mod out the kernel of $\mb L$, i.e. the subspace of $\bs\theta$ with weighted divergence $g=0$. This is naturally linked to the assumption of irrotationality~\eqref{initial irrotational condition} which guarantees, we show this in the key result of this section -- Proposition~\ref{L-estimate-tan}, that we can in fact dynamically control  $\|\partial_s^a\pt^\beta\bs\theta\|_{3}^2+\|\partial_s^a\grad\pt^\beta\bs\theta\|_{4}^2$ modulo lower order nonlinear terms.



\subsubsection{Lagrangian description of irrotationality}


From~\eqref{E:LINDEF}
it is clear that any $H^2$ vectorfield $\bs\theta$ such that $g=\grad\cdot(\bar w^3\bs\theta)=0$ is in the kernel of the operator $\mb L$.
In particular, to obtain strict coercivity of $\mb L$ we restrict ourselves to $\langle\cdot,\cdot\rangle_3$-orthogonal complement of
$K=\{\bs\theta:\grad\cdot(\bar w^3\bs\theta)=0\}$. Note that $\{\bs\theta=\grad\vartheta\}\subseteq K^\perp$ since for any $\bs\theta_0\in K$ we have
\begin{align*}
\<\grad\vartheta,\bs\theta_0\>_3=\int\grad\vartheta\cdot\bs\theta_0\bar w^3\d\mb x=\int\vartheta\grad\cdot(\bs\theta_0\bar w^3)\d\mb x=0.
\end{align*}

Therefore, the natural assumption to hope for the strict coercivity of the term on the left-hand side of~\eqref{E:COERCIVEBOUNDKEY2} is that 
$\bs\theta$ is in fact a gradient. In this section we show that this is true to the top order if we assume that the fluid is irrotational. The challenge is that the irrotationality condition
in the Lagrangian variables~\eqref{initial irrotational condition} is expressed at the level of the $s$-derivative of the flow map, and a careful analysis is necessary to obtain satisfactory lower bounds.



\begin{lemma}\label{L:CURL1}
Let $\bs\theta$ be a solution of~\eqref{E:EP in self-similar} in the sense of Theorem \ref{T:LOCAL}, given on its maximal interval of existence. Assume further that the fluid is irrotational, i.e. initially~\eqref{initial irrotational condition} holds. Then for $a>0$ we have
\begin{align}
\partial_s\bs\theta&=\grad\brac{\tilde{H}+{1\over 2}\b|\bs\theta+\mb x|^2}-(\partial_s\theta^k)\grad\theta^k \label{E:CURL1}\\
\partial_s^a\bs\theta&=\grad H_a - \sum_{j=0}^{\lfloor{a-1\over 2}\rfloor}C_{a,j}(\partial_s^{a-j}\theta^k)\grad\partial_s^j\theta^k  \label{E:CURL2}
\end{align}
for some real constants $C_{a,j}$, $j\in\{1,\dots,\lfloor{a-1\over 2}\rfloor\}$ and $H^1$-functions $H_a$ and $\tilde{H}$.
\end{lemma}


\begin{proof}
Since the Euler-Poisson equation preserves the fluid irrotational condition, \eqref{initial irrotational condition} implies that $\curl_A\partial_t\bs\eta=\mb 0$ for $t$, or equivalently in Eulerian coordinates $\grad\times\mb u=0$. Since any curl-free vector field can be written as a gradient, we have $\mb u=\grad\hat H$ for some $\hat H$, or equivalently $\partial_t\bs\eta=A\grad H$ for some $H$ in Lagrangian coordinates. Since
\begin{align*}
\partial_t\bs\eta=\lambda^{-3/2}\partial_s(\lambda(\bs\theta+\mb x))
=\lambda^{-3/2}((\bs\theta+\mb x)\partial_s\lambda+\lambda\partial_s\bs\theta)
=\lambda^{-1/2}(\partial_s\bs\theta-\b(\bs\theta+\mb x)),
\end{align*}
this means on the level of $\bs\theta$ we have
\[\partial_s\bs\theta-\b(\bs\theta+\mb x)=\A\grad\tilde{H}.\]
Hence we have
\begin{align*}
\partial_s\bs\theta&=\grad\tilde{H}+(\A-I)\grad\tilde{H}+\b(\bs\theta+\mb x)
=\grad\tilde{H}+(I-\A^{-1})\A\grad\tilde{H}+\b(\bs\theta+\mb x)\nonumber\\
&=\grad\tilde{H}-(\partial_s\theta^k-\b(\theta^k+x^k))\grad\theta^k+\b(\bs\theta+\mb x)
=\grad\brac{\tilde{H}+{1\over 2}\b|\bs\theta+\mb x|^2}-(\partial_s\theta^k)\grad\theta^k.
\end{align*}
This proves~\eqref{E:CURL1}. To prove~\eqref{E:CURL2}, we will use induction. We have shown that it is true for $a=1$. Suppose it is true for some $a\ge1$. Then
\begin{align*}
\partial_s^{a+1}\bs\theta&=\grad\partial_sH_a - \partial_s\sum_{j=0}^{\lfloor{a-1\over 2}\rfloor}C_{a,j}(\partial_s^{a-j}\theta^k)\grad\partial_s^j\theta^k\\
&=\grad\partial_sH_a - \sum_{j=0}^{\lfloor{a-1\over 2}\rfloor}C_{a,j}(\partial_s^{a+1-j}\theta^k)\grad\partial_s^j\theta^k - \sum_{j=0}^{\lfloor{a-1\over 2}\rfloor}C_{a,j}(\partial_s^{a-j}\theta^k)\grad\partial_s^{j+1}\theta^k\\
&=\grad\partial_sH_a - \sum_{j=0}^{\lfloor{a-1\over 2}\rfloor}C_{a,j}(\partial_s^{a+1-j}\theta^k)\grad\partial_s^j\theta^k - \sum_{j=1}^{\lfloor{a-1\over 2}\rfloor+1}C_{a,j-1}(\partial_s^{a+1-j}\theta^k)\grad\partial_s^j\theta^k.
\end{align*}
Note that $\lfloor{a-1\over 2}\rfloor+1>\lfloor{a\over 2}\rfloor$ if and only if $a$ is odd. Assume therefore that $a=2a'+1$ for some $a'\in\mathbb N\cup\{0\}$. Then $\lfloor{a-1\over 2}\rfloor+1=a'+1$ and $\lfloor{a\over 2}\rfloor=a'$. When $j=a'+1$ in the last sum, we have
\begin{align*}
C_{2a'+1,a'}(\partial_s^{2a'+2-(a'+1)}\theta^k)\grad\partial_s^{a'+1}\theta^k&=C_{2a'+1,a'}(\partial_s^{a'+1}\theta^k)\grad\partial_s^{a'+1}\theta^k
={1\over 2}C_{2a'+1,a'}\grad(\partial_s^{a'+1}\theta^k)^2
\end{align*}
which can be absorbed into $H_{a+1}$. Therefore
\begin{align*}
\partial_s^{a+1}\bs\theta
&=\grad\brac{\partial_sH_a+{1\over 2}\mb 1[a\text{ odd}]C_{a,a'}\grad(\partial_s^{a'+1}\theta^k)^2}\\
&\quad - \sum_{j=0}^{\lfloor{a-1\over 2}\rfloor}C_{a,j}(\partial_s^{a+1-j}\theta^k)\grad\partial_s^j\theta^k - \sum_{j=1}^{\lfloor{a\over 2}\rfloor}C_{a,j-1}(\partial_s^{a+1-j}\theta^k)\grad\partial_s^j\theta^k\\
&=\grad H_{a+1} - \sum_{j=0}^{\lfloor{a\over 2}\rfloor}C_{a+1,j}(\partial_s^{a-j}\theta^k)\grad\partial_s^j\theta^k,
\end{align*}
where $\mb 1[\star]$ denotes the Iverson bracket (see Definition \ref{Special notations}). This completes the induction argument.
\end{proof}


\begin{remark}
The above lemma is a purely structural statement about (suitably smooth) irrotational fields. Strictly speaking we do not need $\bs\theta$ to be a solution of the Euler-Poisson system~\eqref{E:EP in self-similar}. 
\end{remark}


With this we can now show that the curl of $\partial_s^a\pr^b\pt^\beta\bs\theta$ equals lower order terms and non-linear terms.


\begin{lemma}\label{pre-curl-highorder}
Let $\bs\theta$ be a solution of~\eqref{E:EP in self-similar} in the sense of Theorem \ref{T:LOCAL}, given on its maximal interval of existence. Assume further that the fluid is irrotational, i.e. initially~\eqref{initial irrotational condition} holds. Then for $a>0$ we have
\begin{align}
\grad\times\partial_s^a\bs\theta=-\partial_s^{a-1}((\grad\partial_s\theta^k)\times\grad\theta^k). \label{E:CURLHIGH1}
\end{align}
Moreover, for some constants $C_{\gamma,\beta}>0$ we have
\begin{align}
\grad\times\partial_s^a\pt^\beta\bs\theta&=-\partial_s^{a-1}\pt^\beta((\grad\partial_s\theta^k)\times\grad\theta^k)+\sum_{|\gamma|<|\beta|}C_{\gamma,\beta}\<\grad\partial_s^a\pt^\gamma\bs\theta\>,\label{E:CURLHIGH2}\\
\grad\times\partial_s^a\pr^b\pt^\beta\bs\theta&=-\partial_s^{a-1}\pr^b\pt^\beta((\grad\partial_s\theta^k)\times\grad\theta^k)+\sum_{|\gamma|+d<|\beta|+b}C_{\gamma,\beta}\<\grad\partial_s^a\pr^d\pt^\gamma\bs\theta\>, \label{E:CURLHIGH3}
\end{align}
where we recall notations defined in Definition~\ref{Special notations}.
\end{lemma}


\begin{proof}
Apply $\grad\times\partial_s^{a-1}$ to~\eqref{E:CURL1}
to get~\eqref{E:CURLHIGH1}.
Formulas~\eqref{E:CURLHIGH2}--\eqref{E:CURLHIGH3} follow trivially when $|\beta|=0=b$. Now assume formula~\eqref{E:CURLHIGH2} is true for a multi-index $\beta$, $|\beta|\geq 0$. Then
\begin{align*}
\grad\times\partial_s^a\pt_j\pt^\beta\bs\theta
&=\<\grad\partial_s^a\pt^\beta\bs\theta\>+\pt_j\grad\times\partial_s^a\pt^\beta\bs\theta
\\
&=\<\grad\partial_s^a\pt^\beta\bs\theta\>-\partial_s^{a-1}\pt_j\pt^\beta((\grad\partial_s\theta^k)\times\grad\theta^k)+\sum_{|\gamma|<|\beta|}C_{\gamma,\beta}\<\pt_j\grad\partial_s^a\pt^\gamma\bs\theta\>\\
&=-\partial_s^{a-1}\pt_j\pt^\beta((\grad\partial_s\theta^k)\times\grad\theta^k)+\sum_{|\gamma|<|\beta|+1}C_{\gamma,\beta}'\<\grad\partial_s^a\pt^\gamma\bs\theta\>,
\end{align*}
where we recall the notation from Definition~\ref{Special notations} and the commutation relation $[\pt_j,\grad]=\<\grad\>$ from Lemma \ref{Commutation relations}.
The proof then follows by induction. The proof of~\eqref{E:CURLHIGH3} is similar, using the commutation relation $[\pr,\grad]=\<\grad\>$ from Lemma \ref{Commutation relations}.
\end{proof}




\begin{corollary}\label{curl-highorder}
Let $\bs\theta$ be a solution of~\eqref{E:EP in self-similar} in the sense of Theorem \ref{T:LOCAL}, given on its maximal interval of existence. Assume further that the fluid is irrotational, i.e. initially~\eqref{initial irrotational condition} holds. 
Let $n\geq 20$. 
\begin{enumerate}
\item For $a+|\beta|\leq n$ with $a>0$ we have
\begin{align*}
\|\grad\times\partial_s^a\pt^\beta\bs\theta\|_4^2\lesssim S_{n,|\beta|-1,0}+(E_n+Z_n^2)E_n.
\end{align*}
\item For $a+|\beta|+b\leq n$ with $a>0$ we have
\begin{align*}
\|\grad\times\partial_s^a\pr^b\pt^\beta\bs\theta\|_{4+b}^2\lesssim S_{n,|\beta|+b-1}+(E_n+Z_n^2)E_n.
\end{align*}
\end{enumerate}
\end{corollary}


\begin{proof}
Use Lemma \ref{pre-curl-highorder} and note that
\begin{align*}
\norm{\partial_s^{a-1}\pt^\beta((\grad\partial_s\theta^k)\times\grad\theta^k)}_4^2&\lesssim (E_n+Z_n^2)E_{n}\\
\norm{\sum_{|\gamma|<|\beta|}C_{\gamma,\beta}\<\grad\partial_s^a\pt^\gamma\bs\theta\>}_4^2&\lesssim S_{n,|\beta|-1,0},
\end{align*}
which yields the first claim. The second claim follows similarly. 
\end{proof}






\subsubsection{Coercivity of $\mb L$}


The lemmas in the last subsection showed that $\partial_s^a\bs\theta$ is a gradient on the linear level, which will ultimately help us show that $\|\partial_s^a\pt^\beta\bs\theta\|_{3}^2+\|\partial_s^a\grad\pt^\beta\bs\theta\|_{4}^2$ can be ``controlled" by the linearised dynamics.
We start by showing we can control $\|\partial_s^a\bs\theta\|^2_3$ in the following lemma.


\begin{lemma}\label{cor2}
Let $\bs\theta$ be a solution of~\eqref{E:EP in self-similar} in the sense of Theorem \ref{T:LOCAL}, given on its maximal interval of existence. Assume further that the fluid is irrotational, i.e. initially~\eqref{initial irrotational condition} holds. Let $n\ge 20$. Then we have the bound
\be\label{E:PARTONECOERCIVE}
\|\partial_s^a\bs\theta\|^2_3\lesssim\int_{B_R}\bar w^{-2}|\grad\cdot(\bar w^3\partial_s^a\bs\theta)|^2\d\mb x+(E_n+Z_n^2)^{1/2}E_{n}\qquad\text{for}\qquad 0<a\leq n.
\ee
\end{lemma}


\begin{proof}
Let $g=\grad\cdot(\bar w^3\partial_s^a\bs\theta)$. Multiply both sides of this equation by $H_a$ and integrate over $B_R$ to get
\begin{align*}
\int_{B_R}\bar w^3(\grad H_a)\cdot\partial_s^a\bs\theta\;\d\mb x&=-\int_{B_R}gH_a\d\mb x=-\int_{B_R}g(H_a-(H_a)_{B_{2R/3}})\d\mb x\\
&\leq\epsilon^{-1}\int_{B_R}\bar w^{-2}g^2\d\mb x+\epsilon\int_{B_R}(H_a-(H_a)_{B_{2R/3}})^2\bar w^2\d\mb x\\
&\leq\epsilon^{-1}\int_{B_R}\bar w^{-2}g^2\d\mb x+\epsilon C'\int_{B_R}|\grad H_a|^2\bar w^4\d\mb x
\end{align*}
where we have used the Hardy-Poincar\'e inequality in the last line, see Theorem~\ref{Hardy-Poincare inequality}. From this and Lemma~\ref{L:CURL1} we get
\begin{align*}
\int_{B_R}\bar w^3|\partial_s^a\bs\theta|^2\d\mb x
&\leq\epsilon^{-1}\int_{B_R}\bar w^{-2}g^2\d\mb x+\epsilon C'\int_{B_R}|\partial_s^a\bs\theta|^2\bar w^4\d\mb x
+(1+\epsilon)C''(E_n+Z_n^2)^{1/2}E_{n}
\end{align*}
where we bound for example
\begin{align*}
&\abs{\int_{B_R}\bar w^3\brac{\sum_{j=0}^{\lfloor{a-1\over 2}\rfloor}C_{a,j}(\partial_s^{a-j}\theta^k)\grad\partial_s^j\theta^k}\cdot\partial_s^a\bs\theta\;\d\mb x}\\
&\lesssim(E_n+Z_n^2)^{1/2}\abs{\int_{B_R}\bar w^3\brac{\sum_{j=0}^{\lfloor{a-1\over 2}\rfloor}C_{a,j}(\partial_s^{a-j}\theta^k)}\cdot\partial_s^a\bs\theta\;\d\mb x}
\lesssim(E_n+Z_n^2)^{1/2}S_{n}
\end{align*}
Choosing $\epsilon$ small enough, we get~\eqref{E:PARTONECOERCIVE}.
\end{proof}


Before proving the key result of this section, we have the following structural decomposition, which holds for any sufficiently smooth vectorfield $\bs\theta$.


\begin{lemma}\label{g-identity}
For any $\bs\theta$ such that $\|\bs\theta\|_3+\|\grad\bs\theta\|_4<\infty$ we have
\begin{align*}
&\int_{B_R}{4\over 3}\bar w^{-2}|\grad\cdot(\bar w^3\bs\theta)|^2\d\mb x\\
&=\int_{B_R}\brac{\bar w^4\brac{{1\over 3}|\grad\cdot\bs\theta|^2+|\grad\bs\theta|^2+[\curl\bs\theta]^k_l\partial_k\theta^l}-4\bar w^3\theta^k\theta^l\partial_k\partial_l\bar w}\d\mb x\\
&=\int_{B_R}\bigg(\bar w^4\brac{{1\over 3}|\grad\cdot\bs\theta|^2+|\grad\bs\theta|^2-{1\over 2}|\curl\bs\theta|^2}
-\bar w^3\brac{\bar w''|\bs\theta\cdot\mb e_r|^2+{\bar w'\over r}(|\bs\theta|^2-|\bs\theta\cdot\mb e_r|^2)}\bigg)\d\mb x.
\end{align*}
where $\mb e_r$ denotes the radial unit vector $\mb x/|\mb x|$.
\end{lemma}


\begin{proof}
The first line follows from the following identity:
\begin{align*}
-{4\over 3}\grad(\bar w^{-2}\grad\cdot(\bar w^3\bs\theta ))
&=-{1\over 3\bar w^3}\grad(\bar w^4\grad\cdot\bs\theta )-{1\over\bar w^3}\partial_k(\bar w^4\grad\theta^k)-4\bs\theta \cdot\grad\grad\bar w\\
&=-{1\over 3\bar w^3}\grad(\bar w^4\grad\cdot\bs\theta )-{1\over\bar w^3}\partial_k(\bar w^4(\partial_k\bs\theta +[\curl\bs\theta ]^k_\bullet)
-4\bs\theta \cdot\grad\grad\bar w.
\end{align*}
And then the second line follows from
\begin{align*}
[\curl\bs\theta]^k_l[\curl\bs\theta]^k_l&=(\partial_l\theta^k-\partial_k\theta^l)(\partial_l\theta^k-\partial_k\theta^l)
=(\partial_k\theta^l-\partial_l\theta^k)\partial_k\theta^l-(\partial_l\theta^k-\partial_k\theta^l)\partial_k\theta^l\\
&=-2(\partial_l\theta^k-\partial_k\theta^l)\partial_k\theta^l
=-2[\curl\bs\theta]^k_l\partial_k\theta^l
\end{align*}
and
\begin{align*}
\theta^k\theta^l\partial_k\partial_l\bar w
&=\theta^k\theta^l\partial_k\brac{\bar w'{x^l\over r}}
=\theta^k\theta^l\brac{\bar w''{x^lx^k\over r^2}+\bar w'{\delta^l_k\over r}-\bar w'{x^lx^k\over r^3}}
=\bar w''|\bs\theta\cdot\mb e_r|^2+{\bar w'\over r}(|\bs\theta|^2-|\bs\theta\cdot\mb e_r|^2). 
\end{align*}
\end{proof}

Using this, we can now prove that we can control $\|\partial_s^a\bs\theta\|_{3}^2+\|\partial_s^a\grad\bs\theta\|_{4}^2$.

\begin{proposition}\label{L-estimate}
Let $n\geq 20$. Let $\bs\theta$ be a solution of~\eqref{E:EP in self-similar} in the sense of Theorem \ref{T:LOCAL}, given on its maximal interval of existence. 
Assume further that the energy, momentum, and irrotationality constraints~\eqref{initial momentum condition},~\eqref{initial energy condition}, and~\eqref{initial irrotational condition} hold respectively.
Then for any $0<a\leq n$ we have
\begin{align}
\norm{\partial_s^a\bs\theta}_{3}^2+\norm{\partial_s^a\grad\bs\theta}_{4}^2\lesssim|\b|^{-2}\brac{{49\over 50}\norm{\partial_s^{a+1}\bs\theta}_{3}^2+\<\mb L\partial_s^a\bs\theta,\partial_s^a\bs\theta\>}+C_\delta(E_n+Z_n^2)^{1/2}E_{n}
\end{align}
\end{proposition}
\begin{proof}
Combining Proposition \ref{cor1} and Lemma \ref{cor2} we have, for small $\epsilon$,
\begin{multline*}
\underbrace{\epsilon\int_{B_R}\bar w^{-2}|\grad\cdot(\bar w^3\partial_s^a\bs\theta)|^2\d\mb x+\|\partial_s^a\bs\theta\|_3^2}_{:=M}
\lesssim|\b|^{-2}\brac{\<\mb L\partial_s^a\bs\theta,\partial_s^a\bs\theta\>_3+{49\over 50}\|\partial_s^{a+1}\bs\theta\|_3^2}+C_\delta(E_n+Z_n^2)^{1/2}E_{n}.
\end{multline*}
Note that by Corollary \ref{curl-highorder} $\|\curl\partial_s^a\bs\theta\|_4^2\lesssim(E_n+Z_n^2)E_{n}$.
Using Lemma \ref{g-identity} we have
\begin{align*}
M&=\int_{\R^3}\brac{\epsilon\bar w^{-2}|\grad\cdot(\bar w^3\partial_s^a\bs\theta )|^2+\bar w^3|\partial_s^a\bs\theta |^2}\d\mb x\\
&\geq\int_{\R^3}\brac{\epsilon{3\over 4}\bar w^4\brac{|\grad\partial_s^a\bs\theta |^2-{1\over 2}|\curl\partial_s^a\bs\theta|^2}-\epsilon{3\over 4}\bar w^3\bar w''|\partial_s^a\bs\theta\cdot\mb e_r|^2+\bar w^3|\partial_s^a\bs\theta |^2}\d\mb x\\
&\geq\int_{\R^3}\brac{\epsilon{3\over 4}\bar w^4|\grad\partial_s^a\bs\theta |^2-\epsilon{3\over 4}\bar w^3\bar w''|\partial_s^a\bs\theta\cdot\mb e_r|^2+\bar w^3|\partial_s^a\bs\theta|^2}\d\mb x-\epsilon C(E_n+Z_n^2)E_{n}
\end{align*}
Choosing $\epsilon$ small enough, we then have
\begin{align*}
M+(E_n+Z_n^2)S_{n}&\gtrsim\int_{\R^3}\brac{\bar w^4|\grad\partial_s^a\bs\theta |^2+\bar w^3|\partial_s^a\bs\theta |^2}\d\mb x. 
\end{align*}
\end{proof}

Next we will upgrade our estimate to control $\|\partial_s^a\pt^\beta\bs\theta\|_{3}^2+\|\partial_s^a\grad\pt^\beta\bs\theta\|_{4}^2$ for $|\beta|>0$. First we will need the following lemma.

\begin{lemma}\label{L-estimate-tan lemma}
For any vector field $\bs\theta$ and $\epsilon>0$,
\begin{align*}
\|\bs\theta\|_3^2\lesssim\epsilon\|\pr\bs\theta\|_4^2+(1+\epsilon^{-1})\|\bs\theta\|_4^2.
\end{align*}
\end{lemma}
\begin{proof}
We have
\begin{align*}
-\int_{B_R}|\bs\theta|^2r\bar w'\bar w^3\d\mb x
&=-{1\over 4}\int_{B_R}|\bs\theta|^2\pr\bar w^4\d\mb x
={1\over 2}\int_{B_R}\bar w^4\bs\theta\cdot\pr\bs\theta\;\d\mb x+{3\over 4}\int_{B_R}|\bs\theta|^2\bar w^4\d\mb x\\
&\leq\epsilon\|\pr\bs\theta\|_4^2+{1\over 4}(3+\epsilon^{-1})\|\bs\theta\|_4^2.
\end{align*}
Now
\begin{align*}
\|\bs\theta\|_3^2\lesssim\|\bs\theta\|_4^2-\int_{B_R}|\bs\theta|^2r\bar w'\bar w^3\d\mb x\leq\epsilon\|\pr\bs\theta\|_4^2+{1\over 4}(7+\epsilon^{-1})\|\bs\theta\|_4^2. 
\end{align*}
\end{proof}

\begin{proposition}\label{L-estimate-tan}
Let $n\geq 20$. Let $\bs\theta$ be a solution of~\eqref{E:EP in self-similar} in the sense of Theorem \ref{T:LOCAL}, given on its maximal interval of existence. 
Assume further that the energy, momentum, and irrotationality constraints~\eqref{initial momentum condition},~\eqref{initial energy condition}, and~\eqref{initial irrotational condition} hold respectively. Then for any $a+|\beta|\leq n$ with $a,|\beta|>0$ we have
\begin{multline}\label{E:INDONSET}
\norm{\partial_s^a\pt^\beta\bs\theta}_{3}^2+\norm{\partial_s^a\grad\pt^\beta\bs\theta}_{4}^2\\\lesssim|\b|^{-2}\brac{{49\over 50}\norm{\partial_s^{a+1}\pt^\beta\bs\theta}_{3}^2+\<\mb L\partial_s^a\pt^\beta\bs\theta,\partial_s^a\pt^\beta\bs\theta\>}+CS_{n,|\beta|-1,0}+C_\delta(E_n+Z_n^2)^{1/2}E_{n}
\end{multline}
\end{proposition}
\begin{proof}
By Proposition \ref{cor1} and Lemma \ref{g-identity} we have
\begin{align*}
&\|\grad\partial_s^a\pt^{\beta}\bs\theta\|_4^2-{1\over 2}\|\curl\partial_s^a\pt^{\beta}\bs\theta\|_4^2-\int_{B_R}\bar w''\bar w^3|\partial_s^a\pt^{\beta}\bs\theta\cdot\mb e_r|^2\d\mb x\\
&\leq\int_{B_R}{4\over 3}\bar w^{-2}|\grad\cdot(\bar w^3\partial_s^a\pt^{\beta}\bs\theta)|^2\d\mb x\\
&\lesssim|\b|^{-2}\brac{\<\mb L\partial_s^a\pt^{\beta}\bs\theta,\partial_s^a\pt^{\beta}\bs\theta\>_3+{49\over 50}\|\partial_s^{a+1}\pt^{\beta}\bs\theta\|_3^2}+C_\delta(E_n+Z_n^2)E_n.
\end{align*}
By Corollary \ref{curl-highorder} we have
\begin{align*}
&\|\grad\partial_s^a\pt^{\beta}\bs\theta\|_4^2-\int_{B_R}\bar w''\bar w^3|\partial_s^a\pt^{\beta}\bs\theta\cdot\mb e_r|^2\d\mb x\\
&\lesssim|\b|^{-2}\brac{\<\mb L\partial_s^a\pt^{\beta}\bs\theta,\partial_s^a\pt^{\beta}\bs\theta\>_3+{49\over 50}\|\partial_s^{a+1}\pt^{\beta}\bs\theta\|_3^2}+CS_{n,|\beta|-1,0}+C_\delta(E_n+Z_n^2)E_n.
\end{align*}
Now by Lemma \ref{L-estimate-tan lemma}, we have
${1\over 2\epsilon}\|\partial_s^a\pt^{\beta}\bs\theta\|_3^2-{1\over 2}\|\grad\partial_s^a\pt^{\beta}\bs\theta\|_4^2\lesssim_{\epsilon}\|\partial_s^a\pt^{\beta}\bs\theta\|_4^2\leq S_{n,|\beta|-1,0}.$
Adding this and the above equation, and chosing $\epsilon$ small enough we get
\begin{multline*}
\|\partial_s^a\pt^{\beta}\bs\theta\|_3^2+\|\grad\partial_s^a\pt^{\beta}\bs\theta\|_4^2\\\lesssim|\b|^{-2}\brac{\<\mb L\partial_s^a\pt^{\beta}\bs\theta,\partial_s^a\pt^{\beta}\bs\theta\>_3+{49\over 50}\|\partial_s^{a+1}\pt^{\beta}\bs\theta\|_3^2}+CS_{n,|\beta|-1,0}+C_\delta(E_n+Z_n^2)E_n. 
\end{multline*}
\end{proof}


\begin{remark}
Estimate~\eqref{E:INDONSET} features an order 1 term $CS_{n,|\beta|-1,0}$ on the right-hand side. This could be problematic for the closure of the estimates, but
the key point is that this term is effectively decoupled, as it features one tangential derivative less. This will allows us later to close the estimates via induction on the 
order of derivatives in the problem.
\end{remark}



\subsection{Reduction to linear problem}\label{S:ENERGYESTIMATES}


 In order to prove the bound~\eqref{E:NONLINEARBOUND}, we will need to apply the coercivity estimates from Section~\ref{S:CI}. In particular, we must control the non-linear terms 
 in order to effectively reduce the problem to a linear one. In Sections~\ref{Estimating the non-linear part of the pressure term} and~\ref{Estimating the linear and non-linear part of the gravity term} we will prove high-order energy bounds for the nonlinear contributions from the pressure and the gravity term respectively. We will also prove high-order energy bounds for the full gravity term (including the linear part) in Section~\ref{Estimating the linear and non-linear part of the gravity term} that we will need for induction on radial derivatives. Then using these, we will reduce the full non-linear problem to the linear one in Section~\ref{Reduction to linear problem}.  This will then allow us to prove energy estimates and our main theorem in Section~\ref{S:EE2}.

\subsubsection{Estimating the non-linear part of the pressure term}
\label{Estimating the non-linear part of the pressure term}

In this subsection we will estimate the non-linear part of the pressure term $\partial_s^a\pr^b\pt^{\beta}\mb P$ \eqref{E:EP in self-similar}, and show that it can be bounded by $(\E_n+\Z_n^2)^{1/2}\E_n$. More precisely, when doing energy estimates, the term $\<\partial_s^a\pr^b\pt^\beta\mb P,\partial_s^a\pr^b\pt^\beta\bs\theta\>$ and $\<\partial_s^a\pr^b\pt^\beta\mb P,\partial_s^{a+1}\pr^b\pt^\beta\bs\theta\>$ will arise, we will show that $\mb P$ here can be reduced to $\mb P_{d,L}$ modulo remainder terms that can be estimated. We will use results from section \ref{Pertaining the pressure term}.

Using Lemma \ref{P-inner-product}, we will now estimate the difference between ``$\mb P_{b}$'' and ``$\mb P_{b,L}$''.

\begin{proposition}\label{P-reduction-2}
Let $n\geq 20$ and $a+|\beta|+b\leq n$ with $a>0$. For any $\bs\theta$ that satisfies our a priori assumption \eqref{A priori assumption} we have
\begin{align*}
\abs{\int_0^s\<\mb P_b\partial_s^a\pr^b\pt^\beta\bs\theta,\partial_s^{a+1}\pr^b\pt^\beta\bs\theta\>_{3+b}\d\tau-{1\over 2}\eva{\<\mb P_{b,L}\partial_s^a\pr^b\pt^\beta\bs\theta,\partial_s^a\pr^b\pt^\beta\bs\theta\>_{3+b}}_0^s}&\lesssim (\E_n+\Z_n^2)^{1/2}\E_n\\
\abs{\int_0^s\<\mb P_b\partial_s^a\pr^b\pt^\beta\bs\theta,\partial_s^a\pr^b\pt^\beta\bs\theta\>_{3+b}\d\tau-\int_0^s\<\mb P_{b,L}\partial_s^a\pr^b\pt^\beta\bs\theta,\partial_s^a\pr^b\pt^\beta\bs\theta\>_{3+b}\d\tau}&\lesssim (\E_n+\Z_n^2)^{1/2}\E_n
\end{align*}
\end{proposition}
\begin{proof}
By Lemma~\ref{P-inner-product}
\begin{align*}
&\<\mb P_b\partial_s^a\pr^b\pt^\beta\bs\theta,\partial_s^{a+1}\pr^b\pt^\beta\bs\theta\>_{3+b}\\
&=\int\bigg((\cApar_m\partial_s^a\pr^b\pt^\beta\bs\theta)\cdot(\cApar_m\partial_s^{a+1}\pr^b\pt^\beta\bs\theta)
+{1\over 3}(\div_{\A}\partial_s^a\pr^b\pt^\beta\bs\theta)(\div_{\A}\partial_s^{a+1}\pr^b\pt^\beta\bs\theta)\\
&\qquad\quad-{1\over 2}[\curl_{\A}\partial_s^a\pr^b\pt^\beta\bs\theta]^m_j[\curl_{\A}\partial_s^{a+1}\pr^b\pt^\beta\bs\theta]^m_j\bigg)\J^{-1/3}\bar w^{4+b}\d\mb x\\
&=\int\bigg((\cApar_m\partial_s^a\pr^b\pt^\beta\bs\theta)\cdot\partial_s(\cApar_m\partial_s^a\pr^b\pt^\beta\bs\theta)
+{1\over 3}(\div_{\A}\partial_s^a\pr^b\pt^\beta\bs\theta)\partial_s(\div_{\A}\partial_s^a\pr^b\pt^\beta\bs\theta)\\
&\qquad\quad-{1\over 2}[\curl_{\A}\partial_s^a\pr^b\pt^\beta\bs\theta]^m_j\partial_s[\curl_{\A}\partial_s^a\pr^b\pt^\beta\bs\theta]^m_j\bigg)\J^{-1/3}\bar w^{4+b}\d\mb x
+\Rd[(E_n+Z_n^2)^{1/2}E_n]\\
&={1\over 2}\partial_s\int\bigg(|\cApar_m\partial_s^a\pr^b\pt^\beta\bs\theta|^2+{1\over 3}|\div_{\A}\partial_s^a\pr^b\pt^\beta\bs\theta|^2
-{1\over 2}|[\curl_{\A}\partial_s^a\pr^b\pt^\beta\bs\theta]|^2\bigg)\J^{-1/3}\bar w^{4+b}\d\mb x\\
&\quad+\Rd[(E_n+Z_n^2)^{1/2}E_n]\\
&={1\over 2}\partial_s\int\bigg(|\partial_m\partial_s^a\pr^b\pt^\beta\bs\theta|^2+{1\over 3}|\div\partial_s^a\pr^b\pt^\beta\bs\theta|
^2-{1\over 2} | [\curl\partial_s^a\pr^b\pt^\beta\bs\theta]|^2\bigg)\bar w^{4+b}\d\mb x\\
&\quad+\partial_s\Rd[(E_n+Z_n^2)^{1/2}E_n]+\Rd[(E_n+Z_n^2)^{1/2}E_n]\\
&={1\over 2}\partial_s\<\mb P_{b,L}\partial_s^a\pr^b\pt^\beta\bs\theta,\partial_s^a\pr^b\pt^\beta\bs\theta\>_{3+b}+\partial_s\Rd[(E_n+Z_n^2)^{1/2}E_n]
+\Rd[(E_n+Z_n^2)^{1/2}E_n]
\end{align*}
where we recall notation $\Rd[\star]$ introduced in Definition~\ref{Special notations}. Integrating in time we get the first equation. For the second equation, note that
\begin{align*}
\<\mb P_b\partial_s^a\pr^b\pt^\beta\bs\theta,\partial_s^a\pr^b\pt^\beta\bs\theta\>_{3+b}=\<\mb P_{b,L}\partial_s^a\pr^b\pt^\beta\bs\theta,\partial_s^a\pr^b\pt^\beta\bs\theta\>_3+\Rd[(E_n+Z_n^2)^{1/2}E_n].
\end{align*}
Integrating in time we get the second equation. 
\end{proof}

And now we will estimate the difference between ``$\mb P$'' and ``$\mb P_{b}$''.

\begin{proposition}\label{P-reduction-1}
Let $n\geq 20$. For any $\bs\theta$ that satisfies our a priori assumption \eqref{A priori assumption} we have
\begin{enumerate}
\item For $a+|\beta|\leq n$ with $a>0$ we have
\begin{align*}
\abs{\int_0^s\<\partial_s^a\pt^\beta\mb P-\mb P_0\partial_s^a\pt^\beta\bs\theta,\partial_s^{a+1}\pt^\beta\bs\theta\>_3\d\tau}&\lesssim\S_{n,|\beta|-1,0}^{1/2}\S_{n,|\beta|,0}^{1/2}+(\E_n+\Z_n^2)^{1/2}\E_n\\
\abs{\int_0^s\<\partial_s^a\pt^\beta\mb P-\mb P_0\partial_s^a\pt^\beta\bs\theta,\partial_s^a\pt^\beta\bs\theta\>_3\d\tau}&\lesssim\S_{n,|\beta|-1,0}^{1/2}\S_{n,|\beta|,0}^{1/2}+(\E_n+\Z_n^2)^{1/2}\E_n
\end{align*}
\item For $a+|\beta|+b\leq n$ with $a>0$ we have
\begin{multline*}
\abs{\int_0^s\<\partial_s^a\pr^b\pt^\beta\mb P-\mb P_b\partial_s^a\pr^b\pt^\beta\bs\theta,\partial_s^{a+1}\pr^b\pt^\beta\bs\theta\>_{3+b}\d\tau}\\\lesssim(\S_{n,|\beta|+b-1}^{1/2}+\S_{n,|\beta|+b,b-1}^{1/2})\S_{n,|\beta|+b}^{1/2}+(\E_n+\Z_n^2)^{1/2}\E_n
\end{multline*}
\begin{multline*}
\abs{\int_0^s\<\partial_s^a\pr^b\pt^\beta\mb P-\mb P_b\partial_s^a\pr^b\pt^\beta\bs\theta,\partial_s^a\pr^b\pt^\beta\bs\theta\>_{3+b}\d\tau}\\\lesssim(\S_{n,|\beta|+b-1}^{1/2}+\S_{n,|\beta|+b,b-1}^{1/2})\S_{n,|\beta|+b}^{1/2}+(\E_n+\Z_n^2)^{1/2}\E_n
\end{multline*}
\end{enumerate}
\end{proposition}
\begin{proof}
\begin{enumerate}
\item Using Lemma \ref{pressure-outer-commutator-tangential} we have
\begin{align*}
&\abs{\int_0^s\<\partial_s^a\pt^\beta\mb P-\mb P_L\partial_s^a\pt^\beta\bs\theta,\partial_s^{a+1}\pt^\beta\bs\theta\>_3\d\tau}\\
&\leq\abs{\int_0^s\int_{B_R}\partial_k(\bar w^4(T_{R:a,\beta})^k_i)\partial_s^{a+1}\pt^\beta\theta^i\;\d\mb x\d\tau}\\
&\quad+\abs{\int_0^s\int_{B_R}\sum_{|\beta'|\leq|\beta|-1}\<C\grad(\bar w^4\partial_s^a\pt^{\beta'}T)\>\<\partial_s^{a+1}\pt^\beta\bs\theta\>\;\d\mb x\d\tau}\\
&\leq\abs{\int_0^s\int_{B_R}\bar w^4\partial_s(T_{R:a,\beta})^k_i\partial_s^{a}\partial_k\pt^\beta\theta^i\;\d\mb x\d\tau}
+\abs{\int_0^s\int_{B_R}\sum_{|\beta'|\leq|\beta|-1}\bar w^4\<C\partial_s^{a+1}\pt^{\beta'}T\>\<\partial_s^{a}\grad\pt^\beta\bs\theta\>\;\d\mb x\d\tau}\\
&\quad+(E_n+Z_n^2)(0)^{1/2}E_{n}(0)+(E_n+Z_n^2)(s)^{1/2}E_{n}(s)\\
&\quad+S_{n,|\beta|-1,0}(0)^{1/2}S_{n,|\beta|,0}(0)^{1/2}+S_{n,|\beta|-1,0}(s)^{1/2}S_{n,|\beta|,0}(s)^{1/2}\\
&\lesssim\S_{n,|\beta|-1,0}^{1/2}\S_{n,|\beta|,0}^{1/2}+(\E_n+\Z_n^2)^{1/2}\E_n.
\end{align*}
Proof of the second formula is similar and easier.
\item By Lemma \ref{pressure-outer-commutator} we need to estimate the following.
\begin{align*}
&\abs{\int_0^s\int_{B_R}\brac{\sum_{\substack{b'\leq b\\|\beta'|\leq|\beta|-1}}+\sum_{\substack{b'\leq b-1\\|\beta'|\leq|\beta|+1}}}\<C\omega\partial_s^a\pr^{b'}\pt^{\beta'}T\>\<\partial_s^{a+1}\pr^b\pt^\beta\bs\theta\>\bar w^{3+b}\d\mb x\d\tau}\\
&\lesssim\abs{\int_0^s\int_{B_R}\sum_{\substack{b'\leq b\\|\beta'|\leq|\beta|-1}}\<C\omega T_T[\partial_s^{a+1}\pr^{b'}\pt^{\beta'}\grad\bs\theta]\>\<\partial_s^{a}\pr^b\pt^\beta\bs\theta\>\bar w^{3+b}\d\mb x\d\tau}\\
&\quad+\abs{\int_0^s\int_{B_R}\sum_{\substack{b'\leq b-1\\|\beta'|\leq|\beta|+1}}\<C\omega T_T[\partial_s^a\pr^{b'}\pt^{\beta'}\grad\bs\theta]\>\<\partial_s^{a+1}\pr^b\pt^\beta\bs\theta\>\bar w^{3+b}\d\mb x\d\tau}\\
&\quad+\S_{n,|\beta|+b-1}^{1/2}\S_{n,|\beta|+b}^{1/2}+(\E_n+\Z_n^2)^{1/2}\E_n\\
&\lesssim(\S_{n,|\beta|+b-1}^{1/2}+\S_{n,|\beta|+b,b-1}^{1/2})\S_{n,|\beta|+b}^{1/2}+(\E_n+\Z_n^2)^{1/2}\E_n
\end{align*}
\begin{align*}
&\abs{\int_0^s\int_{B_R}\brac{\sum_{\substack{b'\leq b-1\\|\beta'|\leq|\beta|-1}}+\sum_{\substack{b'\leq b-2\\|\beta'|\leq|\beta|}}}\<C\partial_s^a\pr^{b'}\pt^{\beta'}\grad T\>\<\partial_s^{a+1}\pr^b\pt^\beta\bs\theta\>\bar w^{3+b}\d\mb x\d\tau}\\
&\lesssim\abs{\int_0^s\int_{B_R}\brac{\sum_{\substack{b'\leq b-1\\|\beta'|\leq|\beta|-1}}+\sum_{\substack{b'\leq b-2\\|\beta'|\leq|\beta|}}}\<C\partial_s^{a+1}\pr^{b'}\pt^{\beta'}T\>\<\partial_s^{a}\grad\pr^b\pt^\beta\bs\theta\>\bar w^{3+b}\d\mb x\d\tau}\\
&\quad+\abs{\int_0^s\int_{B_R}\brac{\sum_{\substack{b'\leq b-1\\|\beta'|\leq|\beta|-1}}+\sum_{\substack{b'\leq b-2\\|\beta'|\leq|\beta|}}}\<C\partial_s^{a+1}\pr^{b'}\pt^{\beta'}T\>\<\omega\partial_s^{a}\pr^b\pt^\beta\bs\theta\>\bar w^{2+b}\d\mb x\d\tau}\\
&\quad+\S_{n,|\beta|+b-1}^{1/2}\S_{n,|\beta|+n}^{1/2}+(\E_n+\Z_n^2)^{1/2}\E_n\\
&\lesssim\S_{n,|\beta|+b-1}^{1/2}\S_{n,|\beta|+n}^{1/2}+(\E_n+\Z_n^2)^{1/2}\E_n
\end{align*}
\begin{align*}
&\abs{\int_0^s\int_{B_R}\brac{\sum_{\substack{b'\leq b\\|\beta'|\leq|\beta|-1}}+\sum_{\substack{b'\leq b-1\\|\beta'|\leq|\beta|}}}\<C\bar w\pr^{b'}\pt^{\beta'}\grad T\>\<\partial_s^{a+1}\pr^b\pt^\beta\bs\theta\>\bar w^{3+b}\d\mb x\d\tau}\\
&\lesssim\abs{\int_0^s\int_{B_R}\brac{\sum_{\substack{b'\leq b\\|\beta'|\leq|\beta|-1}}+\sum_{\substack{b'\leq b-1\\|\beta'|\leq|\beta|}}}\<C\partial_s^{a+1}\pr^{b'}\pt^{\beta'}T\>\<\partial_s^{a}\grad\pr^b\pt^\beta\bs\theta\>\bar w^{4+b}\d\mb x\d\tau}\\
&\quad+\abs{\int_0^s\int_{B_R}\brac{\sum_{\substack{b'\leq b\\|\beta'|\leq|\beta|-1}}+\sum_{\substack{b'\leq b-1\\|\beta'|\leq|\beta|}}}\<C\partial_s^{a+1}\pr^{b'}\pt^{\beta'}T\>\<\omega\partial_s^{a}\pr^b\pt^\beta\bs\theta\>\bar w^{3+b}\d\mb x\d\tau}\\
&\quad+\S_{n,|\beta|+b-1}^{1/2}\S_{n,|\beta|+n}^{1/2}+(\E_n+\Z_n^2)^{1/2}\E_n\\
&\lesssim\S_{n,|\beta|+b-1}^{1/2}\S_{n,|\beta|+n}^{1/2}+(\E_n+\Z_n^2)^{1/2}\E_n
\end{align*}
This proves the first formula. Proof of the second formula is similar and easier. 
\end{enumerate}
\end{proof}

\subsubsection{Estimating the linear and non-linear part of the gravity term}
\label{Estimating the linear and non-linear part of the gravity term}

In this subsection we will estimate the gravity term $\partial_s^a\pr^b\pt^{\beta}\mb G$ \eqref{E:EP in self-similar} and show that it can be bounded by $E_n$. We will also estimate the non-linear part of $\partial_s^a\pt^{\beta}\mb G$, and show that it can be bounded by $(\E_n+\Z_n^2)^{1/2}\E_n$. We will use results from Section \ref{Pertaining the gravity term}.

Since the gravity term is a non-local term, we need to estimate convolution-like operator. However, rather than the convolution kernel $|\mb x-\mb z|^{-1}$ we actually need to estimate $|\bs\xi(\mb x)-\bs\xi(\mb z)|^{-1}$. Lemma \ref{distance estimate} and the following lemma tell us how to reduce the latter to the former, which will allows us to estimate using the Young's convolution inequality.

\begin{lemma}\label{lemma for kernel}
Let $\bs\xi$ and $\bs\theta$ be as in \eqref{E:THETADEF}, and $\bs\theta$ satisfies our a priori assumption \eqref{A priori assumption}. Let $n\geq 21$ and $a+|\beta|\leq n$ with $a>0$.
\begin{enumerate}
\item When $a+|\beta|>n/2$ we have
\begin{align*}
\abs{\partial_s^{a}(\pt_{\mb x}+\pt_{\mb z})^{\beta}\brac{1\over|\bs\xi(\mb x)-\bs\xi(\mb z)|}}
&\lesssim{1\over|\mb x-\mb z|^2}\sum_{\substack{n/2<a'+|\gamma|\leq n\\a'>0}} |\partial_s^{a'}\pt_{\mb x}^\gamma\bs\theta(\mb x)-\partial_s^{a'}\pt_{\mb z}^\gamma\bs\theta(\mb z)|\\
&\quad+{E_n^{1/2}\over|\mb x-\mb z|^2}\sum_{n/2<|\gamma|\leq n}|\pt_{\mb x}^\gamma\bs\xi(\mb x)-\pt_{\mb z}^\gamma\bs\xi(\mb z)|
\end{align*}
\item When $|\beta|>n/2$ we have
\begin{align*}
\abs{(\pt_{\mb x}+\pt_{\mb z})^{\beta}\brac{1\over|\bs\xi(\mb x)-\bs\xi(\mb z)|}}&\lesssim{1\over|\mb x-\mb z|^2}\sum_{n/2<|\gamma|\leq n}|\pt_{\mb x}^\gamma\bs\xi(\mb x)-\pt_{\mb z}^\gamma\bs\xi(\mb z)|
\end{align*}
\item When $a+|\beta|\leq n/2$ we have
\begin{align*}
\abs{\partial_{i,\mb z}\partial_s^{a}(\pt_{\mb x}+\pt_{\mb z})^{\beta}\brac{1\over|\bs\xi(\mb x)-\bs\xi(\mb z)|}}&\lesssim{E_n^{1/2}\over|\mb x-\mb z|^2}
\end{align*}
\item When $|\beta|\leq n/2$ we have
\begin{align*}
\abs{\partial_{i,\mb z}(\pt_{\mb x}+\pt_{\mb z})^{\beta}\brac{1\over|\bs\xi(\mb x)-\bs\xi(\mb z)|}}&\lesssim{1\over|\mb x-\mb z|^2}
\end{align*}
\end{enumerate}
\end{lemma}
\begin{proof}
These follows from Lemma \ref{distance estimate}, the embedding theorems \ref{Near boundary embedding theorem} and \ref{Near origin embedding theorem}, the a priori bounds $E_n,Z_n\lesssim 1$ \eqref{A priori assumption}, and the following.
\begin{align*}
\partial_s^{a}(\pt_{\mb x}+\pt_{\mb z})^{\beta}\brac{1\over|\bs\xi(\mb x)-\bs\xi(\mb z)|}
&=\sum_{m=1}^{a+|\beta|}\sum_{\substack{\sum_{i=1}^m(a_i+a_i')=a\\\sum_{i=1}^m(\beta_i+\beta_i')=\beta\\|a_i|+|\beta_i|>0}}{(-1)^m(2m)!\over m!2^m }{1\over|\bs\xi(\mb x)-\bs\xi(\mb z)|^{1+2m}}\\
&\qquad\qquad\prod_{i=1}^m(\partial_s^{a_i}\pt_{\mb x}^{\beta_i}\bs\xi(\mb x)-\partial_s^{a_i}\pt_{\mb z}^{\beta_i}\bs\xi(\mb z))\cdot(\partial_s^{a_i'}\pt_{\mb x}^{\beta_i'}\bs\xi(\mb x)-\partial_s^{a_i'}\pt_{\mb z}^{\beta_i'}\bs\xi(\mb z)). 
\end{align*}
\end{proof}

We will now prove the main results of this subsection.

\begin{proposition}\label{G estimate}
Let $n\geq 21$ and suppose $\bs\theta$ satisfies our a priori assumption \eqref{A priori assumption}. For $a+|\beta|+b\leq n$ with $a>0$ we have
\begin{align}
\|\partial_s^a\pr^b\pt^\beta\mb G\|_{3+b}^2&\lesssim E_n.
\end{align}
\end{proposition}
\begin{proof}
By definition 
\begin{align*}
\mb G&=\K_{\bs\xi}\grad\cdot(\A\bar w^3)-\mathcal{K}\grad\bar w^3
=-\int_{\R^3}{\partial_k(\A^k\bar w^3)\over|\bs\xi(\mb x)-\bs\xi(\mb z)|}\d\mb z+\int_{\R^3}{\grad\bar w^3\over|\mb x-\mb z|}\d\mb z
\end{align*}
Consider first when $b=0$. Since $a>0$, by Lemma~\ref{L:ENERGYLEMMA1} we have
\begin{align*}
\partial_s^a\pt^\beta\mb G(\mb x)
&=\partial_s^a\pt^\beta\K_{\bs\xi}\grad\cdot(\A\bar w^3)(\mb x)
=-\partial_s^a\pt^\beta\int_{\R^3}{\partial_k(\A^k\bar w^3)\over|\bs\xi(\mb x)-\bs\xi(\mb z)|}\d\mb z\\
&=-\int_{\R^3}\sum_{\substack{a_1+a_2=a\\\beta_1+\beta_2=\beta}}\partial_s^{a_1}(\pt_{\mb x}+\pt_{\mb z})^{\beta_1}\brac{1\over|\bs\xi(\mb x)-\bs\xi(\mb z)|}\partial_s^{a_2}\pt_{\mb z}^{\beta_2}\partial_k(\A^k\bar w^3)(\mb z)\d\mb z\\
&=-\int_{\R^3}\sum_{\substack{a_1+a_2=a\\\beta_1+\beta_2=\beta\\a_1+|\beta_1|>n/2}}\partial_s^{a_1}(\pt_{\mb x}+\pt_{\mb z})^{\beta_1}\brac{1\over|\bs\xi(\mb x)-\bs\xi(\mb z)|}\partial_s^{a_2}\pt_{\mb z}^{\beta_2}\partial_k(\A^k\bar w^3)(\mb z)\d\mb z\\
&\quad-\int_{\R^3}\sum_{\substack{a_1+a_2=a\\\beta_1+\beta_2=\beta\\a_1+|\beta_1|\leq n/2}}\partial_s^{a_1}(\pt_{\mb x}+\pt_{\mb z})^{\beta_1}\brac{1\over|\bs\xi(\mb x)-\bs\xi(\mb z)|}\partial_s^{a_2}\pt_{\mb z}^{\beta_2}\partial_k(\A^k\bar w^3)(\mb z)\d\mb z\\
&=-\int_{\R^3}\sum_{\substack{a_1+a_2=a\\\beta_1+\beta_2=\beta\\a_1+|\beta_1|>n/2}}\partial_s^{a_1}(\pt_{\mb x}+\pt_{\mb z})^{\beta_1}\brac{1\over|\bs\xi(\mb x)-\bs\xi(\mb z)|}\partial_s^{a_2}\pt_{\mb z}^{\beta_2}\partial_k(\A^k\bar w^3)(\mb z)\d\mb z\\
&\quad+\int_{\R^3}\sum_{\substack{a_1+a_2=a\\\beta_1+\beta_2\leq\beta\\a_1+|\beta_1|\leq n/2}}\<\grad_{\mb z}\>\partial_s^{a_1}(\pt_{\mb x}+\pt_{\mb z})^{\beta_1}\brac{1\over|\bs\xi(\mb x)-\bs\xi(\mb z)|}(\<\partial_s^{a_2}\pt_{\mb z}^{\beta_2}\A\>\bar w^3)(\mb z)\d\mb z
\end{align*}
Now using Lemma \ref{lemma for kernel}, we get
\begin{align*}
|\partial_s^a\pt^\beta\mb G(\mb x)|
&\lesssim\int_{\R^3}{1\over|\mb x-\mb z|^2}\sum_{\substack{n/2<a'+|\gamma|\leq n\\a'>0}}|\partial_s^{a'}\pt_{\mb x}^\gamma\bs\theta(\mb x)-\partial_s^{a'}\pt_{\mb z}^\gamma\bs\theta(\mb z)|\d\mb z\\
&\quad+\int_{\R^3}{E_n^{1/2}\over|\mb x-\mb z|^2}\sum_{n/2<|\gamma|\leq n}|\pt_{\mb x}^\gamma\bs\xi(\mb x)-\pt_{\mb z}^\gamma\bs\xi(\mb z)|\d\mb z\\
&\quad+\int_{\R^3}\sum_{\substack{0<a_2\leq a\\\beta_2\leq\beta}}{1\over|\mb x-\mb z|^2}(\<\partial_s^{a_2}\pt_{\mb z}^{\beta_2}\A\>\bar w^3)(\mb z)\d\mb z
+\int_{\R^3}\sum_{\beta_2\leq\beta}{E_n^{1/2}\over|\mb x-\mb z|^2}(\<\pt_{\mb z}^{\beta_2}\A\>\bar w^3)(\mb z)\d\mb z
\end{align*}
Now using Young's convolution inequality we get
\begin{align*}
\|\partial_s^a\pt^\beta\mb G(\mb x)\|_{L^2(\R^3)}\lesssim E_n^{1/2}.
\end{align*}
Hence $\|\partial_s^a\pt^\beta\mb G\|_{3}^2\lesssim E_n$. From the above proof, with small modification, we can further see that
\begin{alignat*}{4}
\|\partial_s^a\pt^\beta\mb G(\mb x)\|_{L^\infty(\R^3)}&\lesssim E_n^{1/2}\qquad &\text{when}&&\qquad a+|\beta|&\leq n/2\\
\|\pt^\beta\mb G(\mb x)\|_{L^\infty(\R^3)}&\lesssim 1\qquad &\text{when}&&\qquad |\beta|&\leq n/2.
\end{alignat*}

Now we deal with the case $b>0$. Let
\begin{align*}
W_{n,c}&=\sum_{\substack{a+|\beta|+b\leq n\\a>0\\|\beta|+b\leq c}}\|\partial_s^a\pr^b\pt^\beta\mb G\|_{3+b}^2\\
W_{n,c,d}&=\sum_{\substack{a+|\beta|+b\leq n\\a>0\\|\beta|+b\leq c\\b\leq d}}\|\partial_s^a\pr^b\pt^\beta\mb G\|_{3+b}^2\\
V_{n,c}&=\sum_{\substack{a+|\beta|+b\leq n\\a>0\\|\beta|+b\leq c}}\sup_{\R^3}\brac{\bar w^b|\partial_s^a\pr^b\pt^\beta\mb G|^2}\\
V_{n,c,d}&=\sum_{\substack{a+|\beta|+b\leq n\\a>0\\|\beta|+b\leq c\\b\leq d}}\sup_{\R^3}\brac{\bar w^b|\partial_s^a\pr^b\pt^\beta\mb G|^2}.
\end{align*}
For $a+|\beta|+b\leq n/2$, using the above lemmas \ref{Breakdown for X_r} and \ref{div and curl of G} we have
\begin{align*}
\bar w^b|\partial_s^a\pr^b\pt^\beta\mb G|^2
&\lesssim\bar w^b|r\partial_s^a\grad\cdot\pr^{b-1}\pt^\beta\mb G|^2+\bar w^b|r\partial_s^a\grad\times\pr^{b-1}\pt^\beta\mb G|^2+\sum_{k=1}^3\bar w^b|\partial_s^a\pr^{b-1}\pt_k\pt^\beta\mb G|^2\\
&\lesssim\bar w^b|r\partial_s^a\pr^{b-1}\pt^\beta\grad\cdot\mb G|^2+\bar w^b|r\partial_s^a\pr^{b-1}\pt^\beta\grad\times\mb G|^2+V_{n,b+|\beta|-1}+V_{n,b+|\beta|,b-1}\\
&\lesssim\bar w^b|r\partial_s^a\pr^{b-1}\pt^\beta(I-\A)\grad\cdot\mb G|^2+\bar w^b|r\partial_s^a\pr^{b-1}\pt^\beta(I-\A)\grad\times\mb G|^2\\
&\quad+E_n+V_{n,b+|\beta|-1}+V_{n,b+|\beta|,b-1}\\
&\lesssim\bar w^b|r(I-\A)\partial_s^a\pr^{b-1}\pt^\beta\grad\cdot\mb G|^2+\bar w^b|r(I-\A)\partial_s^a\pr^{b-1}\pt^\beta\grad\times\mb G|^2\\
&\quad+E_n+V_{n,b+|\beta|-1}+V_{n,b+|\beta|,b-1}\\
&\lesssim\bar w^b(E_n+Z_n^2)|r\partial_s^a\pr^{b-1}\pt^\beta\grad\mb G|^2+E_n+V_{n,b+|\beta|-1}+V_{n,b+|\beta|,b-1}
\end{align*}
So
\begin{align*}
V_{n,b+|\beta|,b}\lesssim(E_n+Z_n^2)V_{n,b+|\beta|,b}+E_n+V_{n,b+|\beta|-1}+V_{n,b+|\beta|,b-1}
\end{align*}
By a priori assumption \eqref{A priori assumption}, we have $E_n+Z_n^2\ll 1$, so
\begin{align*}
V_{n,b+|\beta|,b}\lesssim E_n+V_{n,b+|\beta|-1}+V_{n,b+|\beta|,b-1}.
\end{align*}
We know $V_{n',c,0}\lesssim E_n$ for all $c\leq n'\leq n/2$, so by induction we get $V_{n',c,d}\lesssim E_n$ for all $d\leq c\leq n'\leq n/2$.

Now for $a+|\beta|+b\leq n$, using the above lemmas \ref{Breakdown for X_r} and \ref{div and curl of G} and results for $V$ we have
\begin{align*}
\|\partial_s^a\pr^b\pt^\beta\mb G\|_{3+b}^2
&\lesssim\|r\partial_s^a\grad\cdot\pr^{b-1}\pt^\beta\mb G\|_{3+b}^2+\|r\partial_s^a\grad\times\pr^{b-1}\pt^\beta\mb G\|_{3+b}^2+\sum_{k=1}^3\|\partial_s^a\pr^{b-1}\pt_k\pt^\beta\mb G\|_{3+b}^2\\
&\lesssim\|r\partial_s^a\pr^{b-1}\pt^\beta\grad\cdot\mb G\|_{3+b}^2+\|r\partial_s^a\pr^{b-1}\pt^\beta\grad\times\mb G\|_{3+b}^2+W_{n,b+|\beta|-1}+W_{n,b+|\beta|,b-1}\\
&\lesssim\|r\partial_s^a\pr^{b-1}\pt^\beta(1-\A)\grad\cdot\mb G\|_{3+b}^2+\|r\partial_s^a\pr^{b-1}\pt^\beta(1-\A)\grad\times\mb G\|_{3+b}^2\\
&\quad +E_n+W_{n,b+|\beta|-1}+W_{n,b+|\beta|,b-1}\\
&\lesssim\|r(1-\A)\partial_s^a\pr^{b-1}\pt^\beta\grad\cdot\mb G\|_{3+b}^2+\|r(1-\A)\partial_s^a\pr^{b-1}\pt^\beta\grad\times\mb G\|_{3+b}^2\\
&\quad +E_n+W_{n,b+|\beta|-1}+W_{n,b+|\beta|,b-1}\\
&\lesssim(E_n+Z_n^2)\|r\partial_s^a\pr^{b-1}\pt^\beta\grad\mb G\|_{3+b}^2+E_n+W_{n,b+|\beta|-1}+W_{n,b+|\beta|,b-1}
\end{align*}
So
\begin{align*}
W_{n,b+|\beta|,b}\lesssim(E_n+Z_n^2)W_{n,b+|\beta|,b}+E_n+W_{n,b+|\beta|-1}+W_{n,b+|\beta|,b-1}
\end{align*}
By a priori assumption \eqref{A priori assumption}, we have $E_n+Z_n^2\ll 1$, so
\begin{align*}
W_{n,b+|\beta|,b}\lesssim E_n+W_{n,b+|\beta|-1}+W_{n,b+|\beta|,b-1}.
\end{align*}
We know $W_{n,c,0}\lesssim E_n$ for all $c$, so by induction we get $W_{n,c,d}\lesssim E_n$ for all $d\leq c\leq n$. 
\end{proof}

We are now in the position to estimate the difference between the high order derivatives of nonlinear gravity term $\mb G$~\eqref{E:G} and its linearised part $\mb G_L$~\eqref{E:linearised G}.

\begin{proposition}\label{G non-linear bound}
Let $n\geq 21$ and suppose $\bs\theta$ satisfies our a priori assumption \eqref{A priori assumption}. For $a+|\beta|\leq n$ with $a>0$ we have
\begin{align*}
\abs{\int_0^s\<\partial_s^a\pt^\beta\mb G-\mb G_L\partial_s^a\pt^\beta\bs\theta,\partial_s^{a+1}\pt^\beta\bs\theta\>_3\d\tau}&\lesssim(\E_n+\Z_n^2)^{1/2}\E_n\\
\abs{\int_0^s\<\partial_s^a\pt^\beta\mb G-\mb G_L\partial_s^a\pt^\beta\bs\theta,\partial_s^a\pt^\beta\bs\theta\>_3\d\tau}&\lesssim(\E_n+\Z_n^2)^{1/2}\E_n
\end{align*}
\end{proposition}


\begin{proof}
Since $\|\partial_s^{a+1}\pt^\beta\bs\theta\|_3+\|\partial_s^{a}\pt^\beta\bs\theta\|_3\lesssim E_n^{1/2}$, it suffice to prove that
\begin{align*}
\|\partial_s^a\pt^\beta\mb G-\mb G_L\partial_s^a\pt^\beta\bs\theta\|_3\lesssim(E_n+Z_n^2)^{1/2}E_n^{1/2}.
\end{align*}
Recall from Lemma~\ref{L:GRAVITYONE} that
\begin{align*}
\mb G-\mb G_{L}\bs\theta&=\underbrace{\K_{\bs\xi}(\A_l^i(\partial_k\theta^l)(\grad\theta^k)\partial_i\bar w^3-\bar w^3(\A^i_m\A^l_\bullet-I^i_mI^l_\bullet)\partial_i\partial_l\theta^m)}_{:=M_1}\\
&\quad\underbrace{-(\K_{\bs\xi}-\K)\partial_i(\bar w^3\grad\theta^i)}_{:=M_2}+\underbrace{(\K_{\bs\xi}-\mathcal{K}-\K_{\bs\xi}^{(1)})\grad\bar w^3}_{:=M_3}.
\end{align*}
Now $\|\partial_s^a\pt^\beta M_1\|_3$ can be estimated in a similar way as the previous Proposition \ref{G estimate}, and $\|\partial_s^a\pt^\beta M_3\|_3$ can be estimated in the same way as in Lemma \ref{energy gravity lemma} in equation \eqref{estimate second order gravity}. Now in the same way as in Lemma~\ref{K_2 lemma} and recalling $K_1$~\eqref{E:K1DEF} we can show that
\begin{align*}
|\partial_s^a(\pt_{\mb x}+\pt_{\mb z})^\beta K_1(\mb x,\mb z)|&\lesssim{1\over|\mb x-\mb z|^2}\sum_{\substack{0<a'\leq a\\\beta'\leq\beta}}|\partial_s^{a'}\pt^{\beta'}\bs\theta(\mb x)-\partial_s^{a'}\pt^{\beta'}\bs\theta(\mb z)|
+{E_n^{1/2}\over|\mb x-\mb z|^2}\sum_{\beta'\leq\beta}|\pt^{\beta'}\bs\theta(\mb x)-\pt^{\beta'}\bs\theta(\mb z)|\\
|(\pt_{\mb x}+\pt_{\mb z})^\beta K_1(\mb x,\mb z)|&\lesssim{1\over|\mb x-\mb z|^2}\sum_{\beta'\leq\beta}|\pt^{\beta'}\bs\theta(\mb x)-\pt^{\beta'}\bs\theta(\mb z)|
\end{align*}
And when $a+|\beta|\leq n/2$,
\begin{align*}
|\partial_{i,\mb z}\partial_s^a(\pt_{\mb x}+\pt_{\mb z})^\beta K_1(\mb x,\mb z)|&\lesssim{E_n^{1/2}\over|\mb x-\mb z|^2}
\end{align*}
and when $|\beta|\leq n/2$,
\begin{align*}
|\partial_{i,\mb z}(\pt_{\mb x}+\pt_{\mb z})^\beta K_1(\mb x,\mb z)|&\lesssim{(E_n+Z_n^2)^{1/2}\over|\mb x-\mb z|^2}.
\end{align*}
Using these bounds (in the same way we use Lemma \ref{lemma for kernel} in the proof of the previous Proposition \ref{G estimate}), we can estimate $\|\partial_s^a\pt^\beta M_2\|_3$. 
\end{proof}



\subsubsection{Reduction to linear problem}\label{Reduction to linear problem}


Having estimated the non-linear parts of the equation in the last two subsections, in this section we will use them to reduce our problem to the linear problem for which we have the coercivity result that we can apply. We only need to do this for the case with no radial derivatives, the case with radial derivatives can be obtained by induction.

\begin{lemma}\label{L breakdown}
For any $\bs\theta$ that satisfies our a priori assumption \eqref{A priori assumption} we have
\begin{align*}
\int_0^s\<\mb G_L\partial_s^a\pt^\beta\bs\theta,\partial_s^{a+1}\pt^\beta\bs\theta\>_3\d\tau&={1\over 2}\eva{\<\mb G_L\partial_s^a\pt^\beta\bs\theta,\partial_s^a\pt^\beta\bs\theta\>_3}_0^s\\
\<\mb L\partial_s^a\pt^\beta\bs\theta,\partial_s^a\pt^\beta\bs\theta\>_3&=\delta\|\partial_s^a\pt^\beta\bs\theta\|_3^2+\<\mb P_{0,L}\partial_s^a\pt^\beta\bs\theta,\partial_s^a\pt^\beta\bs\theta\>_3
+\<\mb G_L\partial_s^a\pt^\beta\bs\theta,\partial_s^a\pt^\beta\bs\theta\>_3,
\end{align*}
where we recall~\eqref{E:LINDEF},~\eqref{E:linearised G} and~\eqref{E:PDLDEF}.
\end{lemma}
\begin{proof}
We have from~\eqref{E:linearised G}
\begin{align*}
\<\mb G_L\partial_s^a\pt^\beta\bs\theta,\partial_s^{a+1}\pt^\beta\bs\theta\>_3
&=\int\bigg((\partial_s^a\pt^\beta\theta^i)(\partial_s^{a+1}\pt^\beta\theta^j)\bar w^3\partial_i\partial_j\K\bar w^3\\
&\qquad\qquad-(4\pi)^{-1}(\grad\K\grad\cdot(\bar w^3\partial_s^a\pt^\beta\bs\theta))\cdot(\grad\K\grad\cdot(\bar w^3\partial_s^{a+1}\pt^\beta\bs\theta))\bigg)\d\mb x\\
&={1\over 2}\partial_s\int\brac{(\partial_s^a\pt^\beta\theta^i)(\partial_s^{a}\theta^j)\bar w^3\partial_i\partial_j\K\bar w^3-(4\pi)^{-1}|\grad\K\grad\cdot(\bar w^3\partial_s^a\pt^\beta\bs\theta)|^2}\d\mb x\\
&={1\over 2}\partial_s\<\mb G_L\partial_s^a\pt^\beta\bs\theta,\partial_s^a\pt^\beta\bs\theta\>_3.
\end{align*}
The second formula follows from the definition of $\mb L$, $\mb P_{0,L}$ and $\mb G_L$. 
\end{proof}

The following theorem reduces the full non-linear problem to the linear one.

\begin{theorem}\label{estimate-theorem}
Let $n\geq 20$ and suppose $\bs\theta$ satisfies our a priori assumption \eqref{A priori assumption}. For $a+|\beta|\leq n$ with $a>0$ we have
\begin{multline}
\abs{\int_0^s\<\partial_s^a\pt^\beta(\delta\bs\theta+\mb P+\mb G),\partial_s^{a+1}\pt^\beta\bs\theta\>_3\d\tau-{1\over 2}\eva{\<\mb L\partial_s^a\pt^\beta\bs\theta,\partial_s^a\pt^\beta\bs\theta\>_3}_0^s}\\\lesssim\S_{n,|\beta|-1,0}^{1/2}\S_{n,|\beta|,0}^{1/2}+(\E_n+\Z_n^2)^{1/2}\E_n
\end{multline}
\begin{multline}
\abs{\int_0^s\<\partial_s^a\pt^\beta(\delta\bs\theta+\mb P+\mb G),\partial_s^{a}\bs\theta\>_3\d\tau-\int_0^s\<\mb L\partial_s^a\pt^\beta\bs\theta,\partial_s^a\pt^\beta\bs\theta\>_3\d\tau}\\
\lesssim\S_{n,|\beta|-1,0}^{1/2}\S_{n,|\beta|,0}^{1/2}+(\E_n+\Z_n^2)^{1/2}\E_n
\end{multline}
\end{theorem}
\begin{proof}
Using Lemma~\ref{L breakdown} and Propositions \ref{P-reduction-2}, \ref{P-reduction-1}, and~\ref{G non-linear bound}, we conclude the proof. 
\end{proof}

This theorem above reduces the non-linear problem for time and tangential derivatives to the linear problem. Now applying our linear coercivity results from before, we get the following coercivity result for our non-linear problem, allowing us to control $\|\partial_s^{a+1}\pt^\beta\bs\theta\|_{3+b}^2+\|\partial_s^a\pt^\beta\bs\theta\|_{3+b}^2+\|\partial_s^a\grad\pt^\beta\bs\theta\|_{4+b}^2$.

\begin{corollary}\label{estimate-theorem-cor-1}
Let $n\geq 20$. Let $\bs\theta$ be a solution of~\eqref{E:EP in self-similar} in the sense of Theorem \ref{T:LOCAL}, given on its maximal interval of existence. 
Assume further that the energy, momentum, and irrotationality constraints~\eqref{initial momentum condition},~\eqref{initial energy condition}, and~\eqref{initial irrotational condition} hold respectively. Then for $a+|\beta|\leq n$ with $a>0$ we have
\begin{align}
&\norm{\partial_s^{a+1}\pt^\beta\bs\theta}_{3}^2+\norm{\partial_s^a\pt^\beta\bs\theta}_{3}^2+\norm{\partial_s^a\grad\pt^\beta\bs\theta}_{4}^2\nonumber\\
&\lesssim|\b|^{-2}\brac{CS_{n,|\beta|,0}(0)+{1\over 2}\eva{\norm{\partial_s^{a+1}\pt^\beta\bs\theta}_{3}^2}_0^s+\int_0^s\<\partial_s^a\pt^\beta(\delta\bs\theta+\mb P+\mb G),\partial_s^{a+1}\pt^\beta\bs\theta\>_3\d\tau}\nonumber\\
&\quad+C\brac{\S_{n,|\beta|-1,0}+|\b|^{-2}\S_{n,|\beta|-1,0}^{1/2}\S_{n,|\beta|,0}^{1/2}}+C_\delta(\E_n+\Z_n^2)^{1/2}\E_n
\end{align}
\begin{align}
&\int_0^s\brac{\norm{\partial_s^{a+1}\pt^\beta\bs\theta}_{3}^2+\norm{\partial_s^a\pt^\beta\bs\theta}_{3}^2+\norm{\partial_s^a\grad\pt^\beta\bs\theta}_{4}^2}\d\tau\nonumber\\
&\lesssim|\b|^{-2}\int_0^s\brac{\norm{\partial_s^{a+1}\pt^\beta\bs\theta}_{3}^2+\<\partial_s^a\pt^\beta(\delta\bs\theta+\mb P+\mb G),\partial_s^{a}\pt^\beta\bs\theta\>_3}\d\tau\nonumber\\
&\quad+C\brac{\S_{n,|\beta|-1,0}+|\b|^{-2}\S_{n,|\beta|-1,0}^{1/2}\S_{n,|\beta|,0}^{1/2}}+C_\delta(\E_n+\Z_n^2)^{1/2}\E_n
\end{align}
\end{corollary}
\begin{proof}
Combining Theorem~\ref{estimate-theorem} and Propositions~\ref{L-estimate} and~\ref{L-estimate-tan} we conclude the proof. 
\end{proof}

To control the version with radial derivative $\|\partial_s^{a+1}\pr^b\pt^\beta\bs\theta\|_{3+b}^2+\|\partial_s^a\pr^b\pt^\beta\bs\theta\|_{3+b}^2+\|\partial_s^a\grad\pr^b\pt^\beta\bs\theta\|_{4+b}^2$, we do not need to apply the linear coercivity result like Theorem \ref{estimate-theorem} above. This is because we get control of $\|\partial_s^a\grad\pr^b\pt^\beta\bs\theta\|_{4+b}^2$ directly from the pressure term, while the control of $\|\partial_s^{a+1}\pr^b\pt^\beta\bs\theta\|_{3+b}^2+\|\partial_s^a\pr^b\pt^\beta\bs\theta\|_{3+b}^2$ and the gravity term we get automatically from induction from the step with one less space derivative, as follows.

\begin{corollary}\label{estimate-theorem-cor-2}
Let $n\geq 21$ and suppose $\bs\theta$ satisfies our a priori assumption \eqref{A priori assumption}. For $a+|\beta|+b\leq n$ with $a,b>0$ we have
\begin{align}
&\norm{\partial_s^{a+1}\pr^b\pt^\beta\bs\theta}_{3+b}^2+\norm{\partial_s^a\pr^b\pt^\beta\bs\theta}_{3+b}^2+\norm{\partial_s^a\grad\pr^b\pt^\beta\bs\theta}_{4+b}^2\nonumber\\
&\lesssim CS_{n,|\beta|+b,b}(0)+{1\over 2}\norm{\partial_s^{a+1}\pr^b\pt^\beta\bs\theta}_{3+b}^2
+\int_0^s\<\partial_s^a\pr^b\pt^\beta(\delta\bs\theta+\mb P+\mb G),\partial_s^{a+1}\pr^b\pt^\beta\bs\theta\>_{3+b}\d\tau\nonumber\\
&\quad+C\brac{\S_{n,|\beta|+b-1}+(\S_{n,|\beta|+b-1}^{1/2}+\S_{n,|\beta|+b,b-1}^{1/2})\E_n^{1/2}+(\E_n+\Z_n^2)^{1/2}\E_n}
\end{align}
\begin{align}
&\int_0^s\brac{\norm{\partial_s^{a+1}\pr^b\pt^\beta\bs\theta}_{3+b}^2+\norm{\partial_s^a\pr^b\pt^\beta\bs\theta}_{3+b}^2+\norm{\partial_s^a\grad\pr^b\pt^\beta\bs\theta}_{4+b}^2}\d\tau\nonumber\\
&\lesssim\int_0^s\brac{\norm{\partial_s^{a+1}\pr^b\pt^\beta\bs\theta}_{3+b}^2+\<\partial_s^a\pr^b\pt^\beta(\delta\bs\theta+\mb P+\mb G),\partial_s^{a}\pr^b\pt^\beta\bs\theta\>_{3+b}}\d\tau\nonumber\\
&\quad+C\brac{\S_{n,|\beta|+b-1}+(\S_{n,|\beta|+b-1}^{1/2}+\S_{n,|\beta|+b,b-1}^{1/2})\E_n^{1/2}+(\E_n+\Z_n^2)^{1/2}\E_n}
\end{align}
\end{corollary}
\begin{proof}
By Propositions~\ref{P-reduction-1} and~\ref{P-reduction-2}, we can replace $\partial_s^a\pr^b\pt^\beta\mb P$ by $\mb P_{b,L}\partial_s^a\pr^b\pt^\beta\bs\theta$. Now by Lemma \ref{P-inner-product} we have
\begin{align*}
&{1\over 2}\sbrac{\|\partial_s^a\grad\pr^b\pt^\beta\bs\theta\|_{4+b}^2+{1\over 3}\|\partial_s^a\grad\cdot(\pr^b\pt^\beta\bs\theta)\|_{4+b}^2-{1\over 2}\|[\partial_s^a\curl\pr^b\pt^\beta\bs\theta]\|_{4+b}^2}^s_0\\
&=\int_0^s\<\mb P_{b,L}\partial_s^a\pr^b\pt^\beta\bs\theta,\partial_s^{a+1}\pr^b\pt^\beta\bs\theta\>_{3+b}\d\tau
\end{align*}
Now using Corollary \ref{curl-highorder} we get
\begin{align*}
\|\partial_s^a\grad\pr^b\pt^\beta\bs\theta\|_{4+b}^2&\lesssim CS_{n,|\beta|+b,b}(0)+\int_0^s\<\mb P_{b,L}\partial_s^a\pr^b\pt^\beta\bs\theta,\partial_s^{a+1}\pr^b\pt^\beta\bs\theta\>_{3+b}\d\tau\\
&\quad+C\brac{S_{n,|\beta|+b-1}+(E_n+Z_n^2)E_n}.
\end{align*}
Furthermore, note that
\begin{align*}
\norm{\partial_s^{a+1}\pr^b\pt^\beta\bs\theta}_{3+b}^2+\norm{\partial_s^a\pr^b\pt^\beta\bs\theta}_{3+b}^2
&\lesssim\norm{\partial_s^{a+1}\grad\pr^{b-1}\pt^\beta\bs\theta}_{4+(b-1)}^2+\norm{\partial_s^a\grad\pr^{b-1}\pt^\beta\bs\theta}_{4+(b-1)}^2\\&\lesssim S_{n,|\beta|+b-1}.
\end{align*}
Now note that, using Proposition \ref{G estimate},
\begin{align*}
\abs{\int_0^s\<\delta\partial_s^a\pr^b\pt^\beta\bs\theta,\partial_s^{a+1}\pr^b\pt^\beta\bs\theta\>_{3+b}\d\tau}&\lesssim\S_{n,|\beta|+b-1}\\
\abs{\int_0^s\<\partial_s^a\pr^b\pt^\beta\mb G,\partial_s^{a+1}\pr^b\pt^\beta\bs\theta\>_{3+b}\d\tau}&\lesssim\S_{n,|\beta|+b-1}^{1/2}\E_n^{1/2}
\end{align*}
then we are done for the first formula. Proof for the second formula is similar. 
\end{proof}


\subsection{Energy estimates and proof of the main theorem}\label{S:EE2}


In this section we finally commute the momentum equation~\eqref{E:EP in self-similar} and then derive the 
high-order energy estimates. Since the bounds near the vacuum boundary are more delicate as they are sensitive 
to the weights, we present them in Section~\ref{Near boundary energy estimate section} and the estimates away from the
vacuum boundary in Section~\ref{Near origin energy estimate section}. Then finally we will prove our main theorem in section \ref{Bootstrapping scheme and final theorem} using the energy estimates.
 

\subsubsection{Near boundary energy estimate}\label{Near boundary energy estimate section}


In this subsection we will prove the energy estimate for $\S_n$ (recall~\eqref{E:SBULLET}).


\begin{theorem}[Near boundary energy estimate]\label{Near boundary energy estimate}
Let $n\geq 21$,  and assume that $\epsilon>0$ and $|\delta|$ are sufficiently small. Let $\bs\theta$ be a solution of~\eqref{E:EP in self-similar} in the sense of Theorem \ref{T:LOCAL}, given on its maximal interval of existence. 
Assume further that the energy, momentum, and irrotationality constraints~\eqref{initial momentum condition},~\eqref{initial energy condition}, and~\eqref{initial irrotational condition} hold respectively.  Then there exist $m>0$ such that
\begin{align}
\S_n-C\epsilon\E_n\lesssim_{\epsilon}|\b|^{-m}\S_n(0)+C_\delta(\E_n+\Z_n^2)^{1/2}\E_n
\end{align}
whenever our a priori assumption \eqref{A priori assumption} is satisfied. Here we recall Definition~\eqref{E:TOTALNORM} of the total norm $\E_n$.
\end{theorem}


\begin{proof}
Let $a+|\beta|+b\leq n$. Apply $\partial_s^a\pr^b\pt^\beta$ to the momentum equation~\eqref{E:EP in self-similar} to get
\begin{align*}
\partial_s^{a+2}\pr^b\pt^\beta\bs\theta-{\b\over 2}\partial_s^{a+1}\pr^b\pt^\beta\bs\theta+\partial_s^a\pr^b\pt^\beta(\delta\bs\theta+\mb P+\mb G)=0
\end{align*}
Taking the $\langle\cdot,\cdot\rangle_{3+b}$-inner with $\partial_s^{a+1}\pr^b\pt^\beta\bs\theta$ we get
\begin{align*}
0&={1\over 2}\partial_s\|\partial_s^{a+1}\pr^b\pt^\beta\bs\theta\|^2_{3+b}+\<\partial_s^a\pr^b\pt^\beta(\delta\bs\theta+\mb P+\mb G),\partial_s^{a+1}\pr^b\pt^\beta\bs\theta\>_{3+b}
-{\b\over 2}\|\partial_s^{a+1}\pr^b\pt^\beta\bs\theta\|_{3+b}^2.
\end{align*}
On the other hand, taking inner product of the equation with $\partial_s^{a}\pr^b\pt^\beta\bs\theta$ we get
\begin{align*}
0&=\partial_s\<\partial_s^{a+1}\pr^b\pt^\beta\bs\theta,\partial_s^a\pr^b\pt^\beta\bs\theta\>_{3+b}-\|\partial_s^{a+1}\pr^b\pt^\beta\bs\theta\|_{3+b}^2-{\b\over 4}\partial_{3+b}\|\partial_s^a\pr^b\pt^\beta\bs\theta\|_{3+b}^2\\
&\quad+\<\partial_s^a\pr^b\pt^\beta(\delta\bs\theta+\mb P+\mb G),\partial_s^a\pr^b\pt^\beta\bs\theta\>_{3+b}
\end{align*}
where we used the identiity $\<\partial_s^{a+2}\pr^b\pt^\beta\bs\theta,\partial_s^a\pr^b\pt^\beta\bs\theta\>=\partial_s\<\partial_s^{a+1}\pr^b\pt^\beta\bs\theta,\partial_s^a\pr^b\pt^\beta\bs\theta\>-\|\partial_s^{a+1}\pr^b\pt^\beta\bs\theta\|^2$. Multiply the latter equation by $c$, add to it two times the equation before, and then integrate w.r.t. $s$ to obtain
\begin{multline*}
0=\eva{\brac{{1\over 2}\|\partial_s^{a+1}\pr^b\pt^\beta\bs\theta\|_{3+b}^2+c\<\partial_s^{a+1}\pr^b\pt^\beta\bs\theta,\partial_s^a\pr^b\pt^\beta\bs\theta\>_{3+b}-{c\b\over 4}\|\partial_s^a\pr^b\pt^\beta\bs\theta\|_{3+b}^2}}_0^s\\
+\int_0^s\bigg(\<\partial_s^a\pr^b\pt^\beta(\delta\bs\theta+\mb P+\mb G),\partial_s^{a+1}\pr^b\pt^\beta\bs\theta\>_{3+b}+c\<\partial_s^a\pr^b\pt^\beta(\delta\bs\theta+\mb P+\mb G),\partial_s^a\pr^b\pt^\beta\bs\theta\>_{3+b}\\-\brac{c+{\b\over 2}}\|\partial_s^{a+1}\pr^b\pt^\beta\bs\theta\|_{3+b}^2\bigg)\d\tau.
\end{multline*}
\begin{enumerate}
\item When $b=0$, using Corollary \ref{estimate-theorem-cor-1} we get
\begin{align*}
&\norm{\partial_s^{a+1}\pt^\beta\bs\theta}_{3}^2+\norm{\partial_s^a\pt^\beta\bs\theta}_{3}^2+\norm{\partial_s^a\grad\pt^\beta\bs\theta}_{4}^2
+c\int_0^s\brac{\norm{\partial_s^{a+1}\pt^\beta\bs\theta}_{3}^2+\norm{\partial_s^a\pt^\beta\bs\theta}_{3}^2+\norm{\partial_s^a\grad\pt^\beta\bs\theta}_{4}^2}\d\tau\\
&\quad+|\b|^{-2}\eva{\brac{c\<\partial_s^{a+1}\pt^\beta\bs\theta,\partial_s^a\pt^\beta\bs\theta\>_3-{c\b\over 4}\|\partial_s^a\pt^\beta\bs\theta\|_3^2}}_0^s
-|\b|^{-2}\int_0^s\brac{2c+{\b\over 2}}\|\partial_s^{a+1}\pt^\beta\bs\theta\|_3^2\d\tau\\
&\lesssim|\b|^{-2}S_{n,|\beta|,0}(0)+\S_{n,|\beta|-1,0}+|\b|^{-2}\S_{n,|\beta|-1,0}^{1/2}\S_{n,|\beta|,0}^{1/2}+C_\delta(\E_n+\Z_n^2)^{1/2}\E_n.
\end{align*}
Choosing $c$ small enough (e.g. $c=|\b|^2/100$ when $\b\ll 1$), we get
\begin{align*}
&\norm{\partial_s^{a+1}\pt^\beta\bs\theta}_{3}^2+\norm{\partial_s^a\pt^\beta\bs\theta}_{3}^2+\norm{\partial_s^a\grad\pt^\beta\bs\theta}_{4}^2
+\int_0^s\brac{\norm{\partial_s^{a+1}\pt^\beta\bs\theta}_{3}^2+\norm{\partial_s^a\pt^\beta\bs\theta}_{3}^2+\norm{\partial_s^a\grad\pt^\beta\bs\theta}_{4}^2}\d\tau\\
&\lesssim|\b|^{-2}(|\b|^{-2}S_{n,|\beta|,0}(0)+\S_{n,|\beta|-1,0}+|\b|^{-2}\S_{n,|\beta|-1,0}^{1/2}\S_{n,|\beta|,0}^{1/2}
+C_\delta(\E_n+\Z_n^2)^{1/2}\E_n),
\end{align*}
and so (noting that the constant implicit in the notation $\lesssim$ do not depend on $s$)
\begin{align*}
\S_{n,|\beta|,0}&\lesssim|\b|^{-4}\S_{n,|\beta|,0}(0)+|\b|^{-2}\S_{n,|\beta|-1,0}+|\b|^{-4}\S_{n,|\beta|-1,0}^{1/2}\S_{n,|\beta|,0}^{1/2}
+C_\delta(\E_n+\Z_n^2)^{1/2}\E_n.
\end{align*}
In particular when $|\beta|=0$ we have $\S_{n,0,0}\lesssim|\b|^{-4}\S_{n,0,0}(0)+C_\delta(\E_n+\Z_n^2)^{1/2}\E_n$. And so using Young's inequality and by induction on $|\beta|$ we have
\begin{align}
\S_{n,|\beta|,0}&\lesssim|\b|^{-4-8|\beta|}\S_{n,|\beta|,0}(0)+C_\delta(\E_n+\Z_n^2)^{1/2}\E_n\label{E:near boundary estimate no radial}
\end{align}
for all $|\beta|\leq n$.
\item When $b>0$, using Corollary \ref{estimate-theorem-cor-2} we get
\begin{align*}
&\norm{\partial_s^{a+1}\pr^b\pt^\beta\bs\theta}_{3+b}^2+\norm{\partial_s^a\pr^b\pt^\beta\bs\theta}_{3+b}^2+\norm{\partial_s^a\grad\pr^b\pt^\beta\bs\theta}_{4+b}^2\\
&\quad+c\int_0^s\brac{\norm{\partial_s^{a+1}\pr^b\pt^\beta\bs\theta}_{3+b}^2+\norm{\partial_s^a\pr^b\pt^\beta\bs\theta}_{3+b}^2+\norm{\partial_s^a\grad\pr^b\pt^\beta\bs\theta}_{4+b}^2}\d\tau\\
&\quad+\eva{\brac{c\<\partial_s^{a+1}\pr^b\pt^\beta\bs\theta,\partial_s^a\pr^b\pt^\beta\bs\theta\>_3-{c\b\over 4}\|\partial_s^a\pr^b\pt^\beta\bs\theta\|_3^2}}_0^s
-\int_0^s\brac{2c+{\b\over 2}}\|\partial_s^{a+1}\pr^b\pt^\beta\bs\theta\|_3^2\d\tau\\
&\lesssim S_{n,|\beta|+b,b}(0)+\S_{n,|\beta|+b-1}+(\S_{n,|\beta|+b-1}^{1/2}+\S_{n,|\beta|+b,b-1}^{1/2})\E_n^{1/2}
+C_\delta(\E_n+\Z_n^2)^{1/2}\E_n
\end{align*}
Choosing $c$ small enough we get
\begin{align*}
&\norm{\partial_s^{a+1}\pr^b\pt^\beta\bs\theta}_{3+b}^2+\norm{\partial_s^a\pr^b\pt^\beta\bs\theta}_{3+b}^2+\norm{\partial_s^a\grad\pr^b\pt^\beta\bs\theta}_{4+b}^2\\
&\quad+\int_0^s\brac{\norm{\partial_s^{a+1}\pr^b\pt^\beta\bs\theta}_{3+b}^2+\norm{\partial_s^a\pr^b\pt^\beta\bs\theta}_{3+b}^2+\norm{\partial_s^a\grad\pr^b\pt^\beta\bs\theta}_{4+b}^2}\d\tau\\
&\lesssim|\b|^{-1}(S_{n,|\beta|+b,b}(0)+\S_{n,|\beta|+b-1}+(\S_{n,|\beta|+b-1}^{1/2}+\S_{n,|\beta|+b,b-1}^{1/2})\E_n^{1/2}+C_\delta(\E_n+\Z_n^2)^{1/2}\E_n)
\end{align*}
and so
\begin{align*}
\S_{n,|\beta|+b,b}&\lesssim|\b|^{-1}(\S_{n,|\beta|+b,b}(0)+\S_{n,|\beta|+b-1}+\S_{n,|\beta|+b,b-1}\\
&\quad+(\S_{n,|\beta|+b-1}^{1/2}+\S_{n,|\beta|+b,b-1}^{1/2})\E_n^{1/2}+C_\delta(\E_n+\Z_n^2)^{1/2}\E_n).
\end{align*}
or equivalently
\begin{align*}
\S_{n,c,d}&\lesssim|\b|^{-1}(\S_{n,c,d}(0)+\S_{n,c-1}+\S_{n,c,d-1}
+(\S_{n,c-1}^{1/2}+\S_{n,c,d-1}^{1/2})\E_n^{1/2}+C_\delta(\E_n+\Z_n^2)^{1/2}\E_n).
\end{align*}
We already know $\S_{n,c,0}\lesssim|\b|^{-4-8c}\S_{n,c,0}(0)+C_\delta(\E_n+\Z_n^2)^{1/2}\E_n$. And so using Young's inequality and by induction on $c$ and $d$ we have
\begin{align*}
\S_{n,c,d}-C\epsilon\E_n&\lesssim_\epsilon|\b|^{-m}\S_{n,c,d}(0)+C_\delta(\E_n+\Z_n^2)^{1/2}\E_n.
\end{align*}
for all $d\leq c\leq n$. This means we have $\S_n-C\epsilon\E_n\lesssim_\epsilon|\b|^{-m}\S_n(0)+C_\delta(\E_n+\Z_n^2)^{1/2}\E_n$. 
\end{enumerate}
\end{proof}



\subsubsection{Near origin energy estimate}
\label{Near origin energy estimate section}


In this subsection we will prove the energy estimate for $\Q_n$, see~\eqref{E:QBULLET}.


Using Lemma \ref{Hodge bound}, the following lemma shows that in fact we only need to control the divergence $\partial_s^a\grad^c\grad\cdot\bs\theta$ in order to control the near origin energy $Q_n$.


\begin{lemma}\label{grad from divergence lemma}
Let $n\geq 20$ and $c\leq n$. 
Let $\bs\theta$ be a solution of~\eqref{E:EP in self-similar} in the sense of Theorem \ref{T:LOCAL}, given on its maximal interval of existence. 
Assume further that the irrotationality constraint~\eqref{initial irrotational condition}. Then
\begin{align*}
Q_{n,c}\lesssim\|\partial_s^a\grad^c\grad\cdot\bs\theta\|_{3+2(c+1)}^2+Q_{n,c-1}+(E_n+Z_n^2)E_n.
\end{align*}
\end{lemma}


\begin{proof}
Let $a+c\leq n$ with $a>0$. Using the previous Lemma \ref{Hodge bound}, we have
\begin{align*}
\|\partial_s^a\grad^{c+1}\bs\theta\|_{3+2(c+1)}^2&\lesssim\|\partial_s^a\grad^c\grad\cdot\bs\theta\|_{3+2(c+1)}^2+\|\partial_s^a\grad^c\grad\times\bs\theta\|_{3+2(c+1)}^2
+\|\partial_s^a\grad^c\bs\theta\|_{3+2c}^2\\
&\leq\|\partial_s^a\grad^c\grad\cdot\bs\theta\|_{3+2(c+1)}^2+\|\partial_s^a\grad^c\grad\times\bs\theta\|_{3+2(c+1)}^2+Q_{n,c-1}.
\end{align*}
Recalling~\eqref{E:CURL1}  we have
$
\partial_s^a\grad^c\grad\times\bs\theta=-\partial_s^{a-1}\grad^c((\partial_s\grad\theta^k)\times\grad\theta^k)
$
and therefore
\begin{align*}
\|\partial_s^a\grad^c\grad\times\bs\theta\|_{3+2(c+1)}^2\lesssim(E_n+Z_n^2)E_n. 
\end{align*}
\end{proof}


\begin{lemma}
For any tensor field $T$ smooth enough we have
\begin{align*}
\partial^\gamma\brac{\bar w^{-3}\partial_k(\bar w^4T^k)}=\bar w\partial^\gamma\partial_kT^k+\sum_{|\gamma'|\leq|\gamma|}\bar w^{|\gamma|-|\gamma'|}\<C\partial^{\gamma'}T^k\>.
\end{align*}
where we recall notations introduced in Definition~\ref{Special notations}.
\end{lemma}


\begin{proof}
The statement follows easily by induction. 
\end{proof}


\begin{theorem}[Near origin energy estimate]\label{Near origin energy estimate}
Let $n\geq 21$ and $\delta$ small. Let $\bs\theta$ be a solution of~\eqref{E:EP in self-similar} in the sense of Theorem \ref{T:LOCAL}, given on its maximal interval of existence. 
Assume further that the energy, momentum, and irrotationality constraints~\eqref{initial momentum condition},~\eqref{initial energy condition}, and~\eqref{initial irrotational condition} hold respectively.  Then we have
\begin{align}
\Q_n\lesssim|\b|^{-4}\E_n(0)+C_\delta(\E_n+\Z_n^2)^{1/2}\E_n
\end{align}
whenever our a priori assumption \eqref{A priori assumption} is satisfied.
\end{theorem}
\begin{proof}
Recall the momentum equation (\ref{E:EP in self-similar}) is
\begin{align*}
\mb 0=\partial_s^2\bs\theta-{1\over 2}\b\partial_s\bs\theta+\delta\bs\theta+\bar w^{-3}\partial_k(\bar w^4(\underbrace{\A^k\J^{-1/3}-I^k}_{=T}))+\A\grad\Phi-\grad\mathcal{K}\bar w^3.
\end{align*}
where we recall $T$ in \eqref{E:T}. Also recall from \eqref{E:weighted divergence of gravity} $(\A\grad)\cdot(\A\grad)\Phi(\mb x)=4\pi\bar w^3\J^{-1}$. So taking the divergence of the gravity term makes it easy to estimate. From Lemma~\ref{grad from divergence lemma} we also know that to control $Q_n$ it suffices to estimate the divergence.  Let $a+|\gamma|+1\leq n$. Evaluating the dot product of~(\ref{E:EP in self-similar}) with $\partial_s^a\partial^\gamma\A\grad$ we get
\begin{align*}
\mb 0&=\partial_s^a\partial^\gamma\A\grad\cdot\partial_s^2\bs\theta-{1\over 2}\b\partial_s^a\partial^\gamma\A\grad\cdot\partial_s\bs\theta+\delta\partial_s^a\partial^\gamma\A\grad\cdot\bs\theta\\
&\quad+\partial_s^a\partial^\gamma\A\grad\cdot\brac{\bar w^{-3}\partial_k(\bar w^4T^k)}+4\pi\partial_s^a\partial^\gamma(\bar w^3\J^{-1})-(\partial_s^a\partial^\gamma\A\grad)\cdot\grad\mathcal{K}\bar w^3\\
&=\partial_s^a\partial^\gamma\A\grad\cdot\partial_s^2\bs\theta-{1\over 2}\b\partial_s^a\partial^\gamma\A\grad\cdot\partial_s\bs\theta+\delta\partial_s^a\partial^\gamma\A\grad\cdot\bs\theta\\
&\quad+\partial_s^a\partial^\gamma\A\grad\cdot\brac{\bar w^{-3}\partial_k(\bar w^4T^k)}+4\pi\partial_s^a\partial^\gamma(\bar w^3(\J^{-1}-1))
-(\partial_s^a\partial^\gamma(\A-I)\grad)\cdot\grad\mathcal{K}\bar w^3
\end{align*}
From here we will the do two things (i) and (ii) as follows
\begin{enumerate}[(i)]
\item Times the equation with $\bar w^{6+2|\gamma|}\partial_s^{a+1}\partial^\gamma\A\grad\cdot\bs\theta$ and integrate in time and space we get
\begin{align*}
0&=\int_0^s\bigg(\bigg\<\partial_s^a\partial^\gamma\A\grad\cdot\partial_s^2\bs\theta-{1\over 2}\b\partial_s^a\partial^\gamma\A\grad\cdot\partial_s\bs\theta+\delta\partial_s^a\partial^\gamma\A\grad\cdot\bs\theta,
\partial_s^{a+1}\partial^\gamma\A\grad\cdot\bs\theta\bigg\>_{6+2|\gamma|}\\
&\qquad\qquad+\inner{\partial_s^a\partial^\gamma\A\grad\cdot\brac{\bar w^{-3}\partial_k(\bar w^4T^k)},\partial_s^{a+1}\partial^\gamma\A\grad\cdot\bs\theta}_{6+2|\gamma|}\\
&\qquad\qquad+\big\<4\pi\partial_s^a\partial^\gamma(\bar w^3(\J^{-1}-1))-(\partial_s^a\partial^\gamma(\A-I)\grad)\cdot\grad\mathcal{K}\bar w^3,
\partial_s^{a+1}\partial^\gamma\A\grad\cdot\bs\theta\big\>_{6+2|\gamma|}\bigg)\d\tau
\end{align*}
Now commuting $\A$ with space and time derivatives, we get a non-linear remainder $\Rd[(\E_n+\Z_n^2)^{1/2}\E_n]$ (recall notation $\Rd[\star]$ defined in Definition \ref{Special notations}),
\begin{align*}
0&=\int_0^s\bigg(\inner{\partial_s\A\grad\cdot\partial_s^{a+1}\partial^\gamma\bs\theta-{1\over 2}\b\A\grad\cdot\partial_s^{a+1}\partial^\gamma\bs\theta,\A\grad\partial_s^{a+1}\partial^\gamma\cdot\bs\theta}_{6+2|\gamma|}\\
&\qquad\qquad+\inner{\delta\A\grad\cdot\partial_s^a\partial^\gamma\bs\theta,\partial_s\A\grad\cdot\partial_s^{a}\partial^\gamma\bs\theta}_{6+2|\gamma|}\\
&\qquad\qquad+\inner{\A\grad\cdot\partial_s^a\partial^\gamma\brac{\bar w^{-3}\partial_k(\bar w^4T^k)},\partial_s\A\grad\cdot\partial_s^{a}\partial^\gamma\bs\theta}_{6+2|\gamma|}\\
&\qquad\qquad+\big\<4\pi(\J^{-1}-1)\partial_s^a\partial^\gamma(\bar w^3)-(\A-I)\grad\cdot\grad\mathcal{K}\partial_s^a\partial^\gamma\bar w^3,
\partial_s^{a+1}\partial^\gamma\A\grad\cdot\bs\theta\big\>_{6+2|\gamma|}\bigg)\d\tau\\
&\quad+\Rd[(\E_n+\Z_n^2)^{1/2}\E_n]
\end{align*}
Now terms in the first two line we factorised, and terms in the last line in the integral we can estimate by $\Q_{n,|\gamma|}^{1/2}\Q_{n,|\gamma|+1}^{1/2}$ and $\Q_{n-1}^{1/2}\Q_{n,|\gamma|+1}^{1/2}$,
\begin{align*}
0&=\int_0^s\bigg({1\over 2}\partial_s\norm{\A\grad\cdot\partial_s^{a+1}\partial^\gamma\bs\theta}_{6+2|\gamma|}^2-{1\over 2}\b\norm{\A\grad\cdot\partial_s^{a+1}\partial^\gamma\bs\theta}_{6+2|\gamma|}^2
+{1\over 2}\delta\partial_s\norm{\A\grad\cdot\partial_s^a\partial^\gamma\bs\theta}_{6+2|\gamma|}^2\\
&\qquad\qquad+\inner{\A\grad\cdot\brac{\bar w\partial_s^a\partial^\gamma\partial_k T^k},\partial_s\A\grad\cdot\partial_s^{a}\partial^\gamma\bs\theta}_{6+2|\gamma|}\bigg)\d\tau\\
&\quad+\Rd[(\E_n+\Z_n^2)^{1/2}\E_n]+\Rd[\Q_{n,|\gamma|}^{1/2}\Q_{n,|\gamma|+1}^{1/2}]+\Rd[\Q_{n-1}^{1/2}\Q_{n,|\gamma|+1}^{1/2}]
\end{align*}
Now terms that are full time derivatives can be evaluated, and $\partial_s^a\partial^\gamma\partial_k T^k$ can be converted to $T_T[\partial_s^a\partial^\gamma\partial_k\grad\bs\theta]^k$ (recall Lemma \ref{pressure structure lemma}) leaving a reminder that we can estimate with $(\E_n+\Z_n^2)^{1/2}\E_n$.
\begin{align*}
0&={1\over 2}\eva{\brac{\norm{\A\grad\cdot\partial_s^{a+1}\partial^\gamma\bs\theta}_{6+2|\gamma|}^2+\delta\norm{\A\grad\cdot\partial_s^a\partial^\gamma\bs\theta}_{6+2|\gamma|}^2}}_0^s\\
&\quad+\int_0^s\bigg(-{1\over 2}\b\norm{\A\grad\cdot\partial_s^{a+1}\partial^\gamma\bs\theta}_{6+2|\gamma|}^2+\inner{\A\grad\cdot\brac{\bar wT_T[\partial_s^a\partial^\gamma\partial_k\grad\bs\theta]^k},\partial_s\A\grad\cdot\partial_s^{a}\partial^\gamma\bs\theta}_{6+2|\gamma|}\bigg)\d\tau\\
&\quad+\Rd[(\E_n+\Z_n^2)^{1/2}\E_n]+\Rd[\Q_{n,|\gamma|}^{1/2}\Q_{n,|\gamma|+1}^{1/2}]+\Rd[\Q_{n-1}^{1/2}\Q_{n,|\gamma|+1}^{1/2}]
\end{align*}
Now all the term before the term with $T_T$ can be bounded by $\Q_{n,|\gamma|}^{1/2}\Q_{n,|\gamma|+1}^{1/2}$ and $\Q_{n-1}^{1/2}\Q_{n,|\gamma|+1}^{1/2}$, and we integrate by parts on the term with $T_T$,
\begin{align*}
0&=-\int_0^s\inner{\bar wT_T[\partial_s^a\partial^\gamma\partial_k\grad\bs\theta]^k,\partial_s\A\grad(\A\grad\cdot\partial_s^{a}\partial^\gamma\bs\theta)}_{6+2|\gamma|}\d\tau\\
&\quad+\Rd[(\E_n+\Z_n^2)^{1/2}\E_n]+\Rd[\Q_{n,|\gamma|}]+\Rd[\Q_{n,|\gamma|}^{1/2}\Q_{n,|\gamma|+1}^{1/2}]+\Rd[\Q_{n-1}^{1/2}\Q_{n,|\gamma|+1}^{1/2}]
\end{align*}
Now we expend the terms by definition and simplify,
\begin{align*}
0
&=\int_0^s\inner{\J^{-1/3}\brac{\A^k_m\A^l+{1\over 3}\A^k\A^l_m}\partial_s^a\partial^\gamma\partial_k\partial_l\theta^m,\partial_s(\A^j\A^\ell_i\partial_j\partial_\ell\partial_s^{a}\partial^\gamma\theta^i)}_{7+2|\gamma|}\d\tau\\
&\quad+\Rd[(\E_n+\Z_n^2)^{1/2}\E_n]+\Rd[\Q_{n,|\gamma|}]+\Rd[\Q_{n,|\gamma|}^{1/2}\Q_{n,|\gamma|+1}^{1/2}]+\Rd[\Q_{n-1}^{1/2}\Q_{n,|\gamma|+1}^{1/2}]\\
&=\int_0^s\inner{\J^{-1/3}\brac{\A^k_m\A^l_o+{1\over 3}\A^k_o\A^l_m}\partial_k\partial_l\partial_s^a\partial^\gamma\theta^m,\partial_s(\A^j_o\A^\ell_i\partial_j\partial_\ell\partial_s^{a}\partial^\gamma\theta^i)}_{7+2|\gamma|}\d\tau\\
&\quad+\Rd[(\E_n+\Z_n^2)^{1/2}\E_n]+\Rd[\Q_{n,|\gamma|}]+\Rd[\Q_{n,|\gamma|}^{1/2}\Q_{n,|\gamma|+1}^{1/2}]+\Rd[\Q_{n-1}^{1/2}\Q_{n,|\gamma|+1}^{1/2}]\\
&=\int_0^s{4\over 3}\inner{\J^{-1/3}\A^k_o\A^l_m\partial_k\partial_l\partial_s^a\partial^\gamma\theta^m,\partial_s(\A^j_o\A^\ell_i\partial_j\partial_\ell\partial_s^{a}\partial^\gamma\theta^i)}_{7+2|\gamma|}\d\tau\\
&\quad+\Rd[(\E_n+\Z_n^2)^{1/2}\E_n]+\Rd[\Q_{n,|\gamma|}]+\Rd[\Q_{n,|\gamma|}^{1/2}\Q_{n,|\gamma|+1}^{1/2}]+\Rd[\Q_{n-1}^{1/2}\Q_{n,|\gamma|+1}^{1/2}]
\end{align*}
Now the term in the integral can be factorised into a time derivative,
\begin{align*}
0&={2\over 3}\int_0^s\int_{B_R}\J^{-1/3}\partial_s\norm{\A^k\A^l_m\partial_k\partial_l\partial_s^a\partial^\gamma\theta^m}^2\bar w^{7+2|\gamma|}\d\mb x\d\tau\\
&\quad+\Rd[(\E_n+\Z_n^2)^{1/2}\E_n]+\Rd[\Q_{n,|\gamma|}]+\Rd[\Q_{n,|\gamma|}^{1/2}\Q_{n,|\gamma|+1}^{1/2}]+\Rd[\Q_{n-1}^{1/2}\Q_{n,|\gamma|+1}^{1/2}]
\end{align*}
Now we can evaluate the time integral using integration by parts, leaving a remainder term that can be estimated with $(\E_n+\Z_n^2)^{1/2}\E_n$ when the time derivative falls on $\J^{-1/3}$,
\begin{align*}
0&={2\over 3}\eva{\int_{B_R}\J^{-1/3}\norm{\A^k\A^l_m\partial_k\partial_l\partial_s^a\partial^\gamma\theta^m}^2\bar w^{7+2|\gamma|}\d\mb x}_0^s\\
&\quad+\Rd[(\E_n+\Z_n^2)^{1/2}\E_n]+\Rd[\Q_{n,|\gamma|}]+\Rd[\Q_{n,|\gamma|}^{1/2}\Q_{n,|\gamma|+1}^{1/2}]+\Rd[\Q_{n-1}^{1/2}\Q_{n,|\gamma|+1}^{1/2}]\\
&={2\over 3}\eva{\int_{B_R}\norm{\grad\grad\cdot\partial_s^a\partial^\gamma\bs\theta}^2\bar w^{7+2|\gamma|}\d\mb x}_0^s\\
&\quad+\Rd[(\E_n+\Z_n^2)^{1/2}\E_n]+\Rd[\Q_{n,|\gamma|}]+\Rd[\Q_{n,|\gamma|}^{1/2}\Q_{n,|\gamma|+1}^{1/2}]+\Rd[\Q_{n-1}^{1/2}\Q_{n,|\gamma|+1}^{1/2}]
\end{align*}
It follows that
\begin{multline*}
\norm{\grad\grad\cdot\partial_s^a\partial^\gamma\bs\theta}_{3+2(2+|\gamma|)}^2\\\lesssim\Q_{n,|\gamma|+1}(0)+\Q_{n,|\gamma|}+\Q_{n,|\gamma|}^{1/2}\Q_{n,|\gamma|+1}^{1/2}+\Q_{n-1}^{1/2}\Q_{n,|\gamma|+1}^{1/2}+(\E_n+\Z_n^2)^{1/2}\E_n.
\end{multline*}
Using Lemma \ref{grad from divergence lemma} we get
\begin{align*}
Q_{n,|\gamma|+1}\lesssim\Q_{n,|\gamma|+1}(0)+\Q_{n,|\gamma|}+\Q_{n,|\gamma|}^{1/2}\Q_{n,|\gamma|+1}^{1/2}+\Q_{n-1}^{1/2}\Q_{n,|\gamma|+1}^{1/2}+(\E_n+\Z_n^2)^{1/2}\E_n
\end{align*}
\item Times the equation with $\bar w^{6+2|\gamma|}\partial_s^{a}\partial^\gamma\A\grad\cdot\bs\theta$ and integrate in time and space we get
\begin{align*}
0&=\int_0^s\bigg(\bigg\<\partial_s^a\partial^\gamma\A\grad\cdot\partial_s^2\bs\theta-{1\over 2}\b\partial_s^a\partial^\gamma\A\grad\cdot\partial_s\bs\theta+\delta\partial_s^a\partial^\gamma\A\grad\cdot\bs\theta,
\partial_s^{a}\partial^\gamma\A\grad\cdot\bs\theta\bigg\>_{6+2|\gamma|}\\
&\qquad\qquad+\inner{\partial_s^a\partial^\gamma\A\grad\cdot\brac{\bar w^{-3}\partial_k(\bar w^4T^k)},\partial_s^{a}\partial^\gamma\A\grad\cdot\bs\theta}_{6+2|\gamma|}\\
&\qquad\qquad+\big\<4\pi\partial_s^a\partial^\gamma(\bar w^3(\J^{-1}-1))-(\partial_s^a\partial^\gamma(\A-I)\grad)\cdot\grad\mathcal{K}\bar w^3,
\partial_s^{a}\partial^\gamma\A\grad\cdot\bs\theta\big\>_{6+2|\gamma|}\bigg)\d\tau
\end{align*}
Now commuting $\A$ with space and time derivatives, we get a non-linear remainder $\Rd[(\E_n+\Z_n^2)^{1/2}\E_n]$,
\begin{align*}
0&=\int_0^s\bigg(\inner{\partial_s\A\grad\cdot\partial_s^{a+1}\partial^\gamma\bs\theta-{1\over 2}\b\A\grad\cdot\partial_s^{a+1}\partial^\gamma\bs\theta,\A\grad\partial_s^{a}\partial^\gamma\cdot\bs\theta}_{6+2|\gamma|}\\
&\qquad\qquad+\inner{\delta\A\grad\cdot\partial_s^a\partial^\gamma\bs\theta,\A\grad\cdot\partial_s^{a}\partial^\gamma\bs\theta}_{6+2|\gamma|}\\
&\qquad\qquad+\inner{\A\grad\cdot\partial_s^a\partial^\gamma\brac{\bar w^{-3}\partial_k(\bar w^4T^k)},\A\grad\cdot\partial_s^{a}\partial^\gamma\bs\theta}_{6+2|\gamma|}\\
&\qquad\qquad+\big\<4\pi(\J^{-1}-1)\partial_s^a\partial^\gamma(\bar w^3)-(\A-I)\grad\cdot\grad\mathcal{K}\partial_s^a\partial^\gamma\bar w^3,
\partial_s^{a}\partial^\gamma\A\grad\cdot\bs\theta\big\>_{6+2|\gamma|}\bigg)\d\tau\\
&\quad+\Rd[(\E_n+\Z_n^2)^{1/2}\E_n]
\end{align*}
Now all the terms, apart from the top order term involving $T_T$ from the pressure, can be bounded by $(\E_n+\Z_n^2)^{1/2}\E_n]+\Rd[\Q_{n,|\gamma|}+\Q_{n,|\gamma|}+\Q_{n,|\gamma|}^{1/2}\Q_{n,|\gamma|+1}^{1/2}+\Q_{n-1}^{1/2}\Q_{n,|\gamma|+1}^{1/2}$,
\begin{align*}
0&=-\int_0^s\bigg(\inner{\A\grad\cdot\partial_s^{a+1}\partial^\gamma\bs\theta,\partial_s\A\grad\partial_s^{a}\partial^\gamma\cdot\bs\theta}_{6+2|\gamma|}
+\inner{\A\grad\cdot\brac{\bar w\partial_s^a\partial^\gamma\partial_k T^k},\A\grad\cdot\partial_s^{a}\partial^\gamma\bs\theta}_{6+2|\gamma|}\bigg)\d\tau\\
&\quad+\Rd[(\E_n+\Z_n^2)^{1/2}\E_n]+\Rd[\Q_{n,|\gamma|}]+\Rd[\Q_{n,|\gamma|}^{1/2}\Q_{n,|\gamma|+1}^{1/2}]+\Rd[\Q_{n-1}^{1/2}\Q_{n,|\gamma|+1}^{1/2}]\\
&=\int_0^s\inner{\A\grad\cdot\brac{\bar wT_T[\partial_s^a\partial^\gamma\partial_k\grad\bs\theta]^k},\A\grad\cdot\partial_s^{a}\partial^\gamma\bs\theta}_{6+2|\gamma|}\d\tau\\
&\quad+\Rd[(\E_n+\Z_n^2)^{1/2}\E_n]+\Rd[\Q_{n,|\gamma|}]+\Rd[\Q_{n,|\gamma|}^{1/2}\Q_{n,|\gamma|+1}^{1/2}]+\Rd[\Q_{n-1}^{1/2}\Q_{n,|\gamma|+1}^{1/2}]
\end{align*}
Now we integrate by parts,
\begin{align*}
0&=-\int_0^s\inner{\bar wT_T[\partial_s^a\partial^\gamma\partial_k\grad\bs\theta]^k,\A\grad(\A\grad\cdot\partial_s^{a}\partial^\gamma\bs\theta)}_{6+2|\gamma|}\d\tau\\
&\quad+\Rd[(\E_n+\Z_n^2)^{1/2}\E_n]+\Rd[\Q_{n,|\gamma|}]+\Rd[\Q_{n,|\gamma|}^{1/2}\Q_{n,|\gamma|+1}^{1/2}]+\Rd[\Q_{n-1}^{1/2}\Q_{n,|\gamma|+1}^{1/2}]
\end{align*}
Now we expend the terms by defintion and simplify,
\begin{align*}
0&=\int_0^s\inner{\J^{-1/3}\brac{\A^k_m\A^l+{1\over 3}\A^k\A^l_m}\partial_s^a\partial^\gamma\partial_k\partial_l\theta^m,\A^j\A^\ell_i\partial_j\partial_\ell\partial_s^{a}\partial^\gamma\theta^i}_{7+2|\gamma|}\d\tau\\
&\quad+\Rd[(\E_n+\Z_n^2)^{1/2}\E_n]+\Rd[\Q_{n,|\gamma|}]+\Rd[\Q_{n,|\gamma|}^{1/2}\Q_{n,|\gamma|+1}^{1/2}]+\Rd[\Q_{n-1}^{1/2}\Q_{n,|\gamma|+1}^{1/2}]\\
&=\int_0^s\inner{\J^{-1/3}\brac{\A^k_m\A^l_o+{1\over 3}\A^k_o\A^l_m}\partial_k\partial_l\partial_s^a\partial^\gamma\theta^m,\A^j_o\A^\ell_i\partial_j\partial_\ell\partial_s^{a}\partial^\gamma\theta^i}_{7+2|\gamma|}\d\tau\\
&\quad+\Rd[(\E_n+\Z_n^2)^{1/2}\E_n]+\Rd[\Q_{n,|\gamma|}]+\Rd[\Q_{n,|\gamma|}^{1/2}\Q_{n,|\gamma|+1}^{1/2}]+\Rd[\Q_{n-1}^{1/2}\Q_{n,|\gamma|+1}^{1/2}]\\
&=\int_0^s{4\over 3}\inner{\J^{-1/3}\A^k_o\A^l_m\partial_k\partial_l\partial_s^a\partial^\gamma\theta^m,\A^j_o\A^\ell_i\partial_j\partial_\ell\partial_s^{a}\partial^\gamma\theta^i}_{7+2|\gamma|}\d\tau\\
&\quad+\Rd[(\E_n+\Z_n^2)^{1/2}\E_n]+\Rd[\Q_{n,|\gamma|}]+\Rd[\Q_{n,|\gamma|}^{1/2}\Q_{n,|\gamma|+1}^{1/2}]+\Rd[\Q_{n-1}^{1/2}\Q_{n,|\gamma|+1}^{1/2}]\\
&={4\over 3}\int_0^s\int_{B_R}\J^{-1/3}\norm{\A^k\A^l_m\partial_k\partial_l\partial_s^a\partial^\gamma\theta^m}^2\bar w^{7+2|\gamma|}\d\mb x\d\tau\\
&\quad+\Rd[(\E_n+\Z_n^2)^{1/2}\E_n]+\Rd[\Q_{n,|\gamma|}]+\Rd[\Q_{n,|\gamma|}^{1/2}\Q_{n,|\gamma|+1}^{1/2}]+\Rd[\Q_{n-1}^{1/2}\Q_{n,|\gamma|+1}^{1/2}]\\
&={4\over 3}\int_0^s\int_{B_R}\norm{\grad\grad\cdot\partial_s^a\partial^\gamma\bs\theta}^2\bar w^{7+2|\gamma|}\d\mb x\d\tau\\
&\quad+\Rd[(\E_n+\Z_n^2)^{1/2}\E_n]+\Rd[\Q_{n,|\gamma|}]+\Rd[\Q_{n,|\gamma|}^{1/2}\Q_{n,|\gamma|+1}^{1/2}]+\Rd[\Q_{n-1}^{1/2}\Q_{n,|\gamma|+1}^{1/2}]
\end{align*}
It follows that
\begin{multline*}
\int_0^s\norm{\grad\grad\cdot\partial_s^a\partial^\gamma\bs\theta}_{3+2(2+|\gamma|)}^2\d\tau\\\lesssim\Q_{n,|\gamma|+1}(0)+\Q_{n,|\gamma|}+\Q_{n,|\gamma|}^{1/2}\Q_{n,|\gamma|+1}^{1/2}+\Q_{n-1}^{1/2}\Q_{n,|\gamma|+1}^{1/2}+(\E_n+\Z_n^2)^{1/2}\E_n.
\end{multline*}
Using Lemma \ref{grad from divergence lemma} we get
\begin{align*}
\int_0^sQ_{n,|\gamma|+1}\d\tau\lesssim\Q_{n,|\gamma|+1}(0)+\Q_{n,|\gamma|}+\Q_{n,|\gamma|}^{1/2}\Q_{n,|\gamma|+1}^{1/2}+\Q_{n-1}^{1/2}\Q_{n,|\gamma|+1}^{1/2}+(\E_n+\Z_n^2)^{1/2}\E_n
\end{align*}
\end{enumerate}
Combining the results of (i) and (ii) and noting that $\lesssim$ does not depend on $s$, we get that
\begin{align*}
\Q_{n,|\gamma|+1}&\lesssim\Q_{n,|\gamma|+1}(0)+\Q_{n,|\gamma|}+\Q_{n,|\gamma|}^{1/2}\Q_{n,|\gamma|+1}^{1/2}+\Q_{n-1}^{1/2}\Q_{n,|\gamma|+1}^{1/2}+(\E_n+\Z_n^2)^{1/2}\E_n\\
&\lesssim\E_{n}(0)+\Q_{n,|\gamma|}+\Q_{n,|\gamma|}^{1/2}\Q_{n,|\gamma|+1}^{1/2}+\Q_{n-1}^{1/2}\Q_{n,|\gamma|+1}^{1/2}+(\E_n+\Z_n^2)^{1/2}\E_n
\end{align*}
We have by definition and equation (\ref{E:near boundary estimate no radial}) in the previous theorem
\begin{align*}
\Q_{n,0}\lesssim\S_{n,0}&\lesssim|\b|^{-4}\S_n(0)+C_\delta(\E_n+\Z_n^2)^{1/2}\E_n
\leq|\b|^{-4}\E_n(0)+C_\delta(\E_n+\Z_n^2)^{1/2}\E_n
\end{align*}
And so using Young's inequality and by induction we have
\begin{align*}
\Q_{n,d}&\lesssim|\b|^{-4}\E_{n}(0)+C_\delta(\E_n+\Z_n^2)^{1/2}\E_n
\end{align*}
for all $d\leq n$. Therefore we have $\Q_n\lesssim|\b|^{-4}\E_n(0)+C_\delta(\E_n+\Z_n^2)^{1/2}\E_n$. 
\end{proof}

\subsubsection{Bootstrapping scheme and final theorem}
\label{Bootstrapping scheme and final theorem}

In this subsection we will prove our main theorem that the energy $E_n$ decays exponentially while $Z_n$ remains bounded. To do so we will use the bootstrapping scheme in the following lemma and proposition.

\begin{lemma}\label{pre bootstrapping scheme}
Suppose $E:[0,T]\to[0,\infty]$ is continuous and
\begin{align*}
E(t)\leq C_1E(0)+C_2E(t)^{3/2}\qquad\text{whenever}\qquad \sup_{\tau\in[0,t]}E(\tau)\leq C_3.
\end{align*}
where $C_1\geq 1$. Then $E\leq 2C_1E(0)$ whenever $E(0)\leq\min\{(2^5C_1C_2^2)^{-1},C_3/2C_1\}$.
\end{lemma}
\begin{proof}
We will prove this by a standard bootstrap argument. Let
\[I=\set{t\in[0,T]:\sup_{\tau\in[0,t]}E(\tau)\leq\min\{2C_1E(0),C_3\}}.\]
Then $I$ is non-empty (since $0\in I$) and closed (since $E$ is continuous). If $I=[0,T]$, then we are done. Otherwise, let $t_0=\inf\{t\in[0,T]:t\not\in I\}$. We must have $t_0\in I$ since $0\in I$ and $I$ is closed. Then we have 
\[E(t_0)\leq C_1E(0)+C_2(2C_1E(0))^{3/2}\leq {3\over 2}C_1E(0)\leq{3\over 4}C_3.\]
So by continuity of $E$, a neighbourhood of $t_0$ must lie in $I$. But this contradicts the definition of $t_0$. So we must have $I=[0,T]$.
\end{proof}

\begin{proposition}\label{bootstrapping scheme}
Suppose $E,Z:[0,\infty)\to[0,\infty]$ are continuous and for all $t\geq t_0\geq 0$ we have
\begin{align*}
\E_{t_0}(t)&\leq C_0\E_{t_0}(t_0)+C_1Z(t_0)\E_{t_0}(t)+C_2(1+(t-t_0)^k)\E_{t_0}(t)^{3/2}\\
Z(t)&\leq Z(t_0)+C_3(t-t_0)^l\E_{t_0}^{1/2}(t)
\end{align*}
whenever $\sup_{\tau\in[t_0,t]}(E(\tau)+Z(\tau))\leq C_4$, where $k,l\geq 0$ and
\begin{align*}
\E_{t_0}(t)=\sup_{\tau\in[t_0,t]}E(\tau)+\int_{t_0}^tE(\tau)\d\tau.
\end{align*}
Then there exist $\epsilon>0$ such that $\E_0\leq 6C_0\E_0(0)$ whenever $\E_0(0),Z(0)\leq\epsilon$. Moreover, $E(t)\leq 16(4^{-t/32C_0})C_0E(0)$.
\end{proposition}
\begin{proof}
Let $T=32C_0$ and $C_1Z(0)<\min\{1/4,C_1C_4/4\}\leq 1/2$. Then by the above Lemma \ref{pre bootstrapping scheme}, for small enough $\epsilon$, we have $\E_0\leq 4C_0\E_0(0)$ on $[0,T]$ whenever $\E_0(0)\leq\epsilon$. On $[T/2,T]$, there must exist a point $T_1$ such that $E(T_1)\leq{1\over 4}\E_0(0)$, otherwise $\E_0(T)>4C_0\E_0(0)$.

By having a small enough $\epsilon$, we can assume $2C_0^{1/2}C_1C_3T^l\E_0(0)^{1/2}<\min\{1/8,C_1C_4/8\}$. Now
\begin{align*}
C_1Z(T_1)&\leq C_1Z(0)+2C_0^{1/2}C_1C_3T^l\E_0(0)^{1/2}\\
&\leq\min\set{{1\over 4},{C_1C_4\over 4}}+\min\set{{1\over 8},{C_1C_4\over 8}}\leq\min\set{{1\over 2},{C_1C_4\over 2}}.
\end{align*}
Then by the above Lemma \ref{pre bootstrapping scheme}, we get that $\E_{T_1}\leq 4C_0\E_{T_1}(T_1)=4C_0E(T_1)\leq C_0\E_0(0)$ on $[T_1,T_1+T]$. On $[T_1+T/2,T_1+T]$, there must exist a point $T_2$ such that $E(T_2)\leq{1\over 4}\E_{T_1}(T_1)={1\over 4}E(T_1)$, otherwise $\E_{T_1}(T)>4C_0\E_{T_1}(T_1)$.

Repeating inductively, we can get $T_n\in[T_{n-1}+T/2,T_{n-1}+T]$ such that
\begin{align*}
C_1Z(T_n)&\leq C_1Z(T_{n-1})+C_1C_3T^l\E_{T_{n-1}}(T_n)^{1/2}\\
&\leq\brac{{1\over 4}\sum_{m=0}^{n-1}{1\over 2^m}+{1\over 4}{1\over 2^n}}\min\{1,C_1C_4\}\leq{1\over 2}\min\{1,C_1C_4\}\\
\E_{T_n}&\leq 4^{1-n}C_0\E_0(0)\qquad\text{on}\qquad[T_n,T_n+T].
\end{align*}
Now
\begin{align*}
\E_0\leq\E_0(\infty)\leq\E_0(T_1)+\E_{T_1}(T_2)+\E_{T_2}(T_3)+\cdots\leq 6C_0\E_0(0). 
\end{align*}
\end{proof}


Finally, before proving the main theorem, we provide a simple lemma based on the fundamental theorem of calculus, which
relates the $\Z_n$-norm~\eqref{E:ZNTOTAL} to the total energy norm $\E_n$~\eqref{E:TOTALNORM}.


\begin{lemma}\label{bottom estimate lemma}
We have
\begin{align*}
\Z_1(s)&\leq\Z_1(0)+Cs^{1/2}\S_1^{1/2}(s)\\
\Z_n(s)&\leq\Z_n(0)+Cs^{1/2}\E_n^{1/2}(s)
\end{align*}
\end{lemma}


\begin{proof}
We have
$h(s)-h(0)=\int_0^s\partial_sh(\tau)\d\tau
$
and therefore
\begin{align*}
(h(s)-h(0))^2=\brac{\int_0^s\partial_sh(\tau)\d\tau}^2
\leq s\int_0^s(\partial_sh(\tau))^2\d\tau.
\end{align*}
This easily gives
$ \|h(s)-h(0)\|_k^2\leq s\int_0^s\|\partial_sh(\tau)\|_k^2\d\tau$
and thus
$ \|h(s)\|_k\leq\|h(0)\|_k+s^{1/2}\brac{\int_0^s\|\partial_sh(\tau)\|_k^2\d\tau}^{1/2},
$
which concludes the proof.
\end{proof}


\begin{theorem}
Let $n\geq 21$ and $\delta$ small. Let $(\bs\theta, \partial_s\bs\theta)$ be a solution of~\eqref{E:EP in self-similar} in the sense of Theorem \ref{T:LOCAL} and that satisfies \eqref{initial momentum condition}, \eqref{initial energy condition} and \eqref{initial irrotational condition} (i.e. the perturbation does not change the momentum or energy of the star, and correspond to an irrotational flow). Then there is some $m>0$ such that we have
\begin{align}
\E_n(s)\lesssim|\b|^{-m}\E_n(0)+C_\delta\brac{\Z_n(0)\E_n(s)+(1+s^{1/2})\E_n(s)^{3/2}}
\end{align}
whenever our a priori assumption \eqref{A priori assumption} is satisfied. Moreover, there exists $\epsilon_0>0$ such that if $E_n(0)+Z_n(0)^2\leq\epsilon_0$, then we have $\E_n\lesssim\E_n(0)$ with $E_n(s)\lesssim e^{-C|\b|^ms}$ (decaying exponentially on $[0,\infty)$) and $Z_n$ bounded on $[0,\infty)$.
\end{theorem}


\begin{proof}
By the energy estimates in Theorem \ref{Near boundary energy estimate} and \ref{Near origin energy estimate} (for $\S_n$ and $\Q_n$) in the last two subsections and Lemma \ref{bottom estimate lemma} we have (choosing $\epsilon$ small enough)
\begin{align*}
\E_n&\lesssim\S_n+\Q_n-C\epsilon\E_n\lesssim_\epsilon|\b|^{-m}\E_n(0)+C_\delta(\E_n+\Z_n^2)^{1/2}\E_n\\
&\lesssim|\b|^{-m}\E_n(0)+C_\delta\brac{\Z_n(0)\E_n(s)+(1+s^{1/2})\E_n(s)^{3/2}}
\end{align*}
Using Proposition \ref{bootstrapping scheme} above, we get $\E_n\lesssim\E_n(0)$ with $E_n(s)\lesssim e^{-C|\b|^ms}$ and $Z_n$ bounded on $[0,\infty)$.
\end{proof}


\section{Nonradial stability of linearly expanding Goldreich-Weber stars}\label{linear GW}

\subsection{Formulation and statement of the result}

\subsubsection{Equation in linearly-expanding coordinates}



In this section we will take the enthalpy $\bar w$ to be the profile associated with the linearly expanding GW star from Definition~\ref{linear GW def}. To study the linearly expanding GW stars, we want to write our variables as a perturbation from the model GW star. To that end we will use the rescaled variable $\bs\xi$ (equation \eqref{E:KSIDEF}) introduced in Section \ref{Notation} adapted to the expanding background profile and also write the problem in ``linear'' time variables. We introduce the ``linear'' time coordinate $s$ adapted to the expanding profile via
\[
\frac{ds}{dt}=\l(t)^{-1}.
\]
In this new coordinate, $\lambda(s)$ is an increasing function such that
\begin{align}\label{E:lambda in s - linear}
\lambda(s)\sim e^{s\sqrt{\lambda_1^2+2\delta}}\qquad\text{as}\qquad s\to\infty.
\end{align}
We have the following change of coordinate formula $\partial_t=\lambda^{-1}\partial_s$. The condition $\ddot{\lambda}\lambda^2=\delta$ \eqref{E:GW1} becomes
\begin{align}
\delta=\lambda\partial_s(\lambda^{-1}\partial_s\lambda)=\partial_s^2\lambda-{(\partial_s\lambda)^2\over\lambda}
\label{E:delta and b relation - linear}
\end{align}
Then the Euler-Poisson equations~\eqref{E:MOMLAGR} becomes
\begin{align*}
\mb 0
&=\partial_t\mb v+(f_0J_0)^{-1}\partial_k(A^k(f_0J_0)^{4/3}J^{-1/3})+A\grad\psi\\
&=\lambda^{-1}\partial_s(\lambda^{-1}\partial_s(\lambda\bs\xi))+\lambda^{-2}(f_0J_0)^{-1}\partial_k(\A^k(f_0J_0)^{4/3}\J^{-1/3})+\lambda^{-2}\A\grad\Phi
\end{align*}
Times the equation by $\lambda^2$ we get
\begin{align*}
\mb 0&=\lambda\partial_s(\lambda^{-1}\partial_s(\lambda\bs\xi))+(f_0J_0)^{-1}\partial_k(\A^k(f_0J_0)^{4/3}\J^{-1/3})+\A\grad\Phi\\
&=\lambda\brac{\partial_s^2\bs\xi+{\partial_s\lambda\over\lambda}\partial_s\bs\xi+\brac{{\partial_s^2\lambda\over\lambda}-{(\partial_s\lambda)^2\over\lambda^2}}\bs\xi}+(f_0J_0)^{-1}\partial_k(\A^k(f_0J_0)^{4/3}\J^{-1/3})
+\A\grad\Phi\\
&=\brac{\lambda\partial_s^2\bs\xi+(\partial_s\lambda)\partial_s\bs\xi+\delta\bs\xi}+(f_0J_0)^{-1}\partial_k(\A^k(f_0J_0)^{4/3}\J^{-1/3})+\A\grad\Phi
\end{align*}
So the Euler-Poisson equations in terms of $\bs\xi$~\eqref{E:KSIDEF} is:
\begin{align}\label{E:KSIEQUATION - linear}
\lambda\partial_s^2\bs\xi+\lambda'\partial_s\bs\xi+\delta\bs\xi+\frac1{f_0J_0}\partial_k(\A^k(f_0J_0)^{4/3}\J^{-1/3})+\A\grad\Phi =\mb 0,
\end{align}
where $\lambda':=\partial_s\lambda$.

The GW-star is a particular s-independent solution of~\eqref{E:KSIEQUATION - linear} of the form $\bs\xi(\mb x)\equiv \mb x$ and $f_0=\bar w^3$.
Before formulating the stability problem, we must first make the use of the labelling gauge freedom and fix the choice of the initial enthalpy $(f_0J_0)^{1/3}$ for the general perturbation to be exactly
identical to the background enthalpy $\bar w$, i.e. we set
\begin{align}\label{E:GAUGECHOICE - linear}
(f_0J_0)^{1/3} = \bar w \ \qquad \text{ on }\ B_R(\mb 0).
\end{align}
Equation~\eqref{E:GAUGECHOICE - linear} can be re-written in the form $\rho_0\circ\bs\eta_0 \det[\grad\bs\eta_0]=\bar w^3$ on the initial domain $B_R(\mb 0)$. By a result of 
Dacorogna-Moser~\cite{DM} and similarly to~\cite{HaJa2018-1,HaJa2016-2} there exists a choice of an initial bijective map $\bs\eta_0:B_R(\mb 0)\to\Omega(\mb 0)$ so that~\eqref{E:GAUGECHOICE - linear} holds true. The gauge fixing condition~\eqref{E:GAUGECHOICE - linear}
is necessary as it constraints the freedom to arbitrary relabel the particles at the initial time.
%



\begin{lemma}[Euler-Poisson in linearly-expanding coordinate]
With respect to the linearly expanding profile $(\lambda,\bar w)$ from Definition~\ref{linear GW def}, the perturbation $\bs\theta$ defined in~\eqref{E:THETADEF} formally solves 
\begin{align}\label{E:EP in linear}
\lambda\partial_s^2\bs\theta+\lambda'\partial_s\bs\theta+\delta\bs\theta+\mb P+\mb G=\mb 0,
\end{align}
where the nonlinear pressure operator $\mb P$ and the nonlinear gravity operator $\mb G$ are defined in \eqref{E:P} and \eqref{E:GDEF}.
\end{lemma}


\begin{proof}
Recall that the GW-enthalpy satisfies
\begin{align}
\mb 0=\delta\mb x+4\grad\bar w+\grad\K\bar w^3
\label{E:equation for bar-w - linear}
\end{align}
Using the gauge condition~\eqref{E:GAUGECHOICE - linear},
the momentum equation~\eqref{E:KSIEQUATION - linear} becomes
\begin{align*}
\bar w^3\brac{\lambda\partial_s^2\bs\theta+\lambda'\partial_s\bs\theta+\delta\bs\theta}+\partial_k(\bar w^4(\A^k\J^{-1/3}-I^k))+\bar w^3(\A\grad\Phi-\grad\mathcal{K}\bar w^3)=\mb 0.
\end{align*}
Hence, we can write the momentum equation as
\begin{align*}
\mb 0 & =\lambda\partial_s^2\bs\theta+\lambda'\partial_s\bs\theta+\delta\bs\theta+\underbrace{\bar w^{-3}\partial_k(\bar w^4(\A^k\J^{-1/3}-I^k))}_{\mb P}+\underbrace{\A\grad\Phi-\mathcal{K}\grad\bar w^3}_{\mb G}.
\end{align*}
\end{proof}


\subsubsection{High-order energies and the main theorem}



We now introduce high-order weighted Sobolev norm that we will use for our high-order energy method explained in Section~\ref{S:EE2 - linear}. Recall the notation in Section \ref{Notation}. Assuming that $(s,{\bf y})\mapsto \bs\theta(s,\mb y)$ is a sufficiently smooth field, for any $n\in\mathbb N_0$ we let
\begin{align*}
S_n(s)&:=\sum_{\substack{|\beta|+b\leq n}}\brac{\lambda\|\pr^b\pt^{\beta}\partial_s\bs\theta\|_{3+b}^2+\|\pr^b\pt^{\beta}\bs\theta\|_{3+b}^2+\|\grad\pr^b\pt^{\beta}\bs\theta\|_{4+b}^2}\\
Q_n(s)&:=\sum_{c\leq n}\brac{\lambda\|\grad^c\partial_s\bs\theta\|_{3+2c}^2+\|\grad^c\bs\theta\|_{3+2c}^2+\|\grad^{c+1}\bs\theta\|_{4+2c}^2}\\
Z_n(s)&:=\sum_{\substack{|\beta|+b=n}}\lambda\|\pr^b\pt^{\beta}(\A\grad\times\partial_s\bs\theta)\|_{4+b}^2+\lambda\|\grad^n(\A\grad\times\partial_s\bs\theta)\|_{4+2n}^2
\end{align*}
We define the total instant energy via
\begin{align}\label{E:ENDEF - linear}
E_n:=S_n+Q_n+Z_n.
\end{align}
We shall run the energy identity using $E_n$; $Z_n$ controls the curl of the velocity, while the energies $S_n$ and $Q_n$ will be used  for high-order estimates near the vacuum boundary and near the origin respectively.  
In particular, the control afforded by $Q_n$ is stronger near the origin, while $S_n$ is stronger near the boundary. Finally we define
\begin{align}
\S_n(s)&:=\sup_{\tau\in[0,s]}S_n(\tau),\label{E:SBULLET - linear}\\
\Q_n(s)&:=\sup_{\tau\in[0,s]}Q_n(\tau), \label{E:QBULLET - linear}\\
\E_n(s)&:=\sup_{\tau\in[0,s]}E_n(\tau), \label{E:TOTALNORM - linear}
\end{align}
The norms~\eqref{E:SBULLET - linear}--\eqref{E:TOTALNORM - linear} will play the role of the ``left hand side'' in the high-order energy identities.

In this section, we make the following a priori assumption:
\begin{flalign}
\text{\noindent\textbf{A priori assumption:} $E_n\leq\epsilon$ where $\epsilon>0$ is some small constant.} && \label{A priori assumption - linear}
\end{flalign}

We now state our main theorem.


\begin{theorem}[Nonlinear stability of GW stars]\label{T:MAIN - linear}
Let $n\geq 21$. The linearly expanding GW star from Definition~\ref{linear GW def} nonlinearly stable. More precisely, there exists an $\epsilon^*>0$ such that for any initial data $(\bs\theta(0),\partial_s\bs\theta(0))$ satisfying
\begin{align}
E_n(0)&\leq\epsilon^*,
\end{align}
the associated solution $s\mapsto( \bs\theta(s,\ph),  \partial_s\bs\theta(s,\ph) )$ to \eqref{E:EP in linear} exists for all $s\geq 0$ and is unique in the class of all data with finite norm $E_n$. Moreover, there exists a constant $C>0$ such that
\begin{align*}
E_n(s)\leq C\epsilon^*\qquad\text{for all}\qquad s\geq 0.
\end{align*}
\end{theorem}


\begin{remark}
Like in the last section (cf. Remark \ref{Our goal is not to optimise}), it is not our goal to optimise the number $n$ of derivatives in our spaces. 
\end{remark}


{\bf Local-in-time well-posedness.}
The same process as described in section \ref{sec:2.1.3} for the self-similarly expanding GW star can be use to obtain the equivalent well-posedness result in the weighted high-order energy space $E_n$ defined in the current section for the linearly expanding GW star.


\begin{theorem}[Local well-posedness]\label{T:LOCAL - linear}
Let $n\geq 21$. Then for any given initial data $(\bs\theta(0), \partial_s \bs\theta(0))$ such that $E_n(0)<\infty$, there exist some $T>0$ and a unique solution $(\bs\theta, \partial_s\bs\theta):[0,T]\times B_R\to\R^3\times \R^3$ to \eqref{E:EP in linear} such that $E_n(s)\leq 2E_n(0)$ for all $s\in[0,T]$.
\end{theorem}


Theorem~\ref{T:LOCAL - linear} is a starting point for the continuity argument that will culminate in the proof of Theorem~\ref{T:MAIN - linear}.

\subsubsection{Proof strategy}


The basic idea behind the global existence in Theorem~\ref{T:MAIN - linear} is similar to that of the self-similarly expanding GW case in the last section. In fact, it is more straightforward here than in the last section owing to the fact the linearly expanding GW star expands at a faster rate that the self-similarly expanding GW star, and hence there is a stronger dispersion effect. 

In particular, we have the exponentially increasing $\lambda(s)$ factor in the first term in \eqref{E:EP in linear}. This leads to a $\lambda$-factor in front of the ``velocity'' terms in the higher order energy in \eqref{E:ENDEF - linear}. This gives terms on the velocity level (terms with at least one time derivatives) an extra decay that effectively make it subleading order on par with the non-linear term and hence negligible in the dynamics.

This in particular renders the effect of gravity in the dynamics to be secondary:
\[\int_{s_0}^s\<\pr^b\pt^\beta\mb G,\partial_s\pr^b\pt^\beta\bs\theta\>_{3+b}\d\tau\lesssim\E_n(s)\int_{s_0}^s\lambda^{-1/2}\d\tau.\]
Here $\E_n(s)\int_{s_0}^s\lambda^{-1/2}\d\tau$ has and effect similar to the non-linear term $\E^{3/2}$, see Proposition \ref{bootstrapping scheme - linear}  and Theorem \ref{Linear final theorem}.

This leads to a key simplification - we do not need a precise coercivity result like in the self-similarly expanding GW case for the operator $\mb L$ in Section~\ref{self-similar GW}.  In particular we do not need to make the assumption that the fluid is irrotational in this case -- we allow nontrivial vorticity initially and control its time evolution by the curl estimates (Section \ref{linear Vorticity estimates}), similar to~\cite{HaJa2016-2}. 
Note also that linear motion is secondary (bounded) in a linearly expanding coordinate. So a non-zero momentum in the initial data, which in theory should make the overall GW star to travel at constant speed in the direction of the momentum, is automatically encapsulated by the linear expanding coordinate about a linearly expanding GW star centred at the origin.

Many terms that appeared on the primary ``linear level'' in self-similarly expanding case of the last section are now not at leading order any more. As a result, higher time derivatives can be avoided in our higher order energy in \eqref{E:ENDEF - linear}, and we do not need the sophisticated triple induction scheme on the higher order energies that we had to carry out in the proof of Theorem~\ref{T:MAIN}.

\subsection{Pressure estimates}\label{Estimating the non-linear part of the pressure term - linear}

In this section we will estimate the non-linear part of the pressure term $\pr^b\pt^{\beta}\mb P$ and $\partial^\gamma\mb P$ \eqref{E:EP in linear}. More precisely, when doing energy estimates, terms like $\<\pr^b\pt^\beta\mb P,\partial_s\pr^b\pt^\beta\bs\theta\>$ will arise, we will show that $\mb P$ here can be reduced to $\mb P_{d,L}$ modulo remainder terms that can be estimated. We will use results from section \ref{Pertaining the pressure term}.

Using Lemma \ref{P-inner-product}, we will now estimate the difference between ``$\mb P_{b}$'' and ``$\mb P_{b,L}$''.

\begin{proposition}\label{P-reduction-2 - linear}
Let $n\geq 20$. Let $|\beta|+b\leq n$ and $|\gamma|\leq n$. Suppose $\bs\theta$ satisfies our a priori assumption \eqref{A priori assumption - linear}. Then, for any $0\leq s_0\leq s$, we have
\begin{align*}
\abs{\int_{s_0}^s\<\mb P_b\pr^b\pt^\beta\bs\theta,\partial_s\pr^b\pt^\beta\bs\theta\>_{3+b}\d\tau-{1\over 2}\eva{\<\mb P_{b,L}\pr^b\pt^\beta\bs\theta,\pr^b\pt^\beta\bs\theta\>_{3+b}}_{s_0}^s}&\lesssim \E_n(s)^{3/2}\\
\abs{\int_{s_0}^s\<\mb P_{2|\gamma|}\partial^\gamma\bs\theta,\partial_s\partial^\gamma\bs\theta\>_{3+2|\gamma|}\d\tau-{1\over 2}\eva{\<\mb P_{2|\gamma|,L}\partial^\gamma\bs\theta,\partial^\gamma\bs\theta\>_{3+2|\gamma|}}_{s_0}^s}&\lesssim \E_n(s)^{3/2}
\end{align*}
\end{proposition}
\begin{proof}
The proof is similar to Proposition \ref{P-reduction-2} except the reminder term here will be $\partial_s\Rd[E_n^{3/2}]+\Rd[\lambda^{-1/2}E_n^{3/2}]$. Since $\int_0^\infty\lambda^{-1/2}\d s<\infty$, integrating in time we get the first equation. Proof for the second formula is similar.
\end{proof}

And now we will estimate the difference between ``$\mb P$'' and ``$\mb P_{b}$''.

\begin{proposition}\label{P-reduction-1 - linear}
Let $n\geq 20$ and $|\beta|+b\leq n$. Suppose $\bs\theta$ satisfies our a priori assumption \eqref{A priori assumption - linear}. Then, for any $0\leq s_0\leq s$, we have
\begin{align*}
\abs{\int_{s_0}^s\<\pr^b\pt^\beta\mb P-\mb P_b\pr^b\pt^\beta\bs\theta,\partial_s\pr^b\pt^\beta\bs\theta\>_{3+b}\d\tau}&\lesssim\int_{s_0}^s\lambda^{-{1\over 2}}E_n\d\tau\\
\abs{\int_{s_0}^s\<\partial^\gamma\mb P-\mb P_{2|\gamma|}\partial^\gamma\bs\theta,\partial_s\partial^\gamma\bs\theta\>_{3+2|\gamma|}\d\tau}&\lesssim\int_{s_0}^s\lambda^{-{1\over 2}}E_n\d\tau.
\end{align*}
\end{proposition}
\begin{proof}
By Lemma \ref{pressure-outer-commutator} we need to estimate the following.
\begin{align*}
&\abs{\int_{s_0}^s\int_{B_R}\brac{\sum_{\substack{b'\leq b\\|\beta'|\leq|\beta|-1}}+\sum_{\substack{b'\leq b-1\\|\beta'|\leq|\beta|+1}}}\<C\omega\pr^{b'}\pt^{\beta'}T\>\<\partial_s\pr^b\pt^\beta\bs\theta\>\bar w^{3+b}\d\mb x\d\tau}\\
&+\abs{\int_{s_0}^s\int_{B_R}\brac{\sum_{\substack{b'\leq b-1\\|\beta'|\leq|\beta|-1}}+\sum_{\substack{b'\leq b-2\\|\beta'|\leq|\beta|}}}\<C\pr^{b'}\pt^{\beta'}\grad T\>\<\partial_s\pr^b\pt^\beta\bs\theta\>\bar w^{3+b}\d\mb x\d\tau}\\
&+\abs{\int_{s_0}^s\int_{B_R}\brac{\sum_{\substack{b'\leq b\\|\beta'|\leq|\beta|-1}}+\sum_{\substack{b'\leq b-1\\|\beta'|\leq|\beta|}}}\<C\bar w\pr^{b'}\pt^{\beta'}\grad T\>\<\partial_s\pr^b\pt^\beta\bs\theta\>\bar w^{3+b}\d\mb x\d\tau}\\
&\lesssim\int_{s_0}^s\lambda^{-{1\over 2}}E_n\d\tau
\end{align*}
where the $\lambda^{-1/2}$ factor comes from estimating $\|\partial_s\pr^b\pt^\beta\bs\theta\|_{3+b}\lesssim\lambda^{-1/2}E_n^{1/2}$. The terms with $T$ can be estimated noting the structure given in Lemma \ref{pressure structure lemma}. This proves the first formula. The proof for the second formula is similar.
\end{proof}


\subsection{Gravity estimates}
\label{Estimating the linear and non-linear part of the gravity term - linear}

In this subsection we will estimate the gravity term $\pr^b\pt^{\beta}\mb G$ and $\partial^\gamma\mb G$ \eqref{E:EP in linear} and show that it can be bounded by $E_n$. We will use results from Section \ref{Pertaining the gravity term}.

Since the gravity term is a non-local term, we need to estimate convolution-like operator. However, rather than the convolution kernel $|\mb x-\mb z|^{-1}$ we actually need to estimate $|\bs\xi(\mb x)-\bs\xi(\mb z)|^{-1}$. Lemma \ref{distance estimate} and the following lemma tell us how to reduce the latter to the former, which will allows us to estimate using the Young's convolution inequality.

\begin{lemma}\label{lemma for kernel - linear}
Let $\bs\xi$ and $\bs\theta$ be as in \eqref{E:THETADEF}, and $\bs\theta$ satisfies our a priori assumption \eqref{A priori assumption - linear}. Let $n\geq 21$ and $|\beta|\leq n$.
\begin{enumerate}
\item When $|\beta|>n/2$ we have
\begin{align*}
\abs{(\pt_{\mb x}+\pt_{\mb z})^{\beta}\brac{1\over|\bs\xi(\mb x)-\bs\xi(\mb z)|}}&\lesssim{1\over|\mb x-\mb z|^2}\sum_{n/2<|\gamma|\leq n}|\pt_{\mb x}^\gamma\bs\xi(\mb x)-\pt_{\mb z}^\gamma\bs\xi(\mb z)|
\end{align*}
\item When $|\beta|\leq n/2$ we have
\begin{align*}
\abs{\partial_{i,\mb z}(\pt_{\mb x}+\pt_{\mb z})^{\beta}\brac{1\over|\bs\xi(\mb x)-\bs\xi(\mb z)|}}&\lesssim{1\over|\mb x-\mb z|^2}
\end{align*}
\end{enumerate}
\end{lemma}
\begin{proof}
The proof proceed like Lemma \ref{lemma for kernel} but using instead the embedding theorems \ref{Near boundary embedding theorem - linear} and \ref{Near origin embedding theorem - linear} and the a priori bounds $E_n\lesssim 1$ \eqref{A priori assumption - linear}.
\end{proof}

We next derive an helpful lemma for derivatives on $\K_{\bs\xi}-\K$. Recall $K_1$ from \eqref{E:K1DEF} with
\begin{align}
(\K_{\bs\xi}-\mathcal{K})g(\mb x)&=-\int_{\R^3}K_1(\mb x,\mb z)g(\mb z)\d\mb z,
\end{align}

\begin{lemma}\label{K_1 lemma - linear}
Let $n\geq 20$ and $|\beta|\leq n$. We have
\begin{align*}
|(\pt_{\mb x}+\pt_{\mb z})^\beta K_1(\mb x,\mb z)|&\lesssim{1\over|\mb x-\mb z|^2}\sum_{\beta'\leq\beta}|\pt^{\beta'}\bs\theta(\mb x)-\pt^{\beta'}\bs\theta(\mb z)|
\end{align*}
\end{lemma}
\begin{proof}
The proof proceed like Lemma \ref{K_2 lemma} but using instead the embedding theorems \ref{Near boundary embedding theorem - linear} and \ref{Near origin embedding theorem - linear} and the a priori bounds $E_n\lesssim 1$ \eqref{A priori assumption - linear}.
\end{proof}

Finally we can prove the main results of this subsection.

\begin{proposition}[Gravity estimates]\label{G estimate - linear}
Let $n\geq 21$ and suppose $\bs\theta$ satisfies our a priori assumption \eqref{A priori assumption - linear}. Then we have
\begin{alignat*}{4}
\|\pr^b\pt^\beta\mb G\|_{3+b}^2&\lesssim E_n\qquad &\text{when}&&\qquad |\beta|+b&\leq n,\\
\|\bar w^{b/2}\pr^b\pt^\beta\mb G\|_{L^\infty(\R^3)}^2&\lesssim E_n\qquad &\text{when}&&\qquad |\beta|+b&\leq n/2,\\
\|\grad^c\mb G\|_{3+2c}^2&\lesssim E_n\qquad &\text{when}&&\qquad c&\leq n,\\
\|\bar w^{c/2}\grad^c\mb G\|_{L^\infty(\R^3)}^2&\lesssim E_n\qquad &\text{when}&&\qquad c&\leq n/2.
\end{alignat*}
\end{proposition}
\begin{proof}
By definition 
\begin{align*}
\mb G&=\K_{\bs\xi}\grad\cdot(\A\bar w^3)-\mathcal{K}\grad\bar w^3
=-\int_{\R^3}{\partial_k(\A^k\bar w^3)\over|\bs\xi(\mb x)-\bs\xi(\mb z)|}\d\mb z+\int_{\R^3}{\grad\bar w^3\over|\mb x-\mb z|}\d\mb z\\
&=-\int_{\R^3}{\partial_k((\A^k-I^k)\bar w^3)\over|\bs\xi(\mb x)-\bs\xi(\mb z)|}\d\mb z-\int_{\R^3}K_1(\mb x,\mb z)\grad\bar w^3\d\mb z
\end{align*}
By Lemma~\ref{L:ENERGYLEMMA1} we have
\begin{align*}
\pt^\beta\mb G(\mb x)&
=-\pt^\beta\brac{\int_{\R^3}{\partial_k((\A^k-I^k)\bar w^3)\over|\bs\xi(\mb x)-\bs\xi(\mb z)|}\d\mb z+\int_{\R^3}K_1(\mb x,\mb z)\grad\bar w^3\d\mb z}\\
&=-\int_{\R^3}\sum_{\substack{\beta_1+\beta_2=\beta}}(\pt_{\mb x}+\pt_{\mb z})^{\beta_1}\brac{1\over|\bs\xi(\mb x)-\bs\xi(\mb z)|}\pt_{\mb z}^{\beta_2}\partial_k((\A^k-I^k)\bar w^3)(\mb z)\d\mb z\\
&\quad-\int_{\R^3}\sum_{\substack{\beta_1+\beta_2=\beta}}(\pt_{\mb x}+\pt_{\mb z})^{\beta_1}K_1(\mb x,\mb z)(\pt_{\mb z}^{\beta_2}\grad\bar w^3)(\mb z)\d\mb z\\
&=-\int_{\R^3}\sum_{\substack{\beta_1+\beta_2=\beta\\|\beta_1|>n/2}}(\pt_{\mb x}+\pt_{\mb z})^{\beta_1}\brac{1\over|\bs\xi(\mb x)-\bs\xi(\mb z)|}\pt_{\mb z}^{\beta_2}\partial_k((\A^k-I^k)\bar w^3)(\mb z)\d\mb z\\
&\quad-\int_{\R^3}\sum_{\substack{\beta_1+\beta_2=\beta\\|\beta_1|\leq n/2}}(\pt_{\mb x}+\pt_{\mb z})^{\beta_1}\brac{1\over|\bs\xi(\mb x)-\bs\xi(\mb z)|}\pt_{\mb z}^{\beta_2}\partial_k((\A^k-I^k)\bar w^3)(\mb z)\d\mb z\\
&\quad-\int_{\R^3}\sum_{\substack{\beta_1+\beta_2=\beta}}(\pt_{\mb x}+\pt_{\mb z})^{\beta_1}K_1(\mb x,\mb z)(\pt_{\mb z}^{\beta_2}\grad\bar w^3)(\mb z)\d\mb z\\
&=-\int_{\R^3}\sum_{\substack{\beta_1+\beta_2=\beta\\|\beta_1|>n/2}}(\pt_{\mb x}+\pt_{\mb z})^{\beta_1}\brac{1\over|\bs\xi(\mb x)-\bs\xi(\mb z)|}\pt_{\mb z}^{\beta_2}\partial_k((\A^k-I^k)\bar w^3)(\mb z)\d\mb z\\
&\quad+\int_{\R^3}\sum_{\substack{\beta_1+\beta_2\leq\beta\\|\beta_1|\leq n/2}}\<\grad_{\mb z}\>(\pt_{\mb x}+\pt_{\mb z})^{\beta_1}\brac{1\over|\bs\xi(\mb x)-\bs\xi(\mb z)|}(\<\pt_{\mb z}^{\beta_2}(\A-I)\>\bar w^3)(\mb z)\d\mb z\\
&\quad-\int_{\R^3}\sum_{\substack{\beta_1+\beta_2=\beta}}(\pt_{\mb x}+\pt_{\mb z})^{\beta_1}K_1(\mb x,\mb z)(\pt_{\mb z}^{\beta_2}\grad\bar w^3)(\mb z)\d\mb z\\
\end{align*}
Now using Lemma \ref{lemma for kernel - linear} and \ref{K_1 lemma - linear}, we get
\begin{align*}
|\partial_s^a\pt^\beta\mb G(\mb x)|
&\lesssim\int_{\R^3}{E_n^{1/2}\over|\mb x-\mb z|^2}\sum_{n/2<|\gamma|\leq n}|\pt_{\mb x}^\gamma\bs\xi(\mb x)-\pt_{\mb z}^\gamma\bs\xi(\mb z)|\bar w^2\d\mb z\\
&\quad+\int_{\R^3}\sum_{\substack{\beta_2\leq\beta}}{1\over|\mb x-\mb z|^2}(\<\pt_{\mb z}^{\beta_2}(\A-I)\>\bar w^3)(\mb z)\d\mb z\\
&\quad+\int_{\R^3}\sum_{\substack{\beta_1+\beta_2=\beta}}\sum_{\beta'\leq\beta_1}{|\pt^{\beta'}\bs\theta(\mb x)-\pt^{\beta'}\bs\theta(\mb z)|\over|\mb x-\mb z|^2}(\pt_{\mb z}^{\beta_2}\grad\bar w^3)(\mb z)\d\mb z
\end{align*}
Now using Young's convolution inequality we get
\begin{align*}
\|\pt^\beta\mb G(\mb x)\|_{L^2(\R^3)}\lesssim E_n^{1/2}.
\end{align*}
Hence $\|\pt^\beta\mb G\|_{3}^2\lesssim E_n$. From the above proof, with small modification, we can further see that
\begin{alignat*}{4}
\|\pt^\beta\mb G(\mb x)\|_{L^\infty(\R^3)}&\lesssim E_n^{1/2}\qquad &\text{when}&&\qquad |\beta|&\leq n/2
\end{alignat*}

Now we deal with the case $b>0$. Let
\begin{align*}
W_{n}&=\sum_{\substack{|\beta|+b\leq n}}\|\pr^b\pt^\beta\mb G\|_{3+b}^2\\
W_{n,d}&=\sum_{\substack{|\beta|+b\leq n\\b\leq d}}\|\pr^b\pt^\beta\mb G\|_{3+b}^2\\
V_{n}&=\sum_{\substack{|\beta|+b\leq n}}\sup_{\R^3}\brac{\bar w^b|\pr^b\pt^\beta\mb G|^2}\\
V_{n,d}&=\sum_{\substack{|\beta|+b\leq n\\b\leq d}}\sup_{\R^3}\brac{\bar w^b|\pr^b\pt^\beta\mb G|^2}.
\end{align*}
For $|\beta|+b\leq n/2$, using the above lemmas \ref{Breakdown for X_r} and \ref{div and curl of G} we have
\begin{align*}
\bar w^b|\pr^b\pt^\beta\mb G|^2
&\lesssim\bar w^b|r\grad\cdot\pr^{b-1}\pt^\beta\mb G|^2+\bar w^b|r\grad\times\pr^{b-1}\pt^\beta\mb G|^2+\sum_{k=1}^3\bar w^b|\pr^{b-1}\pt_k\pt^\beta\mb G|^2\\
&\lesssim\bar w^b|r\pr^{b-1}\pt^\beta\grad\cdot\mb G|^2+\bar w^b|r\pr^{b-1}\pt^\beta\grad\times\mb G|^2+V_{b+|\beta|-1}+V_{b+|\beta|,b-1}\\
&\lesssim\bar w^b|r\pr^{b-1}\pt^\beta((I-\A)\grad\cdot\mb G)|^2+\bar w^b|r\pr^{b-1}\pt^\beta((I-\A)\grad\times\mb G)|^2\\
&\quad+E_n+V_{b+|\beta|-1}+V_{b+|\beta|,b-1}\\
&\lesssim\bar w^b|r(I-\A)\pr^{b-1}\pt^\beta\grad\cdot\mb G|^2+\bar w^b|r(I-\A)\pr^{b-1}\pt^\beta\grad\times\mb G|^2\\
&\quad+E_n+V_{b+|\beta|-1}+V_{b+|\beta|,b-1}\\
&\lesssim\bar w^bE_n|r\pr^{b-1}\pt^\beta\grad\mb G|^2+E_n+V_{b+|\beta|-1}+V_{b+|\beta|,b-1}
\end{align*}
So
\begin{align*}
V_{b+|\beta|,b}\lesssim E_nV_{b+|\beta|,b}+E_n+V_{b+|\beta|-1}+V_{b+|\beta|,b-1}
\end{align*}
By a priori assumption \eqref{A priori assumption - linear}, we have $E_n\ll 1$, so
\begin{align*}
V_{b+|\beta|,b}\lesssim E_n+V_{b+|\beta|-1}+V_{b+|\beta|,b-1}.
\end{align*}
We know $V_{n',0}\lesssim E_n$ for all $n'\leq n/2$, so by induction we get $V_{n',d}\lesssim E_n$ for all $d\leq n'\leq n/2$.

Now for $|\beta|+b\leq n$, using the above lemmas \ref{Breakdown for X_r} and \ref{div and curl of G} and results for $V$ we have
\begin{align*}
\|\pr^b\pt^\beta\mb G\|_{3+b}^2
&\lesssim\|r\grad\cdot\pr^{b-1}\pt^\beta\mb G\|_{3+b}^2+\|r\grad\times\pr^{b-1}\pt^\beta\mb G\|_{3+b}^2+\sum_{k=1}^3\|\pr^{b-1}\pt_k\pt^\beta\mb G\|_{3+b}^2\\
&\lesssim\|r\pr^{b-1}\pt^\beta\grad\cdot\mb G\|_{3+b}^2+\|r\pr^{b-1}\pt^\beta\grad\times\mb G\|_{3+b}^2+W_{b+|\beta|-1}+W_{b+|\beta|,b-1}\\
&\lesssim\|r\pr^{b-1}\pt^\beta((1-\A)\grad\cdot\mb G)\|_{3+b}^2+\|r\pr^{b-1}\pt^\beta((1-\A)\grad\times\mb G)\|_{3+b}^2\\
&\quad +E_n+W_{b+|\beta|-1}+W_{b+|\beta|,b-1}\\
&\lesssim\|r(1-\A)\pr^{b-1}\pt^\beta\grad\cdot\mb G\|_{3+b}^2+\|r(1-\A)\pr^{b-1}\pt^\beta\grad\times\mb G\|_{3+b}^2\\
&\quad +E_n+W_{b+|\beta|-1}+W_{b+|\beta|,b-1}\\
&\lesssim E_n\|r\pr^{b-1}\pt^\beta\grad\mb G\|_{3+b}^2+E_n+W_{b+|\beta|-1}+W_{b+|\beta|,b-1}
\end{align*}
So
\begin{align*}
W_{b+|\beta|,b}\lesssim E_nW_{b+|\beta|,b}+E_n+W_{b+|\beta|-1}+W_{b+|\beta|,b-1}
\end{align*}
By a priori assumption \eqref{A priori assumption - linear}, we have $E_n\ll 1$, so
\begin{align*}
W_{b+|\beta|,b}\lesssim E_n+W_{b+|\beta|-1}+W_{b+|\beta|,b-1}.
\end{align*}
We know $W_{n,0}\lesssim E_n$, so by induction we get $W_{n,d}\lesssim E_n$ for all $d\leq n$.

Let
\begin{align*}
Y_n&=\sum_{c\leq n}\|\grad^c\mb G\|_{3+2c}^2
\end{align*}
By Sobolev embeddings like those used to prove the embedding theorems \ref{Near boundary embedding theorem - linear} and \ref{Near origin embedding theorem - linear}, we have that
\begin{align*}
\|\bar w^{c/2}\grad^c\mb G\|_{L^\infty(\R^3)}^2\lesssim E_n+Y_n\qquad \text{when}\qquad c\leq n/2.
\end{align*}
Now for $c\leq n$, using the above lemmas \ref{Hodge bound} and \ref{div and curl of G} we have
\begin{align*}
\|\grad^c\mb G\|_{3+2c}^2
&\lesssim\|\grad^{c-1}\grad\cdot\mb G\|_{3+2c}^2+\|\grad^{c-1}\grad\times\mb G\|_{3+2c}^2+\|\grad^{c-1}\mb G\|_{3+2(c-1)}^2\\
&\lesssim\|\grad^{c-1}((1-\A)\grad\cdot\mb G)\|_{3+2c}^2+\|\grad^{c-1}((1-\A)\grad\times\mb G)\|_{3+2c}^2+E_n+Y_{c-1}\\
&\lesssim\|(1-\A)\grad^{c-1}\grad\cdot\mb G\|_{3+2c}^2+\|(1-\A)\grad^{c-1}\grad\times\mb G\|_{3+2c}^2+E_n(1+Y_n)+Y_{c-1}\\
&\lesssim E_n\|\grad^{c-1}\grad\mb G\|_{3+2c}^2+E_n(1+Y_n)+Y_{c-1}
\end{align*}
So
\begin{align*}
Y_c\lesssim E_nY_n+E_n+Y_{c-1}
\end{align*}
By induction on $c\leq n$ we have
\begin{align*}
Y_n\lesssim E_nY_n+E_n
\end{align*}
By a priori assumption \eqref{A priori assumption - linear}, we have $E_n\ll 1$, so we get $Y_n\lesssim E_n$.
\end{proof}


\subsection{Vorticity estimates}\label{linear Vorticity estimates}


In this section we will estimate the vorticity $Z_n$, $\|\grad\times\pr^b\pt^{\beta}\bs\theta\|_{4+b}$ and $\|\grad\times\partial^\gamma\bs\theta\|_{4+2|\gamma|}$ which will be needed to control the ``curl'' part of the pressure term as seem in Lemma \ref{P-inner-product}.

By taking curl to the Euler-Poisson equation \eqref{E:EP in linear} we can essentially get rid of the pressure and gravity terms, which allows us to estimate the curl separately.

\begin{lemma}\label{Vorticity lemma - linear}
Let $\bs\theta$ be a solution of~\eqref{E:EP in linear} in the sense of Theorem \ref{T:LOCAL - linear}. Then for any $s\geq s_0\geq 0$ we have
\begin{align*}
\lambda(s)^{1\over 2}(\A\grad\times\partial_s\bs\theta)(s)&={\lambda(s_0)\over\lambda(s)^{1/2}}(\A\grad\times\partial_s\bs\theta)(s_0)
+\lambda(s)^{-{1\over 2}}\int_{s_0}^s(\partial_s\A)\grad\times\partial_s\bs\theta\d s'\\
(\A\grad\times\bs\theta)(s)&=(\A\grad\times\bs\theta)(s_0)+\int_{s_0}^s(\partial_s\A)\grad\times\bs\theta\;\d s'
+\int_{s_0}^s{\lambda(s_0)\over\lambda(s)}\d s'(\A\grad\times\partial_s\bs\theta)(s_0)\\
&\quad+\int_{s_0}^s\lambda(s')^{-1}\int_{s_0}^{s'}(\partial_s\A)\grad\times\partial_s\bs\theta\d s''\d s'.
\end{align*}
\end{lemma}
\begin{proof}
Recall \eqref{E:KSIEQUATION - linear} is
\begin{align}
\mb 0=\lambda\partial_s^2\bs\xi+\lambda'\partial_s\bs\xi+\delta\bs\xi+\frac1{w^3}\partial_k(\A^kw^4\J^{-1/3})+\A\grad\Phi
\end{align}
Now note that
\begin{align*}
\frac1{w^3}\partial_k(\A^kw^4\J^{-1/3})
&=\bar w^{-3}\partial_k(\bar w^4\a^k\J^{-4/3})
=\bar w^{-3}\a^k\partial_k(\bar w^4\J^{-4/3}) \\
&=(\bar w\J^{-1/3})^{-1}\A^k\partial_k(\bar w^4\J^{-4/3})
={4\over 3}\A\grad(\bar w^3\J^{-1})
\end{align*}
and $\A\grad\times\bs\xi=\epsilon_{\bullet jk}\A_j^l\partial_l\xi_k=\epsilon_{\bullet jk}\delta_{jk}=\mb 0$. 
So taking $\A\grad\times$ to \eqref{E:KSIEQUATION - linear} we get
\begin{align*}
\mb 0&=\lambda\A\grad\times\partial_s^2\bs\theta+\lambda'\A\grad\times\partial_s\bs\theta
=\lambda\brac{\partial_s(\A\grad\times\partial_s\bs\theta)-(\partial_s\A)\grad\times\partial_s\bs\theta}+\lambda'\A\grad\times\partial_s\bs\theta\\
&=\partial_s(\lambda\A\grad\times\partial_s\bs\theta)-(\partial_s\A)\grad\times\partial_s\bs\theta
\end{align*}
So
\begin{align*}
&\lambda(s)(\A\grad\times\partial_s\bs\theta)(s)
=\lambda(s_0)(\A\grad\times\partial_s\bs\theta)(s_0)+\int_{s_0}^s(\partial_s\A)\grad\times\partial_s\bs\theta\d s'
\end{align*}
So
\begin{align*}
\lambda\partial_s(\A\grad\times\bs\theta)(s)&=\lambda(\partial_s\A)\grad\times\bs\theta(s)+\lambda(s_0)(\A\grad\times\partial_s\bs\theta)(s_0)
+\int_{s_0}^s(\partial_s\A)\grad\times\partial_s\bs\theta\d s'
\end{align*}
So
\begin{align*}
(\A\grad\times\bs\theta)(s)&=(\A\grad\times\bs\theta)(s_0)+\int_{s_0}^s(\partial_s\A)\grad\times\bs\theta\;\d s'
+\int_{s_0}^s{\lambda(s_0)\over\lambda(s)}\d s'(\A\grad\times\partial_s\bs\theta)(s_0)\\
&\quad+\int_{s_0}^s\lambda(s')^{-1}\int_{s_0}^{s'}(\partial_s\A)\grad\times\partial_s\bs\theta\d s''\d s'.
\end{align*}
\end{proof}

\begin{proposition}\label{Vorticity prop 1 - linear}
Let $\bs\theta$ be a solution of~\eqref{E:EP in linear} in the sense of Theorem \ref{T:LOCAL - linear}. Let $n\geq 21$. Then for any $s\geq s_0\geq 0$, we have
\begin{align*}
Z_n(s)\lesssim Z_n(s_0)+{(s-s_0)^2\over\lambda(s)}\E_n(s)+\lambda(s)^{-1}\E_n(s)^2.
\end{align*}
\end{proposition}
\begin{proof}
Take $\pr^b\pt^{\beta}$ ($b+\beta=n$) to the first equation in Lemma \ref{Vorticity lemma - linear} we get
\begin{align*}
\lambda(s)^{1\over 2}\pr^b\pt^{\beta}(\A\grad\times\partial_s\bs\theta)(s)&={\lambda(s_0)\over\lambda(s)^{1/2}}\pr^b\pt^{\beta}(\A\grad\times\partial_s\bs\theta)(s_0)
+\lambda(s)^{-{1\over 2}}\pr^b\pt^{\beta}\int_{s_0}^s(\partial_s\A)\grad\times\partial_s\bs\theta\d s'
\end{align*}
Since $\lambda$ is an increasing function \eqref{E:lambda in s - linear}, we have $\lambda(s_0)^{1/2}/\lambda(s)^{1/2}\leq q$ which we can used to bound the first term on the RHS. And for the other terms we can estimate for example
\begin{align*}
&\int_{B_R}\abs{\lambda(s)^{-{1\over 2}}\int_{s_0}^s(\pr^b\pt^{\beta}\partial_s\A)\grad\times\partial_s\bs\theta\;\d s'}^2\bar w^{4+b}\d\mb x\\
&\leq\lambda(s)^{-1}\int_{B_R}\abs{\int_{s_0}^s(\pr^b\pt^{\beta}\partial_s\A)\grad\times\partial_s\bs\theta\;\d s'}^2\bar w^{4+b}\d\mb x\\
&\lesssim\lambda(s)^{-1}\int_{B_R}\abs{\int_{s_0}^s(\pr^b\pt^{\beta}\A)\grad\times\partial_s^2\bs\theta\;\d s'}^2\bar w^{4+b}\d\mb x+\lambda(s)^{-1}\E_n(s)^2\\
&\lesssim{(s-s_0)^2\over\lambda(s)}\E_n(s)+\lambda(s)^{-1}\E_n(s)^2
\end{align*}
where we used $\lambda\partial_s^2\bs\theta=-\lambda'\partial_s\bs\theta-\delta\bs\theta-\mb P-\mb G$ \eqref{E:EP in linear}. 

We can then repeat this process for $\partial^\gamma$ ($|\gamma|=n$) in place of $\pr^b\pt^{\beta}$ to complete the proof.
\end{proof}

\begin{proposition}\label{Vorticity prop 2 - linear}
Let $\bs\theta$ be a solution of~\eqref{E:EP in linear} in the sense of Theorem \ref{T:LOCAL - linear}. Let $n\geq 21$. Then for any $s\geq s_0\geq 0$ and $n'\leq n$, we have
\begin{align*}
\sum_{b+\beta\leq n'}\|\grad\times\pr^b\pt^{\beta}\bs\theta(s)\|_{4+b}^2&\lesssim E_n(s_0)+S_{n'-1}(s)+\E_n(s)^2
+\int_{s_0}^s{1+(s'-s_0)^2\over\lambda(s')}\d s'\;\E_n(s)\\
\sum_{|\gamma|\leq n}\|\grad\times\partial^\gamma\bs\theta(s)\|_{4+2|\gamma|}^2&\lesssim E_n(s_0)+\E_n(s)^2+\int_{s_0}^s{1+(s'-s_0)^2\over\lambda(s')}\d s'\;\E_n(s).
\end{align*}
\end{proposition}
\begin{proof}
Take $\pr^b\pt^{\beta}$ ($b+\beta=n$) to the second equation in Lemma \ref{Vorticity lemma - linear} and then estimate in a similar way to Lemma \ref{Vorticity prop 1 - linear}. We estimate
\begin{align*}
\norm{\int_{s_0}^s\lambda(s')^{-1}\int_{s_0}^{s'}\star\;\d s''\d s'}_{4+b}^2
&\leq\int_{B_R}\brac{\int_{s_0}^s\lambda(s')^{-1}\d s'}\brac{\int_{s_0}^s\lambda(s')^{-1}\abs{\int_{s_0}^{s'}\star\;\d s''}^2\d s'}\bar w^{4+b}\d\mb x\\
&\lesssim\int_{s_0}^s\lambda(s')^{-1}\int_{B_R}\abs{\int_{s_0}^{s'}\star\;\d s''}^2\bar w^{4+b}\d\mb x\d s'\\
&\lesssim\int_{s_0}^s{1+(s'-s_0)^2\over\lambda(s')}\d s'\;\E_n(s)
\end{align*}
where we used the Cauchy–Schwarz inequality to get the first inequality, and the fact that $\int_0^\infty\lambda^{-1}\d s'<\infty$ for the second inequality (see \eqref{E:lambda in s - linear}). And
\begin{align*}
\sum_{b+\beta\leq n}\|[\grad\times,\pr^b\pt^{\beta}]\bs\theta(s)\|_{4+b}^2&\lesssim S_{n'-1}(s)\\
\sum_{b+\beta\leq n}\|\pr^b\pt^{\beta}((\A-I)\grad\times\bs\theta)(s)\|_{4+b}^2&\lesssim\E_n(s)^2.
\end{align*}
Then we get the first formula. Proof for the second formula is similar.
\end{proof}



\subsection{Energy estimates and proof of the main theorem}\label{S:EE2 - linear}


In this section we finally commute the momentum equation~\eqref{E:EP in linear} and then derive the 
high-order energy estimates. Then finally we will prove our main theorem using the energy estimates.

\begin{theorem}[Energy estimates]\label{Energy estimate - linear}
Let $n\geq 21$. Let $\bs\theta$ be a solution of~\eqref{E:EP in linear} in the sense of Theorem \ref{T:LOCAL - linear}, given on its maximal interval of existence. Then
\begin{align}
\E_n(s)\lesssim\E_n(s_0)+\E_n(s)^{3/2}+\E_n(s)\int_{s_0}^s\lambda^{-1/2}\d\tau.
\end{align}
for any $s\geq s'\geq 0$ whenever our a priori assumption \eqref{A priori assumption - linear} is satisfied. Here we recall Definition~\eqref{E:TOTALNORM - linear} of the total norm $\E_n$.
\end{theorem}


\begin{proof}
Since $\E_n=\S_n+\Q_n+\Z_n$, we need to prove the formula with LHS each of these three component terms.

We first deal with the $\S_n$ part. Let $|\beta|+b\leq n$. Apply $\pr^b\pt^\beta$ to the momentum equation~\eqref{E:EP in linear} to get
\begin{align*}
\lambda\partial_s^{2}\pr^b\pt^\beta\bs\theta+\lambda'\partial_s\pr^b\pt^\beta\bs\theta+\pr^b\pt^\beta(\delta\bs\theta+\mb P+\mb G)=0
\end{align*}
Taking the $\langle\cdot,\cdot\rangle_{3+b}$-inner with $\partial_s\pr^b\pt^\beta\bs\theta$ we get
\begin{align*}
0&={1\over 2}\lambda\partial_s\|\partial_s\pr^b\pt^\beta\bs\theta\|^2_{3+b}+\lambda'\|\partial_s\pr^b\pt^\beta\bs\theta\|_{3+b}^2
+\<\pr^b\pt^\beta(\delta\bs\theta+\mb P+\mb G),\partial_s\pr^b\pt^\beta\bs\theta\>_{3+b}\\
&={1\over 2}\partial_s\brac{\lambda\|\partial_s\pr^b\pt^\beta\bs\theta\|^2_{3+b}}+{1\over 2}\lambda'\|\partial_s\pr^b\pt^\beta\bs\theta\|_{3+b}^2
+\<\pr^b\pt^\beta(\delta\bs\theta+\mb P+\mb G),\partial_s\pr^b\pt^\beta\bs\theta\>_{3+b}
\end{align*}
Integrate in time we get
\begin{align*}
0&={1\over 2}\eva{\brac{\lambda\|\partial_s\pr^b\pt^\beta\bs\theta\|^2_{3+b}}}_{s_0}^s+{1\over 2}\int_{s_0}^s\lambda'\|\partial_s\pr^b\pt^\beta\bs\theta\|_{3+b}^2\d\tau
+\int_{s_0}^s\<\pr^b\pt^\beta(\delta\bs\theta+\mb P+\mb G),\partial_s\pr^b\pt^\beta\bs\theta\>_{3+b}\d\tau
\end{align*}
By Proposition \ref{P-reduction-1 - linear}, \ref{P-reduction-2 - linear}, \ref{G estimate - linear} and Lemma \ref{P-inner-product} we get
\begin{align*}
&{1\over 2}\bigg(\lambda\|\partial_s\pr^b\pt^\beta\bs\theta\|^2_{3+b}+\|\grad\pr^b\pt^\beta\bs\theta\|^2_{4+b}+{1\over 3}\|\div\pr^b\pt^\beta\bs\theta\|^2_{4+b}
-{1\over 2}\|\curl\pr^b\pt^\beta\bs\theta\|^2_{4+b}\bigg)\bigg|_{s_0}^s\\
&\quad+{1\over 2}\int_{s_0}^s\lambda'\|\partial_s\pr^b\pt^\beta\bs\theta\|_3^2\d\tau\\
&\lesssim\E_n(s)^{3/2}+\int_{s_0}^s\lambda^{-1/2}E_n\d\tau
\end{align*}
Using Proposition \ref{Vorticity prop 2 - linear} we get
\begin{align*}
&\lambda(s)\|\partial_s\pr^b\pt^\beta\bs\theta(s)\|^2_{3+b}+\|\grad\pr^b\pt^\beta\bs\theta(s)\|^2_{4+b}\\
&\lesssim E_n(s_0)+E_{n-1}(s)+\E_n(s)^2+\int_{s_0}^s{1+(s'-s_0)^2\over\lambda(s')}\d\tau\;\E_n(s)+\E_n(s)^{3/2}
+\int_{s_0}^s\lambda^{-1/2}E_n\d\tau\\
&\lesssim E_n(s_0)+E_{n-1}(s)+\E_n(s)^{3/2}+\E_n(s)\int_{s_0}^s\lambda^{-1/2}\d\tau
\end{align*}
where we used \eqref{E:lambda in s - linear} and \eqref{A priori assumption - linear}. Add to it
\begin{align*}
\|\pr^b\pt^\beta\bs\theta(s)\|^2_{3+b}&=\|\pr^b\pt^\beta\bs\theta(s_0)\|^2_{3+b}+2\int_{s_0}^s\<\pr^b\pt^\beta\bs\theta(s),\partial_s\pr^b\pt^\beta\bs\theta(s)\>_{3+b}\d\tau\\
&\lesssim E_n(s_0)+\int_{s_0}^s\lambda^{-1/2}E_n\d\tau,
\end{align*}
sum over $|\beta|+b\leq n'\leq n$ and we get
\begin{align*}
S_{n'}(s)\lesssim E_n(s_0)+S_{n'-1}(s)+\E_n(s)^{3/2}+\E_n(s)\int_{s_0}^s\lambda^{-1/2}\d\tau.
\end{align*}
Induct on $n'$ we get
\begin{align*}
S_n(s)\lesssim E_n(s_0)+\E_n(s)^{3/2}+\E_n(s)\int_{s_0}^s\lambda^{-1/2}\d\tau.
\end{align*}

To prove the $\Q_n$ part, we repeat the above with $\partial^\gamma$ in place of $\pr^b\pt^\beta$, and weight $3+2|\gamma|$ instead of weight $3+b$.

Finally the $\Z_n$ part is given by Proposition \ref{Vorticity prop 1 - linear} noting that
\begin{align*}
\lambda(s)^{-1}\E_n(s)^2+{(s-s_0)^2\over\lambda(s)}\E_n(s)\lesssim\E_n(s)^{3/2}+\E_n(s)\int_{s_0}^s\lambda^{-1/2}\d\tau.
\end{align*}
where we used \eqref{E:lambda in s - linear} and \eqref{A priori assumption - linear}.
\end{proof}


To proof our main theorem that the energy $E_n$ remains bounded, we will use the bootstrapping scheme in the following lemma and proposition.

\begin{lemma}\label{pre bootstrapping scheme - linear}
Suppose $E:[0,T]\to[0,\infty]$ is continuous and
\begin{align*}
E(t)\leq C_1E(0)+C_2E(t)^{3/2}\qquad\text{whenever}\qquad \sup_{\tau\in[0,t]}E(\tau)\leq C_3.
\end{align*}
where $C_1\geq 1$. Then $E\leq 2C_1E(0)$ whenever $E(0)\leq\min\{(2^5C_1C_2^2)^{-1},C_3/2C_1\}$.
\end{lemma}
\begin{proof}
Same as Lemma \ref{pre bootstrapping scheme}.
\end{proof}

\begin{proposition}\label{bootstrapping scheme - linear}
Suppose $E:[0,\infty)\to[0,\infty]$ are continuous and for all $s\geq s_0\geq 0$ we have
\begin{align*}
E(s)\leq C_1E(s_0)+C_2E(s)^{3/2}+C_3F(s_0,s)E(s)\quad\text{whenever}\quad\sup_{\tau\in[s_0,s]}E(\tau)\leq C_4
\end{align*}
where $F:\{(s_0,s)\in[0,\infty)\times[0,\infty):s\geq s_0\}\to[0,\infty)$ is a function such that
\begin{enumerate}
\item $\lim_{s_0\to\infty}\sup_{s\geq s_0}F(s_0,s)=0$;
\item $\lim_{\delta'\to 0}\sup_{|s-s_0|\leq\delta'}F(s_0,s)=0$.
\end{enumerate}
Then there exist $\epsilon^*>0$ such that $E\lesssim_{C_1}E(0)$ whenever $E(0)\leq\epsilon^*$.
\end{proposition}
\begin{proof}
Pick $s_\infty$ large enough so that
\[C_3\sup_{s\geq s_\infty}F(s_\infty,s)<{1\over 2}.\]
Then by Lemma \ref{pre bootstrapping scheme - linear} there exist $\epsilon_\infty>0$ such that $\sup_{s\in[s_\infty,\infty)}\leq 4C_1E(s_\infty)$ whenever $E(s_\infty)\leq\epsilon_\infty$.

Now pick $\delta'$ small enough so that
\[C_3\sup_{|s-s_0|\leq\delta'}F(s_0,s)<{1\over 2}.\]
Then by Lemma \ref{pre bootstrapping scheme - linear} there exist $\epsilon_0>0$ such that $\sup_{s\in[m\delta',(m+1)\delta']}\leq 4C_1E(m\delta')$ whenever $E(m\delta')\leq\epsilon_0$.

Let $\epsilon^*\leq\min\{\epsilon_0,\epsilon_\infty\}/(4C_1)^{\lceil s_\infty/\delta'\rceil}$. Then $E(s)\leq(4C_1)^{\lceil s_\infty/\delta'\rceil +1}E(0)$ for all $s\geq 0$ whenever $E(0)\leq\epsilon^*$.
\end{proof}



\begin{theorem}\label{Linear final theorem}
Let $n\geq 21$. Let $(\bs\theta, \partial_s\bs\theta)$ be a solution of~\eqref{E:EP in linear} in the sense of Theorem \ref{T:LOCAL - linear}. Then there exists $\epsilon^*>0$ such that if $E_n(0)\leq\epsilon^*$, then we have $\E_n\lesssim\E_n(0)$.
\end{theorem}


\begin{proof}
By the energy estimates in Theorem \ref{Energy estimate - linear} we have
\begin{align}
\E_n(s)\lesssim\E_n(s_0)+\E_n(s)^{3/2}+\E_n(s)\int_{s_0}^s\lambda^{-1/2}\d\tau.
\end{align}
Applying Proposition \ref{bootstrapping scheme - linear} above with $E=\E_n$ and $F(s_0,s)=\int_{s_0}^s\lambda^{-1/2}\d\tau$ (which satisfies the properties required for the proposition because of \eqref{E:lambda in s - linear}) we get the desired result.
\end{proof}


\appendix
\section{Appendix}\label{GW appendix}

\subsection{Differentiation and commutation properties}


Here we first collect some standard results on how derivatives interact with $\J$ and $\A$, which can be found for example in~\cite{JaMa2015}. After that we state various derivative commutators frequently used in the article.


\begin{lemma}\label{derivative-formula}
Recall notations defined in Definition \ref{Special notations}. We have
\begin{align*}
\A^i_j-I^i_j&=-\A^i_k\partial_j\theta^k\\
\partial_\bullet\J&=\J\A^l_m\partial_\bullet\partial_l\theta^m\\
\partial_\bullet\A^i_j&=-\A^i_m\A^l_j\partial_\bullet\partial_l\theta^m
\end{align*}
\end{lemma}


\begin{proof}
Since $\A=(\grad\bs\xi)^{-1}$, we have
\begin{align*}
I^i_j=\A^i_k\partial_j\xi^k=\A^i_k(I^k_j+\partial_j\theta^k).
\end{align*}
It can be proven that if $U:t\mapsto U(t)$ is a differentiable map of invertible square matrices, then
\begin{enumerate}
\item $\displaystyle{\d\det U\over\d t}=\det(U)\tr\brac{U^{-1}{\d U\over\d t}}$;
\item $\displaystyle{\d U^{-1}\over\d t}=-U^{-1}{\d A\over\d t}U^{-1}$.
\end{enumerate}
Using i. we get $\partial_{\bullet}J=JA^k_i\partial_{\bullet}\partial_k\eta^i$, and using ii. we get $\partial_\bullet A^i_j=-A^i_k(\partial_\bullet\partial_l\eta^k)A^l_j$. Converting to $\A$ and $\J$ by tracing the definition and keeping track of the factors of $\lambda$, we get the stated formulas.
\end{proof}


We commonly use various commutation properties between the Cartesian, radial, angular derivatives, and their Lagrangian counterparts.


\begin{lemma}[Commutation relations]\label{Commutation relations}
We have the following commutation relations
\begin{align*}
[\pr,\grad]&=-\grad\\
[\pr,\mb x]&=\mb x\\
[\pt_i,\partial_j]&=-\epsilon_{ijk}\partial_k\\
[\pt_i,x^j]&=-\epsilon_{ijk}x^k\\
[\pr,\K]&=2\K\\
[\pt_i,\K]&=0\\
[\partial_s,\cApar_j]&=-(\cApar_j\partial_s\theta^m)\cApar_m\\
[\grad,\cApar_j]&=-(\cApar_j\grad\theta^m)\cApar_m\\
[\pr,\cApar_j]&=-(\A^l_j\pr\partial_l\theta^m)\cApar_m-\cApar_j\\
[\pt_i,\cApar_j]&=-(\A^l_j\pt_i\partial_l\theta^m)\cApar_m-\epsilon_{ikl}\A^k_j\partial_l\\
[\pr,\pt_i]&=0\\
[\pt_i,\pt_{i'}]&=-\pt_{ii'}.
\end{align*}
\end{lemma}
\begin{proof}
We have
\begin{align*}
[\pr,\partial_j]&=x^i\partial_i\partial_j-\partial_jx^i\partial_i=-\partial_j\\
[\pr,x^j]&=x^i\partial_ix^j-x^jx^i\partial_i=x^j\\
[\pt_i,\partial_j]&=\epsilon_{ilk}x^l\partial_k\partial_j-\epsilon_{ilk}\partial_jx^l\partial_k=-\epsilon_{ijk}\partial_k\\
[\pt_i,x^j]&=\epsilon_{ilk}x^l\partial_kx^j-\epsilon_{ilk}x^jx^l\partial_k=\epsilon_{ilj}x^l\\
[\partial_s,\cApar_j]&=\partial_s\A^i_j\partial_i-\A^i_j\partial_i\partial_s
=-\A^i_m\A^l_j(\partial_l\partial_s\theta^m)\partial_i\\
[\partial_i,\cApar_j]&=\partial_i\A^k_j\partial_k-\A^k_j\partial_k\partial_i
=-\A^k_m\A^l_j(\partial_l\partial_i\theta^m)\partial_k\\
[\pr,\cApar_j]&=\pr(\A_j^k\partial_k)-\A_j^k\partial_k\pr=-(\A^k_m\A^l_j\pr\partial_l\theta^m)\partial_k+\A^k_j\pr\partial_k-\A^k_j\partial_k\pr\\
&=-(\A^k_m\A^l_j\pr\partial_l\theta^m)\partial_k-\A^k_j\partial_k\\
[\pt_i,\cApar_j]&=\pt_i(\A_j^k\partial_k)-\A_j^k\partial_k\pt_i=-(\A^k_m\A^l_j\pt_i\partial_l\theta^m)\partial_k+\A^k_j\pt_i\partial_k-\A^k_j\partial_k\pt_i\\
&=-(\A^l_j\pt_i\partial_l\theta^m)\cApar_m-\epsilon_{ikl}\A^k_j\partial_l\\
[\pr,\pt_i]&=\epsilon_{ijk}x^l\partial_l(x^j\partial_k)-\epsilon_{ijk}x^j\partial_k(x^l\partial_l)=\epsilon_{ijk}\brac{\delta^j_lx^l\partial_k+x^lx^j\partial_l\partial_k-\delta^l_kx^j\partial_l-x^jx^l\partial_k\partial_l}
\\
&=\epsilon_{ijk}\brac{x^j\partial_k-x^j\partial_k}
=0\\
[\pt_i,\pt_{i'}]&=\epsilon_{ijk}\epsilon_{i'j'k'}\brac{x^j\partial_k(x^{j'}\partial_{k'})-x^{j'}\partial_{k'}(x^j\partial_k)}
=\epsilon_{ijk}\epsilon_{i'j'k'}\brac{\delta_k^{j'}x^j\partial_{k'}-\delta_{k'}^jx^{j'}\partial_k}\\
&=\epsilon_{ijk}\epsilon_{i'kk'}x^j\partial_{k'}-\epsilon_{ik'k}\epsilon_{i'j'k'}x^{j'}\partial_k
=\epsilon_{kij}\epsilon_{kk'i'}x^j\partial_{k'}-\epsilon_{k'ki}\epsilon_{k'i'j'}x^{j'}\partial_k\\
&=(\delta_{ik'}\delta_{ji'}-\delta_{ii'}\delta_{jk'})x^j\partial_{k'}-(\delta_{ki'}\delta_{ij'}-\delta_{kj'}\delta_{ii'})x^{j'}\partial_k
=x^{i'}\partial_i-x^i\partial_{i'}=-\pt_{ii'}\\
\K(\mb x\cdot\grad g)(\mb y)&=-\int{\mb x\cdot\grad g(\mb x)\over|\mb y-\mb x|}\d\mb x=\int\brac{{g(\mb x)\grad\cdot\mb x\over|\mb y-\mb x|}+g(\mb x)\mb x\cdot\grad_{\mb x}{1\over|\mb y-\mb x|}}\d\mb x\\
&=\int\brac{{3g(\mb x)\over|\mb y-\mb x|}+g(\mb x)\mb x\cdot{\mb y-\mb x\over|\mb y-\mb x|^3}}\d\mb x
=\int\brac{{2g(\mb x)\over|\mb y-\mb x|}+g(\mb x)\mb y\cdot{\mb y-\mb x\over|\mb y-\mb x|^3}}\d\mb x\\
&=-2\K g+\mb y\cdot\grad\K g\\
\K(\pt_ig)(\mb y)&=-\int{\epsilon_{ijk}x^j\partial_kg(\mb x)\over|\mb y-\mb x|}\d\mb x
=\int g(\mb x)\epsilon_{ijk}x^j{y^k-x^k\over|\mb y-\mb x|^3}\d\mb x
\\
&=\int g(\mb x)\epsilon_{ijk}x^j{y^k\over|\mb y-\mb x|^3}\d\mb x=\int g(\mb x)\epsilon_{ijk}y^j{-x^k\over|\mb y-\mb x|^3}\d\mb x\\
&=\int g(\mb x)\epsilon_{ijk}y^j{y^k-x^k\over|\mb y-\mb x|^3}\d\mb x=(\pt_i\K g)(\mb y). 
\end{align*}
\end{proof}



\subsection{Spherical harmonics}\label{A:SPHERICALHARMONICS}


Spherical harmonics has a real as well as complex version. For the definition and basic properties of the complex version, see \cite{Jackson}. The relation between complex spherical harmonics $Y^m_l:S^2\to\C$ and real spherical harmonics $Y_{lm}:S^2\to\R$ are
\begin{align*}
Y^m_l=\begin{cases}{1\over\sqrt 2}(Y_{l,-m}-iY_{lm})& m<0\\
Y_{l0}& m=0\\
{(-1)^m\over\sqrt 2}(Y_{lm}+iY_{l,-m})& m>0\end{cases}
\end{align*}
We also have the relation
$
(Y^m_l)^*=(-1)^mY^{-m}_l.
$
The zeroth and first order real spherical harmonics are given by
\begin{alignat*}{3}
Y_{0,0}(\mb x)&={1\over\sqrt{4\pi}},& \qquad\qquad \ 
Y_{1,-1}(\mb x)&=\sqrt{3\over 4\pi}{x^2\over|\mb x|},\\
Y_{1,0}(\mb x)&=\sqrt{3\over 4\pi}{x^3\over|\mb x|},& \qquad\qquad \
Y_{1,1}(\mb x)&=\sqrt{3\over 4\pi}{x^1\over|\mb x|}.
\end{alignat*}
The spherical harmonics satisfy the following orthonormal conditions
\begin{align*}
\int_{S^2}Y_{lm}Y_{l'm'}\d S=\delta_{ll'}\delta_{mm'}=\int_{S^2}Y_{m}^m(Y_{l'}^{m'})^*\d S
\end{align*}
and they form a basis for $L^2(S^2)$ \cite{Atkinson Han} so that, in particular, any function $g\in L^2(S^2)$ has a spherical harmonics expansion
\[g=\sum_{l=0}^\infty\sum_{m=-l}^l g_{lm}Y_{lm},\qquad\qquad g_{lm}\in\R\]
that converge in $L^2(S^2)$. More generally, a function $g\in L^2(B_R)$ has a spherical harmonics expansion in $L^2(B_R)$,
\begin{align}
g=\sum_{l=0}^\infty\sum_{m=-l}^l g_{lm}(r)Y_{lm},\qquad\qquad g_{lm}:[0,R]\to\R.\label{spherical harmonics expansion}
\end{align}
Indeed, since $L^2(B_R)=L^2([0,R];L^2(S^2),r^2)=L^2([0,R];L^2(\partial B_r))$, or in other words
\[\int_{B_R}|\ph|\ \d\mb x=\int_0^R\int_{\partial B_r}|\ph|\ \d S\d r,\]
$g|_{\partial B_r}$ must be in $L^2(\partial B_r)$ for almost every $r\in[0,R]$. So a spherical harmonics expansion exist for almost every $r$. Now
\begin{align*}
\norm{g-\sum_{l=0}^N\sum_{m=-l}^l g_{lm}Y_{lm}}_{L^2(B_R)}^2&=\int_0^R\norm{g-\sum_{l=0}^N\sum_{m=-l}^l g_{lm}Y_{lm}}_{L^2(\partial B_r)}^2\d r\\
&\to 0\qquad\text{as}\qquad N\to\infty
\end{align*}
by dominated convergence theorem (where the dominating function is $4\|g\|_{L^2(\partial B_r)}^2$). Hence \eqref{spherical harmonics expansion} converge in $L^2(B_R)$. Similarly, functions in $L^2(B_R,\bar w^{-2})$ and $L^2(\R^3)$ have a spherical harmonics expansion.

The following lemma allows us to expand gravitational potentials in spherical harmonics.

\begin{lemma}\label{potential decomposition}
For $\mb x,\mb y\in\R^3$ we have
\begin{align*}
{1\over|\mb x-\mb y|}&=4\pi\sum_{l=0}^\infty\sum_{m=-l}^l{1\over 2l+1}{\min\{|\mb x|,|\mb y|\}^l\over\max\{|\mb x|,|\mb y|\}^{l+1}}Y_{lm}(\mb y)Y_{lm}(\mb x)
\end{align*}
and this expression converge uniformly for $(\mb x,\mb y)$ in any compact set in $\{(\mb r,\mb r')\in\R^6:|\mb r|\not=|\mb r'|\}$.
\end{lemma}
\begin{proof}
From~\cite{Jackson} we have
\begin{align*}
{1\over|\mb x-\mb y|}&=4\pi\sum_{l=0}^\infty\sum_{m=-l}^l{1\over 2l+1}{\min\{|\mb x|,|\mb y|\}^l\over\max\{|\mb x|,|\mb y|\}^{l+1}}Y^m_{l}(\mb y)^*Y^m_{l}(\mb x)
\end{align*}
One derivation of this formula is as follows. Assume $r'=|\mb r'|<|\mb r|=r$, otherwise swap $\mb r'$ and $\mb r$. By the law of cosines,
\[{\displaystyle {\frac {1}{|\mathbf {r} -\mathbf {r} ' |}}={\frac {1}{\sqrt {r^{2}+(r')^{2}-2rr'\cos \gamma }}}={\frac {1}{r{\sqrt {1+h^{2}-2h\cos \gamma }}}}\quad {\hbox{with}}\quad h:={\frac {r'}{r}}.}\]
We find here the generating function of the Legendre polynomials $P_{\ell }(\cos \gamma )$:
\begin{align}
{\displaystyle {\frac {1}{\sqrt {1+h^{2}-2h\cos \gamma }}}=\sum _{\ell =0}^{\infty }h^{\ell }P_{\ell }(\cos \gamma ).}\label{Generating function of the Legendre polynomials}
\end{align}
Use of the spherical harmonic addition theorem
\[{\displaystyle P_{\ell }(\cos \gamma )={\frac {4\pi }{2\ell +1}}\sum _{m=-\ell }^{\ell }(-1)^{m}Y_{\ell }^{-m}(\theta ,\varphi )Y_{\ell }^{m}(\theta ',\varphi ')}\]
gives our first formula.
Since $|P_\ell(\cos\gamma)|\leq 1$ for all $\ell$, the power series in (\ref{Generating function of the Legendre polynomials}) has radius of convergence 1, and uniform convergence for any compact set in $B_1$. By Identity theorem for analytic functions, the equality of \eqref{Generating function of the Legendre polynomials} holds for $h<1$. We thus conclude that the expansion for $|\mb r-\mb r'|^{-1}$ converge uniformly on any compact set in $\{(\mb r,\mb r')\in\R^6:|\mb r|\not=|\mb r'|\}$.
Moreover, in real spherical harmonics,
\begin{align*}
{1\over|\mb x-\mb y|}
&=4\pi\sum_{l=0}^\infty\sum_{m=-l}^l{1\over 2l+1}{\min\{|\mb x|,|\mb y|\}^l\over\max\{|\mb x|,|\mb y|\}^{l+1}}Y^m_{l}(\mb y)^*Y^m_{l}(\mb x)\\
&=4\pi\sum_{l=0}^\infty\sum_{m=-l}^l{(-1)^m\over 2l+1}{\min\{|\mb x|,|\mb y|\}^l\over\max\{|\mb x|,|\mb y|\}^{l+1}}Y^{-m}_{l}(\mb y)Y^m_{l}(\mb x)\\
&=4\pi\sum_{l=0}^\infty\sum_{m=1}^l{1\over 2}{1\over 2l+1}{\min\{|\mb x|,|\mb y|\}^l\over\max\{|\mb x|,|\mb y|\}^{l+1}}(Y_{l,m}(\mb y)-iY_{l,-m}(\mb y))(Y_{lm}(\mb x)+iY_{l,-m}(\mb x))\\
&\quad+4\pi\sum_{l=0}^\infty{1\over 2l+1}{\min\{|\mb x|,|\mb y|\}^l\over\max\{|\mb x|,|\mb y|\}^{l+1}}Y_{l0}(\mb y)Y_{l0}(\mb x)\\
&\quad+4\pi\sum_{l=0}^\infty\sum_{m=-l}^{-1}{1\over 2}{1\over 2l+1}{\min\{|\mb x|,|\mb y|\}^l\over\max\{|\mb x|,|\mb y|\}^{l+1}}(Y_{l,-m}(\mb y)+iY_{lm}(\mb y))(Y_{l,-m}(\mb x)-iY_{lm}(\mb x))\\
&=4\pi\sum_{l=0}^\infty\sum_{m=-l}^l{1\over 2l+1}{\min\{|\mb x|,|\mb y|\}^l\over\max\{|\mb x|,|\mb y|\}^{l+1}}Y_{lm}(\mb y)Y_{lm}(\mb x)
\end{align*}
\end{proof}

There also exist a vectorial version of spherical harmonics which allow the expansion of $L^2$ vector fields, details can be found in \cite{Barrera Estevez Giraldo, Freeden Schreiner}.

\subsection{Hardy-Poincar\'e inequality and embeddings}

Here we will state and prove a version of the Hardy-Poincar\'e inequality and the embedding theorems, which are important for our analysis.


Denote $B_R=B_R(\R^m)$ the ball of radius $R$ in $\R^m$, and $L^2(B_R,w)$ the $L^2$ space on $B_R$ weighted by $w$. Denote $d_{\partial B_R}$ the distance function to $\partial B_R$.

\begin{theorem}[Hardy-Poincar\'e inequality]\label{Hardy-Poincare inequality}
Let $R>0$ and $k\geq 0$. For any $\theta\in H^1_\loc(B_R(\R^m))$ we have
\begin{align}
\|\theta-\theta_{B_{(m-1)R/m}}\|_{L^2(B_R,d_{\partial B_R}^k)}\lesssim\|\grad\theta\|_{L^2(B_R,d_{\partial B_R}^{k+2})},
\end{align}
where $\theta_{B_r}$ denotes the average of $\theta$ on $B_r$.
\end{theorem}
\begin{proof}
Using only the standard Poincar\'e inequality and elementary methods, we will provide here a proof for the case $k>0$. However, in this paper we also used the case $k=0$. A slightly different version of the case $k=0$ was first proven in \cite{Boas Straube} which makes use of the Hardy inequality. A proof of the $k=0$ case for the version here (and a more general form) can be found in \cite{Drelichman Duran} by Drelichman and Dur\'an, the proof of which (with very slight modification) will work for the $k=0$ as well as the $k>0$ case.

It suffice to show this for $\theta\in C^1(\bar B_R)$ since $C^1(\bar B_R)$ is dense in the type of Sobolev spaces we are considering \cite{Kufner}. In particular we can assume $\lim_{|\mb x|\to R}\theta(\mb x)d_{\partial B_R}(\mb x)^{(k+1)/2}=0$ for integration by parts later.

Let $w$ be a smooth function on $B_R$ such that $d_{\partial B_R}\lesssim w\lesssim d_{\partial B_R}$ and $w=d_{\partial B_R}$ on $B_{R}\setminus B_{(m-1)R/m}$. Let $G=\theta w^{k/2}$. Then
\begin{align*}
\int_{B_R}|\grad\theta|^2w^{k+2}\d\mb x
&=\int_{B_R}|\grad(w^{-k/2}G)|^2w^{k+2}\d\mb x
=\int_{B_R}|{w^{-k/2}\grad G-{k\over 2}Gw^{-k/2-1}\grad w }|^2w^{k+2}\d\mb x\\
&=\int_{B_R}|{w\grad G-{k\over 2}G\grad w }|^2\d\mb x
=\int_{B_R}\brac{w^2|\grad G|^2-k(G\grad G)(w\grad w)+{k^2\over 4}|\grad w|^2G^2}\d\mb x\\
&=\int_{B_R}\brac{w^2|\grad G|^2-{k\over 4}(\grad w^2)(\grad G^2)+{k^2\over 4}|\grad w|^2G^2}\d\mb x\\
&=\int_{B_R}\brac{w^2|\grad G|^2+{k\over 4}\brac{\lpc w^2+k|\grad w|^2}G^2}\d\mb x
\end{align*}
On $B_{R}\setminus B_{(m-1)R/m}$ we have
\begin{align*}
k(\lpc w^2+k|\grad w|^2)&={k\over r^{n-1}}{\d\over\d r}\brac{r^{n-1}{\d(R-r)^2\over\d r}}+k^2\brac{{\d(R-r)\over\d r}}^2\\
&={k\over r^{n-1}}{\d\over\d r}\brac{r^{n-1}(-2R+2r)}+k^2
=k\brac{-2(n-1){R\over r}+2n}+k^2\\
&\geq\begin{cases}k^2&k\geq 0\\k^2+2k&k\leq 0\end{cases}
\end{align*}
which is strictly positive when $k\in\R\setminus[-2,0]$. So we have
\begin{align*}
{k^2+\min\{0,2k\}\over 4}\int_{B_R}\theta^2w^k\d\mb x
&={k^2+\min\{0,2k\}\over 4}\int_{B_R}G^2\d\mb x\\
&\leq\int_{B_R}|\grad\theta|^2w^{k+2}\d\mb x+{k\over 4}(k+\|\lpc w^2\|_\infty+k\|\grad w\|_\infty)\int_{B_{(m-1)R/m}}G^2\d\mb x\\
&\lesssim\int_{B_R}|\grad\theta|^2w^{k+2}\d\mb x+\int_{B_{(m-1)R/m}}\theta^2\d\mb x\\
&\lesssim\int_{B_R}|\grad\theta|^2w^{k+2}\d\mb x+\brac{\int_{B_{(m-1)R/m}}\theta\d\mb x}^2+\int_{B_{(m-1)R/m}}|\grad\theta|^2\d\mb x\\
&\lesssim\int_{B_R}|\grad\theta|^2w^{k+2}\d\mb x+\brac{\int_{B_{(m-1)R/m}}\theta\d\mb x}^2
\end{align*}
where we used the Poincar\'e inequality on $B_{(m-1)R/m}$. Replacing $\theta$ with $\theta-\theta_{B_{(m-1)R/m}}$ (valid when $k>-1$) we see that
\begin{align*}
\int_{B_R}(\theta-\theta_{B_{(m-1)R/m}})^2w^k\d\mb x\lesssim\int_{B_R}|\grad\theta|^2w^{k+2}\d\mb x. 
\end{align*}
\end{proof}



An immediate corollary of the Hardy-Poincar\'e inequality is the following.


\begin{corollary}\label{weight upgrade}
Let $k,l\geq 0$. We have
\begin{align}
\|\theta\|_k\lesssim\|\theta\|_l+\|\grad\theta\|_{k+2}
\end{align}
\end{corollary}


\begin{proof}
We have 
\begin{align*}
\|\theta\|_k&\leq\|\theta-\theta_{B_{(m-1)R/m}}\|_k+\|\theta_{B_{(m-1)R/m}}\|_k\lesssim\theta_{B_{(m-1)R/m}}\|1\|_k+\|\grad\theta\|_{k+2}\\
&\leq\|\theta\|_{L^2(B_{(m-1)R/m})}\|1\|_{L^2(B_{(m-1)R/m})}\|1\|_k+\|\grad\theta\|_{k+2}
\lesssim\|\theta\|_l+\|\grad\theta\|_{k+2}. 
\end{align*}
\end{proof}


Using this corollary we will next derive the embedding theorems. We will show that terms with less than $n/2$ the derivatives can be estimated in the $L^\infty$ norm by $E_n$ or $E_n+Z_n^2$, which is relevant for the energy estimates. For this we will need the following lemmas.


\begin{lemma}\label{weight upgrade for energy}
We have
\begin{align*}
\sum_{c=0}^n\|\grad^{n-c}\pr^b\pt^{\beta}\bs\theta\|_{\max\{3+b+n-2c,0\}}^2&\lesssim\sum_{c=0}^n\|\grad^c\pr^b\pt^{\beta}\bs\theta\|_{3+b+c}^2\\
\|\bar w^{\lfloor b/2\rfloor}\pr^b\pt^{\beta}\bs\theta(s)\|_{H^n}^2&\lesssim\sum_{c=0}^{4+2n}\|\grad^c\pr^b\pt^{\beta}\bs\theta\|_{3+b+c}^2\\
\|\bar w^{\lfloor b/2\rfloor}\grad\pr^b\pt^{\beta}\bs\theta(s)\|_{H^n}^2&\lesssim\sum_{c=1}^{6+2n}\|\grad^c\pr^b\pt^{\beta}\bs\theta\|_{3+b+c}^2
\end{align*}
\end{lemma}


\begin{proof}
The first formula follows from repeated application of the above Corollary \ref{weight upgrade}. The formulas after follows from the first with $n$ replaced by $4+2n$ and $6+2n$ respectively. 
\end{proof}


\begin{lemma}\label{Embedding base}
We have
\begin{align*}
\|\bar w^{\lfloor b/2\rfloor}\pr^b\pt^{\beta}\bs\theta(s)\|_{L^\infty}^2&\lesssim\sum_{c=0}^{8}\|\grad^c\pr^b\pt^{\beta}\bs\theta\|_{3+b+c}^2\\
&\lesssim\sum_{b'+|\beta'|\leq 8+b+|\beta|}\|\pr^{b'}\pt^{\beta'}\bs\theta\|_{3+b'}^2+\sum_{c\leq 8+b+|\beta|}\|\grad^c\bs\theta\|_{3+2c}^2\\
\|\bar w^{\lfloor b/2\rfloor}\grad\pr^b\pt^{\beta}\bs\theta(s)\|_{L^\infty}^2&\lesssim\sum_{c=1}^{10}\|\grad^c\pr^b\pt^{\beta}\bs\theta\|_{3+b+c}^2\\
&\lesssim\sum_{b'+|\beta'|\leq 10+b+|\beta|}\|\pr^{b'}\pt^{\beta'}\bs\theta\|_{3+b'}^2+\sum_{c\leq 10+b+|\beta|}\|\grad^c\bs\theta\|_{3+2c}^2
\end{align*}
\end{lemma}


\begin{proof}
This follows from Lemma~\ref{weight upgrade for energy} and the embedding $H^2\hookrightarrow L^\infty$. 
\end{proof}


\subsubsection{Embedding theorems for self-similarly expanding GW stars}

Using Lemmas \ref{weight upgrade for energy} and \ref{Embedding base} we can derive the Embedding theorems for self-similarly expanding GW stars, relevant for Section \ref{self-similar GW}.

\begin{theorem}[Near boundary embedding theorem]\label{Near boundary embedding theorem}
We have
\begin{align*}
\sum_{a+|\beta|+b\leq n}\|\bar w^{\lfloor b/2\rfloor}\partial_s^a\pr^b\pt^{\beta}\bs\theta(s)\|_{L^\infty}^2
&\lesssim\sum_{a+|\beta|+b\leq 8+n}\|\partial_s^{a}\pr^{b}\pt^{\beta}\bs\theta\|_{3+b}^2+\sum_{a+c\leq 8+n}\|\partial_s^a\grad^c\bs\theta\|_{3+2c}^2\\
&\lesssim E_{n+8}+Z_{n+8}^2\\
\sum_{\substack{a+|\beta|+b\leq n\\a>0}}\|\bar w^{\lfloor b/2\rfloor}\partial_s^a\pr^b\pt^{\beta}\bs\theta(s)\|_{L^\infty}^2
&\lesssim\sum_{\substack{a+|\beta|+b\leq 8+n\\a>0}}\|\partial_s^{a}\pr^{b}\pt^{\beta}\bs\theta\|_{3+b}^2+\sum_{\substack{a+c\leq 8+n\\a>0}}\|\partial_s^a\grad^c\bs\theta\|_{3+2c}^2\\
&\lesssim E_{n+8}\\
\sum_{a+|\beta|+b\leq n}\|\bar w^{\lfloor b/2\rfloor}\grad\partial_s^a\pr^b\pt^{\beta}\bs\theta(s)\|_{L^\infty}^2
&\lesssim\sum_{a+|\beta|+b\leq 10+n}\|\partial_s^{a}\pr^{b}\pt^{\beta}\bs\theta\|_{3+b}^2+\sum_{a+c\leq 10+n}\|\partial_s^a\grad^c\bs\theta\|_{3+2c}^2\\
&\lesssim E_{n+10}+Z_{n+10}^2\\
\sum_{\substack{a+|\beta|+b\leq n\\a>0}}\|\bar w^{\lfloor b/2\rfloor}\grad\partial_s^a\pr^b\pt^{\beta}\bs\theta(s)\|_{L^\infty}^2
&\lesssim\sum_{\substack{a+|\beta|+b\leq 10+n\\a>0}}\|\partial_s^{a}\pr^{b}\pt^{\beta}\bs\theta\|_{3+b}^2+\sum_{\substack{a+c\leq 10+n\\a>0}}\|\partial_s^a\grad^c\bs\theta\|_{3+2c}^2\\
&\lesssim E_{n+10}
\end{align*}
\end{theorem}


\begin{proof}
The proof is a direct consequence of Lemma~\ref{Embedding base}. 
\end{proof}


\begin{theorem}[Near origin embedding theorem]\label{Near origin embedding theorem}
We have
\begin{align*}
\sum_{a+c\leq n+1}\|\bar w^c\partial_s^a\grad^{c}\bs\theta(s)\|_{L^\infty}^2&\lesssim E_{n+10}+Z_{n+10}^2
\end{align*}
\end{theorem}
\begin{proof}
Similar to above \ref{Near boundary embedding theorem}. 
\end{proof}

\subsubsection{Embedding theorems for linearly expanding GW stars}

Using Lemmas \ref{weight upgrade for energy} and \ref{Embedding base} we can derive the Embedding theorems for linearly expanding GW stars, relevant for Section \ref{linear GW}.

\begin{theorem}[Near boundary embedding theorem]\label{Near boundary embedding theorem - linear}
We have
\begin{align*}
\sum_{|\beta|+b\leq n}\|\bar w^{\lfloor b/2\rfloor}\pr^b\pt^{\beta}\bs\theta(s)\|_{L^\infty}^2&\lesssim\sum_{|\beta|+b\leq 8+n}\|\pr^{b}\pt^{\beta}\bs\theta\|_{3+b}^2+\sum_{c\leq 8+n}\|\grad^c\bs\theta\|_{3+2c}^2\\
&\lesssim E_{n+8}\\
\sum_{|\beta|+b\leq n}\|\bar w^{\lfloor b/2\rfloor}\partial_s\pr^b\pt^{\beta}\bs\theta(s)\|_{L^\infty}^2&\lesssim\sum_{|\beta|+b\leq 8+n}\|\partial_s\pr^{b}\pt^{\beta}\bs\theta\|_{3+b}^2+\sum_{c\leq 8+n}\|\partial_s\grad^c\bs\theta\|_{3+2c}^2\\
&\lesssim\lambda^{-1}E_{n+8}\\
\sum_{|\beta|+b\leq n}\|\bar w^{\lfloor b/2\rfloor}\grad\pr^b\pt^{\beta}\bs\theta(s)\|_{L^\infty}^2&\lesssim\sum_{|\beta|+b\leq 10+n}\|\pr^{b}\pt^{\beta}\bs\theta\|_{3+b}^2+\sum_{c\leq 10+n}\|\grad^c\bs\theta\|_{3+2c}^2\\
&\lesssim E_{n+10}
\end{align*}
\end{theorem}


\begin{proof}
The proof is a direct consequence of Lemma~\ref{Embedding base}. 
\end{proof}


\begin{theorem}[Near origin embedding theorem]\label{Near origin embedding theorem - linear}
We have
\begin{align*}
\sum_{c\leq n+1}\|\bar w^c\grad^{c}\bs\theta(s)\|_{L^\infty}^2&\lesssim E_{n+10}\\
\sum_{c\leq n}\|\bar w^c\partial_s\grad^{c}\bs\theta(s)\|_{L^\infty}^2&\lesssim\lambda^{-1}E_{n+10}
\end{align*}
\end{theorem}
\begin{proof}
Similar to above \ref{Near boundary embedding theorem - linear}. 
\end{proof}

\end{document}